%% file: Manuscript.tex
\documentclass[10pt, reqno]{amsart}
\usepackage{amssymb}
\usepackage{amsthm}
\usepackage{nccmath}
\usepackage{hyperref}
\usepackage{array}
\usepackage{xparse}
\usepackage{enumitem}
\usepackage{tikz}
\usepackage{mathtools}
\usepackage[font=scriptsize]{caption}

\makeatletter
\def\black@#1{%
    \noalign{%
        \ifdim#1>\displaywidth
            \dimen@\prevdepth
            \nointerlineskip
            \vskip-\ht\strutbox@
            \vskip-\dp\strutbox@
            \vbox{\noindent\hbox to \displaywidth{\hbox to#1{\strut@\hfill}}}%
            \prevdepth\dimen@
        \fi
    }%
}
\makeatother

\makeatletter
\renewcommand{\tocsection}[3]{%
  \indentlabel{\@ifnotempty{#2}{\bfseries\ignorespaces#1 #2\quad}}\bfseries#3}
\renewcommand{\tocsubsection}[3]{%
  \indentlabel{\@ifnotempty{#2}{\ignorespaces#1 #2\quad}}#3}

\newcommand\@dotsep{4.5}
\def\@tocline#1#2#3#4#5#6#7{\relax
  \ifnum #1>\c@tocdepth 
  \else
    \par \addpenalty\@secpenalty\addvspace{#2}%
    \begingroup \hyphenpenalty\@M
    \@ifempty{#4}{%
      \@tempdima\csname r@tocindent\number#1\endcsname\relax
    }{%
      \@tempdima#4\relax
    }%
    \parindent\z@ \leftskip#3\relax \advance\leftskip\@tempdima\relax
    \rightskip\@pnumwidth plus1em \parfillskip-\@pnumwidth
    #5\leavevmode\hskip-\@tempdima{#6}\nobreak
    \leaders\hbox{$\m@th\mkern \@dotsep mu\hbox{.}\mkern \@dotsep mu$}\hfill
    \nobreak
    \hbox to\@pnumwidth{\@tocpagenum{\ifnum#1=1\bfseries\fi#7}}\par
    \nobreak
    \endgroup
  \fi}
\AtBeginDocument{%
\expandafter\renewcommand\csname r@tocindent0\endcsname{0pt}
}
\def\l@subsection{\@tocline{2}{0pt}{2.5pc}{5pc}{}}
\makeatother

\makeatletter
 \def\@testdef #1#2#3{%
   \def\reserved@a{#3}\expandafter \ifx \csname #1@#2\endcsname
  \reserved@a  \else
 \typeout{^^Jlabel #2 changed:^^J%
 \meaning\reserved@a^^J%
 \expandafter\meaning\csname #1@#2\endcsname^^J}%
 \@tempswatrue \fi}
\makeatother

\calclayout

\input{definitions}

\begin{document}

\title[water waves]{Angled crested like water waves with surface tension: Wellposedness of the problem }
\author{Siddhant Agrawal}
\address{Department of Mathematics, University of Massachusetts Amherst, MA 01003}
\email{agrawal@math.umass.edu}
\date{\today}

\begin{abstract}
We consider the capillary-gravity water wave equation in two dimensions. We assume that the fluid is inviscid, incompressible, irrotational and the air density is zero. We construct an energy functional and prove a local wellposedness result without assuming the Taylor sign condition. When the surface tension $\sigma$ is zero, the energy reduces to a lower order version of the energy obtained by Kinsey and Wu \cite{KiWu18} and allows angled crest interfaces.  For positive surface tension, the energy does not allow angled crest interfaces but admits initial data with large curvature of the order of $\sigma^{-\onebythree + \ep} $ for any $\ep >0$.
\end{abstract}

\maketitle
\tableofcontents

\section{Introduction}

We are concerned with the motion of a fluid in dimension two with a free boundary. In this work we will identify 2D vectors with complex numbers. The fluid region $\Omega(t)$ and the air is separated by an interface $\Sigma(t)$, with the fluid being below the air region and $\Sigma(t)$ being a one dimensional curve. We assume that the interface tends point-wise to the real line at infinity but we do not assume that the interface is a graph. The air and the fluid are assumed to have constant densities of 0 and 1 respectively. The fluid is also assumed to be inviscid, incompressible, irrotational and we assume that the bottom is at infinite depth. The gravitational field is assumed to be a constant vector $-i$ pointing in the downward direction. The motion of the fluid is then governed by the Euler equation
\begin{equation}\label{euler-capillary}
\begin{aligned}
& \mathbf{v_t + (v.\nabla)v} = -i -\nabla P  \qq\text{ on } \Omega (t) 	\\
& \tx{div } \mathbf{ v} = 0, \quad \tx{curl } \mathbf{ v }=0 \qq\text{ on } \Omega(t) 
\end{aligned}
\end{equation}
Along with the boundary conditions
\begin{equation}\label{boundary}
\begin{aligned}
& P = -\sigma \partial_{s}\thvar \qq\text{ on } \Sigma (t) \\
& (1,\mathbf{v}) \tx{ is tangent to the free surface } (t, \Sigma(t)) \\
& \mathbf{v} \to 0, \grad P \to -i \qq\text{ as } |(x,y)| \to \infty
\end{aligned}
\end{equation}
Here $\thvar = $ angle the interface makes with the $x -axis$, $\ps = $ arc length derivative, $\sigma = $ coefficient of surface tension $\geq0$. 

The earliest results on local well-posedness for the Cauchy problem are for small data in 2D and were obtained by Nalimov \cite{Na74}, Yoshihara \cite{Yo82,Yo83} and Craig \cite{Cr85}. In the case of zero surface tension, Wu \cite{Wu97,Wu99} obtained the proof of local well-posedness for arbitrary data in Sobolev spaces. See also the works in  \cite{ChLi00, Li05, La05, ZhZh08, CaCoFeGaGo13, AlBuZu14, HuIfTa16, AlBuZu18, HaIfTa17, Po19, Ai17, Ai20}. 

In the case of non-zero surface tension, the local well-posedness of the equation in Sobolev spaces was established by Beyer and Gunther in \cite{BeGu98}. See also the works in \cite{Ig01, Am03, Sc05, CoSh07, ShZe11, AlBuZu11, CaCoFeGaGo12, Ng17}. The zero surface tension limit of the water wave equations in Sobolev spaces was proved by Ambrose and Masmoudi \cite{AmMa05, AmMa09}. See also the works in \cite{OgTa02, ShZe08, MiZh09,CaCoFeGaGo12, ShSh19}.

An important quantity related to the well-posedness of the problem in the zero surface tension case is the Taylor sign condition. This says that there should exist a constant $c>0$ such that 
\begin{align*}
-\frac{\partial P}{\partial n} \geq c >0 \quad \tx{ on } \Sigma(t)
\end{align*}
In \cite{Wu97} Wu proved that this condition  is satisfied for the infinite bottom case if the interface is $C^{1,\alpha}$ for $\alpha>0$. This was later shown to be true for flat bottoms and with perturbations  to flat bottom by Lannes \cite{La05}.  In the case of non-zero surface tension, the Taylor sign condition is not needed for establishing the local wellposedness for a fixed $\sigma>0$, but now the time of existence $T$ depends on the value of $\sigma$ and $T\to0$ as $\sigma \to 0$. The Taylor sign condition again becomes important if one studies the zero surface tension limit, as in this case one needs uniform time of existence for $0\leq \sigma\leq \sigma_0$ for some fixed $\sigma_0>0$. In particular observe that assuming the Taylor sign condition to prove the zero surface tension limit forces one to assume quantitative bounds on the $C^{1,\al}$ norm of the initial interface. In fact in the case of non-zero surface tension, all results mentioned above obtain a time of existence $T \lesssim \norm[\infty]{\kap}^{-1}$ where $\kap$ is the initial curvature of the interface. 

In the zero surface tension case, for non $C^{1}$ curves the Taylor sign condition is only satisfied in a weak sense with $-\frac{\partial P}{\partial n} \geq 0$ \cite{Wu97, KiWu18} and this makes the quasilinear equation degenerate. Because of this obtaining a local well-posedness result in this setting becomes highly non trivial as standard energy estimates in Sobolev spaces do not work.  Kinsey and Wu \cite{KiWu18} managed to overcome this difficulty by using a weighted Sobolev norm with the weight depending nonlinearly on the interface and proved a priori estimates for interfaces which can have angled crests. Building upon this work Wu \cite{Wu19} proved a local well-posedness result that allows for angled crested interfaces as initial data. Later on in \cite{Ag20} we studied the evolution of the singularities of these waves. 

In this paper we extend the work of Kinsey and Wu \cite{KiWu18} for $\sigma=0$ to the case of $\sigma\geq 0$. We construct an energy functional $\Ecalsigma(t)$ and prove an a priori estimate \footnote{\thmref{thm:aprioriEsigma} actually uses $\Esigma(t)$ instead of $\Ecalsigma(t)$ however both are equivalent to each other by \propref{prop:equivEsigma}} for it in \thmref{thm:aprioriEsigma} which works for all $\sigma \geq 0$. Using this we prove a local well-posedness result for $\sigma>0$ in \thmref{thm:existence}. The energy $\Ecalsigma(t)$ has several interesting properties:
\begin{enumerate}
\item For $\sigma=0$ the energy $\Ecalsigma(t)$ reduces to a lower order version of the energy of Kinsey and Wu \cite{KiWu18}. In particular it allows singular interfaces such as interfaces with angled crests and cusps. 
\item For $\sigma>0$ the energy $\Ecalsigma(t)$ does not allow any singularities in the interface. In particular it does not allow angled crested interfaces. 
\item For $\sigma>0$ even though the energy $\Ecalsigma(t)$ does not allow singularities in the interface, it does allow interfaces with large curvature. It allows the $\Linfty$ norm of the curvature of the initial interface to grow like $\sigma^{-\onebythree + \ep}$ for any $\ep>0$. In particular for $\sigma$ small and for interfaces close to being angled crested, we obtain a time of existence much larger than all previous results. See \corref{cor:example} for more details.
\item We do not assume the Taylor sign condition in proving the a priori estimate  \thmref{thm:aprioriEsigma} or the local wellposedness result \thmref{thm:existence} and the energy $\Ecalsigma(t)$ is an increasing function of $\sigma$. For initial data in appropriate Sobolev spaces we obtain uniform time of existence of solutions for $0\leq \sigma \leq \sigma_0$ for arbitrary $\sigma_0>0$  thereby recovering the uniform time of existence result of Ambrose and Masmoudi \cite{AmMa05} in this case. \footnote{ Ambrose and Masmoudi \cite{AmMa05} had the restriction of $\sigma_0$ being small which we do not have. See the discussion after \eqref{form:Taylor} and \secref{sec:discussion} for more details.}
\end{enumerate}

The growth rate of $\sigma^{-\onebythree}$ for the $\Linfty$ norm of the curvature can be explained by the following scaling argument. Let us ignore gravity and consider a solution $\Z(\al,t)$ to the capillary water wave equation with surface tension $\sigma$ \footnote{The role of gravity will be clarified in a future work.}. Then for any $\lamb>0$ and $s \in \Rsp$, the function $\Z_\lamb(\al,t) = \lamb^{-1}\Z(\lamb\al, \lamb^s t) $ is a solution to the capillary water wave equation with surface tension $\sigma_\lamb = \lamb^{2s-3}\sigma$. We are interested in the zero surface tension limit, so we want the solutions $\Z_\lamb(\cdot,t)$ to exist in the same time interval $[0,T]$ and hence should have the same time scales. Hence we put $s=0$ to get $\Z_\lamb(\al,t) = \lamb^{-1}\Z(\lamb\al,  t) $ and surface tension $\sigma_\lamb = \lamb^{-3}\sigma$. In this case, the curvature $\kap_\lamb (\al,t) = \lamb\kap(\lamb\al,t)$ which yields $\sigma_\lamb\kap_\lamb^3 (\al,t) = \sigma\kap^3(\lamb\al,t)$. Hence $\norm*[\infty]{\sigma^{\onebythree}\kap}$ is invariant under this scaling and so the curvature grows like $\sigma^{-\onebythree}$ as $\sigma \to 0$.

In a forthcoming paper we show that in an appropriate regime (see \remref{rmk:convergence}) the zero surface tension limit of the solutions obtained in \thmref{thm:existence} allows for angled crested interfaces (more generally non-$C^1$ interfaces) and satisfies the gravity water wave equation. 

We will follow the approach  of Kinsey and Wu \cite{KiWu18} and Wu \cite{Wu19}, and work with free surface equations in conformal coordinates, derive quasilinear equations and use weighted Sobolev norms. The presence of surface tension gives rise to several structural and analytical difficulties which are not present in \cite{KiWu18} and \cite{Wu19}. We explain the difficulties, new ideas and the main results in detail in \secref{sec:discussionandresults}. In addition to the work of Kinsey and Wu \cite{KiWu18} and Wu \cite{Wu19}, we also use ideas from the work of Ambrose and Masmoudi \cite{AmMa05} in choosing appropriate variables to work with. 

The paper is organized as follows: In \secref{sec:notation} we introduce the notation and prove some basic formulae including the formula for the Taylor sign condition. In \secref{sec:discussionandresults} we state our main results and discuss the main ideas and heuristics. In \secref{sec:quasilinear} we derive our quasilinear equations by taking derivatives to the Euler equation. In \secref{sec:aprioriEsigma} we prove the main a priori estimate for the energy $\Esigma(t)$. In \secref{sec:equivalence} we prove the equivalence of the energies $\Esigma(t)$ and $\Ecalsigma(t)$ and explain their relation to the Sobolev norm. In \secref{sec:existence} we prove an existence result in Sobolev space for $\sigma>0$ and also a blow up criterion. Finally in \secref{sec:proof} we complete the proof of our main results. The appendix \secref{sec:appendix} contains all the commonly used identities and estimates used throughout the paper.

\bigskip

\noindent \textbf{Acknowledgment}: This work was part of the author's Ph.D. thesis and he is very grateful to his advisor Prof. Sijue Wu for proposing the problem and for her guidance during this project. The author would also like to thank Prof. Jeffrey Rauch for many helpful discussions. The author was supported in part by NSF Grants DMS-1101434, DMS-1361791 through his advisor.
\section{Notation and Preliminaries}\label{sec:notation}
\bigskip

We will try to be as consistent as possible with the notation used in \cite{KiWu18}. Most of this section is essentially taken directly from \cite{KiWu18} except for the new definitions and some modifications of the formulae due to surface tension. The Fourier transform is defined as
\[
\hat{f}(\xi) = \frac{1}{\sqrt{2\pi}}\int e^{-ix\xi}f(x) \diff x
\] 
We will denote by $ \Dcalsp(\Rsp)$ the space of smooth functions with compact support in $\Rsp$ and $ \Dcalsp'(\Rsp)$ will be the space of distributions.  $\Scalsp(\Rsp)$ will denote the Schwartz space of rapidly decreasing functions and $\Scalsp'(\Rsp)$ is the space of tempered distributions. A Fourier multiplier with symbol $a(\xi)$ is the operator $T_a$ defined formally by the relation $\dis \widehat{T_a{f}} = a(\xi)\hat{f}(\xi)$. The operators $\papabs^s $ for $s\in\Rsp$ are defined as the Fourier multipliers with symbol $\abs{\xi}^s$. 
The Sobolev space $H^s(\Rsp)$ for $s\geq 0$  is the space of functions with  $\norm[H^s]{f} = \norm*[\Ltwo(\diff \xi)]{(1+\abs{\xi}^2)^{\frac{s}{2}}\hat{f}(\xi)} < \infty$. The homogenous Sobolev space $\Hhalf(\Rsp)$ is the space of functions modulo constants with  $\norm[\Hhalf]{f} = \norm*[\Ltwo(\diff \xi)]{\abs{\xi}^\half \hat{f}(\xi)} < \infty$. The Poisson kernel is given by
\begin{align}\label{eq:Poissonkernel}
K_\ep(x) = \frac{\ep}{\pi(\ep^2 + x^2)} \qquad \tx{ for } \ep>0
\end{align}

From now on compositions of functions will always be in the spatial variables. We write $f = f(\cdot,t), g = g(\cdot,t), f \compose g(\cdot,t) :=  f(g(\cdot,t),t)$. Define the operator $U_g$ as given by $U_g f = f\compose g$. Observe that $U_f U_g = U_{g\compose f}$.  Let $[A,B] := AB - BA$ be the commutator of the operators $A$ and $B$. If $A$ is an operator and $f$ is a function, then $(A + f)$ will represent the addition of the operators A and the multiplication operator $T_f$ where $T_f (g) = fg$.  We will denote the spacial coordinates in $\Omega(t) $ with $z = x+iy$, whereas $\zp = \xp + i\yp$ will denote the coordinates in the lower half plane $\Pminus = \cbrac{(x,y) \in \Rsp^2 \suchthat y<0}$. As we will frequently work with holomorphic functions, we will use the holomorphic derivatives $\pz = \half(\px-i\py)$ and $\pzp = \half(\pxp-i\pyp)$. 

Let the interface $\Sigma(t) : \z = \z(\al, t) \in \Csp$ be given by a Lagrangian parametrization $\al$ satisfying $\z_{\al}(\al,t) \neq 0 $ for all  $\al \in \Rsp$. Hence $\zt(\al,t) = \mathbf{v}(\z(\al,t),t)$ is the velocity of the fluid on the interface and $\ztt(\al,t) = (\mathbf{v_t + (v.\nabla)v})(\z(\al,t),t)$ is the acceleration.  As ${\frac{\zal}{\zalabs}(\al,t) = e^{i\thvar(\al,t)}}$ and $\frac{1}{\zalabs}\pal$ is the arc length derivative in Lagrangian coordinates, the pressure can be rewritten as
\[
P(\z(\al,t),t) = i\sigma \frac{1}{\zal}\pal\frac{\zal}{\zalabs} (\al,t)
\]
Note that $ \frac{1}{\zal}\pal\frac{\zal}{\zalabs}$ is purely imaginary. Hence the Euler equation becomes
\[
\ztt(\al,t) + i = -\hat{n}\frac{\partial P}{\partial \hat{n}}(\z(\al,t))  -\hat{t}\frac{\partial P}{\partial \hat{t}}(\z(\al,t)) 
\]
where
\begin{align*}
& \hat{t} = \frac{\zal}{\zalabs} = e^{i\thvar} = \tx{unit tangent vector} \\
& \hat{n} = i\frac{\zal}{\zalabs} = ie^{i\thvar} = \tx{unit outward normal vector}
\end{align*}
Define
\[
\avar(\al,t) = -\frac{1}{\zalabs}\frac{\partial P}{\partial \hat{n}}(\z(\al,t)) \; \in \Rsp
\]
Hence we get
\begin{align}  \label{form:zttbar}
& \ztt + i = i\avar \zal - \frac{\zal}{\zalabs}\frac{1}{\zalabs}\pal(P(\z(\al,t),t)) \notag \\
\tx{Therefore } \quad & \zttbar -i = -i\avar \zalbar -i\sigma \frac{1}{\zal}\pal \frac{1}{\zal}\pal \frac{\zal}{\zalabs} 
\end{align}

Let $\Psi(\cdot,t): \Pminus \to  \Omega(t)$ be conformal maps satisfying $\lim_{\z\to \infty} \Psi_\z(\z,t) =1$  and $\lim_{\z\to \infty} \Psi_t(\z,t) =0$.  With this, the only ambiguity left in the definition of $\Psi$ is that of the choice of translation of the conformal map at $t=0$, which does not play any role in the analysis. Let $\Phi(\cdot,t):\Omega(t) \to \Pminus $ be the inverse of the map $\Psi(\cdot,t)$ and define $\h(\cdot,t):\Rsp \to \Rsp$ as
\begin{align}\label{def:h}
\h(\al,t) = \Phi(\z(\al,t),t)
\end{align}
hence $\h(\cdot,t)$ is a homeomorphism. As we use both Lagrangian and conformal parameterizations, we will denote the Lagrangian parameter by $\al$ and the conformal parameter by $\ap$. Let $\hinv(\ap,t)$ be its inverse i.e.
\[
\h(\hinv(\ap,t),t) = \ap
\] 
From now on, we will fix our Lagrangian parametrization at $t=0$ by imposing
\begin{align*}
h(\al,0)= \al \quad  \tx { for all } \al \in \Rsp
\end{align*}
Hence the Lagrangian parametrization is the same as conformal parametrization at $t=0$. Define the variables
\[
\begin{array}{l l l}
 \Z(\ap,t) = \z\compose \hinv (\ap,t)  & \Zap(\ap,t) = \pap \Z(\ap,t) &  \quad  \tx{ Hence } \quad  \brac*[]{\dfrac{\zal}{\hal}} \compose \hinv = \Zap \\
  \Zt(\ap,t) = \zt\compose \hinv (\ap,t)  & \Ztap(\ap,t) = \pap \Zt(\ap,t) &  \quad   \tx{ Hence } \quad  \brac*[]{\dfrac{\ztal}{\hal}} \compose \hinv = \Ztap \\
 \Ztt(\ap,t) = \ztt\compose \hinv (\ap,t)  & \Zttap(\ap,t) = \pap \Ztt(\ap,t) &  \quad   \tx{ Hence } \quad  \brac*[]{\dfrac{\zttal}{\hal}} \compose \hinv = \Zttap \\
\end{array}
\]
Hence $\Z(\ap,t), \Zt(\ap,t)$ and $\Ztt(\ap,t)$ are the parameterizations of the boundary, the velocity and the acceleration in conformal coordinates and in particular $\Z(\cdot,t)$ is the boundary value of the conformal map $\Psi(\cdot,t)$. Note that as $\Z(\ap,t) = \z(\hinv(\ap,t),t)$ we see that  $\pt \Z \neq \Zt$. Similarly $\pt \Zt \neq \Ztt$.  The substitute for the time derivative is the material derivative. Define the operators
\begingroup
\allowdisplaybreaks
\begin{fleqn}
\begin{align}\label{eq:mainoperators}
\begin{split}
\qquad & \Dt =  \tx{material derivative} = \pt + \bvar \pap \qquad \tx{ where } \bvar = \h_t \compose \hinv \\
 & \Dap = \Dapfrac \qquad \Dapbar = \Dapbarfrac \qquad \Dapabs = \Dapabsfrac \\
& \Hil  =   \text{Hilbert transform  = Fourier multiplier with symbol } -sgn(\xi) \\
& \qquad  \Hil f(\alpha ' ) =  \frac{1}{i\pi} p.v. \int \frac{1}{\alpha ' - \beta'}f(\beta') \diff\beta' \\
& \Ph = \text{Holomorphic projection} = \frac{\Id + \Hil}{2} \\
& \Pa = \text{Antiholomorphic projection} = \frac{\Id - \Hil}{2} \\ 
& \papabs   = \ i\Hil \partial_{\alpha'} = \sqrt{-\Delta} = \text{ Fourier multiplier with symbol } |\xi| \ \\
& \papabs^{1/2}  = \text{ Fourier multiplier with symbol } \abs{\xi}^{1/2} 
\end{split}
\end{align}
\end{fleqn}
\endgroup
Now we have $\Dt \Z = \Zt$ and $\Dt \Zt = \Ztt$ and more generally $\Dt (f(\cdot,t)\compose \hinv) = (\pt f(\cdot,t)) \compose \hinv$ or equivalently $\pt (F(\cdot,t)\compose \h) = (\Dt F(\cdot,t)) \compose \h$. This means that $\Dt = U_h^{-1}\pt U_h$ i.e. $\Dt$ is the material derivative in conformal coordinates.  In this work, the material derivative $\Dt$ is more heavily used as compared to the time derivative $\pt$. 

Define $\U: \Pminusbar \to \Csp $ as
\begin{align*}
\U = \vboldbar \compose \Psi 
\end{align*}
and observe that $\U$ is a holomorphic function on $\Pminus$. Also note that its boundary value is given by $ \Ztbar(\ap,t) = \U(\ap,t)$ for all $\ap\in \Rsp$. The Hilbert transform defined above satisfies the following property.
\begin{lem}[\cite{Ti86}]
Let $1<p<\infty$ and let $F(z)$ be a holomorphic function in the lower half plane with $F(z) \to 0$ as $z\to \infty$. Then the following are equivalent
\begin{enumerate}
\item $\dis \sup_{y<0} \norm[p]{F(\cdot + iy)} < \infty$
\item $F(z)$ has a boundary value $f$, non-tangentially almost everywhere with $f \in L^p$ and $\Hil(f) = f$. 
\end{enumerate}
\end{lem}
In particular this says if $\U$ decays appropriately at infinity, then the boundary value of $\U$ namely $\Ztbar$ will satisfy $\Hil \Ztbar = \Ztbar$.  We can now define the main variables used in this paper
\begin{fleqn}
\begin{align}\label{eq:mainvariables}
\begin{split}
 \qquad & \A = (\avar\hal) \compose \hinv  \\
& \Aonesigma = \A \Zapabs^2   \qquad  \tx{ Hence } \quad \frac{\Aonesigma}{\Zapabs} = -\frac{\partial P}{\partial \hat{n}}\compose \hinv\\
& \Aone = 1 - \Imag[\Zt,\Hil]\Ztapbar  \\
& \g = \thvar \compose \hinv \qquad  \tx{ Hence } \quad \frac{\Zap}{\Zapabs} = e^{i\g} \tx{ and } \Dapabs \g = (\ps \thvar) \compose \hinv = -i\Dap \frac{\Zap}{\Zapabs}  \\
& \Th = (\Id+\Hil)\Dapabs \g = -i(\Id+\Hil) \Dap \frac{\Zap}{\Zapabs} \\
& \w = e^{i\g} = \frac{\Zap}{\Zapabs} \qquad  \tx{ Hence } \quad \Dapabs \w = i \w\Real \Th
\end{split}
\end{align} 
\end{fleqn} 
Observe that $\Real \Th = \kap\compose\hinv$ where $\kap$ is the curvature of the interface. With this notation, by precomposing  $\eqref{form:zttbar}$ with $\hinv$ we get
\begin{align} \label{form:Zttbar_old}
\Zttbar -i = -i\frac{\Aonesigma}{\Zap} -i\sigma \Dap \Dap \frac{\Zap}{\Zapabs}  
\end{align}
Let us now derive the formulae of $\Aonesigma$ and $\bvar$.
\medskip

\subsubsection{}$\!\!\!\mathbf{Formula \ for\  } \Aonesigma$
\medskip

Let $F = \vboldbar$ and hence $F$ is holomorphic in $\Omega(t)$ and $\ztbar = F(\z(\al,t),t)$. Hence
\[
\zttbar = F_t(\z(\al,t),t) + F_\z(\z(\al,t),t)\zt(\al,t) \quad\qquad \ztalbar = F_\z(\z(\al,t),t)\zal(\al,t) 
\]
\begin{align*}
\tx{Hence }\quad  \zttbar  = F_t \compose z + \zt \frac{\ztalbar}{\zal}  \qq \qq \qq
\end{align*}
Precomposing with $\hinv$ we obtain $\dis \enspace  \Zttbar  = F_t\compose\Z + \Zt \frac{\Ztapbar}{\Zap} $. Now multiply by $i\Zap$ and use \eqref{form:Zttbar_old} to get
\[
\Aonesigma = i\Zap F_t\compose\Z + \Zap +i\Zt\Ztapbar -\sigma\pap\Dap\frac{\Zap}{\Zapabs}
\]
Apply $(\Id - \Hil) $ and use the fact that $\Hil(\Zap - 1) = \Zap -1$ and $\Hil 1 = 0$ to obtain
\[
(\Id - \Hil)\Aonesigma = 1 + i[\Zt,\Hil]\Ztapbar - \sigma (\Id - \Hil)\pap\Dap\frac{\Zap}{\Zapabs}
\]
Now take the real part 
\[
\Aonesigma = 1 - \Imag[\Zt,\Hil]\Ztapbar + \sigma \pap \Hil \Dap \frac{\Zap}{\Zapabs}
\]
Hence 
\begin{align} \label{form:Aonesigma}
\Aonesigma = \Aone + \sigma \pap \Hil \Dap \frac{\Zap}{\Zapabs} \qquad \tx{ and in particular }\quad \Aonesigma \big|_{\sigma = 0} = \Aone 
\end{align}
Note that the only non-holomorphic quantity in the above formula is $i\Zt\Ztapbar$. Also note that as $\Aone  = 1 - \Imag[\Zt,\Hil]\Ztapbar$ by the calculation in \cite{KiWu18, Wu15b} we have that $\Aone \geq 1$. From \eqref{form:Aonesigma} the Taylor sign condition term can be written as
\begin{align}\label{form:Taylor}
 -\frac{\partial P}{\partial \hat{n}} \compose \hinv = \frac{\Aonesigma}{\Zapabs} = \frac{1}{\Zapabs}\brac{\Aone + \sigma\papabs(\kap\compose\hinv)}
\end{align}
where $\kap\compose\hinv = -i\Dap\frac{\Zap}{\Zapabs}$ is the curvature in conformal coordinates. For $\sigma=0$, this formula was first derived by Wu \cite{Wu97} to prove the Taylor sign condition for $C^{1,\alpha}$ interfaces with $\alpha>0$ and was crucially used in Kinsey-Wu \cite{KiWu18} to prove a priori estimates for angled crest interfaces. We will also use this formula in an essential way in this paper.  

Observe that if the interface is $C^{1,\alpha}$ then the Taylor sign condition is true for $\sigma=0$ but will fail generically if $\sigma$ is large. Ambrose and Masmoudi \cite{AmMa05} assumed that the Taylor sign condition holds for all $0\leq \sigma\leq \sigma_0$ which is only true if $\sigma_0$ is small. This issue was resolved by Shatah and Zeng \cite{ShZe08} where they only assumed the Taylor sign condition for $\sigma = 0$ which removed the smallness assumption of $\sigma_0$. However in both cases the Taylor sign condition is assumed for $\sigma=0$ which we do not have as we allow interfaces with angled crests for $\sigma=0$ for which only a weak Taylor sign condition $-\frac{\partial P}{\partial \hat{n}} \geq 0$ is satisfied. 

\medskip

\subsubsection{}$\!\!\!\mathbf{Formula \ for\  } \bvar$
\medskip

Recall that $\h(\al,t) = \Phi(\z(\al,t),t)$ and so by taking derivatives we get
\[
\h_t = \Phi_t \compose \z + (\Phi_\z \compose \z)\zt  \quad \qquad \hal = (\Phi_\z \compose \z)\zal 
\]
\begin{align*}
\tx{ Hence } \enspace \h_t = \Phi_t \compose \z + \frac{\hal}{\zal}\zt 
\end{align*}
Precomposing with $ \hinv$ we obtain $\dis  \quad \h_t \compose \hinv = \Phi_t \compose \Z + \frac{\Zt}{\Zap}$. Apply $(\Id - \Hil)$ and take real part, to get
\begin{align}\label{form:bvar} 
\bvar = \Real(\Id-\Hil)\brac{\frac{\Zt}{\Zap}} 
\end{align}
This formula is the same as the one in \cite{KiWu18} and surface tension does not affect the formula.
\smallskip

\subsection{\texorpdfstring{Fundamental Equation \nopunct}{}} \hspace*{\fill} \medskip

Substituting the formula for $\Aonesigma$ in equation \eqref{form:Zttbar_old}, we get
\begin{align*}
\Zttbar -i = -i\frac{ \Aone}{\Zap} -i\sigma \Dap \Hil \Dap \frac{\Zap}{\Zapabs} -i\sigma \Dap \Dap \frac{\Zap}{\Zapabs} 
\end{align*}
Now combine the second and third term and use $\Th = -i(\Id+\Hil) \Dap \frac{\Zap}{\Zapabs}$ to get the fundamental equation
\begin{align}  \label{form:Zttbar}
\Zttbar -i = -i \frac{\Aone}{\Zap} +\sigma \Dap \Th
\end{align}
Note that as $\Aone$ does not depend on $\sigma$, the effect of surface tension is that it adds a holomorphic quantity to the conjugate of the acceleration. We also see that
\begin{align*}
\frac{\Aone}{\Zapabs} = -\frac{\partial P}{\partial \hat{n}} \compose \hinv \Bigr\vert_{\sigma=0} \geq 0
\end{align*}
and hence it represents the Taylor sign condition in the absence of surface tension. As the equation is written in terms of $\Aone$ and not $\Aonesigma$, our energy $\Esigma$ will always be positive irrespective of the value of surface tension.

\subsection{\texorpdfstring{System \nopunct}{}} \hspace*{\fill}\label{sec:systemone} \medskip

To summarize the system is in the variables $(\Zap,\Zt)$ satisfying
\begin{align}\label{eq:systemone}
\begin{split}
\bvar & = \Real(\Id - \Hil)\brac{\frac{\Zt}{\Zap}} \\
\Aone & = 1 - \Imag \sqbrac{\Zt,\Hil}\Ztapbar \\
(\pt + \bvar\pap)\Zap &= \Ztap - \bap\Zap \\
(\pt + \bvar\pap)\Ztbar & = i -i\frac{\Aone}{\Zap} + \frac{\sigma}{\Zap}\pap(\Id + \Hil)\cbrac{\Imag\brac{\frac{1}{\Zap}\pap \frac{\Zap}{\Zapabs} }}
\end{split}
\end{align}
along with the condition that their harmonic extensions, namely $\Psizp(\cdot + iy) = K_{-y}\conv \Zap$ and $\U(\cdot + iy) = K_{-y}\conv \Ztbar$ for all $y<0$, \footnote{here $K_{-y}$ is the Poisson kernel \eqref{eq:Poissonkernel}} are holomorphic functions on $\Pminus$ and satisfy
\begin{align*}
\lim_{c \to \infty} \sup_{\abs{\zp}\geq c}\cbrac{\abs{\Psizp(\zp) - 1} + \abs{\U(\zp)}}  = 0 
\qquad \tx{ and } \quad \Psizp(\zp) \neq 0 \quad \tx{ for all } \zp \in \Pminus 
\end{align*}

We observe that for such a $\Psizp$  we can uniquely define $\log(\Psizp) : \Pminus \to \Csp$ such that $\log(\Psizp)$ is a continuous function with $\Psizp = \exp\cbrac{\log(\Psizp)}$ and $(\log(\Psizp))(\zp) \to 0$ as $\zp \to \infty$. Also note that one can obtain $\Z(\cdot,t)$ by the formula 
\begin{align*}
\Z(\ap,t) = \Z(\ap,0) +  \int_0^t \cbrac{\Zt(\ap,s) - \bvar(\ap,s)\Zap(\ap,s)} \diff s
\end{align*}
In particular instead of the variables $(\Zap,\Zt)$, one can view the system being in variables $(\Z,\Zt)$. 

Another important observation one immediate makes is that in the above system, there is no restriction that the function $\Z(\cdot,t)$ be injective. Even if the curve $\Z(\cdot,t)$ becomes self-intersecting, the system still makes sense and one can still find a solution. Hence the above system allows self-intersecting domains similar to work of \cite{CaCoFeGaGo12, CaCoFeGaGo13} where solutions with splash and splat singularities were constructed. Observe that we assumed that the interface is non-self intersecting while deriving this system from the Euler equation \eqref{euler-capillary} and \eqref{boundary}. If the solution $(\Z,\Zt)(t)$ of \eqref{eq:systemone} leads to a non-self intersecting curve then one can go back and obtain a solution to the Euler equation in a similar way as done in \cite{Wu19} section 2.5. However if the interface becomes self-intersecting, then its relation to the Euler equation is lost and the solution becomes non-physical. From now on we will exclusively focus on the system \eqref{eq:systemone} and hence all results in this paper apply to self-intersecting curves as well.

One can rewrite the function $\h(\al,t)$ defined in \eqref{def:h} as the solution to the ODE
\begin{align*}
\frac{\diff h}{\diff t} &= \bvar(\h, t) \\
h(\al,0) &= \al
\end{align*}
where $\bvar$ is given by \eqref{eq:systemone}. From this we easily see that as long as $\sup_{[0,T]} \norm[\infty]{\bvarap}(t) <\infty$ we can solve this ODE and for any $t\in [0,T]$ we have that $h(\cdot,t)$ is a homeomorphism. Hence it makes sense to talk about the functions $z = \Z\compose \h, \zt = \Zt \compose \h$ which are Lagrangian parametrizations of the interface and the velocity on the boundary. 

\medskip

\section{Main results and discussion}\label{sec:discussionandresults}

\subsection{\texorpdfstring{Results \nopunct}{}}\label{sec:results}\hspace*{\fill} \medskip

For $\sigma\geq 0$ define the energy
\begingroup
\allowdisplaybreaks
\begin{align*}
\Ecalsigmaone & =  \norm[2]{\pap\frac{1}{\Zap}}^2 + \norm*[\bigg][\Hhalf]{\frac{1}{\Zap}\pap\frac{1}{\Zap}}^2  +  \norm[\Hhalf]{\sigma\pap\Th}^2 + \norm*[\Bigg][2]{\sigma^\onebysix\Zap^\half\pap\frac{1}{\Zap}}^6 + \norm*[\Bigg][\infty]{\sigma^\half\Zap^\half\pap\frac{1}{\Zap}}^2  \\*
& \quad  +  \norm*[\Bigg][2]{\frac{\sigma^\half}{\Zap^\half}\pap^2\frac{1}{\Zap}}^2  + \norm*[\Bigg][\Hhalf]{\frac{\sigma^\half}{\Zap^\threebytwo}\pap^2\frac{1}{\Zap}}^2     + \norm[2]{\frac{\sigma}{\Zap}\pap^3\frac{1}{\Zap}}^2 + \norm[\Hhalf]{\frac{\sigma}{\Zap^2}\pap^3\frac{1}{\Zap}}^2 \\
\Ecalsigmatwo &= \norm[2]{\Ztapbar}^2 + \norm*[\Bigg][2]{\frac{1}{\Zap^2}\pap\Ztapbar}^2 + \norm*[\Bigg][2]{\frac{\sigma^\half}{\Zap^\half}\pap\Ztapbar}^2 + \norm*[\Bigg][2]{\frac{\sigma^\half}{\Zap^\fivebytwo}\pap^2\Ztapbar}^2 \\
\Ecalsigma & = \Ecalsigmaone + \Ecalsigmatwo
\end{align*}
\endgroup

Note that $\Ecalsigmaone$ controls terms based solely on the interface $\Z(\cdot,t)$ and $\Ecalsigmatwo$ controls the weighted derivatives of the velocity $\Zt(\cdot,t)$. Hence the energy can be thought of as a weighted Sobolev norm with the weight given by  powers of $\frac{1}{\Zap}$. Also observe that all the terms of $\Ecalsigma$ are the boundary values of holomorphic functions.

Observe that the energy $\Ecalsigma(t)$ is well defined and is finite if $(\Zap-1,\frac{1}{\Zap} - 1, \Zt)(t) \in H^{3.5}(\Rsp)\times H^{3.5}(\Rsp)\times H^3(\Rsp)$. If we ignore the weights in the energy, we get back the Sobolev norm. See \lemref{lem:equivsobolev} for the precise relationship between the energy $\Ecalsigma$ and the Sobolev norm of the solution. 

The energy $\Ecalsigma$ is equivalent to $\Esigma$ which is what we use to prove the main a priori estimate \thmref{thm:aprioriEsigma} ($\Esigma$ is defined in \secref{sec:aprioriEsigma} and the equivalence between $\Ecalsigma$ and $\Esigma$ is proved in \propref{prop:equivEsigma}). A more intuitive explanation of the energy $\Ecalsigma(t)$ can be obtained by applying the fundamental operators  $\Dt, \frac{1}{\Zapabs^2}\pap$ and $\sigma^\onebythree \Dapabs$ occurring in the quasilinear equations \eqref{eq:DapbarZtbar} and \eqref{eq:Th} on the variable $\g = \thvar\compose \hinv$. This is explained in more detail in \secref{sec:discussion}.

For $\sigma=0$, the energy $\Ecalsigma(t)$ allows angled crest interfaces and outward pointing cusps (i.e. interfaces with angled crests and cusps yield $\Ecalsigma < \infty$ if $\sigma=0$). To see this, we only need to show that the first two terms of $\Ecalsigmaone$ allows angled crests and cusps. The argument is exactly the same as explained in \cite{KiWu18, Ag20} and we briefly explain it here. If the interface has an angled crest of angle $\nu\pi$ at $\ap = 0$, then $\Z(\ap) \sim (\ap)^\nu$ near $\ap = 0$. Hence near $\ap = 0$ we have
\begin{align*}
\Zap(\ap) \sim (\ap)^{\nu-1} \qquad \! \frac{1}{\Zap}(\ap) \sim (\ap)^{1-\nu} \qquad \pap\frac{1}{\Zap}(\ap) \sim (\ap)^{-\nu} \qquad \! \frac{1}{\Zap}\pap \frac{1}{\Zap}(\ap) \sim (\ap)^{1 - 2\nu}
\end{align*}
From this we can see that that the first two terms of $\Ecalsigmaone$ allows interfaces with angled crests of angles $\nu\pi$ with $0<\nu<\half$. A similar argument shows that cusps i.e. $\nu = 0$ are also allowed. See \cite{Ag20} for more details. 

For $\sigma>0$, in contrast we see that most of the terms with surface tension in $\Ecalsigmaone$ do not allow angled crest interfaces which can be seen by a similar argument as above. For example consider the term $\norm[\infty]{\sigma^\half\Zap^\half\pap\frac{1}{\Zap}}$ in the energy. Suppose we have an interface with an angled crest of angle $\nu\pi$ at $\ap =0$ with $\nu\geq 0$. Then by the argument above we see that near $\ap = 0$
\begin{align*}
\Zap^\half\pap\frac{1}{\Zap}(\ap) \sim (\ap)^{-\frac{\nu}{2} - \half}
\end{align*}
which is unbounded and so the energy does not allow angled crest interfaces if $\sigma>0$ (i.e. interfaces with angled crests and cusps yield $\Ecalsigma = \infty$ if $\sigma>0$). In fact if $\Ecalsigma<\infty$ and $\sigma>0$, then from \lemref{lem:equivsobolev} and \remref{rem:equivsobolev} we automatically have $\norm[H^2]{\Ztap} + \norm[H^{2.5}]{\pap\Zap} < \infty$ i.e. the data has to be in Sobolev spaces and hence do not allow singularities. However this does not say anything about the size of the Sobolev norm which can be very large. In fact the $\Linfty$ norm of the curvature can be of the order of $\sigma^{-\frac{1}{3}}$ (See \corref{cor:example} below for an example).

The fact that the energy does not allow singularities for $\sigma>0$ is quite natural as one would expect that surface tension should be smoothing and hence stable solutions should not have angled crests. Moreover the natural extension of the  energy of Kinsey-Wu \cite{KiWu18} does not allow any singularities in the interface. Observe that in Kinsey-Wu \cite{KiWu18}, the quantity $\Zttbarap \in \Ltwo$ and $\pap\brac{\frac{\Aone}{\Zap}} \in \Ltwo$. Hence if we assume these quantities remain in $\Ltwo$ when $\sigma>0$, then by the equation \eqref{form:Zttbar} we see that $\sigma\pap\Dap\Th \in \Ltwo$. Hence $\Dapabs\Th \in \Linfty_{loc}$ and as $\Real \Th = \kap\compose\hinv$, we see that $\ps\kap \in \Linfty_{loc}$.  Hence the interface has to be at least $C^{2,1}$ which rules out any type of singularity. 

We now state the main result of this paper.  

\begin{thm}\label{thm:existence}
Let $\sigma >0$ and assume the initial data $(\Z,\Zt)(0)$ satisfies $(\Zap-1,\frac{1}{\Zap} - 1, \Zt)(0) \in H^{3.5}(\Rsp)\times H^{3.5}(\Rsp)\times H^3(\Rsp)$. Then $ \Ecalsigma(0) < \infty$ and there exists $T,C_1 > 0$ depending only on $\Ecalsigma(0)$ such that the initial value problem to  \eqref{eq:systemone} has a unique solution $(\Z,\Zt)(t)$ in the time interval $[0,T]$ satisfying  
$(\Zap-1,\frac{1}{\Zap} - 1, \Zt) \in C^l([0,T], H^{3.5 - \threebytwo l}(\Rsp)\times H^{3.5 - \threebytwo l}(\Rsp)\times H^{3 - \threebytwo l}(\Rsp))$ 
for $l = 0,1$ and $\dis \sup_{t\in [0,T]} \Ecalsigma (t)\leq C_1 <\infty$
\end{thm}
 It is important to note that we do not allow angled crested interfaces as initial data when $\sigma>0$. However the initial data does allow interfaces with large $C^{1,\alpha}$ norm (see below). The most important feature about the above result is that the time of existence depends only on $\Ecalsigma(0)$ and not the Sobolev norm of the initial data. We now explain some of the important points about the result and the energy $\Ecalsigma$. 

\smallskip
\textbf{Properties:} 
\begin{enumerate}
\item No assumptions on the Taylor sign condition: Observe that no assumptions are made on the Taylor sign condition in the above result. Recall that the Taylor sign condition is $ -\frac{\partial P}{\partial \hat{n}} \geq c >0$ and from \eqref{form:Taylor} that the quantity $ -\frac{\partial P}{\partial \hat{n}}$ depends on the value of $\sigma$. 

All previous results on zero surface tension limit assume the Taylor sign condition. In the case of Ambrose and Masmoudi \cite{AmMa05} the condition $ -\frac{\partial P}{\partial \hat{n}} \geq c >0$ is assumed for all $0\leq \sigma\leq \sigma_0$ which forces $\sigma_0$ to be small. In Shatah and Zeng \cite{ShZe08} the condition $ -\frac{\partial P}{\partial \hat{n}} \geq c >0$ is assumed only for $\sigma = 0$ removing the smallness assumption on $\sigma$. However in both cases one still needs control on the $C^{1,\alpha}$ norm of the interface. We do not make any assumptions which allows us to deal with both large $\sigma$ and large $C^{1,\alpha}$ norm of the interface. 

\item Time of existence is independent of $\sigma$: Observe that the energy $\Ecalsigma(t)$ is an increasing function of $\sigma$. Hence for arbitrary $\sigma_0>0$, the time of existence is uniform for all $0<\sigma\leq \sigma_0$. In particular we recover the uniform time of existence part of the result of Ambrose and Masmoudi \cite{AmMa05} in this case. 

\item  Energy allows angled crest solutions for $\sigma = 0$: If we put $\sigma =0$ in the energy $\Ecalsigma$, then it allows solutions with angled crest with angle less than $90 ^\circ $ and also allows cusps. These are exactly the solutions allowed by the energy obtained by Kinsey and Wu in \cite{KiWu18} and the local wellposedness for angled crested solutions has been proven by Wu in \cite{Wu19}. Our energy for $\sigma = 0$ is a lower order version of the energy in \cite{KiWu18} by half spacial derivative.

\item Energy does not allow angled crest solution for $\sigma >0$ but allows large curvature: In the proof of this theorem we show the estimate $\norm*[\big][\infty]{\kap} \leq \sigma^{-\onebythree}C(\Ecalsigma)$ holds and hence for $\sigma>0$ the curvature is bounded, which automatically excludes angled crest solutions. Note however that for small values of $\sigma$, the energy allows data with quite large curvature of the order of $\sigma^{-\frac{1}{3}}$. (See \corref{cor:example} below where for any given $\ep>0$ arbitrarily small, we construct examples where $\Ecalsigma =O(1)$ and the curvature of the initial data grows like $\sigma^{-\onebythree + \ep}$ as $\sigma\to 0$). 

\end{enumerate}

In the above theorem we do not prove existence when $\sigma=0$ because in this case the energy $\Ecalsigma(0)$ allows singular solutions such as angled crested solutions which have already been proved to exist in \cite{Wu19}. The above result is proved via a new a priori estimate \thmref{thm:aprioriEsigma} which is an extension of the a priori estimate of Kinsey and Wu \cite{KiWu18} for $\sigma\geq 0$, and a local existence result for $\sigma>0$ in Sobolev spaces \thmref{thm:existencesobolev}.  

As noted above the main point of the above theorem is that the time of existence depends only on $\Ecalsigma(0)$. The usefulness of the energy $\Ecalsigma(t)$ comes from the fact that there are interfaces (such as smooth interfaces close to being angled crest) for which the $C^{1,\alpha}$ norm (for any $\alpha>0$) of the interface of the initial data is very large but $\Ecalsigma(0)$ remains bounded. This translates into a longer time of existence if we use the energy $\Ecalsigma(t)$ instead of the Sobolev norm. We now give an example which demonstrates this. We will use the following notation:  Let $A \subset \Csp$ be a non-empty set and let $p \in \bar{A}$. Let $f,g:A \to \Csp$ be functions such that $g(z) \neq 0$ for all $z$ in a punctured neighborhood of $p$. We say that $f(z)\sim g(z)$  as  $z \to p$ in $A$, if $\lim_{\z \to p}\frac{f(z)}{g(z)} \in \Csp^*$ where $\Csp^* = \Csp\backslash\cbrac{0}$.

As before let $(\Psi,\U)(\cdot,0):\Pminus \to \Csp$ be holomorphic maps with $\Psi_z \neq 0$ for all $z\in P_{-}$ and with their boundary values being the initial data namely $(\Z,\Ztbar)(\ap,0) = (\Psi,\U)(\ap,0)$ for all $\ap\in\Rsp$. 
Recall the notation namely that for $\zp \in \Pminus$ we write $\zp = \xp + i\yp$. At $t=0$ define the quantity
\begin{align*}
M & =   \sup_{\yp<0}\norm[\Leightbyseven(\Rsp, \diff \xp)]{\Psi_\zp^\threebyfour\partial_\zp \brac{\frac{1}{\Psi_\zp}}} + \sup_{\yp<0}\norm[\Lfourbythree(\Rsp, \diff \xp)]{\Psi_\zp^\half\partial_\zp \brac{\frac{1}{\Psi_\zp}}} + \sup_{\yp<0}\norm[\Ltwo(\Rsp, \diff \xp)]{\partial_\zp \brac{\frac{1}{\Psi_\zp}}}   \\
& \quad  + \sup_{\yp<0}\norm[\Linfty(\Rsp, \diff \xp)]{\frac{1}{\Psi_\zp}\partial_\zp \brac{\frac{1}{\Psi_\zp}}}  + \sup_{\yp<0}\norm[\Lone(\Rsp, \diff \xp)]{\frac{1}{\Psi_\zp}\partial_\zp^2 \brac{\frac{1}{\Psi_\zp}}} + \sup_{\yp<0}\norm[\Ltwo(\Rsp, \diff \xp)]{\frac{1}{\Psi_\zp^2}\partial_\zp^2 \brac{\frac{1}{\Psi_\zp}}}  \\
& \quad + \sup_{\yp<0}\norm[\Lone(\Rsp, \diff \xp)]{\frac{1}{\Psi_\zp^3}\partial_\zp^3 \brac{\frac{1}{\Psi_\zp}}} + \sup_{\yp<0}\norm[\Ltwo(\Rsp,\diff \xp)]{\frac{1}{\Psi_\zp} - 1} + \sup_{\yp<0}\norm[H^{3.5}(\Rsp,\diff \xp)]{\U}
\end{align*}
It is easy to check that if $M<\infty$, then the initial data satisfies the hypothesis of Theorem 3.9 of Wu \cite{Wu19} and we get a unique solution $(\Z,\Zt)(t)$ to \eqref{eq:systemone} for $\sigma=0$. Also by exactly the same argument as in section 5 of \cite{Ag20}, $M<\infty$ allows interfaces with angled crests of angles $\nu\pi$ with $0<\nu<\half$ and also allows certain cusps (see \cite{Ag20} for more details). Also note that the condition $M<\infty$ forces the gradient of the velocity to be zero at the singularities of the interfaces (see \cite{Ag20} for more details). With this we can now state the main corollary. 
\vspace*{-0.6cm}
\begin{figure}[ht]
\centering
\begin{tikzpicture}[scale=0.55]
\draw [thick, xshift=-25pt] (-4,-0.5) to [out=0,in=-89] (-2,4) to [out=-87+180,in=180] (-0,0) to [out=0,in=180+45] (6,3.5) to [out=-65,in=180] (10,0) to [out=0,in=150] (14,4) to [out=180+38,in=180] (17,-0.5);
\draw [densely dashed, xshift=-25pt, yshift=-20pt] (-4,0.2-0.4) to [out=0,in=-100] (-1.9,2-0.0) to [out=-100+180,in=180] (-0.26,0.3) to [out=0,in=180+45] (5.8,3.5) to (6,3.5) to [out=-65,in=180] (10,0.2) to [out= 0,in=90](13.2,3.8) to [out=-90,in=180] (17,0.2-0.4);
\node at (17.1,-0.3) {\ \ $Z(\alpha,t)$};
\node at (17.7,-0.8-0.3) {\!\!$Z^{\ep,\sigma}(\alpha,t)$};
\draw [thin,xshift=-25pt,  <->] (-4.3,0-0.45) to (-4.3,-0.52-0.4) ;
\node at (-5.5,-0.3-0.4) {${\ep}$};
\node at (1.8,2.6) {$\rho_{air} =0$};
\node at (4.2,0) {$\rho_{water} =1$};
\end{tikzpicture}
\caption{Waves with and without surface tension}
\end{figure}
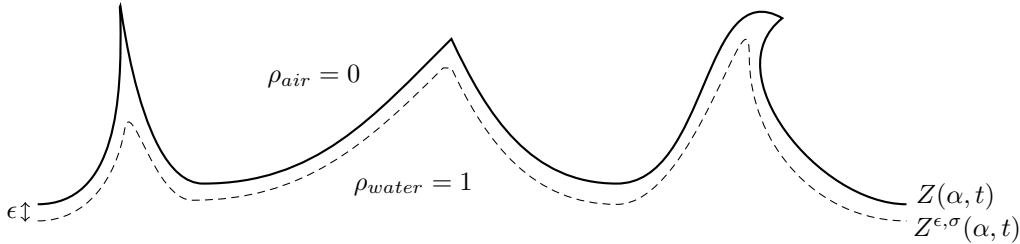
\vspace*{-0.1cm}
\begin{cor}\label{cor:example}
Consider an initial data $(\Z,\Zt)(0)$ with $M<\infty$. Let $(\Z,\Zt)(t)$ be the unique solution of equation \eqref{eq:systemone} for  $\sigma = 0$ with initial data $(\Z,\Zt)(0)$  as obtained in \cite{Wu19}. For $0<\epsilon \leq 1$ and $ \sigma \geq 0$ denote by $(\Z^{\epsilon,\sigma}, \Zt^{\epsilon,\sigma})(t)$ the unique solution to the equation \eqref{eq:systemone}  with surface tension $\sigma $ and with initial data  $(\Z^{\epsilon,\sigma}, \Zt^{\epsilon,\sigma})(0) = (\Z\conv P_\epsilon, \Zt\conv P_\epsilon)(0)$ where $P_\epsilon$ is the Poisson kernel. Then we have the following
\begin{enumerate}
\item For any $c>0$, there exists $T,C_1>0$ depending only on $c$ and $M$ such that for all $\sigma\geq 0$ and $0<\ep\leq 1$ satisfying $\dis \frac{\sigma}{\ep^\threebytwo} \leq c$, the solutions $(\Z^{\epsilon,\sigma}, \Zt^{\epsilon,\sigma})(t)$ exist in the time interval $[0,T]$ with $\sup_{t\in[0,T]} \Ecalsigma(\Z^{\ep,\sigma},\Zt^{\ep,\sigma})(t) \leq C_1 <\infty$.
\item If the initial interface $\Z(\cdot,0)$ has only one angled crest of angle $\nu\pi$ with $0<\nu<\half$, then the $\Linfty$ norm of the curvature $\kap^{\ep,\sigma}$ of the initial interface $\Z^{\ep,\sigma}(\cdot,0)$ satisfies $\norm[\infty]{\kap^{\ep,\sigma}} \sim \ep^{-\nu}$ as $\ep\to0$. In particular for any $0<\del<\onebythree$ arbitrarily small, choosing $\nu = \half - \threebytwo\del$ and $\sigma = \ep^\threebytwo$ we obtain $\norm[\infty]{\kap^{\ep,\sigma}} \sim \sigma^{-\onebythree + \del}$ as $\sigma \to 0$. Hence \thmref{thm:existence} allows initial interfaces with large curvature when $\sigma$ is small.  
\end{enumerate}
\end{cor}

\begin{rmk}\label{rmk:convergence}
In a companion paper \cite{Ag20a} we prove that if in addition $\dis \max\cbrac{\frac{\sigma}{\ep^{3/2}}, \ep} \to 0$, then the solutions $(\Z^{\epsilon,\sigma}, \Zt^{\epsilon,\sigma})(t) \to (\Z, \Zt)(t)$ in an appropriate norm. 
\end{rmk}

\begin{rmk}\label{rem:scaling}
Note that in the corollary, we prove uniform time of existence under the condition $\frac{\sigma}{\ep^\threebytwo} \lesssim 1$ and in particular in this case $\sigma \to 0$ as $\ep \to 0$. We do not know what happens when $\sigma = 1$ and $\ep \to 0$. It is our belief that one cannot expect uniform time of existence in this case for general initial data. To see why, observe that as explained in the introduction if we ignore gravity then the relevant scaling for the problem of zero surface tension limit is $\Z_\lamb(\al,t) = \lamb^{-1}\Z(\lamb\al, t) $ with $\sigma_\lamb = \lamb^{-3}\sigma$ and this scaling leaves the quantity $\norm*[\infty]{\sigma^\onebythree\kappa}$ invariant. Because of this we believe that one cannot expect a general local wellposedness result which gives a time of existence of order 1 when $\norm*[\infty]{\sigma^\onebythree\kappa} \to \infty$. Note that as explained in the second point in the corollary above, the scaling $\frac{\sigma}{\ep^\threebytwo} \lesssim 1$ allows initial interfaces with the $\Linfty$ norm of the curvature of the order of $\sigma^{-\onebythree + \del}$ for arbitrary small $\delta>0$, thereby getting us arbitrary close to the rate $\norm*[\infty]{\sigma^\onebythree\kappa} \sim 1$.

%
\end{rmk}

The novelty of the above result lies in the fact that all previous results on surface tension for large data obtain a time of existence $T \lesssim \norm[\infty]{\kap}^{-1}$ where $\kap$ is the initial curvature, even if $\sigma$ is very small. Hence if the interface $\Z(\cdot,0)$ has a single angled crest of angle $\nu\pi$ then $Z^{\ep,\sigma}(\cdot, 0)$ has curvature $\norm[\infty]{\kap^{\ep,\sigma}} \sim \ep^{-\nu}$, which yields $T \lesssim \ep^\nu \to 0$ as $\ep \to 0$. The above corollary says that these solutions in fact exist on a much longer time interval and that the time of existence is at least $O(1)$ even as $\ep \to 0$, provided there is a balance between surface tension  and smoothness $\sigma\lesssim \ep^\threebytwo$.

The scaling factor of $\sigma/\ep^\threebytwo$ comes about as a compatibility condition of two of the main operators in the quasilinear equations  \eqref{eq:Th} and \eqref{eq:DapbarZtbar} (see \secref{sec:discussion} for more details). Observe that
\begin{align*}
\brac{\frac{1}{\Zapabs^2}\pap}^{\n-\threebytwo}\brac{\frac{\sigma^\onebythree}{\Zapabs}\pap}^{\n 3} \sim \sigma \pap^\threebytwo
\end{align*}
which naturally gives us the factor $\sigma/\ep^\threebytwo$. 


\subsection{\texorpdfstring{Discussion}{}}\label{sec:discussion} \hspace*{\fill} \medskip

We now give a brief heuristic explanation into the nature of our results. 
\begin{enumerate}[leftmargin =*, align=left, label=\arabic*)]
\item Taylor sign condition:

In Ambrose and Masmoudi \cite{AmMa05} the Taylor sign condition $ -\frac{\partial P}{\partial \hat{n}} \geq c >0$ is assumed for all $0\leq \sigma \leq \sigma_0$ which forces $\sigma_0$ to be small. To see this why this assumption was made, one can do a similar computation as was done in \cite{AmMa05} in conformal coordinates and obtain a quasilinear equation of the form
\begin{align}\label{eq:quasig}
\cbrac{\Dt^2  + i\Hil\brac{\frac{\Aonesigma}{\Zapabs^2}\pap} - i\sigma\Hil \brac{\frac{1}{\Zapabs}\pap}^3} g = l.o.t
\end{align}
where $g = \thvar\compose\hinv$ and $i\Hil\pap = \papabs$. Hence heuristically the equation looks like
\begin{align*}
\cbrac{\pt^2 + \brac{\frac{\Aonesigma}{\Zapabs}}\papabs + \sigma\papabs^3} \g = l.o.t
\end{align*}
Note that the coefficient in the second term is Taylor sign condition term namely $\frac{\Aonesigma}{\Zapabs} = -\frac{\partial P}{\partial \hat{n}}\compose\hinv$ and is given by the formula \eqref{form:Taylor}.  Hence to obtain a positive energy from this equation, the Taylor sign condition at $t=0$ was assumed i.e. $\frac{\Aonesigma}{\Zapabs}  \geq c>0$. This condition is satisfied if $\sigma$ is small enough depending on the initial data, but will fail generically if $\sigma$ is large. 

In Shatah and Zeng \cite{ShZe08} it is observed that one can in fact obtain an energy with the coefficient being $ -\frac{\partial P}{\partial \hat{n}} \big\vert_{\sigma=0}$ instead of $ -\frac{\partial P}{\partial \hat{n}}$ which removes the smallness restriction on $\sigma$. However this was achieved using a variational formulation and as we are interested in studying angled crests, it is not clear how to use that approach in our problem. 

To overcome this issue we make the simple  observation that if we take an arc length derivative $\Dapabs$ or a material derivative $\Dt$ to the equation \eqref{eq:quasig}, then the structure of the equation improves. In fact by taking these derivatives to the equation \eqref{eq:quasig} one obtains an equation of the form
\begin{align}\label{eq:quasif}
\cbrac{\Dt^2  + i\Hil\brac{\frac{\Aone}{\Zapabs^2}\pap} - i\sigma\Hil \brac{\frac{1}{\Zapabs}\pap}^3} f = l.o.t
\end{align}
where $f = \Dapabs\g$ or $f = \Dt\g$. Note crucially that the second term now has $\Aone$ instead of $\Aonesigma$. As $\Aone\geq 1$ irrespective of $\sigma$, we do not need to make any assumptions on $\sigma$ being small.  Note that $ \Dapabs\g = \Real \Th $ and $\Dt\g = -\Imag(\Dapbar\Ztbar)$ from \eqref{form:Dtg} and we derive equations for $\Th$ and $\Dapbar\Ztbar$ in \eqref{eq:Th} and \eqref{eq:DapbarZtbar} respectively. Of course we still need to deal with the issue that $ -\frac{\partial P}{\partial \hat{n}} \big\vert_{\sigma=0} = \frac{\Aone}{\Zapabs}$ does not have a positive lower bound. This is resolved in a similar manner as Kinsey and Wu \cite{KiWu18} and is explained below.

\item Heuristic energy estimate:

The main goal of this paper is to extend the work of Kinsey and Wu \cite{KiWu18} to the case of non-zero surface tension. In \cite{KiWu18} a priori estimates are given for angled crested interfaces in the case of $\sigma=0$ for angles $\nu\pi$ with $0<\nu<\half$. Let us do a heuristic energy estimate to understand the difficulties. If the interface is $C^{1,\alpha}$ then we have  $0<c_1 \leq \frac{1}{\Zapabs} \leq c_2 < \infty$ and hence for smooth enough interfaces the main operator in \eqref{eq:quasif} behaves like $\pt^2 + \papabs + \sigma \papabs^3$ for which standard energy estimates in Sobolev spaces work. If the interface has an angled crest of angle $\nu\pi$ at $\ap=0$, then $\Z(\ap) \sim (\ap)^\nu$ and hence $ \frac{1}{\Zapabs}(\ap) \sim \abs{\ap}^{1-\nu}$ near $\ap= 0$ and hence the quasilinear equation near $\ap =0$ behaves like
\begin{align}\label{eq:quasiheuristic}
\begin{split}
& \cbrac{\pt^2 + \abs{\ap}^{2-2\nu}\papabs + \sigma\abs{\ap}^{3-3\nu}\papabs^3}f \\
 & = \abs{\ap}^{1-2\nu}f + \sigma \abs{\ap}^{2-3\nu}\papabs^2f + \sigma \abs{\ap}^{1-3\nu}\papabs f + \sigma\abs{\ap}^{-3\nu}f  + \tx{other l.o.t}
\end{split}
\end{align}
We have included a few simplified versions of the lower order terms to demonstrate the issues in proving an energy estimate. We obtain our energies by multiplying the above equation with either $\pt f$ or $\papabs\pt f$ and then integrating. If we multiply the equation by $\pt f$ and integrate, we obtain the energy
\begin{align}\label{eq:energyheuristic}
E_\sigma(t) = \half\norm[2]{\pt f}^2 + \half\norm[\Hhalf]{\abs{\ap}^{1-\nu} f}^2 + \half\sigma\norm[\Hhalf]{\abs{\ap}^{\threebytwo - \threebytwo\nu} \papabs f}^2
\end{align}
For simplicity we can also add the term $\half\norm[2]{f}^2$ to the energy which is compatible with $\pt f \in \Ltwo$ obtained from the energy. To close the energy estimate, we need to control the $\Ltwo$ norm of the right hand side of \eqref{eq:quasiheuristic} (we ignore the commutator terms in this heuristic). Hence to control the first term $\abs{\ap}^{1-2\nu} f \in \Ltwo$ , we obtain the restriction $\nu\leq \half$ which is one of the main reasons of the restrictions on the angles in Kinsey and Wu \cite{KiWu18}. Note that we cannot control the term $\sigma\abs{\ap}^{2-3\nu}\papabs^2 f \in \Ltwo$ as we only have control of $3/2$ derivatives on $f$. For smooth enough interfaces, it was observed by Ambrose-Masmoudi \cite{AmMa05} that by carefully choosing $f$ (by using variables derived from $\thvar$), this term does not appear in the quasilinear equation and we follow the same approach.  We do not use the modified tangential velocity as in \cite{AmMa05} but instead use the more natural  material derivative $\Dt$ along with the variable $\thvar$ to obtain our highest order quasilinear equation. Hence the equation looks like
\begin{align}\label{eq:quasiheuristictwo}
\begin{split}
& \cbrac{\pt^2 + \abs{\ap}^{2-2\nu}\papabs + \sigma\abs{\ap}^{3-3\nu}\papabs^3}f \\
 & = \abs{\ap}^{1-2\nu}f  + \sigma \abs{\ap}^{1-3\nu}\papabs f + \sigma\abs{\ap}^{-3\nu}f  + \tx{other l.o.t}
\end{split}
\end{align}

Now we need to control $\sigma \abs{\ap}^{1-3\nu}\papabs f \in \Ltwo$ and $\sigma\abs{\ap}^{-3\nu}f \in \Ltwo$. As we only have $f\in \Ltwo$, there is no way we can control the term $\sigma\abs{\ap}^{-3\nu}f \in \Ltwo$ and this is the reason why we do not allow angled crest data if $\sigma>0$. If we work with the smooth interface $\Z^\ep = \Z*P_\ep$ where $P_\ep$ is the Poisson kernel, then this has the effect of changing $\abs{\ap} \mapsto \abs{-i\ep + \ap}$ near $\ap =0$. Hence to close the energy estimate, we obtain the restriction $\sigma\ep^{-3\nu} \lesssim 1$. For the interface $\Z^\ep$, the curvature $\kap \sim \ep^{-\nu}$ and hence this can be written as $\sigma\kap^3 \lesssim 1$. A similar argument for $\sigma \abs{\ap}^{1-3\nu}\papabs f \in \Ltwo$ also yields the same restriction. Note that this is the exact scaling as was mentioned in the introduction and \remref{rem:scaling}. Also note that these restrictions do not depend on the choice of $f$, but is purely a consequence of the structure of the quasilinear equation and attempting to prove an $\Ltwo$ based energy estimate.

The main goal is now to construct and prove an energy estimate which mimics this heuristic energy estimate. A key difficulty in implementing this is to find a suitable $f$ and obtain a corresponding quasilinear equation. We need the following properties for $f$
\begin{enumerate}
\item $f$ needs to have appropriate weights so that the energy such as \eqref{eq:energyheuristic} allows angled crests solutions of angles less than $90\degree$ when $\sigma=0$
\item The quasilinear equation for $f$ should not have terms like $\sigma\abs{\ap}^{2-3\nu}\papabs^2 f$ in the errors, i.e. it should look like \eqref{eq:quasiheuristictwo}  and not \eqref{eq:quasiheuristic}
\end{enumerate}
These are rather severe restrictions on $f$ and it is not clear whether one can find such a function. Observe that if $f$ has such properties, then no weighted derivate of the form $w\pal f$ will satisfy the same properties (Kinsey-Wu \cite{KiWu18} used the weighted derivative $\frac{1}{\Zapabs^2}\pap$ and Ambrose-Masmoudi \cite{AmMa05} used the arc length derivative $\frac{1}{\Zapabs}\pap$ to obtain the higher order energies).  Fortunately we observe that the material derivative $\Dt f$ will satisfy both the properties if $f$ satisfies both properties. In this paper we use $f = \Dapbar\Ztbar$ as the variable for the energy of the highest order in $\Esigma$ and it is related to the material derivative of the angle via $\Dt\g = -\Imag(\Dapbar\Ztbar)$.

\item Constructing the energy:

Observe that the quasilinear equations we derive in \secref{sec:quasilinear} are essentially of the form
\begin{align}\label{eq:quasifthree}
\cbrac{\Dt^2  + i\frac{\Aone}{\Zapabs}\brac{\frac{1}{\Zapabs}\pal}\Hil - i\sigma \brac{\frac{1}{\Zapabs}\pap}^3\Hil} f = l.o.t
\end{align}
for suitable $f$. To obtain an energy we multiply by $\papabs\Dt f$ and integrate to get the energy
\begin{align*}
\norm[\Hhalf]{\Dt f}^2 + \norm[2]{\frac{\sqrt{\Aone}}{\Zapabs}\pap f}^2 + \sigma\norm[2]{\frac{1}{\Zapabs^\half}\pap\Dapabs f}^2
\end{align*}
Observe that as $\Dt\g = -\Imag(\Dapbar\Ztbar)$ we have from \eqref{form:Imag} that $\Ph\Dt\g \approx \frac{i}{2}\Dapbar\Ztbar$. The energy $\Esigma$ defined in \secref{sec:aprioriEsigma} consists of the following
\begin{enumerate}
\item $\Esigmazero$: ad-hoc but carefully chosen lower order terms introduced to help in closing the energy estimate but which still allows the scaling $\sigma/\ep^\threebytwo$ in \corref{cor:example}. \footnote{For example if we introduce the lower order term $\norm[\infty]{\sigma^\onebythree\Zap} \sim \norm[\infty]{T_b^{-1}T_c\Ph(\g)}$ in $\Esigmazero$, then it significantly simplifies the proof of the energy estimate and still allows angled crested interfaces for $\sigma = 0$. However the introduction of this term does not allow the scaling $\sigma/\ep^\threebytwo$ in \corref{cor:example} and one would then get a weaker result.}
\item $\Esigmaone$: Energy as above for $f \approx \Ztbar\Zap \approx \brac{\frac{1}{\Zapabs^2}\pap}^{-1}\Ph \Dt\g$
\item $\Esigmatwo$: Energy as above for $f \approx \papabs^\half\nobrac{\Ztbar} \approx \brac{\frac{1}{\Zapabs^2}\pap}^{-\half}\Ph \Dt\g$
\item $\Esigmathree$: Energy as above for $f \approx \papabs^{-\half}\Th \approx \brac{\frac{1}{\Zapabs^2}\pap}^{\half}\Ph \g$
\item $\Esigmafour$: Energy as above for $f \approx \Dapbar\Ztbar \approx -2i\Ph \Dt\g$
\end{enumerate}

\medskip
Observe that the variables used as $f$ consist of applying the operators $\Dt$ and $\frac{1}{\Zapabs^2}\pap$ on the variable $\g$. These show up as they are part of the fundamental operators in the quasilinear equation \eqref{eq:quasifthree} which are namely 
\begin{align*}
T_a = \Dt = \tx{ material derivative } \qq T_b = \frac{1}{\Zapabs^2}\pap \qq \tx{ and }\quad  T_c = \frac{\sigma^\onebythree}{\Zapabs}\pap
\end{align*}
where we have ignored $\Aone$ as we prove in the energy estimate that $1\leq \Aone \leq 1 + \norm[2]{\Ztapbar}^2$ and hence we can consider $\Aone \approx 1$. All terms in the energy $\Ecalsigma$ and in fact all the terms showing up in \secref{sec:quantEsigma} can be written entirely in terms of these operators and $\g$ (and also using the operator $\Ph$). 

To understand how the energy looks like in the arc length coordinate system, we define the operators
\begin{align*}
T_1 = \Dt = \tx{ material derivative } \qq T_2 =  \atay\ps \qq \tx{ and }\quad  T_3 = \sigma^\onebythree\ps
\end{align*}
where $\dis \atay = -\frac{\partial P}{\partial \hat{n}}\bigg\rvert_{\sigma=0, \mathbf{v} \equiv 0}$ and $\ps$ is the arc length derivative. The $\Dt$ above is by abuse of notation the material derivative in arc length coordinates. It is clear from \eqref{form:Taylor} that $\atay$ in conformal coordinates is exactly $\frac{1}{\Zapabs}$ and we also see that the arc length derivate $\ps$ in conformal coordinates is the operator $\frac{1}{\Zapabs}\pap$. To get the analog of $\Ph$ in arc length coordinates, we let $\Phol$ denote the linear operator defined by the property that for any smooth real valued function $f:\Sigma\to \Rsp$ vanishing at infinity, $\Phol(f):\Sigma\to\Csp$ is the boundary value of a holomorphic function on $\Omega$ vanishing at infinity with $\Real\{\Phol(f) \} = f/2$.

We can now give a heuristic representation of $\Ecalsigma$ in both conformal and arc length coordinate systems using these operators. Let $g, \thvar$ be the angle of the interface with the x-axis in conformal and arc length coordinates respectively. We have the following representation for $\Ecalsigmaone$: 

\begin{enumerate}[leftmargin =*, align=left, label=\arabic*)]
\item $\dis \norm[2]{\pap\frac{1}{\Zap}} \sim \norm[\Hhalf]{T_b^\half\Ph(\g)} \sim \norm[\Hhalf]{T_2^{\half}\Phol(\thvar)}  $ 

\item $\dis \norm[\Hhalf]{\frac{1}{\Zap}\pap\frac{1}{\Zap}} \sim \norm[\Hhalf]{T_b\Ph(\g)} \sim \norm[\Hhalf]{T_2\Phol(\thvar)}$

\item $\dis \norm[2]{\sigma^\onebysix\Zapabs^{\half}\pap\frac{1}{\Zap}}  \sim \norm[\Hhalf]{T_c^\half\Ph(\g)}  \sim \norm[\Hhalf]{T_3^\half\Phol(\thvar)} $

\item $\dis \norm[\infty]{\sigma^\half\Zapabs^\half\pap\frac{1}{\Zap}} \sim \norm[\infty]{T_b^{-\half}T_c^{\threebytwo}\Ph(\g)} \sim \norm[\infty]{T_2^{-\half}T_3^{\threebytwo}\Phol(\thvar)}$

\item  $\dis \norm*[\Bigg][2]{\frac{\sigma^\half}{\Zap^\half}\pap^2\frac{1}{\Zap}} \sim \norm[\Hhalf]{T_c^{\threebytwo}\Ph(\g)}  \sim \norm[\Hhalf]{T_3^{\threebytwo}\Phol(\thvar)} $

\item $\dis \norm*[\Bigg][\Hhalf]{\frac{\sigma^\half}{\Zap^\threebytwo}\pap^2\frac{1}{\Zap}} \sim \norm[\Hhalf]{T_b^\half T_c^\threebytwo \Ph(\g)}  \sim \norm[\Hhalf]{T_2^\half T_3^\threebytwo \Phol(\thvar)} $

\item $\dis \norm[\Hhalf]{\sigma\pap\Th} \sim \norm[\Hhalf]{T_b^{-1}T_c^3 \Ph(\g)} \sim \norm[\Hhalf]{T_2^{-1}T_3^3 \Phol(\thvar)}$

\item $\dis \norm[2]{\frac{\sigma}{\Zap}\pap^3\frac{1}{\Zap}} \sim \norm[\Hhalf]{T_b^{-\half}T_c^3 \Ph(\g)} \sim \norm[\Hhalf]{T_2^{-\half}T_3^3 \Phol(\thvar)} $

\item $\dis \norm[\Hhalf]{\frac{\sigma}{\Zap^2}\pap^3\frac{1}{\Zap}} \sim \norm[\Hhalf]{T_c^3 \Ph(\g)} \sim \norm[\Hhalf]{T_3^3 \Phol(\thvar)} $

\end{enumerate}

As mentioned before we have $\Ph\Dt\g \approx \frac{i}{2}\Dapbar\Ztbar$. Hence $\Ecalsigmatwo$ has the heuristic representation 

\begin{enumerate}[leftmargin =*, align=left, label=\arabic*)]
\item $\dis \norm[2]{\Ztapbar} \sim \norm[\Hhalf]{T_aT_b^{-\half}\Ph(\g)} \sim \norm[\Hhalf]{T_1T_2^{-\half}\Phol(\thvar)}$

\item $\dis  \norm*[\Bigg][2]{\frac{1}{\Zap^2}\pap\Ztapbar} \sim \norm[\Hhalf]{T_aT_b^{\half}\Ph(\g)}  \sim \norm[\Hhalf]{T_1T_2^{\half}\Phol(\thvar)} $

\item $\dis \norm*[\Bigg][2]{\frac{\sigma^\half}{\Zap^\half}\pap\Ztapbar} \sim \norm[\Hhalf]{T_aT_b^{-1}T_c^{\threebytwo}\Ph(\g)} \sim \norm[\Hhalf]{T_1T_2^{-1}T_3^{\threebytwo}\Phol(\thvar)}$

\item $\dis \norm*[\Bigg][2]{\frac{\sigma^\half}{\Zap^\fivebytwo}\pap^2\Ztapbar} \sim \norm[\Hhalf]{T_aT_c^{\threebytwo}\Ph(\g)} \sim \norm[\Hhalf]{T_1T_3^{\threebytwo}\Phol(\thvar)}$
\end{enumerate}
The operators $\Ph$ and $\Phol$ are not fundamental in the representation and are more of a technical nature. One can ignore them to get the essence of the energy. Also heuristically for the arc length representation, one can write $T_1 \theta$ as the arc length derivative of the velocity on the boundary. For example if the velocity on the boundary in arc length coordinates is $\zt$, then one can heuristically write $\dis \norm[2]{\Ztapbar}  \sim \norm[\Hhalf]{T_1T_2^{-\half}\Phol(\thvar)} \sim \norm[\Hhalf]{T_2^{-\half}\Phol(\ps \zt)}$. Similarly for other terms of $\Ecalsigmatwo$. 

Writing the energy with variables being the free surface elevation $\eta$ and the velocity potential $\phi$ is a little more complicated in general due to the restriction that the interface has to be a graph. However if for example there is a single singularity of the interface at the origin and the data is symmetric, then the energy in these variables can be written easily by converting it from the representation in arc length coordinates. However even then one has to be a little careful as the energy $\Ecalsigma$ for $\sigma = 0$ allows interfaces with cusps (as shown in \cite{Ag20}), which would make the slope of the interface at the cusp infinite and the interface non chord-arc, and this may create some difficulties in proving a local wellposedness result in these coordinates.

\item Analytical difficulties: 

In addition to the structural difficulties due to surface tension explained above, we also face numerous analytical difficulties. Even in the special case of $\sigma=0$, the energy $\Ecalsigma\vert_{\sigma=0}$ is lower order as compared to the energy in Kinsey-Wu \cite{KiWu18} by half weighted spacial derivative and we crucially do not have $\Dap\Zttbar \in \Linfty$. This makes our estimates even in the case of $\sigma=0$ much more subtle. In addition we now have a lot of nonlinear terms due to surface tension which we need to control. To overcome these issues, we define weighted function spaces adapted to our problem and prove estimates for these spaces (see \lemref{lem:CW}). We use these function spaces along with estimates from harmonic analysis to control the nonlinear terms.

\end{enumerate}

\section{The quasilinear equations}\label{sec:quasilinear}

We will now use the fundamental equation $\eqref{form:Zttbar}$ as the starting point to derive our quasilinear equations. Our main equation is for the variable $\Dapbar \Ztbar$ which is obtained by applying the operators $\Dapbar\Dt$ to the equation $\eqref{form:Zttbar}$. We also obtain equations for $\Ztbar, \Ztapbar$ and $\Th$ which should be thought of lower order and auxiliary equations. Let us first derive some simple but useful formulas:

\begin{enumerate}[leftmargin =*, align=left, label=\alph*)]
\item We have
\[
\frac{\Zap}{\Zapabs}\pap\frac{1}{\Zap} = \w\pap\brac{\frac{\wbar}{\Zapabs}} = \pap \frac{1}{\Zapabs} + \w\Dapabs \wbar
\]
Observe that $\dis \pap \frac{1}{\Zapabs}$ is real valued and $\dis \w\Dapabs \wbar$ is purely imaginary. From this we obtain the relations
\begin{align} \label{form:RealImagTh}
\Real\brac{\frac{\Zap}{\Zapabs}\pap\frac{1}{\Zap} } = \pap \frac{1}{\Zapabs} \quad \qquad \Imag\brac{\frac{\Zap}{\Zapabs}\pap\frac{1}{\Zap}} = i\brac{ \wbar\Dapabs \w} = -\Real \Th
\end{align}

\item From \eqref{eq:mainvariables} we see that $\g = \Imag\cbrac{\log (\Zap)}$ and hence from \eqref{eq:systemone} we see that
\begin{align*}
\Dt \g = \Imag\cbrac{\frac{\Dt\Zap}{\Zap}} = \Imag\cbrac{\frac{\Ztap - \bvarap\Zap}{\Zap}}
\end{align*}
Now as $\bvarap$ is real valued we have
\begin{align}  \label{form:Dtg}
\Dt \g = - \Imag (\Dapbar \Ztbar) 
\end{align}

\item Observe that for any complex valued function $f$, $\Hil (\Real f) =  i\Imag(\Hil f)$ and  $\Hil(i \Imag f) = \Real (\Hil f)$. Hence we get the following useful identities
\begin{align}
(\Id + \Hil)(\Real f) = f -i\Imag (\Id - \Hil) f  \label{form:Real}\\
(\Id + \Hil)(i \Imag f) = f - \Real (\Id - \Hil)f  \label{form:Imag}
\end{align}
Now we observe that 
\begin{align*}
\Aone  = 1 -\Imag \sqbrac{\Zt,\Hil}\Ztapbar  = \Real(\Id - \Hil)\cbrac{1 + i\Zt\Ztapbar}
\end{align*}
and hence using \eqref{form:Imag} we obtain
\begin{align}\label{form:Aonenew}
\Aone = 1 + i\Zt\Ztapbar -i(\Id + \Hil)\cbrac{\Real(\Zt\Ztapbar)}
\end{align}
Similarly as $\bvar = \Real(\Id - \Hil)\brac{\frac{\Zt}{\Zap}}$ we again use \eqref{form:Imag} to get
\begin{align*}
\bvar = \frac{\Zt}{\Zap} -i(\Id + \Hil)\cbrac{\Imag\brac{\frac{\Zt}{\Zap}}}
\end{align*}
and hence 
\begin{align}\label{form:bvarapnew}
\bvarap = \Dap\Zt + \Zt\pap\frac{1}{\Zap}  -i\pap(\Id + \Hil)\cbrac{\Imag\brac{\frac{\Zt}{\Zap}}}
\end{align}

\item We now record some frequently used commutator identities. They are easily seen by differentiating
\begin{align}\label{eq:commutator}
 [\pap, \Dt ] &= \bap \pap   \qquad &[\Dapabs, \Dt] &= \Real{(\Dap \Zt)} \Dapabs = \Real{(\Dapbar \Ztbar) \Dapabs} \\   \relax 
 [\Dap, \Dt]  &=  \brac{\Dap \Zt} \Dap   \qquad &  [\Dapbar, \Dt] &= \brac{\Dapbar \Ztbar } \Dapbar 
\end{align}
Using the commutator relation ${[\pap, \Dt ]} = \bap \pap$ we obtain the following formulae
\begin{align}
\Dt \Zapabs & = \Dt e^{\Real\log \Zap} =  \Zapabs \cbrac{\Real(\Dap\Zt) - \bvarap} \label{form:DtZapabs} \\ 
 \Dt \frac{1}{\Zap} & = \frac{-1}{\Zap}(\Dap\Zt - \bvarap) = \frac{1}{\Zap}\cbrac{(\bvarap - \Dap\Zt - \Dapbar \Ztbar) + \Dapbar \Ztbar} \label{form:DtoneoverZap}
\end{align}
Observe that $(\bvarap - \Dap\Zt - \Dapbar \Ztbar)$ is real valued and this fact will be useful later on.
\medskip

 \item As we will frequently work with the operator $\Dapabs^3$ we record some commonly used expansions
 \begin{align} 
 \Dapabs^2 f & = \brac{\pap\frac{1}{\Zapabs}}\Dapabs f + \frac{1}{\Zapabs^2}\pap^2 f   \notag\\
 \begin{split} \label{form:Dapabs^3f}
 \Dapabs^3 f & =   \brac{\frac{1}{\Zapabs}\pap^2\frac{1}{\Zapabs}}\Dapabs f + \brac{\pap\frac{1}{\Zapabs}}^2 \Dapabs f + 3\brac{\pap\frac{1}{\Zapabs}}\frac{1}{\Zapabs^2}\pap^2 f  \\
 &  \quad + \frac{1}{\Zapabs^3}\pap^3 f 
 \end{split} \\
 \begin{split} \label{form2:Dapabs^3f}
 \Dapabs^3 f & = \pap\cbrac{\frac{1}{\Zapabs^\threebytwo}\pap\brac{\frac{1}{\Zapabs^\threebytwo}\pap f}} -\frac{3}{2}\brac{\pap\frac{1}{\Zapabs}}\frac{1}{\Zapabs^2}\pap^2 f \\
 & \quad -2\brac{\pap\frac{1}{\Zapabs}}^2\Dapabs f -\frac{1}{2}\brac{\frac{1}{\Zapabs}\pap^2\frac{1}{\Zapabs}}\Dapabs f
 \end{split}
 \end{align}

\end{enumerate}
\medskip

We will now derive formulas for $\Th, \Dt \Th$ and $\Dt^2 \Th$. All three of them are derived similarly. 
\medskip


\subsubsection{}$\!\!\!\mathbf{Formula \ for\  } \Th$
\medskip

We know from \eqref{form:RealImagTh} that $\dis \Real \Th = -\Imag\brac{\frac{\Zap}{\Zapabs} \pap \frac{1}{\Zap}}$. Applying $(\Id + \Hil)$ to this formula and using the identities \eqref{form:Real} and \eqref{form:Imag} we get
\begin{align} \label{form:Th}
\Th = i \frac{\Zap}{\Zapabs} \pap \frac{1}{\Zap} - i \Real (\Id - \Hil) \brac{\frac{\Zap}{\Zapabs} \pap \frac{1}{\Zap}} 
\end{align}
As $\pap \frac{1}{\Zap}$ is holomorphic, this implies that the second term in the above formula is lower order. Hence $\Th \approx  i\frac{\Zap}{\Zapabs} \pap \frac{1}{\Zap}  $ and therefore $\Th $ and $\pap \frac{1}{\Zap}$ have the same regularity. 
\medskip


\subsubsection{}$\!\!\!\mathbf{Formula \ for\  } \Dt\Th$
\medskip

Apply $\Dapabs$ on the formula for $\Dt \g$ in \eqref{form:Dtg} to obtain
\begin{align*}
\Dt \Dapabs \g & = - \Imag{(\Dapabs \Dapbar \Ztbar)} -\Real{(\Dapbar \Ztbar)} \Dapabs \g 
\end{align*}
As $\Dapabs \g = \Real \Th$, hence $\Real{(\Dapbar \Ztbar)} \Dapabs \g = \Real{\cbrac{(\Dapbar \Ztbar) \Real \Th}} = \Imag{\cbrac{i (\Dapbar \Ztbar) \Real \Th}}$. Also observe that $\Dt\Dapabs \g = \Real (\Dt \Th)$. Hence we have
\begin{align}  \label{form:RealDtTh}
\Real (\Dt \Th)   = - \Imag\cbrac{(\Dapabs + i \Real \Th) \Dapbar \Ztbar}
\end{align}
Now apply $(\Id + \Hil)$ on both sides and use the identities \eqref{form:Real} and \eqref{form:Imag} to get
\begin{align} \label{form:DtTh}
\Dt \Th = i\brac{\Dapabs + i \Real \Th} \Dapbar \Ztbar - i \Real (\Id - \Hil)\cbrac{(\Dapabs + i \Real \Th) \Dapbar \Ztbar}  + i \Imag \nobrac{(\Id - \Hil) \Dt \Th} 
\end{align}
Note that $\Dap \Ztbar$ and $\Th $ are holomorphic this causes the second and third term in the above formula to be of lower order. For example observe that as $(\Id - \Hil)\Th =0$, we have
\begin{align*}
(\Id - \Hil)\Dt\Th = \sqbrac{\Dt,\Hil}\Th = \sqbrac{\bvar,\Hil}\pap\Th
\end{align*}
Hence $\Dt \Th \approx i(\Dapabs + i \Real \Th) \Dapbar \Ztbar$.

\medskip


\subsubsection{}$\!\!\!\mathbf{Formula \ for\  } \Dt^2\Th$
\medskip

Apply $\Dt$ on the formula for $\Real( \Dt \Th)$ in  \eqref{form:RealDtTh} to obtain
\begin{flalign*}
\enspace \Real (\Dt^2\Th)   &= - \Imag\cbrac{\Dt(\Dapabs + i \Real \Th) \Dapbar \Ztbar} & \\
  & = - \Imag\cbrac{(\Dapabs + i \Real \Th) \Dt\Dapbar \Ztbar} + \Imag\cbrac{\Real(\Dapbar\Ztbar)\Dapabs\Dapbar\Ztbar  -i\Real (\Dt \Th)\Dapbar\Ztbar}
\end{flalign*}
Now apply $(\Id + \Hil)$ on both sides and use the identities \eqref{form:Real} and \eqref{form:Imag} to get
\begin{align} \label{form:Dt^2Th}
\begin{split}
\Dt^2\Th &= i(\Dapabs + i \Real \Th)\Dt \Dapbar \Ztbar - i \Real (\Id - \Hil)\cbrac{(\Dapabs + i \Real \Th) \Dt\Dapbar \Ztbar} \\
  & \quad  +  i \Imag \nobrac{(\Id - \Hil) \Dt^2 \Th} +  (\Id + \Hil) \Imag\cbrac{\Real(\Dapbar\Ztbar)\Dapabs\Dapbar\Ztbar  -i\Real (\Dt \Th)\Dapbar\Ztbar} 
\end{split}
\end{align}
Again in this formula only the first term is the main term and all other terms are lower order. Hence $\Dt^2 \Th \approx  i(\Dapabs + i \Real \Th)\Dt \Dapbar \Ztbar $.
 
\medskip


\subsection{\texorpdfstring{Equation for $\Ztbar$ \nopunct}{}} \hspace*{\fill} \medskip

Apply $\Dt$ to the fundamental equation \eqref{form:Zttbar}
\begin{align*}
\Ztttbar  = -i \frac{\Dt\Aone}{\Zap} -i\frac{\Aone}{\Zap}\brac{\Zap \Dt\frac{1}{\Zap}} -\sigma \brac{\Dap\Zt}\Dap\Th +\sigma\Dap\Dt\Th 
\end{align*}
Now use the formula for $\dis \Dt \frac{1}{\Zap}$ from \eqref{form:DtoneoverZap} and $\Dt\Th$ from \eqref{form:DtTh} to obtain
\begin{align*}
\begin{split}
\Ztttbar & = -i\frac{1}{\Zap}\brac{\Dt\Aone +\Aone(\bap -\Dap\Zt - \Dapbar\Ztbar)} -i\frac{\Aone}{\Zap} \Dapbar\Ztbar -\sigma \brac{\Dap\Zt}\Dap\Th \\
 & \quad +i\sigma\Dap(\Dapabs + i\Real \Th)\Dapbar\Ztbar -i\sigma\Dap\Real(\Id - \Hil) \cbrac{(\Dapabs +i\Real\Th)\Dapbar \Ztbar}  \\
 & \quad +i\sigma\Dap\Imag(\Id - \Hil) \Dt\Th \\
\end{split}
\end{align*}
Let us define the real valued variable $\Jone$ as
\begin{align} \label{form:Jone}
\begin{split}
\Jone &=  \Dt\Aone +  \Aone\brac{\bvarap - \Dap\Zt - \Dapbar\Ztbar} + \sigma \pap \Real \brac{\Id - \Hil} \cbrac{(\Dapabs + i\Real \Th) \Dapbar \Ztbar} \\
& \quad - \sigma\pap \Imag\brac{\Id - \Hil} \Dt\Th
\end{split}
\end{align}
Using this we get
\begin{align}\label{eq:Ztbar}
\Ztttbar + i\frac{\Aone}{\Zap}\Dapbar\Ztbar -i\sigma\Dap(\Dapabs + i\Real\Th)\Dapbar \Ztbar = -\sigma(\Dap\Zt) \Dap\Th -i\frac{\Jone}{\Zap}
\end{align}
We modify this equation slightly to get an equation appropriate for the computation of the lower order term in the energy. Rewrite the above equation as
\begin{align*}
\begin{split}
& \Ztttbar + i\frac{\Aone}{\Zap}\Dapbar\Ztbar -i\sigma\Dap\cbrac{\brac{\Dapabs\frac{1}{\Zapbar}} \Ztapbar + \frac{1}{\Zapbar}\Dapabs\Ztapbar } \\
& = -\sigma(\Dap\Zt)\Dap\Th -\sigma\Dap\cbrac{(\Real\Th) \Dapbar \Ztbar} - i\frac{\Jone}{\Zap}
\end{split}
\end{align*}
Multiply by $\Zap$ and rearrange to get
\begin{align} \label{eq:ZtZap}
\begin{split}
& \Ztttbar\Zap + i\Aone\Dapbar\Ztbar -i\sigma\pap\brac{\frac{1}{\Zapbar}\Dapabs\Ztapbar  } \\
& = i\sigma\pap\cbrac{\brac{\Dapabs \frac{1}{\Zapbar}}\Ztapbar } -\sigma(\Dap\Zt)\pap\Th -\sigma\pap\cbrac{(\Real\Th)\Dapbar\Ztbar} -i\Jone
\end{split}
\end{align}
This equation gives rise to the energy $\Esigmaone$ in the energy estimate. 
\medskip

\subsection{\texorpdfstring{Equation for $\Dapbar \Ztbar$ \nopunct}{}} \hspace*{\fill} \medskip

Apply $\Dapbar$ to the equation \eqref{eq:Ztbar} and use commutator identities to get
\begin{align*}
\begin{split}
& \Dt^2\Dapbar\Ztbar + i\frac{\Aone}{\Zapabs^2}\pap\Dapbar\Ztbar -i\sigma \Dapbar \Dap (\Dapabs +i\Real\Th)\Dapbar \Ztbar \\
& = - (\Dapbar\Ztbar)\Dapbar\Zttbar -2(\Dapbar\Ztbar)(\Dt\Dapbar\Ztbar) -i\Dapbar\brac{\frac{\Aone}{\Zap}}(\Dapbar\Ztbar) \\
& \quad -\sigma \Dapbar\cbrac{(\Dap\Zt)\Dap\Th } -i\brac{\Dapbar\frac{1}{\Zap}} \Jone - \frac{i}{\Zapabs^2}\pap\Jone
\end{split}
\end{align*}
Observe from \eqref{form:Zttbar} that $-(\Dapbar\Ztbar)\Dapbar\brac{\Zttbar +i\frac{\Aone}{\Zap}} = -\sigma(\Dapbar\Ztbar) \Dapbar\Dap\Th$. Hence we get
\begin{align*}
\begin{split}
& \Dt^2\Dapbar\Ztbar + i\frac{\Aone}{\Zapabs^2}\pap\Dapbar\Ztbar -i\sigma \Dapbar \Dap (\Dapabs +i\Real\Th)\Dapbar \Ztbar \\
& = -2(\Dapbar\Ztbar)(\Dt\Dapbar\Ztbar) -2\sigma\Real(\Dap\Zt)\Dapbar\Dap\Th -\sigma(\Dapbar\Dap\Zt)\Dap\Th \\
& \quad -i\brac{\Dapbar\frac{1}{\Zap}} \Jone - \frac{i}{\Zapabs^2}\pap\Jone
\end{split}
\end{align*}
Now
\begin{align}\label{eq:DapbarDap}
 \Dapbar \Dap  = \Dapbarfrac \brac{\frac{\Zapabs}{\Zap} \Dapabs} = \brac{\Dapbarfrac \frac{\Zapabs}{\Zap} } \Dapabs + \Dapabs^2 = \brac{\Dapabs -i\Real \Th} \Dapabs 
\end{align}
\begin{flalign*}
\text{Hence} \quad & \Dapbar \Dap  \brac{\Dapabs +i\Real \Th}  &\\
& = \brac{\Dapabs -i \Real \Th} \Dapabs \brac{\Dapabs +i\Real \Th} &\\
& =  \brac{\Dapabs -i \Real \Th} \brac{\Dapabs +i \Real \Th} \Dapabs + \brac{\Dapabs -i \Real \Th} \brac{i\Real \brac{\Dapabs \Th}} &\\
& = \brac{\Dapabs^2 +i\Real \brac{\Dapabs \Th} + \brac{\Real \Th}^2}\Dapabs + i\Real\brac{\Dapabs^2\Th} + i\Real\brac{\Dapabs \Th}\Dapabs &\\
 & \quad + \brac{\Real \Th} \nobrac{\Real\brac{\Dapabs \Th}} &\\
 & = \Dapabs^3 + \brac{2i\Real \brac{\Dapabs \Th} + \brac{\Real \Th}^2} \Dapabs + i\Real \brac{\Dapabs^2 \Th} + \brac{\Real \Th} \Real \brac{\Dapabs \Th}&
\end{flalign*}
Therefore $ \Dapbar \Dap  \brac{\Dapabs +i\Real \Th} \approx  \Dapabs^3$. Hence we get the main equation for $\Dapbar\Ztbar$
\begin{equation} \label{eq:DapbarZtbar}
 \brac{\Dt^2 +i\frac{\Aone}{\Zapabs^2}\pap  -i\sigma\Dapabs^3} \Dapbar\Ztbar  =  \Rone -i\brac{\Dapbar\frac{1}{\Zap}} \Jone -i\frac{1}{\Zapabs^2}\pap \Jone 
\end{equation}
where
\begin{align} \label{form:Rone}
\begin{split}  
\Rone & =  -2(\Dapbar\Ztbar)(\Dt\Dapbar\Ztbar) -2\sigma\Real(\Dap\Zt)\Dapbar\Dap\Th -\sigma(\Dapbar\Dap\Zt)\Dap\Th \\ 
& \quad  +i\sigma\brac{2i\Real \brac{\Dapabs \Th} + \brac{\Real \Th}^2} \Dapabs \Dapbar\Ztbar -\sigma \Real \brac{\Dapabs^2 \Th}\Dapbar\Ztbar  \\
& \quad   + i\sigma \brac{\Real \Th} \brac{ \Real \brac{\Dapabs \Th}}\Dapbar\Ztbar 
\end{split}
\end{align}
and $\Jone$ was defined in \eqref{form:Jone}. This equation gives rise to the energy $\Esigmafour$ in the energy estimate. 
\medskip

\subsection{\texorpdfstring{Equation for $\Ztapbar$ \nopunct}{}} \hspace*{\fill} \medskip

Multiply the equation for $\Dapbar\Ztbar$ in \eqref{eq:DapbarZtbar} by $\Zbarap$ to get the equation
\begin{align} \label{eq:Ztbarap}
\begin{split}
& \brac{\Dt^2 +i\frac{\Aone}{\Zapabs^2}\pap  -i\sigma\Dapabs^3} \Ztbarap \\
 & =  \Rone\Zapbar -i\brac{\pap\frac{1}{\Zap}} \Jone -i\Dap \Jone - \Zapbar\sqbrac{\Dt^2 +i\frac{\Aone}{\Zapabs^2}\pap -i\sigma\Dapabs^3, \frac{1}{\Zapbar}} \Ztapbar
\end{split}
\end{align}

This equation gives rise to the energy $\Esigmatwo$ in the energy estimate. This equation will also be useful to get estimates for the term $\Dap \Jone$.
\medskip

\subsection{\texorpdfstring{Equation for $\Th$ \nopunct}{}} \hspace*{\fill} \medskip

Apply $\Dapbar$ to the fundamental equation \eqref{form:Zttbar} and use $\Dapbar \Dap  = \brac{\Dapabs -i\Real \Th} \Dapabs$ from \eqref{eq:DapbarDap} to obtain
\begin{align*}
 \Dt\Dapbar\Ztbar  +i \Dapbar \brac{\frac{\Aone}{\Zap}} -\sigma  \brac{\Dapabs -i\Real \Th} \Dapabs\Th = -\brac{\Dapbar \Ztbar}^2
\end{align*}\\
Now applying the operator $i(\Dapabs +i\Real\Th)$ we obtain
\begin{align*}
\begin{split}
 & i(\Dapabs +i\Real\Th)  \Dt\Dapbar\Ztbar  - \Dapabs \Dapbar \brac{\frac{\Aone}{\Zap}} - i\sigma (\Dapabs +i\Real\Th) \brac{\Dapabs -i\Real \Th} \Dapabs\Th \\
&= -2i\brac{\Dapbar \Ztbar}\brac{\Dapabs \Dapbar \Ztbar} +\brac{\Real \Th}\cbrac{ \brac{\Dapbar \Ztbar}^2 +i\Dapbar\brac{\frac{\Aone}{\Zap}}}
\end{split}
\end{align*}
Observe that
\begin{align*}
(\Dapabs +i\Real\Th) \brac{\Dapabs -i\Real \Th} \Dapabs \Th =  \Dapabs^3\Th -i\Real\brac{\Dapabs\Th} \Dapabs\Th  + \brac{\Real\Th}^2 \Dapabs\Th
\end{align*}
Hence we get
\begin{align*}
\begin{split}
 & i(\Dapabs +i\Real\Th)  \Dt\Dapbar\Ztbar  - \Dapabs \Dapbar \brac{\frac{\Aone}{\Zap}} - i\sigma \Dapabs^3\Th \\
& =  -2i\brac{\Dapbar \Ztbar}\brac{\Dapabs \Dapbar \Ztbar} + \brac{\Real \Th }\cbrac{ \brac{\Dapbar \Ztbar}^2 +i\Dapbar\brac{\frac{\Aone}{\Zap}} +i\sigma\brac{\Real\Th} \Dapabs\Th} \\
& \quad +\sigma \Real\brac{\Dapabs\Th} \Dapabs\Th 
\end{split}
\end{align*}
Now recall from \eqref{form:Th} that $\Th = i \frac{\Zap}{\Zapabs} \pap \frac{1}{\Zap} - i \Real (\Id - \Hil) \brac{\frac{\Zap}{\Zapabs} \pap \frac{1}{\Zap}}$. Therefore
\begin{align*}
 -\Dapabs \Dapbar\brac{\frac{\Aone}{\Zap}} & =  -\Dapabs\cbrac{\frac{\Aone}{\Zapabs} \brac{\frac{\Zap}{\Zapabs}\pap\frac{1}{\Zap}} + \frac{1}{\Zapabs^2}\pap \Aone}  \\
 & =  i\frac{\Aone}{\Zapabs^2}\pap\Th   - \frac{\Aone}{\Zapabs^2}\pap\Real (\Id - \Hil) \brac{\frac{\Zap}{\Zapabs} \pap \frac{1}{\Zap}}  \\
 & \quad - \brac{\Dapabs\frac{\Aone}{\Zapabs}}\brac{\frac{\Zap}{\Zapabs}\pap \frac{1}{\Zap}} -\Dapabs\brac{\frac{1}{\Zapabs^2}\pap \Aone}
\end{align*}
Hence we have
\begin{align*}
\begin{split}
 & i(\Dapabs +i\Real\Th)  \Dt\Dapbar\Ztbar  +  i\frac{\Aone}{\Zapabs^2}\pap\Th  - i\sigma \Dapabs^3\Th \\
& =  -2i\brac{\Dapbar \Ztbar}\brac{\Dapabs \Dapbar \Ztbar} + \brac{\Real \Th }\cbrac{ \brac{\Dapbar \Ztbar}^2 +i\Dapbar\brac{\frac{\Aone}{\Zap}} +i\sigma\brac{\Real\Th} \Dapabs\Th} \\
& \quad +\sigma \Real\brac{\Dapabs\Th} \Dapabs\Th + \brac{\Dapabs\frac{\Aone}{\Zapabs}}\brac{\frac{\Zap}{\Zapabs}\pap \frac{1}{\Zap}} + \Dapabs\brac{\frac{1}{\Zapabs^2}\pap \Aone}  \\
& \quad + \frac{\Aone}{\Zapabs^2}\pap\Real (\Id - \Hil) \brac{\frac{\Zap}{\Zapabs} \pap \frac{1}{\Zap}}
\end{split}
\end{align*}
Recall from \eqref{form:Dt^2Th} that 
\begin{align*} 
\begin{split}
\Dt^2\Th &= i(\Dapabs + i \Real \Th)\Dt \Dapbar \Ztbar - i \Real (\Id - \Hil)\cbrac{(\Dapabs + i \Real \Th) \Dt\Dapbar \Ztbar}  +  i \Imag \nobrac{(\Id - \Hil) \Dt^2 \Th}\\
  & \quad +  (\Id + \Hil) \Imag\cbrac{\Real(\Dapbar\Ztbar)\Dapabs\Dapbar\Ztbar  -i\Real (\Dt \Th)\Dapbar\Ztbar} 
\end{split}
\end{align*}
Hence replacing the term $i(\Dapabs + i \Real \Th)\Dt \Dapbar \Ztbar$ in the equation with $\Dt^2\Th$ we get our main equation as
\begin{align} \label{eq:Th}
 \brac{\Dt^2 +i\frac{\Aone}{\Zapabs^2}\pap  -i\sigma\Dapabs^3} \Th = \Rtwo +i\Jtwo
\end{align}
where
\begin{align}\label{form:Rtwo}
\begin{split}
\Rtwo &=  -2i\brac{\Dapbar \Ztbar}\brac{\Dapabs \Dapbar \Ztbar} + \brac{\Real \Th }\cbrac{ \brac{\Dapbar \Ztbar}^2 +i\Dapbar\brac{\frac{\Aone}{\Zap}} +i\sigma\brac{\Real\Th} \Dapabs\Th} \\
& \quad +\sigma \Real\brac{\Dapabs\Th} \Dapabs\Th + \brac{\Dapabs\frac{\Aone}{\Zapabs}}\brac{\frac{\Zap}{\Zapabs}\pap \frac{1}{\Zap}} + \Dapabs\brac{\frac{1}{\Zapabs^2}\pap \Aone}  \\
& \quad +  (\Id + \Hil) \Imag\cbrac{\Real(\Dapbar\Ztbar)\Dapabs\Dapbar\Ztbar  -i\Real (\Dt \Th)\Dapbar\Ztbar}\\
& \quad + \frac{\Aone}{\Zapabs^2}\pap\Real (\Id - \Hil) \brac{\frac{\Zap}{\Zapabs} \pap \frac{1}{\Zap}} 
\end{split} \\
 \Jtwo&= \Imag (\Id - \Hil) \brac*[\big]{\Dt^2 \Th} - \Real (\Id - \Hil)\cbrac{(\Dapabs + i \Real \Th) \Dt\Dapbar \Ztbar} 
\end{align}
Note that the variable $\Jtwo$ is real valued. This equation gives rise to the energy $\Esigmathree$ in the energy estimate.

\section{Main a priori estimate}\label{sec:aprioriEsigma}
\medskip

We now describe our main a priori estimate. Define
\begin{align*}
& \Esigmazero =  \norm[\infty]{\sigma^{\half} \Zapabs^\half \pap\frac{1}{\Zap}}^2  +   \norm[2]{\sigma^\onebysix\Zapabs^\half\pap\frac{1}{\Zap}}^6  +  \norm[2]{\pap\frac{1}{\Zap}}^2  + \norm[2]{\frac{\sigma^\half}{\Zapabs^\half}\pap^2\frac{1}{\Zap}}^2 \\
& \Esigmaone = \norm[\Hhalf]{(\Zttbar -i)\Zap}^2 + \norm[2]{\sqrt{\Aone}\Ztapbar}^2 + \norm[2]{\frac{\sigma^\half}{\Zapabs^\half}\pap\Ztapbar}^2 \\
& \Esigmatwo = \norm[2]{\Dt\Ztapbar}^2 + \norm[\Hhalf]{\sqrt{\Aone}\frac{\Ztapbar}{\Zapabs}}^2 + \norm[\Hhalf]{\frac{\sigma^\half}{\Zapabs^\threebytwo}\pap\Ztapbar}^2\\
& \Esigmathree = \norm[2]{\Dt\Th}^2 + \norm[\Hhalf]{\sqrt{\Aone}\frac{\Th}{\Zapabs}}^2 + \norm[\Hhalf]{\frac{\sigma^\half}{\Zapabs^\threebytwo}\pap\Th}^2\\
& \Esigmafour = \norm[\Hhalf]{\Dt\Dapbar\Ztbar}^2 + \norm[2]{\sqrt{\Aone}\Dapabs\Dapbar\Ztbar}^2 + \norm[2]{\frac{\sigma^\half}{\Zapabs^\half}\pap\Dapabs\Dapbar\Ztbar}^2 \\
& \Esigma = \Esigmazero + \Esigmaone + \Esigmatwo + \Esigmathree + \Esigmafour
\end{align*}

Observe that the variables used above are all very natural. $\Zt$ and $\Ztt$ are the velocity and acceleration on the boundary respectively, $\Th$ is twice the holomorphic projection of the curvature and $\Dapbar \Ztbar$ is related to the material derivative of the angle by the relation $\Imag \brac{\Dapbar \Ztbar} = - \brac{\pt \thvar}\compose \hinv$ from \eqref{form:Dtg}. The weight ${\frac{1}{\Zapabs}}$ is related to the Taylor sign condition from \eqref{form:Taylor}. The energies $E_{\sigma,i}$ for $1\leq i \leq 4$ are obtained from the quasilinear equations derived in \secref{sec:quasilinear} whereas the energy $\Esigmazero$ is added as a lower order term. For $\sigma=0$, the energies $\Esigmathree$ and $\Esigmafour$ are equivalent, but for $\sigma>0$ the energy $\Esigmafour$ is higher order than $\Esigmathree$. 

\begin{thm}\label{thm:aprioriEsigma}
Let $\sigma\geq 0$ and let $(\Z, \Zt)(t)$ be a smooth solution to \eqref{eq:systemone} in $[0,T]$ for $T>0$, such that for all $s\geq 3$ we have $(\Zap-1,\frac{1}{\Zap} - 1, \Zt) \in \Linfty([0,T], H^{s+\half}(\Rsp)\times H^{s+\half}(\Rsp)\times H^{s}(\Rsp))$. Then $\sup_{t\in[0,T]} \Esigma(t) < \infty$ and there exists a polynomial P with universal non-negative coefficients such that for all $t \in [0,T)$ we have
\[
\frac{d\Esigma(t)}{dt} \leq P(\Esigma(t))
\]
\end{thm}

\begin{rem} \label{rem:aprioriEsigma}
We mention a minor technical point in the statement of the theorem. The energy $\Esigmazero$ contains a term which is the $L^\infty$ norm of a function and hence may not in general be differentiable in time even for smooth solutions. Hence for this term we replace the time derivative by the upper Dini derivative $  \limsup_{s \to 0^{+}} \frac{\norm[\infty]{f(t+s)}-\norm[\infty]{f(t)}}{s}$
\end{rem}

This theorem along with \thmref{thm:existencesobolev} will allow us to prove \thmref{thm:existence}. The above result is an extension of the a priori energy estimate obtained by Kinsey and Wu \cite{KiWu18} to the case of non-zero surface tension. Note that the hypothesis of the theorem easily implies that $\sup_{t\in[0,T]} \Esigma(t) < \infty$. This can also be seen from \propref{prop:equivEsigma} and \lemref{lem:equivsobolev}. The rest of the section is devoted to proving the a priori estimate.

In this section whenever we write $\dis f \in \Ltwo$, what we mean is that there exists a universal polynomial $P$ with nonnegative coefficients such that $\dis \norm*[2]{f} \leq P(\Esigma)$. Similar definitions for $\dis f\in \Hhalf$ or $\dis  f \in \Linfty$. We define the norm $ \dis \norm[\Linfty\cap\Hhalf]{f} = \norm[\infty]{f} + \norm[\Hhalf]{f} $. We also define two new spaces $\Ccal$ and $\Wcal$:
\begin{enumerate}
\item  If $\dis w \in \Linfty$ and $\dis \Dapabs w \in \Ltwo$, then we say $w\in \Wcal$. Define 
\[
 \norm[\Wcal]{w} = \norm[\infty]{w}+ \norm[2]{\Dapabs w}
 \]

\item If $\dis f \in \Hhalf$ and $\dis f\Zapabs \in \Ltwo$, then we say $f\in \mathcal{C}$. Define
\[
\norm[\Ccal]{f} = \norm[\Hhalf]{f} + \brac*[\bigg]{1+ \norm*[\bigg][2]{\pap\frac{1}{\Zapabs}}}\norm[2]{f\Zapabs}
 \]
\end{enumerate}

We also define the norm $\norm[\Wcal\cap\Ccal]{f} = \norm[\Wcal]{f} + \norm[\Ccal]{f}$. The reason for the introduction of these spaces is that we will frequently have situations where $f \in \Hhalf, w\in \Linfty$ and we want $fw \in \Hhalf$. We will also have situations where $f\in \Hhalf, g\Zapabs\in \Ltwo$ and we want $fg\Zapabs\in \Ltwo$. Clearly these are not true in general but in special cases this can be proved and the following lemma addresses this issue for a majority of the situations we encounter.

\begin{lem} \label{lem:CW} 
The following properties hold for the spaces $\Wcal$ and $\Ccal$
\begin{enumerate}[leftmargin =*, align=left]
\item If $w_1,w_2 \in \Wcal$, then $w_1w_2 \in \Wcal$. Moreover $\norm[\Wcal]{w_1w_2} \leq \norm[\Wcal]{w_1}\norm[\Wcal]{w_2}$
\item If $f\in \Ccal$ and $w\in \Wcal$, then $fw \in \Ccal$.  Moreover $\norm[\Ccal]{fw} \lesssim \norm[\Ccal]{f}\norm[\Wcal]{w}$
\item If $f,g \in \Ccal$, then $fg\Zapabs \in \Ltwo$. Moreover $\norm[2]{fg\Zapabs} \lesssim \norm[\Ccal]{f}\norm[\Ccal]{g}$
\end{enumerate}
\end{lem}
In the lemma and in the definitions of $\Ccal$ and $\Wcal$, the function $\dis \frac{1}{\Zapabs}$ is used as a weight but there is nothing special about this function. We can define similar spaces and prove the lemma for any weight. The only property used of the weight is that $\dis \norm*[\Big][2]{\pap \frac{1}{\Zapabs}} < \infty$. See  \propref{prop:Hhalfweight} in the appendix for more details and for the proof of the lemma. In our case we are able to use the weight $\dis \frac{1}{\Zapabs}$ as $\dis \norm*[\Big][2]{\pap \frac{1}{\Zapabs}}$ is controlled by the energy $\Esigma$.

In this section we will sometimes use the function $\Zaphalf$. This is defined as 
\[
 \Zaphalf = e^{\half\log(\Zap)} \quad \enspace \tx{ where } \log(\Zap) \to 0 \tx{ as } \abs{\ap} \to \infty \\
\]
Note that there is no ambiguity in the definition of $\log(\Zap)$ as it is continuous and  we have fixed its value at infinity. We also use the following notation
\begin{align*}
\sqbrac{f_1, f_2;  f_3}(\ap) = \frac{1}{i\pi} \int \brac{\frac{f_1(\ap) - f_1(\bp)}{\ap - \bp}}\brac{\frac{f_2(\ap) - f_2(\bp)}{\ap-\bp}} f_3(\bp) \diff \bp
\end{align*}
This notation will be used in some of the computations later on in the section. 
  
\subsection{\texorpdfstring{Quantities controlled by $\Esigma$ \nopunct}{}}\label{sec:quantEsigma}  \hspace*{\fill} \medskip

Now we come to main part of the section. Here we control all the important terms controlled by the energy $\Esigma$. We will frequently use the estimates proved in the appendix to control the terms such as \propref{prop:Hardy}, \propref{prop:commutator}, \propref{prop:Leibniz} and \propref{prop:triple}. 

For $\sigma=0$ the energy $\Esigma$ is lower order as compared to the energy in Kinsey-Wu \cite{KiWu18} by half weighted spacial derivative. In particular we do not have control of $\Dap\Zttbar \in \Linfty$ which was heavily used in Kinsey-Wu \cite{KiWu18} but we only have $\Dap\Zttbar \in \Hhalf$. Because of this, the energy estimate becomes much more subtle even in the case of $\sigma=0$, and we need to prove stronger control of existing terms. For e.g. in \cite{KiWu18} it is shown that $\Aone, \Dapbar\Ztbar, \bvarap, \frac{1}{\Zapabs^2}\pap\Aone  \in \Linfty$ and we show that in fact $\Aone, \Dapbar\Ztbar, \bvarap, \frac{1}{\Zapabs^2}\pap\Aone \in \Linfty\cap\Hhalf$. Most of the terms for the $\sigma=0$ case controlled here in $\Hhalf$  are new. Of course estimates for terms involving surface tension are also all new. 

\bigskip
\begin{enumerate}[widest = 99, leftmargin =*, align=left, label=\arabic*)]
\item $\dis \Ztapbar \in \Ltwo, \Dapabs\Dapbar\Ztbar \in \Ltwo$ 
\medskip \\
Proof: This is true as $\Aone \geq 1$ and as $\Esigmaone$ and $\Esigmafour$ are part of the energy
\medskip

\item $\dis \Aone \in \Linfty \cap \Hhalf$ 
\smallskip \\
Proof: Recall that $\displaystyle \Aone = 1 - \Imag[\Zt,\Hil]\Ztapbar$ and hence
 \[
 \dis \norm*[\infty]{\Aone} \leq 1 + \norm*[\infty]{[\Zt,\Hil]\Ztapbar} \lesssim 1 + \norm*[2]{\Ztapbar}^2
 \]
  by \propref{prop:commutator}. Similarly by \propref{prop:commutator} and Sobolev embedding we have
\[
\dis \norm*[\Hhalf]{\Aone} \leq \norm*[\big][2]{\papabs^\half [\Zt,\Hil]\Ztapbar} \lesssim \norm*[\big][BMO]{\papabs^\half \Zt} \norm*[2]{\Ztapbar} \lesssim \norm*[2]{\Ztapbar}^2
\]
\smallskip

\item $\dis \pap\frac{1}{\Zap} \in \Ltwo, \pap\frac{1}{\Zapabs} \in \Ltwo, \Dapabs \w \in \Ltwo$ and hence $\w \in \Wcal$
\medskip\\
Proof: Observe that $\dis \pap\frac{1}{\Zap} \in \Ltwo$ as it is part of the energy $\Esigmazero$. From \eqref{form:RealImagTh} we easily get that $\dis \pap\frac{1}{\Zapabs} $ and $\dis \Dapabs \w$ are in $\Ltwo$. Also as $\norm[\infty]{\w} = 1$ and $\Dapabs w \in \Ltwo$ we get that  $w\in \Wcal$. Now that we have shown that $\dis \pap\frac{1}{\Zapabs} \in \Ltwo$, we can use \lemref{lem:CW} from now on.
\medskip

\item $\dis \Dapbar\Ztbar \in \Linfty, \Dapabs\Ztbar \in \Linfty, \Dap\Ztbar \in \Linfty$  
\medskip\\
Proof: We only need to prove that $\Dapbar\Ztbar \in \Linfty$ and the rest follows. Observe that
\begin{align*}
 \pap \brac{\Dapbar\Ztbar}^2  = 2 (\Ztapbar)\brac{\Dapbar\Dapbar \Ztbar}
\end{align*}
As $\Dapbar \Ztbar $ decays at infinity, by integrating we get
\[
 \norm[\infty]{(\Dapbar \Ztbar)^2} \lesssim \int \abs{\Ztapbar}\abs{\Dapbar\Dapbar\Ztbar}\diff\ap \lesssim      \norm[\Ltwo]{\Ztapbar} \norm[\Ltwo]{\Dapabs\Dapbar\Ztbar}
 \]
 Hence  $\dis \norm[\infty]{\Dapbar \Ztbar} \lesssim \sqrt{\norm[\Ltwo]{\Ztapbar} \norm[\Ltwo]{\Dapabs\Dapbar\Ztbar}}$
\medskip

\item $\dis\Dapbar^2\Ztbar \in \Ltwo, \Dapabs^2\Ztbar \in \Ltwo, \Dap^2\Ztbar \in \Ltwo$ 
\medskip\\
Proof: We already know that $\Dapabs \Dapbar \Ztbar \in \Ltwo$ and hence $\Dapbar^2\Ztbar \in \Ltwo$. Now 
\[
\Dapbar^2\Ztbar = \Dapbar\brac{\w\Dapabs\Ztbar} = (\Dapbar\w)\Dapabs\Ztbar + \w^2\Dapabs^2\Ztbar
\]
Now observe that $\dis \Dapabs \w \in \Ltwo $ and $\dis \Dapabs\Ztbar \in \Linfty$ and hence the first term is in $\dis \Ltwo$. Hence we have $\dis \Dapabs^2 \Ztbar \in \Ltwo$. A similar argument works for the rest.
\medskip

\item $\dis  \Dapbar\Ztbar \in \Wcal\cap\Ccal, \Dapabs\Ztbar \in \Wcal\cap\Ccal, \Dap\Ztbar \in \Wcal\cap\Ccal$
\medskip\\
Proof: We will first prove $\dis  \Dapbar\Ztbar \in \Wcal\cap\Ccal$ . Observe that $\Dap\Ztbar \in \Linfty$, $\Dapabs\Dap\Zt \in \Ltwo$ and hence we have $\Dap\Ztbar \in \Wcal$. Now as $\Dap\Ztbar$ is holomorphic i.e. $\dis \Hil \Dap\Ztbar = \Dap\Ztbar$ we see that $\papabs\Dap\Ztbar = i\pap \Dap\Ztbar$. Hence we have
\begin{align*}
\lpar \norm[\Hhalf]{\Dap\Ztbar}^2   = \int \brac{\Dapbar\Zt}\brac{\papabs\Dap\Ztbar }\diff\ap  = i \int \brac{\Ztap}\brac{\Dapbar\Dap\Ztbar}\diff\ap
\end{align*}
Hence $\norm[\Hhalf]{\Dap\Ztbar} \lesssim \sqrt{\norm[\Ltwo]{\Ztbarap}\norm[\Ltwo]{\Dapabs\Dap\Ztbar}}$. Also as $\brac{\Dap\Ztbar} \Zapabs = \wbar \Ztapbar \in \Ltwo$, we have $\Dap\Ztbar \in \Ccal$. Now as $\Dapabs\Ztbar = (\Dap\Ztbar) \w$\, we have $\dis \norm[\Wcal\cap\Ccal]{\Dapabs\Ztbar} \lesssim \norm[\Wcal\cap\Ccal]{\Dap\Ztbar}\norm[\Wcal]{\w}$ by \lemref{lem:CW}. The rest are proved similarly.
\medskip

\item $\dis \pap \Pa\brac{\frac{\Zt}{\Zap}} \in \Linfty$
\begin{flalign*}
\lpar \tx{Proof: We see that }\enspace 2\pap \Pa \brac*[\Big]{\frac{\Zt}{\Zap}} & = (\Id - \Hil)\brac{\Dap\Zt} + (\Id - \Hil)\brac{\Zt\pap\frac{1}{\Zap}} &\\
& = 2\Dap\Zt - (\Id + \Hil)(\Dap\Zt) + (\Id - \Hil)\brac{\Zt\pap\frac{1}{\Zap}} \\
& = 2\Dap\Zt + \sqbrac{\frac{1}{\Zap},\Hil}\Ztap + \sqbrac{\Zt, \Hil}\pap\frac{1}{\Zap}  
\end{flalign*}
Hence $\dis  \quad \dis \norm*[\Big][\infty]{\pap\Pa\brac*[\Big]{\frac{\Zt}{\Zap}}}  \lesssim \norm[\infty]{\Dap\Zt} + \norm[2]{\Ztapbar}\norm*[\Big][2]{\pap\frac{1}{\Zap}}$ by \propref{prop:commutator}
\medskip

\item $\dis \Dapabs \Aone \in \Ltwo$ and hence $\dis \Aone \in \Wcal,  \sqrt{\Aone} \in \Wcal, \frac{1}{\Aone} \in \Wcal, \frac{1}{\sqrt{\Aone}} \in \Wcal$
\begin{flalign*}
\lpar \tx{Proof: Observe that }\enspace \Dapabs\Aone & = \Real \cbrac{\frac{\w}{\Zap}(\Id - \Hil)\pap\Aone} & \\
 & = \Real \cbrac{\w(\Id - \Hil)\Dap\Aone} -\Real\cbrac{\w\sqbrac{\frac{1}{\Zap},\Hil}\pap\Aone} 
\end{flalign*}
Using the formula of $\Aone$ from \eqref{form:Aonenew} we see that 
\begin{align*}
(\Id - \Hil)\Dap\Aone & = i(\Id - \Hil)\brac{(\Dap\Zt)\Ztapbar} +i(\Id - \Hil)\brac{\frac{\Zt}{\Zap}\pap\Ztapbar} \\
& = i(\Id - \Hil)\brac{(\Dap\Zt)\Ztapbar}+ i\sqbrac{\Pa\brac{\frac{\Zt}{\Zap}},\Hil}\pap\Ztapbar
\end{align*}
Hence using \propref{prop:commutator} we have 
\begin{align*}
 \norm[2]{\Dapabs\Aone} & \lesssim \norm[2]{(\Id - \Hil)\Dap\Aone} + \norm*[\Big][2]{\pap\frac{1}{\Zap}}\norm[\infty]{\Aone} \\
 & \lesssim \norm[\infty]{\Dap\Zt}\norm[2]{\Ztapbar} + \norm*[\Big][\infty]{\pap\Pa\brac*[\Big]{\frac{\Zt}{\Zap}}}\norm[2]{\Ztapbar} + \norm*[\Big][2]{\pap\frac{1}{\Zap}}\norm[\infty]{\Aone}
\end{align*}
Now as $\Aone \in \Linfty$ and $\Dapabs\Aone \in \Ltwo$, we have that $\Aone \in \Wcal$. Similarly using the fact that $\Aone \geq 1$, we easily get that $\dis \sqrt{\Aone} \in \Wcal, \frac{1}{\Aone} \in \Wcal, \frac{1}{\sqrt{\Aone}} \in \Wcal$
\medskip

\item $\dis \Th \in \Ltwo$, $\dis \Dt\Th \in \Ltwo$
\medskip\\
Proof: Using \eqref{form:Th} and the fact that the Hilbert transform is bounded on $\Ltwo$, we easily see that $\dis \norm[2]{\Th} \lesssim \norm*[\Big][2]{\pap\frac{1}{\Zap}}$. We have $\dis \Dt\Th \in \Ltwo$ as it part of the energy $\Esigmathree$
\medskip

\item $\dis \frac{\Th}{\Zapabs} \in \Ccal$
\medskip\\
Proof: We know from $\Esigmathree$ that $\dis \frac{\sqrt{\Aone}}{\Zapabs} \Th \in \Hhalf$. Now as $\dis \norm*[2]{\sqrt{\Aone}\Th} \lesssim \norm[\infty]{\Aone}^\half\norm[2]{\Th}$  we now have $\dis \frac{\sqrt{\Aone}}{\Zapabs} \Th \in \Ccal$. Hence  we get $\dis \norm[\Ccal]{\frac{\Th}{\Zapabs}} \lesssim \norm[\Ccal]{\frac{\sqrt{\Aone}}{\Zapabs}\Th}\norm[\Wcal]{\frac{1}{\sqrt{\Aone}}}$ from \lemref{lem:CW}
\medskip

\item $\dis \Dap\frac{1}{\Zap}\in \Ccal$, $ \dis \Dapabs\frac{1}{\Zap}\in \Ccal$, $ \dis \Dapabs\frac{1}{\Zapabs}\in \Ccal$, $\dis \frac{1}{\Zapabs^2}\pap \w \in \Ccal$
\medskip\\
Proof: Observe from \eqref{form:Th} that
\begin{align*}
\qquad \frac{\Th}{\Zapabs} &= i \frac{\Zap}{\Zapabs^2} \pap \frac{1}{\Zap} - i \Real\cbrac{\frac{1}{\Zapabs} (\Id - \Hil) \brac{\frac{\Zap}{\Zapabs} \pap \frac{1}{\Zap}} } \\
& = i \Dapbar\frac{1}{\Zap}  + i \Real\cbrac{\sqbrac{\frac{1}{\Zapabs},\Hil}  \brac{\frac{\Zap}{\Zapabs} \pap \frac{1}{\Zap}} }  - i \Real\cbrac{\sqbrac{\frac{1}{\Zapbar},\Hil}  \brac{\pap \frac{1}{\Zap}} } 
\end{align*}
Hence $\dis \norm*[\Big][\Hhalf]{\Dapbar\frac{1}{\Zap}} \lesssim \norm*[\bigg][\Hhalf]{\frac{\Th}{\Zapabs}} + \norm*[\Big][2]{\pap\frac{1}{\Zap}}^2$ from \propref{prop:commutator} which implies that $\dis \Dapbar \frac{1}{\Zap} \in \Ccal$. As $\wbar \in \Wcal$, by \lemref{lem:CW} we get $\dis \Dap\frac{1}{\Zap}\in \Ccal$ and $\dis \Dapabs \frac{1}{\Zap} \in \Ccal$. Observe that
\[
\Real\brac{\Dapbar\frac{1}{\Zap} } = \Dapabs \frac{1}{\Zapabs} \quad \qquad \Imag\brac{\Dapbar\frac{1}{\Zap}} =  i\brac{ \frac{\wbar}{\Zapabs^2}\pap \w} 
\]
Hence $\dis  \Dapabs \frac{1}{\Zapabs} \in \Ccal$ and $\dis \frac{\wbar}{\Zapabs^2}\pap \w \in \Ccal$. Now again using $\w \in \Wcal$ and \lemref{lem:CW} we easily obtain $\dis \frac{1}{\Zapabs^2}\pap \w \in \Ccal$
\medskip

\item $\dis \frac{1}{\Zapabs^2}\pap \Aone \in \Linfty\cap\Hhalf  $ and hence $\dis \frac{1}{\Zapabs^2}\pap \Aone \in \Ccal$
\medskip\\
Proof: Observe that $\dis  \frac{1}{\Zapabs^2}\pap\Aone  = \Real \cbrac{\frac{\w^2\wbar^2}{\Zapabs}(\Id - \Hil)\Dapabs\Aone}$ and hence we first show that $\dis \frac{\wbar^2}{\Zapabs}(\Id - \Hil)\Dapabs\Aone \in \Linfty \cap \Ccal$. Now
\[
\frac{\wbar^2}{\Zapabs}(\Id - \Hil)\Dapabs\Aone = (\Id - \Hil)\brac*[\bigg]{\frac{1}{\Zap^2}\pap\Aone} - \sqbrac{\frac{\wbar^2}{\Zapabs},\Hil}\Dapabs\Aone
\]
Using the formula of $\Aone$ from \eqref{form:Aonenew} we see that  
\begin{align*}
\lpar (\Id - \Hil)\brac*[\bigg]{\frac{1}{\Zap^2}\pap\Aone} & = i(\Id - \Hil)\cbrac*[\bigg]{\brac*[\bigg]{\frac{\Ztap}{\Zap^2}}\Ztapbar} +i(\Id - \Hil)\cbrac*[\bigg]{\Zt\brac*[\bigg]{\frac{1}{\Zap^2}\pap\Ztapbar}} \\
& = i\sqbrac*[\bigg]{ \frac{\Ztap}{\Zap^2} ,\Hil}\Ztapbar + i\sqbrac{ \Zt,\Hil} \brac*[\bigg]{\frac{1}{\Zap^2}\pap\Ztapbar}
\end{align*}
Hence from \propref{prop:commutator} we have 
\begin{align*}
\begin{split}
\lpar \norm[\Linfty\cap\Hhalf]{\frac{\wbar^2}{\Zapabs}(\Id - \Hil)\Dapabs\Aone} & \lesssim \norm[2]{\Ztapbar}\brac{\norm*[\bigg][2]{\pap\frac{\Ztapbar}{\Zap^2}} + \norm*[\bigg][2]{\frac{1}{\Zap^2}\pap\Ztapbar}}  \\
& \quad  + \norm*[\Big][2]{\pap\frac{1}{\Zap}}\norm[2]{\Dapabs\Aone}
\end{split}
\end{align*}
and as $\dis \Dapabs \Aone \in \Ltwo$, we have $\dis \frac{\wbar^2}{\Zapabs}(\Id - \Hil)\Dapabs\Aone \in \Linfty \cap \Ccal$. Now using the fact that $\w \in \Wcal$ and \lemref{lem:CW}, we can conclude that $\dis \frac{1}{\Zapabs^2}\pap \Aone \in L^\infty\cap\Ccal.$
\medskip

\item $\dis \frac{1}{\Zapabs^3}\pap^2\Aone \in \Ltwo $, $\dis \Dapabs\brac*[\bigg]{\frac{1}{\Zapabs^2}\pap\Aone} \in \Ltwo $ and hence $\dis \frac{1}{\Zapabs^2}\pap \Aone \in \Wcal$
\medskip\\
Proof: Observe that $\dis \Dapabs\brac*[\bigg]{\frac{1}{\Zapabs^2}\pap\Aone}  = \Real \cbrac*[\bigg]{\frac{\w^3\wbar^3}{\Zapabs}(\Id - \Hil)\pap\brac*[\bigg]{\frac{1}{\Zapabs^2}\pap\Aone}}$ and hence it is enough to show that  $\dis \frac{\wbar^3}{\Zapabs}(\Id - \Hil)\pap\brac*[\bigg]{\frac{1}{\Zapabs^2}\pap\Aone} \in \Ltwo$. Now
\begin{flalign*}
\lpar & \frac{\wbar^3}{\Zapabs}(\Id - \Hil)\pap\brac*[\bigg]{\frac{1}{\Zapabs^2}\pap\Aone} \\
 & = (\Id - \Hil) \cbrac*[\bigg]{\wbar^2\Dap \brac*[\bigg]{\frac{1}{\Zapabs^2}\pap\Aone}} - \sqbrac{ \frac{\wbar^3}{\Zapabs} ,\Hil}\pap \brac*[\bigg]{\frac{1}{\Zapabs^2}\pap\Aone}  \\
 & = (\Id - \Hil) \cbrac*[\bigg]{\Dap \brac*[\bigg]{\frac{1}{\Zap^2}\pap\Aone} - 2\brac*[\bigg]{\frac{\wbar}{\Zapabs^2}\pap\Aone}\brac{\Dap\wbar}} - \sqbrac{ \frac{\wbar^3}{\Zapabs} ,\Hil}\pap \brac*[\bigg]{\frac{1}{\Zapabs^2}\pap\Aone}
\end{flalign*}
Using the formula of $\Aone$ from \eqref{form:Aonenew} we see that  
\begin{flalign*}
&& (\Id - \Hil) \cbrac*[\bigg]{\Dap \brac*[\bigg]{\frac{1}{\Zap^2}\pap\Aone}} & = i(\Id - \Hil)\cbrac*[\bigg]{\Dap\cbrac*[\bigg]{\brac*[\bigg]{\frac{\Ztap}{\Zap^2}}\Ztapbar + \Zt\brac*[\bigg]{\frac{1}{\Zap^2}\pap\Ztapbar} }}  \\
&&& = i(\Id - \Hil)\cbrac*[\bigg]{\brac*[\bigg]{\pap\frac{\Ztap}{\Zap^2}}\brac{\Dap\Ztbar} + 2\brac{\Dap\Zt} \brac*[\bigg]{\frac{1}{\Zap^2}\pap\Ztapbar} } \\
&&& \quad + i\sqbrac{\Pa\brac{\frac{\Zt}{\Zap}},\Hil}\pap \brac*[\bigg]{\frac{1}{\Zap^2}\pap\Ztapbar} 
\end{flalign*}
Hence from \propref{prop:commutator} we have
\begin{flalign*}
\lpar \norm*[\bigg][2]{\Dapabs\brac*[\bigg]{\frac{1}{\Zapabs^2} \pap\Aone}} & \lesssim 
\norm*[\bigg][2]{\frac{1}{\Zap^2}\pap\Ztapbar}\brac{ \norm[\infty]{\Dap\Zt} +  \norm*[\Big][\infty]{\pap\Pa\brac{\frac{\Zt}{\Zap}}}} \\
& \quad + \norm*[\bigg][2]{\pap\frac{\Ztap}{\Zap^2}}\norm[\infty]{\Dap\Ztbar} + \norm[2]{\pap\frac{1}{\Zap}}\norm*[\bigg][\infty]{\frac{1}{\Zapabs^2}\pap\Aone} 
\end{flalign*}
Now the other term is easily controlled
\begin{align*}
\norm*[\bigg][2]{\frac{1}{\Zapabs^3}\pap^2\Aone} \lesssim \norm[2]{\pap\frac{1}{\Zapabs}}\norm*[\bigg][\infty]{\frac{1}{\Zapabs^2}\pap\Aone} + \norm*[\bigg][2]{\Dapabs\brac*[\bigg]{\frac{1}{\Zapabs^2} \pap\Aone}}
\end{align*}
As $\dis \frac{1}{\Zapabs^2}\pap\Aone \in \Linfty$ and $\dis \Dapabs\brac{\frac{1}{\Zapabs^2}\pap\Aone} \in \Ltwo $ we get that $\dis \frac{1}{\Zapabs^2}\pap\Aone \in \Wcal$.
\medskip

\item $\dis \bvarap \in \Linfty\cap\Hhalf $ and $\Hil(\bvarap) \in \Linfty\cap\Hhalf $
\medskip\\
Proof: Using the formula of $\bvarap$ from \eqref{form:bvarapnew} we see that  
\begin{align}\label{form:IminusHbvarap}
\begin{split}
(\Id -\Hil)\bvarap & = (\Id - \Hil)\brac{\frac{\Ztap}{\Zap}} + \sqbrac{\Zt,\Hil}\brac*[\Big]{\pap \frac{1}{\Zap}} \\
& = \sqbrac{\frac{1}{\Zap},\Hil}\Ztap + 2\Dap\Zt + \sqbrac{\Zt,\Hil}\brac*[\Big]{\pap \frac{1}{\Zap}} 
\end{split}
\end{align}
Hence $\dis  \norm[\Linfty\cap\Hhalf]{(\Id - \Hil)\bvarap}  \lesssim \norm[2]{\Ztapbar}\norm[2]{\pap\frac{1}{\Zap}} + \norm[\Linfty\cap\Hhalf]{\Dap\Zt} $ from \propref{prop:commutator}. As $\bvarap$ is real valued, this implies $\dis \bvarap \in L^\infty\cap\Hhalf $ and $\Hil(\bvarap) \in L^\infty\cap\Hhalf $
\medskip

\item $\dis \Dapabs \bvarap \in \Ltwo$ and hence $\bvarap \in \Wcal$ 
\medskip\\
Proof: Observe that
\begin{align*}
 \Dapabs\bvarap & = \Real \cbrac{\frac{\w}{\Zap}(\Id - \Hil)\pap\bvarap} \\
 & = \Real \cbrac{\w(\Id - \Hil)\Dap\bvarap} -\Real\cbrac{\w\sqbrac{\frac{1}{\Zap},\Hil}\pap\bvarap}
\end{align*}
Using the formula of $\bvarap$ from \eqref{form:bvarapnew} we see that 
\begin{align*}
\lpar (\Id -\Hil)\Dap\bvarap & = (\Id - \Hil)\cbrac{\Dap^2\Zt + (\Dap\Zt)\brac*[\Big]{\pap \frac{1}{\Zap}} + \frac{\Zt}{\Zap}\pap \brac*[\Big]{\pap \frac{1}{\Zap}} } \\
& = (\Id - \Hil)\cbrac{\Dap^2\Zt + (\Dap\Zt)\brac*[\Big]{\pap \frac{1}{\Zap}}} + \sqbrac{\Pa\brac{\frac{\Zt}{\Zap}},\Hil}\pap \brac*[\Big]{\pap \frac{1}{\Zap}}
\end{align*}
Hence $\dis \norm[2]{\Dapabs\bvarap} \lesssim \norm[2]{\Dap^2\Zt} + \norm*[\Big][2]{\pap\frac{1}{\Zap}}\cbrac{ \norm[\infty]{\Dap\Zt}  +  \norm*[\Big][\infty]{\pap\Pa\brac{\frac{\Zt}{\Zap}}} + \norm[\infty]{\bvarap} } $  from \propref{prop:commutator}
\medskip
 
\item $\dis \pap\Dt\frac{1}{\Zap} \in \Ltwo$, $\dis \Dt\pap\frac{1}{\Zap} \in \Ltwo$
\medskip\\
Proof: Recall from \eqref{form:DtoneoverZap} that $\dis \Dt\frac{1}{\Zap} = \frac{1}{\Zap}(\bvarap - \Dap\Zt)$ and hence
\[
\pap\Dt\frac{1}{\Zap} = \brac*[\Big]{\pap\frac{1}{\Zap}}(\bvarap - \Dap\Zt) + \Dap\bvarap - \Dap^2\Zt
\]
Hence $\dis \norm*[\Big][2]{\pap\Dt\frac{1}{\Zap}} \lesssim \norm*[\Big][2]{\pap\frac{1}{\Zap}}\brac{\norm[\infty]{\bvarap}+\norm[\infty]{\Dap\Zt}} + \norm[2]{\Dapabs\bvarap} + \norm[2]{\Dap^2\Zt} $. Similarly we have $\dis \norm*[\Big][2]{\Dt\pap\frac{1}{\Zap}} \lesssim  \norm*[\Big][2]{\pap\Dt\frac{1}{\Zap}}+ \norm[\infty]{\bvarap}\norm*[\Big][2]{\pap\frac{1}{\Zap}}$
\medskip
 
\item $\dis \Zttapbar \in \Ltwo $
\medskip\\
Proof: From $\Esigmatwo$ we have that $\Dt\Ztapbar \in \Ltwo$. Hence $\norm[2]{\Zttapbar} \lesssim \norm[2]{\Dt\Ztapbar} + \norm[\infty]{\bvarap}\norm[2]{\Ztapbar}$
\medskip

\item $\dis \Dapbar \Zttbar \in \Ccal, \Dapabs \Zttbar \in \Ccal, \Dt\Dapbar\Ztbar \in \Ccal \tx{ and }\Dt\Dapabs\Ztbar \in \Ccal$
\medskip\\
Proof: From $\Esigmafour$ we have that $\Dt\Dapbar\Ztbar \in \Hhalf$. Observe that 
\[
\Dt\Dapbar\Ztbar = \Dapbar\Zttbar - (\Dapbar\Ztbar)^2
\]
and as $\Dapbar\Ztbar \in \Ccal\cap\Wcal$, by using \lemref{lem:CW} we get that $(\Dapbar\Ztbar)^2 \in \Ccal$. Hence $\Dapbar\Zttbar \in \Hhalf$. As $\Zttapbar \in \Ltwo$ we get that $\Dapbar\Zttbar \in \Ccal$. By again using the equation above, we get that $\Dt\Dapbar\Ztbar \in \Ccal$. By using $\wbar\in \Wcal$ and that $\Dapbar\Zttbar \in \Ccal $ in \lemref{lem:CW}, we obtain $\Dapabs\Zttbar \in \Ccal$. Now observe that 
\[
\Dt\Dapabs\Ztbar = \Dapabs\Zttbar - \Real(\Dapbar\Ztbar)\Dapabs\Ztbar
\]
As $\Dapbar\Ztbar \in \Ccal$ we get that $\Real(\Dapbar\Ztbar) \in \Ccal$. Also as $\Dapabs\Ztbar \in \Wcal$, using \lemref{lem:CW} we obtain $ \Real(\Dapbar\Ztbar)\Dapabs\Ztbar \in \Ccal$. Hence $\Dt\Dapabs\Ztbar \in \Ccal$.
\medskip

\item $\dis \Dt\Aone \in \Linfty\cap\Hhalf$
\medskip\\
Proof: Recall that $\dis \Aone = 1 - \Imag[\Zt,\Hil]\Ztapbar $. This implies from \propref{prop:tripleidentity}
\begin{align*}
\Dt\Aone &= -\Imag\cbrac{\sqbrac{\Ztt,\Hil}\Ztapbar + \sqbrac{\Zt,\Hil}\Zttapbar - \sqbrac{\bvar, \Zt; \Ztapbar}}
\end{align*}
Hence $\dis \norm[\Linfty\cap\Hhalf]{\Dt\Aone}  \lesssim \norm[2]{\Ztapbar}\norm[2]{\Zttapbar} + \norm[\infty]{\bvarap}\norm[2]{\Ztapbar}^2$ from \propref{prop:commutator} and \propref{prop:triple}.
\medskip

\item $\Dt(\bvarap - \Dap\Zt - \Dapbar\Ztbar) \in \Linfty\cap\Hhalf$ and hence $\dis \Dt\bvarap \in \Hhalf, \pap\Dt \bvar \in \Hhalf $
\medskip\\
Proof: Using the formula of $\bvarap$ from \eqref{form:bvarapnew} we see that 
\[
\bvarap - \Dap\Zt - \Dapbar\Ztbar =  \Zt\brac*[\Big]{\pap \frac{1}{\Zap}} - \Dapbar\Ztbar -i\pap(\Id + \Hil)\cbrac{\Imag\brac{\frac{\Zt}{\Zap}}}
\]
Observe that $(\bvarap - \Dap\Zt - \Dapbar\Ztbar)$ is real valued and hence by applying $\Real(\Id - \Hil)$ we get
\[
\bvarap - \Dap\Zt - \Dapbar\Ztbar = \Real\cbrac{ \sqbrac{\Zt,\Hil}\brac*[\Big]{\pap \frac{1}{\Zap}} -\sqbrac{\frac{1}{\Zapbar},\Hil}\Ztapbar  } 
\]
Applying $\Dt$ and using \propref{prop:tripleidentity} we obtain
\begin{flalign*}
&& \Dt(\bvarap - \Dap\Zt - \Dapbar\Ztbar) &= \Real   \bigg\{\sqbrac{\Ztt,\Hil}\brac*[\Big]{\pap \frac{1}{\Zap}} + \sqbrac{\Zt,\Hil}\brac*[\Big]{\pap \Dt\frac{1}{\Zap}}  - \sqbrac{\bvar,\Zt; \pap\frac{1}{\Zap}} \\
&&& \qquad \quad  -\sqbrac{\Dt\frac{1}{\Zapbar},\Hil}\Ztapbar  -\sqbrac{\frac{1}{\Zapbar},\Hil}\Zttapbar  + \sqbrac{\bvar,\frac{1}{\Zapbar};\Ztapbar }   \bigg\}
\end{flalign*}
Hence from \propref{prop:commutator} and \propref{prop:triple} we get 
\begin{align*}
\norm[\Linfty\cap\Hhalf]{\Dt(\bvarap - \Dap\Zt - \Dapbar\Ztbar) } & \lesssim \norm[2]{\Zttapbar}\norm*[\Big][2]{\pap\frac{1}{\Zap}} + \norm[2]{\Ztapbar}\norm[2]{\pap\Dt\frac{1}{\Zap}} \\
& \quad + \norm[\infty]{\bvarap}\norm[2]{\Ztapbar}\norm*[\Big][2]{\pap\frac{1}{\Zap}}
\end{align*}
As $\Dt\Dap\Zt \in \Ccal$ and $\Dt\Dapbar\Ztbar \in \Ccal$, we get that $\Dt\bvarap \in \Hhalf$. Now as $
\pap\Dt\bvar = (\bvarap)^2 + \Dt\bvarap $ we get $\norm[\Hhalf]{\pap\Dt\bvar} \lesssim \norm[\Hhalf]{\bvarap}\norm[\infty]{\bvarap} + \norm[\Hhalf]{\Dt\bvarap}$
\medskip

\item $\dis \sigma^{\half}  \Zapabs^\half \pap\frac{1}{\Zap} \in \Linfty$,  $\dis \sigma^{\half}  \Zapabs^\half \pap\frac{1}{\Zapabs} \in \Linfty$, $\dis \frac{ \sigma^\half}{\Zapabs^\half}\pap\w \in \Linfty$, $\dis \sigma^\half\Zapabs^\half\Real\Th \in \Linfty$
\medskip\\
Proof: $\dis \sigma^{\half}  \Zapabs^\half \pap\frac{1}{\Zap} \in \Linfty$ as it part of the energy $\Esigmazero$. Using \eqref{form:RealImagTh}  we easily obtain $\dis \sigma^{\half}  \Zapabs^\half \pap\frac{1}{\Zapabs} \in \Linfty$ and $\dis \frac{ \sigma^\half}{\Zapabs^\half}\pap\w \in \Linfty$. Now from \eqref{form:RealImagTh} we have $\dis \Real\Th = -i\Dap\w$ and this implies that  $\dis \sigma^\half\Zapabs^\half\Real\Th \in \Linfty$
\medskip

\item $\dis \sigma^{\onebysix}  \Zapabs^\half \pap\frac{1}{\Zap} \in \Ltwo$,  $\dis \sigma^{\onebysix}  \Zapabs^\half \pap\frac{1}{\Zapabs} \in \Ltwo$, $\dis \frac{ \sigma^\onebysix}{\Zapabs^\half}\pap\w \in \Ltwo$
\medskip\\
Proof: $\dis \sigma^{\onebysix}  \Zapabs^\half \pap\frac{1}{\Zap} \in \Ltwo$ as it part of the energy $\Esigmazero$. Again using \eqref{form:RealImagTh} we can control the other terms.
\medskip

\item $\dis \sigma \pap\Th \in \Hhalf$
\medskip\\
Proof:  We first note that $\dis (\Zttbar -i)\Zap \in \Hhalf$ as it part of the energy $\Esigmaone$. But from the fundamental equation \eqref{form:Zttbar} we get
\[
(\Zttbar -i)\Zap= -i \Aone +\sigma \pap \Th
\]
We have already proven that $\Aone \in \Hhalf$ and hence $\sigma \pap\Th \in \Hhalf$
\medskip

\item $\dis \sigma^\twobythree\pap\Th \in \Ltwo$
\medskip\\
Proof: As $\dis \Th \in \Ltwo$ and $ \sigma\pap\Th \in \Hhalf$ we obtain the estimate from \lemref{lem:interpolation}. 
\medskip

\item $\dis \sigma^\twobythree \nobrac{\pap^2\frac{1}{\Zap}} \in \Ltwo$, $\dis \sigma^\twobythree \nobrac{\pap^2\frac{1}{\Zapabs}} \in \Ltwo$, $\dis \frac{\sigma^\twobythree}{\Zapabs}\pap^2\w \in \Ltwo$, $\dis \sigma^\twobythree\pap\Dapabs\w \in \Ltwo$
\medskip\\
Proof: Differentiating the equation \eqref{form:Th} we get
\[
\sigma^\twobythree \pap \Th = i\sigma^\twobythree \pap \brac{\frac{\Zap}{\Zapabs} \pap \frac{1}{\Zap}} -i\sigma^\twobythree \Real\cbrac{\pap \sqbrac{ \frac{\w}{\Zaphalf} ,\Hil}\brac{\Zaphalf\pap\frac{1}{\Zap}} }
\]
Hence from \propref{prop:commutator} we get
\[
 \norm[2]{\sigma^\twobythree \pap \brac{\frac{\Zap}{\Zapabs} \pap \frac{1}{\Zap}}} \lesssim \norm[2]{\sigma^\twobythree \pap \Th} + \norm*[\bigg][\infty]{\sigma^\half \pap \frac{\w}{\Zaphalf}}\norm[2]{\sigma^\onebysix \Zaphalf \pap \frac{1}{\Zap}} 
\]
From this and \eqref{form:RealImagTh} we get
\begingroup
\allowdisplaybreaks
\begin{align*}
 \lpar  \norm[2]{\sigma^\twobythree \nobrac{\pap^2\frac{1}{\Zap}}}  & \lesssim  \norm[2]{\sigma^\twobythree \pap \brac{\frac{\Zap}{\Zapabs} \pap \frac{1}{\Zap}}} + \norm*[\bigg][\infty]{\frac{\sigma^\half}{\Zapabs^\half}\pap\w}\norm[2]{\sigma^\onebysix \Zapabs^\half\pap\frac{1}{\Zap}} \\
  \norm[2]{\sigma^\twobythree \nobrac{\pap^2\frac{1}{\Zapabs}}} & \lesssim  \norm[2]{\sigma^\twobythree \pap \brac{\frac{\Zap}{\Zapabs} \pap \frac{1}{\Zap}}} \\
  \norm*[\bigg][2]{\frac{\sigma^\twobythree}{\Zapabs}\pap^2\w} & \lesssim  \norm[2]{\sigma^\twobythree \pap \brac{\frac{\Zap}{\Zapabs} \pap \frac{1}{\Zap}}} + \norm*[\bigg][\infty]{\frac{\sigma^\half}{\Zapabs^\half}\pap\w}\norm[2]{\sigma^\onebysix \Zapabs^\half\pap\frac{1}{\Zap}}
\end{align*}
\endgroup
and we easily obtain $\dis \sigma^\twobythree\pap\Dapabs\w \in \Ltwo$ from $\dis \frac{\sigma^\twobythree}{\Zapabs}\pap^2\w \in \Ltwo$ and we have
\[
\norm[2]{\sigma^\twobythree\pap\Dapabs\w} \lesssim  \norm*[\bigg][2]{\frac{\sigma^\twobythree}{\Zapabs}\pap^2\w} + \norm*[\bigg][\infty]{\frac{\sigma^\half}{\Zapabs^\half}\pap\w}\norm[2]{\sigma^\onebysix \Zapabs^\half\pap\frac{1}{\Zapabs}}
\]
\medskip

\item $\dis \sigma^\onebythree \Th \in \Linfty\cap \Hhalf$
\medskip\\
Proof: As $\dis \Th \in \Ltwo$ and $\sigma^\twobythree\pap\Th \in \Ltwo$ we have $\dis \sigma^\onebythree \Th \in  \Hhalf$ from \lemref{lem:interpolation}. Now as $\Th$ decays at infinity we have
\begin{align*}
\norm*[\big][\infty]{\sigma^\onebythree \Th}^2 = \norm*[\big][\infty]{\sigma^\twobythree \Th^2} \lesssim \sigma^\twobythree \int \abs{\pap (\Th^2)}\diff\ap \lesssim \norm*[2]{\Th}\norm*[2]{\sigma^\twobythree \pap\Th}
\end{align*}
\medskip

\item $\dis \sigma^\onebythree \pap \frac{1}{\Zap} \in \Linfty\cap\Hhalf$,  $\dis \sigma^\onebythree \pap \frac{1}{\Zapabs} \in \Linfty\cap\Hhalf$, $\dis \sigma^\onebythree \Dapabs\w \in \Linfty\cap\Hhalf$
\medskip\\
Proof: This is proved by exactly the same argument used above to show $\dis \sigma^\onebythree \Th \in \Linfty\cap \Hhalf$ 
\medskip

\item $\sigma \pap\Dap\Th \in \Ltwo$, $\sigma \Dapabs\pap\Th \in \Ltwo$, $\sigma \pap\Dapabs\Th \in \Ltwo$
\medskip\\
Proof: Taking a derivative in the fundamental equation \eqref{form:Zttbar} we get
\[
\Zttapbar = -i\Dap\Aone -i\Aone\pap\frac{1}{\Zap}   +\sigma \pap \Dap \Th
\]
Hence $\dis \norm[2]{\sigma \pap \Dap\Th} \lesssim \norm[2]{\Zttapbar} + \norm[2]{\Dapabs\Aone} + \norm[\infty]{\Aone}\norm*[\Big][2]{\pap\frac{1}{\Zap}}$. From this we get that $\dis \norm[2]{\sigma  \Dapabs\pap\Th} \lesssim \norm[2]{\sigma \pap \Dap\Th} + \norm*[\Big][\infty]{\sigma^\onebythree \pap \frac{1}{\Zap}}\norm*[\big][2]{\sigma^\twobythree \pap \Th}$. We can prove $\sigma \pap\Dapabs\Th \in \Ltwo$ similarly. 
\medskip

\item $\dis \frac{\sigma}{\Zapabs}\pap^3\frac{1}{\Zap} \in \Ltwo$, $\dis \frac{\sigma}{\Zapabs}\pap^3\frac{1}{\Zapabs} \in \Ltwo$, $\dis \frac{\sigma}{\Zapabs^2}\pap^3\w \in \Ltwo$
\medskip\\
Proof: We first observe that
\begin{align*}
\lpar \norm[2]{\frac{\sigma}{\Zapabs}\pap^2\brac*[\bigg]{\frac{\Zap}{\Zapabs}\pap\frac{1}{\Zap}} - \frac{\sigma}{\Zapbar}\pap^3\frac{1}{\Zap}} & \lesssim \norm*[\bigg][2]{\frac{\sigma^\twobythree}{\Zapabs}\pap^2\w}\norm[\infty]{\sigma^\onebythree \pap\frac{1}{\Zap}}  \\
& \quad + \norm*[\big][\infty]{\sigma^\onebythree \Dapabs\w}\norm*[\Big][2]{\sigma^\twobythree\pap^2\frac{1}{\Zap}}
\end{align*}
Hence the difference between them is controlled. This implies that we replace them with each other whenever we want. Now differentiating \eqref{form:Th} we get
\begin{align*}
\lpar \frac{\sigma }{\Zapabs}\pap^2\Th &= i \frac{\sigma}{\Zapabs}\pap^2\brac{ \frac{\Zap}{\Zapabs}\pap\frac{1}{\Zap}} + i\Real \cbrac{ \sqbrac{\frac{\sigma}{\Zapabs},\Hil}\pap^2\brac{ \frac{\Zap}{\Zapabs}\pap\frac{1}{\Zap}} } \\
& \quad -i\Real(\Id - \Hil)\cbrac{ \frac{\sigma}{\Zapabs}\pap^2\brac{ \frac{\Zap}{\Zapabs}\pap\frac{1}{\Zap}}}
\end{align*} 
Now we can replace $ \dis (\Id - \Hil)\cbrac{ \frac{\sigma}{\Zapabs}\pap^2\brac{ \frac{\Zap}{\Zapabs}\pap\frac{1}{\Zap}}}$ above with $\dis (\Id - \Hil)\cbrac{ \frac{\sigma}{\Zapbar}\pap^3\frac{1}{\Zap}}$ and rewrite it as $\dis  \sqbrac{\frac{\sigma}{\Zapbar},\Hil} \pap^3\frac{1}{\Zap}$. Hence from \propref{prop:commutator} we have
\begin{align*}
& \norm[2]{\frac{\sigma}{\Zapabs}\pap^2\brac{ \frac{\Zap}{\Zapabs}\pap\frac{1}{\Zap}}} \\
&   \lesssim \norm*[\Big][2]{\frac{\sigma }{\Zapabs}\pap^2\Th} + \norm*[\Big][\infty]{\sigma^\onebythree\pap\frac{1}{\Zap}}\cbrac{ \norm*[\Big][2]{\sigma^\twobythree \pap^2\frac{1}{\Zapabs}} + \norm*[\Big][2]{\sigma^\twobythree \pap^2\frac{1}{\Zapbar}}} \\
& \quad +   \norm*[\Big][2]{\frac{\sigma^\twobythree}{\Zapabs}\pap^2\w}\norm*[\Big][\infty]{\sigma^\onebythree \pap\frac{1}{\Zap}} + \norm*[\big][\infty]{\sigma^\onebythree \Dapabs\w}\norm*[\Big][2]{\sigma^\twobythree\pap^2\frac{1}{\Zap}}
\end{align*}
Hence $\dis \frac{\sigma}{\Zapabs}\pap^3\frac{1}{\Zap} \in \Ltwo$. By using \eqref{form:RealImagTh} we get that $\dis \frac{\sigma}{\Zapabs}\pap^3\frac{1}{\Zapabs} \in \Ltwo$,$\dis \frac{\sigma}{\Zapabs}\pap^2\Dap\w \in \Ltwo$ and so
\begin{align*}
\lpar \norm*[\bigg][2]{\frac{\sigma}{\Zapabs^2}\pap^3\w} & \lesssim \norm*[\bigg][2]{\frac{\sigma}{\Zapabs}\pap^2\Dap\w} +  \norm*[\Big][2]{\frac{\sigma^\twobythree}{\Zapabs}\pap^2\w}\norm*[\Big][\infty]{\sigma^\onebythree \pap\frac{1}{\Zap}}  \\
 & \quad + \norm*[\big][\infty]{\sigma^\onebythree \Dapabs\w}\norm*[\Big][2]{\sigma^\twobythree\pap^2\frac{1}{\Zap}} 
\end{align*}
\medskip

\item $\dis \frac{\sigma^\half}{\Zapabs^\half}\pap^2\frac{1}{\Zap} \in \Ltwo$, $\dis \frac{\sigma^\half}{\Zapabs^\half}\pap^2\frac{1}{\Zapabs} \in \Ltwo$, $\dis \frac{\sigma^\half}{\Zapabs^\threebytwo}\pap^2\w \in \Ltwo$ and $\dis \frac{\sigma^\half}{\Zapabs^\half}\pap\Th \in \Ltwo $
\medskip\\
Proof: $\dis \frac{\sigma^\half}{\Zapabs^\half}\pap^2\frac{1}{\Zap} \in \Ltwo$ as it part of the energy $\Esigmazero$. Now using \eqref{form:RealImagTh} we get
\begin{align*}
\norm[2]{\frac{\sigma^\half}{\Zapabs^\half}\pap^2\frac{1}{\Zapabs}} & \lesssim \norm[2]{\frac{\sigma^\half}{\Zapabs^\half}\pap^2\frac{1}{\Zap}} + \norm[\infty]{\frac{\sigma^\half}{\Zapabs^\half}\pap\w}\norm[2]{\pap\frac{1}{\Zap}} \\
\norm[2]{\frac{\sigma^\half}{\Zapabs^\threebytwo}\pap^2\w} & \lesssim \norm[2]{\frac{\sigma^\half}{\Zapabs^\half}\pap^2\frac{1}{\Zap}} + \norm[\infty]{\frac{\sigma^\half}{\Zapabs^\half}\pap\w}\norm[2]{\pap\frac{1}{\Zap}}
\end{align*}
Now differentiating the equation \eqref{form:Th} we get using \propref{prop:commutator}
\begin{align*}
\norm[2]{\frac{\sigma^\half}{\Zapabs^\half}\pap\Th }  \lesssim \norm[2]{\frac{\sigma^\half}{\Zapabs^\half}\pap^2\frac{1}{\Zap}} + \brac{\norm[\infty]{\frac{\sigma^\half}{\Zapabs^\half}\pap\w} + \norm[\infty]{\sigma^\half\Zapabs^\half\pap\frac{1}{\Zapabs}} }\norm[2]{\pap\frac{1}{\Zap}}
\end{align*}
\medskip

\item $\dis \sigma^{\half}  \Zapabs^\half \pap\frac{1}{\Zap} \in \Wcal$,  $\dis \sigma^{\half}  \Zapabs^\half \pap\frac{1}{\Zapabs} \in \Wcal$, $\dis \frac{ \sigma^\half}{\Zapabs^\half}\pap\w \in \Wcal$, $\dis \sigma^\half\Zapabs^\half\Real\Th \in \Wcal$
\medskip\\
Proof: We will only show that  $\dis \sigma^{\half}  \Zapabs^\half \pap\frac{1}{\Zap} \in \Wcal$ and the rest are proved similarly. As $\dis \sigma^{\half}  \Zapabs^\half \pap\frac{1}{\Zap} \in \Linfty$ we only need to show $\dis  \Dapabs\brac{ \sigma^{\half} \Zapabs^\half \pap\frac{1}{\Zap}} \in \Ltwo$. Now
\begin{align*}
\lpar \norm[2]{\Dapabs\brac{ \sigma^{\half} \Zapabs^\half \pap\frac{1}{\Zap}}}  \lesssim \norm*[\bigg][2]{ \frac{\sigma^\half}{\Zapabs^\half}\pap^2\frac{1}{\Zap}} + \norm[\infty]{\sigma^\half\Zapabs^\half\pap\frac{1}{\Zapabs}}\norm[2]{\pap\frac{1}{\Zap}}
\end{align*}
\medskip

\item $\dis \frac{\sigma^\fivebysix}{\Zapabs^\half}\pap\Th \in \Linfty\cap\Hhalf$
\medskip\\
Proof:  As $\dis \frac{\sigma^\fivebysix}{\Zapabs^\half}\pap\Th$ decays at infinity, we use  \propref{prop:LinftyHhalf} with $\dis w = \frac{\sigma^\onebysix}{\Zapabs^\half}$ to get
\[
\lpar \norm*[\bigg][\Linfty\cap\Hhalf]{\frac{\sigma^\fivebysix}{\Zapabs^\half}\pap\Th}^2 \lesssim \norm*[\big][2]{\sigma^\twobythree\pap\Th}\norm*[\big][2]{\sigma\pap\Dapabs\Th} + \norm*[\big][2]{\sigma^\twobythree\pap\Th}^2\norm*[\bigg][2]{\sigma^\onebysix \Zapabs^\half\pap\frac{1}{\Zapabs}}^2
\]
\medskip

\item $\dis \frac{\sigma^\fivebysix}{\Zapabs^\half}\pap^2\frac{1}{\Zap} \in \Linfty\cap\Hhalf$, $\dis \frac{\sigma^\fivebysix}{\Zapabs^\half}\pap^2\frac{1}{\Zapabs} \in \Linfty\cap\Hhalf$, $\dis \frac{\sigma^\fivebysix}{\Zapabs^\threebytwo}\pap^2\w \in \Linfty\cap\Hhalf$
\medskip\\
Proof: This is proved by exactly the same argument used above to show $\dis \frac{\sigma^\fivebysix}{\Zapabs^\half}\pap\Th \in \Linfty\cap\Hhalf$
\medskip

\item $\dis \frac{\sigma^\half}{\Zapabs^\threebytwo}\pap\Th \in \Ccal$
\medskip\\
Proof: It was proved earlier that $\dis \frac{\sigma^\half}{\Zapabs^\half}\pap\Th \in \Ltwo$. Also $\dis \frac{\sigma^\half}{\Zapabs^\threebytwo}\pap\Th \in \Hhalf$ as it part of the energy $\dis \Esigmathree$ 
\medskip

\item $\dis \frac{\sigma^\half}{\Zapabs^\threebytwo}\pap^2\frac{1}{\Zap} \in \Ccal$, $\dis \frac{\sigma^\half}{\Zapabs^\threebytwo}\pap^2\frac{1}{\Zapabs} \in \Ccal$, $\dis \frac{\sigma^\half}{\Zapabs^\fivebytwo}\pap^2\w \in \Ccal$
\medskip\\
Proof: As  $\dis \frac{\sigma^\half}{\Zapabs^\half}\pap^2\frac{1}{\Zap} \in \Ltwo$, $\dis \frac{\sigma^\half}{\Zapabs^\half}\pap^2\frac{1}{\Zapabs} \in \Ltwo$, $\dis \frac{\sigma^\half}{\Zapabs^\threebytwo}\pap^2\w \in \Ltwo$ we only have to prove the $\Hhalf$ estimates. Using \lemref{lem:CW} we get
\[
\lpar \norm*[\bigg][\Ccal]{\frac{\sigma^\half}{\Zapabs^\threebytwo}\pap\brac{\frac{\Zap}{\Zapabs}\pap\frac{1}{\Zap}} - \frac{\w\sigma^\half}{\Zapabs^\threebytwo}\pap^2\frac{1}{\Zap}} \lesssim \norm*[\bigg][\Wcal]{\frac{\sigma^\half}{\Zapabs^\half}\pap\w}\norm[\Ccal]{\Dapabs\frac{1}{\Zap}}
\]
Hence the difference between them is controlled. This implies that we replace them with each other whenever we want. Now by differentiating \eqref{form:Th} we get
\begin{align*}
\lpar \frac{\sigma^\half}{\Zapabs^\threebytwo}\pap\Th &= i\frac{\sigma^\half}{\Zapabs^\threebytwo}\pap\brac{\frac{\Zap}{\Zapabs}\pap\frac{1}{\Zap}} + i\Real\cbrac{ \sqbrac{\frac{\sigma^\half}{\Zapabs^\threebytwo},\Hil}\pap \brac{\frac{\Zap}{\Zapabs}\pap\frac{1}{\Zap}} } \\
& \quad -i\Real(\Id - \Hil)\cbrac{\frac{\sigma^\half}{\Zapabs^\threebytwo}\pap\brac{\frac{\Zap}{\Zapabs}\pap\frac{1}{\Zap}}}
\end{align*}
Now we replace $\dis (\Id - \Hil)\cbrac{\frac{\sigma^\half}{\Zapabs^\threebytwo}\pap\brac{\frac{\Zap}{\Zapabs}\pap\frac{1}{\Zap}}}$ above with $\dis (\Id - \Hil)\cbrac{\frac{\w\sigma^\half}{\Zapabs^\threebytwo}\pap^2\frac{1}{\Zap}}$ and rewrite it as $\dis \sqbrac{\frac{\w\sigma^\half}{\Zapabs^\threebytwo} ,\Hil} \pap^2\frac{1}{\Zap}$. Hence using \propref{prop:commutator} we have 
\begin{align*}
\lpar \norm*[\bigg][\Hhalf]{\frac{\sigma^\half}{\Zapabs^\threebytwo}\pap\brac{\frac{\Zap}{\Zapabs}\pap\frac{1}{\Zap}}} &\lesssim  \norm*[\bigg][\Hhalf]{\frac{\sigma^\half}{\Zapabs^\threebytwo}\pap\Th} + \norm*[\bigg][2]{\sigma^\half\pap^2\frac{1}{\Zapabs^\threebytwo}}\norm[2]{\pap\frac{1}{\Zap}} \\
& \quad + \norm*[\bigg][2]{\sigma^\half\pap^2\frac{\w}{\Zapabs^\threebytwo}}\norm[2]{\pap\frac{1}{\Zap}} +  \norm*[\bigg][\Wcal]{\frac{\sigma^\half}{\Zapabs^\half}\pap\w}\norm[\Ccal]{\Dapabs\frac{1}{\Zap}}
\end{align*}
Note that we can easily show $\dis \sigma^\half\pap^2\frac{1}{\Zapabs^\threebytwo} \in \Ltwo$, $\dis \sigma^\half\pap^2\frac{w}{\Zapabs^\threebytwo} \in \Ltwo$ by  using Leibniz rule and controlling each individual term. Hence $\dis \frac{\sigma^\half}{\Zapabs^\threebytwo}\pap\brac{\frac{\Zap}{\Zapabs}\pap\frac{1}{\Zap}} \in \Ccal$, $\dis \frac{\w\sigma^\half}{\Zapabs^\threebytwo}\pap^2\frac{1}{\Zap} \in \Ccal$. As $\wbar \in \Wcal$ by using \lemref{lem:CW} we get $\dis \frac{\sigma^\half}{\Zapabs^\threebytwo}\pap^2\frac{1}{\Zap} \in \Ccal$. Now using \eqref{form:RealImagTh} we easily get $\dis \frac{\sigma^\half}{\Zapabs^\threebytwo}\pap^2\frac{1}{\Zapabs} \in \Ccal$, $\dis \frac{\sigma^\half}{\Zapabs^\threebytwo}\pap\Dap\w \in \Ccal$. Hence by \lemref{lem:CW} we have
\[
\lpar \norm*[\bigg][\Ccal]{\frac{\sigma^\half}{\Zapabs^\fivebytwo}\pap^2\w} \lesssim \norm*[\bigg][\Ccal]{\frac{\sigma^\half}{\Zapabs^\threebytwo}\pap\Dap\w}\norm[\Wcal]{\w} + \norm[\Ccal]{\Dapabs\frac{1}{\Zap}}\norm*[\bigg][\Wcal]{\frac{\sigma^\half}{\Zapabs^\half}\pap\w}\norm[\Wcal]{\w}
\]
\medskip

\item $\dis \sigma \Dapbar\Dap\Th \in \Ccal$, $\dis \sigma \Dap^2\Th \in \Ccal $, $\dis \sigma \Dapabs^2\Th \in \Ccal $, $\dis \frac{\sigma}{\Zapabs^2}\pap^2\Th \in \Ccal$
\medskip\\
Proof: Applying the derivative $\Dapbar$ to the fundamental equation \eqref{form:Zttbar} we get
\[
\Dapbar\Zttbar  = -i \Aone\Dapbar\frac{1}{\Zap}  - \frac{i}{\Zapabs^2}\pap\Aone +\sigma \Dapbar \Dap \Th
\]
Hence using \lemref{lem:CW} we get
\[
 \norm[\Ccal]{\sigma\Dapbar\Dap\Th} \lesssim \norm[\Ccal]{\Dapbar\Zttbar} + \norm[\Ccal]{\Dapbar\frac{1}{\Zap}}\norm[\Wcal]{\Aone} + \norm*[\bigg][\Ccal]{\frac{1}{\Zapabs^2}\pap\Aone}
\]
Now as $\wbar \in \Wcal$, by \lemref{lem:CW} we get $\sigma\Dap^2\Th \in \Ccal$. Now we see that
\[
\sigma\Dapbar\Dap\Th = \sigma\brac{\w\Zapabs^\half\pap\frac{1}{\Zap}} \brac{\frac{1}{\Zapabs^\threebytwo}\pap\Th} + \frac{\sigma}{\Zapabs^2}\pap^2\Th
\]
Hence again by \lemref{lem:CW} we have
\[
\norm*[\bigg][\Ccal]{\frac{\sigma}{\Zapabs^2}\pap^2\Th} \lesssim \norm[\Ccal]{\sigma\Dapbar\Dap\Th} + \norm*[\bigg][\Ccal]{\frac{\sigma^\half}{\Zapabs^\threebytwo}\pap\Th}\norm[\Wcal]{\sigma^\half\Zapabs^\half\pap\frac{1}{\Zap}}\norm[\Wcal]{\w}
\]
By a similar argument we get $\sigma\Dapabs^2\Th \in \Ccal$
\medskip

\item $\dis \frac{\sigma}{\Zapabs^2}\pap^3\frac{1}{\Zap} \in \Ccal$, $\dis \frac{\sigma}{\Zapabs^2}\pap^3\frac{1}{\Zapabs} \in \Ccal$, $\dis \sigma\pap^3\frac{1}{\Zapabs^3} \in \Ccal$, $\dis \frac{\sigma}{\Zapabs^3}\pap^3\w \in \Ccal$
\medskip\\
Proof: As $\dis \frac{\sigma}{\Zapabs}\pap^3\frac{1}{\Zap} \in \Ltwo$, $\dis \frac{\sigma}{\Zapabs}\pap^3\frac{1}{\Zapabs} \in \Ltwo$, $\dis \frac{\sigma}{\Zapabs^2}\pap^3\w \in \Ltwo$ we only need to show the $\Hhalf$ estimates. Now observe that
\begin{align*}
\lpar \norm*[\bigg][\Ccal]{\frac{\sigma}{\Zapabs^2}\pap^2\brac{\frac{\Zap}{\Zapabs}\pap\frac{1}{\Zap}} - \frac{\w\sigma}{\Zapabs^2}\pap^3\frac{1}{\Zap}} & \lesssim \norm*[\bigg][\Ccal]{\frac{\sigma^\half}{\Zapabs^\fivebytwo}\pap^2\w}\norm[\Wcal]{\sigma^\half\Zapabs^\half\pap\frac{1}{\Zap}} \\
& \quad + \norm*[\bigg][\Wcal]{\frac{\sigma^\half}{\Zapabs^\half}\pap\w}\norm*[\bigg][\Ccal]{\frac{\sigma^\half}{\Zapabs^\threebytwo}\pap^2\frac{1}{\Zap}}
\end{align*}
Hence the difference between them is controlled. This implies that we replace them with each other whenever we want. Now differentiating the equation \eqref{form:Th} we get
\begin{align*}
\lpar \frac{\sigma}{\Zapabs^2}\pap^2\Th &= i\frac{\sigma}{\Zapabs^2}\pap^2\brac{\frac{\Zap}{\Zapabs}\pap\frac{1}{\Zap}} + i\Real\cbrac{ \sqbrac{\frac{\sigma}{\Zapabs^2},\Hil}\pap^2 \brac{\frac{\Zap}{\Zapabs}\pap\frac{1}{\Zap}} } \\
& \quad -i\Real(\Id - \Hil)\cbrac{\frac{\sigma}{\Zapabs^2}\pap^2\brac{\frac{\Zap}{\Zapabs}\pap\frac{1}{\Zap}} }
\end{align*}
Now we replace $\dis (\Id - \Hil)\cbrac{\frac{\sigma}{\Zapabs^2}\pap^2\brac{\frac{\Zap}{\Zapabs}\pap\frac{1}{\Zap}}}$ above with $\dis (\Id - \Hil)\cbrac{ \frac{\w\sigma}{\Zapabs^2}\pap^3\frac{1}{\Zap}}$ and rewrite it as $\dis \sqbrac{\frac{\w\sigma}{\Zapabs^2} ,\Hil} \pap^3\frac{1}{\Zap}$. Hence using \propref{prop:commutator} we have
\begin{flalign*}
\lpar & \norm*[\bigg][\Hhalf]{\frac{\sigma}{\Zapabs^2}\pap^2\brac{\frac{\Zap}{\Zapabs}\pap\frac{1}{\Zap}}}  \\
& \lesssim \norm*[\bigg][\Hhalf]{\frac{\sigma}{\Zapabs^2}\pap^2\Th} + \norm*[\bigg][2]{\sigma\pap^3\frac{1}{\Zapabs^2}}\norm[2]{\frac{\Zap}{\Zapabs}\pap\frac{1}{\Zap}}  + \norm*[\bigg][2]{\sigma\pap^3\frac{\w}{\Zapabs^2}}\norm[2]{\pap\frac{1}{\Zap}} \\
& \quad  + \norm*[\bigg][\Ccal]{\frac{\sigma^\half}{\Zapabs^\fivebytwo}\pap^2\w}\norm[\Wcal]{\sigma^\half\Zapabs^\half\pap\frac{1}{\Zap}} + \norm*[\bigg][\Wcal]{\frac{\sigma^\half}{\Zapabs^\half}\pap\w}\norm*[\bigg][\Ccal]{\frac{\sigma^\half}{\Zapabs^\threebytwo}\pap^2\frac{1}{\Zap}} 
\end{flalign*}
Note that we can easily show $\dis \sigma\pap^3\frac{1}{\Zapabs^2} \in \Ltwo$, $\dis \sigma\pap^3\frac{\w}{\Zapabs^2} \in \Ltwo$ by  using Leibniz rule and controlling each individual term. Hence $\dis \frac{\sigma}{\Zapabs^2}\pap^2\brac{\frac{\Zap}{\Zapabs}\pap\frac{1}{\Zap}} \in \Ccal$, $\dis  \frac{\w\sigma}{\Zapabs^2}\pap^3\frac{1}{\Zap} \in \Ccal$. As $\wbar \in \Wcal$ by using \lemref{lem:CW} we get $\dis \frac{\sigma}{\Zapabs^2}\pap^3\frac{1}{\Zap} \in \Ccal$. Now using \eqref{form:RealImagTh} we easily get $\dis \frac{\sigma}{\Zapabs^2}\pap^3\frac{1}{\Zapabs} \in \Ccal$, $\dis \frac{\sigma}{\Zapabs^2}\pap^2\Dap\w \in \Ccal$. Hence using \lemref{lem:CW} we have
\begingroup
\allowdisplaybreaks
\begin{align*}
\lpar \norm*[\bigg][\Ccal]{\frac{\sigma}{\Zapabs^3}\pap^3\w} & \lesssim  
\begin{aligned}[t]
&  \norm*[\bigg][\Ccal]{\frac{\sigma}{\Zapabs^2}\pap^2\Dap\w}\norm[\Wcal]{\w} + \norm*[\bigg][\Ccal]{\frac{\sigma^\half}{\Zapabs^\fivebytwo}\pap^2\w}\norm[\Wcal]{\sigma^\half\Zapabs^\half\pap\frac{1}{\Zap}}\norm[\Wcal]{\w} \\*
& + \norm*[\bigg][\Ccal]{\frac{\sigma^\half}{\Zapabs^\threebytwo}\pap^2\frac{1}{\Zap}}\norm*[\bigg][\Wcal]{\frac{\sigma^\half}{\Zapabs^\half}\pap\w}\norm[\Wcal]{\w}
\end{aligned}\\
\norm*[\bigg][\Ccal]{\sigma\pap^3\frac{1}{\Zapabs^3}} & \lesssim 
\begin{aligned}[t]
& \norm*[\bigg][\Ccal]{\frac{\sigma}{\Zapabs^2}\pap^3\frac{1}{\Zapabs}} + \norm[\Wcal]{\sigma^\half\Zapabs^\half\pap\frac{1}{\Zapabs}}\norm*[\bigg][\Ccal]{\frac{\sigma^\half}{\Zapabs^\threebytwo}\pap^2\frac{1}{\Zapabs}} \\*
& +  \norm[\Wcal]{\sigma^\half\Zapabs^\half\pap\frac{1}{\Zapabs}}^2\norm[\Ccal]{\frac{1}{\Zapabs}\pap\frac{1}{\Zapabs}}
\end{aligned}
\end{align*}
\endgroup
\medskip

\item $\dis \frac{\sigma^\half}{\Zapabs^\half}\pap\Ztapbar \in \Ltwo$, $\dis \sigma^\half\Zapabs^\half\pap\Dapabs\Ztbar \in \Ltwo$ and $\dis \sigma^\half\Zapabs^\half\pap\Dapbar\Ztbar \in \Ltwo$
\medskip\\
Proof: We have $\dis \frac{\sigma^\half}{\Zapabs^\half}\pap\Ztapbar \in \Ltwo$ as it part of the energy $\Esigmatwo$. Now observe that
\begin{align*}
\norm[2]{\sigma^\half\Zapabs^\half\pap\Dapabs\Ztbar} \lesssim \norm[\infty]{\sigma^\half\Zapabs^\half\pap\frac{1}{\Zapabs}}\norm[2]{\Ztapbar} + \norm[2]{ \frac{\sigma^\half}{\Zapabs^\half}\pap\Ztapbar}
\end{align*}
We prove $\dis \sigma^\half\Zapabs^\half\pap\Dapbar\Ztbar \in \Ltwo$ similarly.
\medskip

\item $\dis \frac{\sigma^\half}{\Zapabs^\fivebytwo}\pap^2\Ztapbar  \in \Ltwo$, $\dis \frac{\sigma^\half}{\Zapabs^\threebytwo}\pap^2\Dapbar\Ztbar \in \Ltwo$, $\dis \frac{\sigma^\half}{\Zapabs^\half}\pap\Dapabs\Dapbar\Ztbar \in \Ltwo$ and in the same way $\dis \frac{\sigma^\half}{\Zapabs^\half}\Dapabs^2\Ztapbar \in \Ltwo$, $\dis \frac{\sigma^\half}{\Zapabs^\half}\pap\Dapbar^2\Ztbar \in \Ltwo$
\medskip\\
Proof: Note that $\dis \frac{\sigma^\half}{\Zapabs^\half}\pap\Dapabs\Dapbar\Ztbar \in \Ltwo$ as it part of the energy $\Esigmafour$. Now we have
\begin{flalign*}
\lpar \norm*[\bigg][2]{\frac{\sigma^\half}{\Zapabs^\threebytwo}\pap^2\Dapbar\Ztbar} \lesssim \norm*[\bigg][2]{\frac{\sigma^\half}{\Zapabs^\half}\pap\Dapabs\Dapbar\Ztbar}   + \norm[\infty]{\sigma^\half\Zapabs^\half\pap\frac{1}{\Zapabs}}\norm[2]{\Dapabs\Dapbar\Ztbar}
\end{flalign*}
Similarly we see that
\begin{align*}
\lpar \norm*[\bigg][2]{\frac{\sigma^\half}{\Zapabs^\fivebytwo}\pap^2\Ztapbar} & \lesssim \norm*[\bigg][2]{\frac{\sigma^\half}{\Zapabs^\threebytwo}\pap^2\Dapbar\Ztbar} + \norm*[\bigg][2]{\frac{\sigma^\half}{\Zapabs^\half}\pap^2\frac{1}{\Zap}}\norm[\infty]{\Dapabs\Ztbar} \\
& \quad + \norm[\infty]{\sigma^\half\Zapabs^\half\pap\frac{1}{\Zapbar}}\norm[2]{\frac{1}{\Zapabs^2}\pap\Ztapbar}
\end{align*}
We also have
\begin{flalign*}
\lpar \norm*[\bigg][2]{\frac{\sigma^\half}{\Zapabs^\half}\Dapabs^2\Ztapbar} \lesssim \norm[\infty]{\sigma^\half\Zapabs^\half\pap\frac{1}{\Zapabs}}\norm*[\bigg][2]{\frac{1}{\Zapabs^2}\pap\Ztapbar} + \norm*[\bigg][2]{\frac{\sigma^\half}{\Zapabs^\fivebytwo}\pap^2\Ztapbar}
\end{flalign*}
The estimate for $\dis \frac{\sigma^\half}{\Zapabs^\half}\pap\Dapbar^2\Ztbar \in \Ltwo$ is shown in a similar way. 
\medskip

\item $\dis \frac{\sigma^\half}{\Zapabs^\threebytwo}\pap\Ztapbar \in \Wcal\cap\Ccal$, $\dis \frac{\sigma^\half}{\Zapabs^\half}\pap\Dapabs\Ztbar \in \Wcal\cap\Ccal$, $\dis \frac{\sigma^\half}{\Zapabs^\half}\pap\Dapbar\Ztbar \in \Wcal\cap\Ccal$ and also $\dis \frac{\sigma^\half}{\Zapabs^\half}\pap\Dap\Ztbar \in \Wcal\cap\Ccal$
\medskip\\
Proof: Note that $\dis \frac{\sigma^\half}{\Zapabs^\half }\pap\Ztapbar \in \Ltwo$ as it part of the energy $\Esigmaone$ and $\dis \frac{\sigma^\half}{\Zapabs^\threebytwo}\pap\Ztapbar \in \Hhalf$ as it part of the energy $\Esigmatwo$. Hence $\dis \frac{\sigma^\half}{\Zapabs^\threebytwo}\pap\Ztapbar \in \Ccal$. Now observe that 
\begin{align*}
\norm*[\Bigg][2]{\Dapabs\brac{ \frac{\sigma^\half}{\Zapabs^\threebytwo}\pap\Ztapbar}} \lesssim  \norm*[\Bigg][2]{\frac{\sigma^\half}{\Zapabs^\fivebytwo}\pap^2\Ztapbar} + \norm*[\bigg][\infty]{\sigma^\half\Zapabs^\half\pap\frac{1}{\Zapabs}}\norm*[\Bigg][2]{\frac{1}{\Zapabs^2}\pap\Ztapbar}
\end{align*}
Now as $\dis \frac{\sigma^\half}{\Zapabs^\threebytwo}\pap\Ztapbar$ decays at infinity, we use  \propref{prop:LinftyHhalf} with $\dis w = \frac{1}{\Zapabs}$ to get
\begin{align*}
\lpar  \norm*[\Bigg][\infty]{ \frac{\sigma^\half}{\Zapabs^\threebytwo}\pap\Ztapbar}^2  \lesssim  \norm*[\Bigg][2]{ \frac{\sigma^\half}{\Zapabs^\half}\pap\Ztapbar}\norm*[\Bigg][2]{\Dapabs\brac{ \frac{\sigma^\half}{\Zapabs^\threebytwo}\pap\Ztapbar}} 
\end{align*}
Hence we have proved that $\dis \frac{\sigma^\half}{\Zapabs^\threebytwo}\pap\Ztapbar \in \Wcal\cap\Ccal$. Now using \lemref{lem:CW} we see that
\begin{flalign*}
&& \norm*[\Bigg][\Wcal\cap\Ccal]{\frac{\sigma^\half}{\Zapabs^\half}\pap\Dapbar\Ztbar} \lesssim \enspace \norm*[\Bigg][\Wcal\cap\Ccal]{\frac{\sigma^\half}{\Zapabs^\threebytwo}\pap\Ztapbar}\norm[\Wcal]{\w}+\norm[\Wcal\cap\Ccal]{\Dapabs\Ztbar}\norm[\Wcal]{\sigma^\half\Zapabs^\half\pap\frac{1}{\Zap}}
\end{flalign*}
We prove $\dis \frac{\sigma^\half}{\Zapabs^\half}\pap\Dapabs\Ztbar \in \Wcal\cap\Ccal$,$\dis \frac{\sigma^\half}{\Zapabs^\half}\pap\Dap\Ztbar \in \Wcal\cap\Ccal$ similarly.
\medskip

\item $\dis \frac{\sigma^\onebysix}{\Zapabs^\threebytwo}\pap\Ztapbar \in \Ltwo$, $\dis \frac{\sigma^\onebysix}{\Zapabs^\half}\pap\Dapabs\Ztbar \in \Ltwo$, $\dis \frac{\sigma^\onebysix}{\Zapabs^\half}\pap\Dapbar\Ztbar \in \Ltwo$
\medskip\\
Proof: We interpolate between $\dis \frac{\sigma^\half}{\Zapabs^\half}\pap\Ztapbar \in \Ltwo$ and $\dis \frac{1}{\Zapabs^2}\pap\Ztapbar \in \Ltwo$. We simply decompose $ \dis \abs{\frac{\sigma^\onebysix}{\Zapabs^\threebytwo}\pap\Ztapbar} = \abs{ \frac{\sigma^\half}{\Zapabs^\half}\pap\Ztapbar}^\onebythree\abs{\frac{1}{\Zapabs^2}\pap\Ztapbar}^\twobythree$  and use Holder inequality to obtain
\begin{align*}
\norm[2]{\frac{\sigma^\onebysix}{\Zapabs^\threebytwo}\pap\Ztapbar} \lesssim \norm[2]{\frac{\sigma^\half}{\Zapabs^\half}\pap\Ztapbar}^\onebythree\norm[2]{\frac{1}{\Zapabs^2}\pap\Ztapbar}^\twobythree
\end{align*}
We also see that
\begin{align*}
\norm[2]{\frac{\sigma^\onebysix}{\Zapabs^\half}\pap\Dapabs\Ztbar} \lesssim \norm[2]{\frac{\sigma^\onebysix}{\Zapabs^\threebytwo}\pap\Ztapbar} + \norm[2]{\sigma^\onebysix\Zapabs^\half\pap\frac{1}{\Zapabs}}\norm[\infty]{\Dapabs\Ztbar}
\end{align*}
The proof of $\dis \frac{\sigma^\onebysix}{\Zapabs^\half}\pap\Dapbar\Ztbar \in \Ltwo$ is similar.
\medskip

\item $\dis \frac{\sigma^\onebythree}{\Zapabs}\pap\Ztapbar \in  \Ltwo$, $\dis \sigma^\onebythree\pap\Dapabs\Ztbar \in  \Ltwo$, $\dis \sigma^\onebythree\pap\Dapbar\Ztbar \in  \Ltwo$
\medskip\\
Proof: We observe that
\begin{align*}
\norm[2]{\frac{\sigma^\onebythree}{\Zapabs}\pap\Ztapbar}^2 \lesssim \norm[2]{\frac{\sigma^\half}{\Zapabs^\half}\pap\Ztapbar}\norm[2]{\frac{\sigma^\onebysix}{\Zapabs^\threebytwo}\pap\Ztapbar}
\end{align*}
Similarly we have
\begin{align*}
\norm[2]{\sigma^\onebythree\pap\Dapabs\Ztbar}^2 \lesssim \norm[2]{\sigma^\half\Zapabs^\half\pap\Dapabs\Ztbar}\norm[2]{\frac{\sigma^\onebysix}{\Zapabs^\half}\pap\Dapabs\Ztbar}
\end{align*}
We prove $\dis \sigma^\onebythree\pap\Dapbar\Ztbar \in  \Ltwo$ in the same way as above.
\medskip

\item $\dis \sigma^\onebysix \frac{\Ztapbar}{\Zapabs^\half} \in \Wcal$
\medskip\\
Proof: We use  \propref{prop:LinftyHhalf} with $\dis w = \frac{\sigma^\onebysix}{\Zapabs^\half}$ to get
\[
\lpar \norm[\infty]{\sigma^\onebysix \frac{\Ztapbar}{\Zapabs^\half}}^2 \lesssim \norm[2]{\Ztapbar}\norm[2]{\sigma^\onebythree\pap\Dapabs\Ztbar} + \norm[2]{\Ztapbar}^2\norm[2]{\sigma^\onebysix\Zapabs^\half\pap\frac{1}{\Zapabs}}^2
\]
We also have
\[
\lpar \norm[2]{\Dapabs\brac*[\bigg]{ \sigma^\onebysix \frac{\Ztapbar}{\Zapabs^\half}}} \lesssim \norm*[\bigg][2]{\frac{\sigma^\onebysix}{\Zapabs^\threebytwo}\pap\Ztapbar} + \norm[2]{\sigma^\onebysix\Zapabs^\half\pap\frac{1}{\Zapabs}}\norm[\infty]{\Dapabs\Ztbar}
\]
\medskip

\item $\dis \sigma^\onebysix\pap\Pa\brac{\frac{\Zt}{\Zaphalf}} \in \Linfty $
\medskip\\
Proof: We see that
\begin{align*}
 2\sigma^\onebysix\pap \Pa \brac{\frac{\Zt}{\Zaphalf}} & = \sigma^\onebysix(\Id - \Hil)\brac{\frac{\Ztap}{\Zaphalf}} + \sigma^\onebysix(\Id - \Hil)\brac{\Zt\pap\frac{1}{\Zaphalf}} \\
& = 2\sigma^\onebysix\frac{\Ztap}{\Zaphalf} + \sigma^\onebysix\sqbrac{\frac{1}{\Zaphalf},\Hil}\Ztap + \sigma^\onebysix\sqbrac{\Zt, \Hil}\brac*[\Bigg]{\pap\frac{1}{\Zaphalf} } 
\end{align*}
Hence using \propref{prop:commutator} we have
\begin{align*}
\norm[\infty]{\sigma^\onebysix\pap\Pa\brac{\frac{\Zt}{\Zaphalf}}} \lesssim \norm[\infty]{\sigma^\onebysix \frac{\Ztapbar}{\Zapabs^\half}} + \norm[2]{\sigma^\onebysix\Zapabs^\half\pap\frac{1}{\Zap}}\norm[2]{\Ztap}
\end{align*}
\medskip

\item $\dis \frac{\sigma^\onebythree}{\Zapabs^2}\pap\Ztapbar \in \Linfty\cap\Hhalf$
\medskip\\
Proof: We first observe that 
\[
\lpar \norm*[\bigg][2]{\pap\brac*[\Bigg]{\frac{\sigma^\half}{\Zapabs^\fivebytwo}\pap\Ztapbar}} \lesssim \norm[2]{\pap\frac{1}{\Zapabs}}\norm*[\bigg][\infty]{\frac{\sigma^\half}{\Zapabs^\threebytwo}\pap\Ztapbar} + \norm*[\bigg][2]{\frac{\sigma^\half}{\Zapabs^\fivebytwo}\pap^2\Ztapbar}
\]
We  now use  \propref{prop:LinftyHhalf} with $\dis w = \frac{\sigma^\onebysix}{\Zapabs^\half}$ to get
\begin{align*}
\lpar & \norm*[\bigg][\Linfty\cap\Hhalf]{\frac{\sigma^\onebythree}{\Zapabs^2}\pap\Ztapbar}^2 \\
&  \lesssim \norm*[\bigg][2]{\frac{\sigma^\onebysix}{\Zapabs^\threebytwo}\pap\Ztapbar}\norm*[\bigg][2]{\pap\brac*[\Bigg]{\frac{\sigma^\half}{\Zapabs^\fivebytwo}\pap\Ztapbar}} + \norm*[\bigg][2]{\frac{\sigma^\onebysix}{\Zapabs^\threebytwo}\pap\Ztapbar}^2\norm[2]{\sigma^\onebysix\Zapabs^\half\pap\frac{1}{\Zapabs}}^2
\end{align*}
\medskip

\item $\dis \sigma^\onebythree\pap\bvarap \in  \Ltwo$
\medskip\\
Proof:  Using the formula of $\bvarap$ from \eqref{form:bvarapnew} we see that 
\begin{align*}
\lpar (\Id -\Hil)\pap\bvarap  = (\Id - \Hil)\brac*[\big]{\pap\Dap\Zt} + (\Id - \Hil)\cbrac{\Ztap\pap\frac{1}{\Zap}} + (\Id - \Hil)\cbrac{\Zt\pap^2\frac{1}{\Zap}}
\end{align*}
Now we see that 
\begin{align*}
\Zt\pap^2\frac{1}{\Zap} = \frac{\Zt}{2}\brac{\Zaphalf\pap\frac{1}{\Zap}}^2 + \frac{\Zt}{\Zaphalf}\pap\brac{\Zaphalf\pap\frac{1}{\Zap}}
\end{align*}
hence 
\begin{align*}
(\Id - \Hil)\cbrac{\Zt\pap^2\frac{1}{\Zap}} = \frac{1}{2}\sqbrac{\Zt,\Hil}\brac{\Zaphalf\pap\frac{1}{\Zap}}^2 + \sqbrac{\Pa\brac{\frac{\Zt}{\Zaphalf}}}\pap\brac{\Zaphalf\pap\frac{1}{\Zap}}
\end{align*}
As $\bvarap $ is real valued, by taking real part of $ (\Id -\Hil)\pap\bvarap$ and using \propref{prop:commutator} we get
\begin{align*}
\norm*[\big][2]{\sigma^\onebythree\pap\bvarap} & \lesssim \norm*[\big][2]{\sigma^\onebythree\pap\Dap\Zt} + \cbrac{\norm[\infty]{\sigma^\onebysix\frac{\Ztap}{\Zapabs^\half}} + \norm[\infty]{\sigma^\onebysix\pap\Pa\brac{\frac{\Zt}{\Zaphalf}}} }\norm[2]{\sigma^\onebysix\Zapabs^\half\pap\frac{1}{\Zap}} \\
& \quad + \norm[2]{\Ztap}\norm[2]{\sigma^\onebysix\Zapabs^\half\pap\frac{1}{\Zap}}^2
\end{align*}
\medskip

\item $\dis \frac{\sigma^\onebysix}{\Zapabs^\half}\pap\bvarap \in \Ltwo$ 
\medskip\\
Proof: This is obtained by interpolating between $\dis \sigma^\onebythree\pap\bvarap \in \Ltwo$ and $\dis \Dapabs\bvarap \in \Ltwo$. We have
\[
\norm[2]{\frac{\sigma^\onebysix}{\Zapabs^\half}\pap\bvarap}^2 \lesssim \norm*[\big][2]{\sigma^\onebythree\pap\bvarap}\norm*[\big][2]{\Dapabs\bvarap}
\]
\medskip

\item $\dis \frac{\sigma^\half}{\Zapabs^\half}\pap\bvarap \in \Linfty$
\medskip\\
Proof: As $\bvarap$ is real valued we have from \eqref{form:IminusHbvarap} 
\begin{align*}
\bvarap = \Real\cbrac{ \sqbrac{\frac{1}{\Zap},\Hil}\Ztap + 2\Dap\Zt + \sqbrac{\Zt,\Hil}\brac*[\Big]{\pap \frac{1}{\Zap}}}
\end{align*}
Now taking the derivative $\dis \frac{\sigma^\half}{\Zapabs^\half}\pap $ and using \propref{prop:tripleidentity} we obtain
\begin{align*}
\frac{\sigma^\half}{\Zapabs^\half}\pap\bvarap = \Real &  \Biggl\{   \sqbrac{\frac{\sigma^\half}{\Zapabs^\half}\pap\frac{1}{\Zap}, \Hil }\Ztap + \sqbrac{\frac{1}{\Zap},\Hil}\pap\brac{\frac{\sigma^\half}{\Zapabs^\half}\Ztap}  \\
& \quad - \sqbrac{\frac{\sigma^\half}{\Zapabs^\half}, \frac{1}{\Zap} ; \Ztap} + \frac{2\sigma^\half}{\Zapabs^\half}\pap\Dap\Zt + \sqbrac{\frac{\sigma^\half}{\Zapabs^\half}\Ztap, \Hil}\brac{\pap\frac{1}{\Zap}} \\
& \quad + \sqbrac{\Zt, \Hil}\pap\brac{\frac{\sigma^\half}{\Zapabs^\half}\pap\frac{1}{\Zap}}  - \sqbrac{\frac{\sigma^\half}{\Zapabs^\half}, \Zt ; \pap\frac{1}{\Zap}} \Biggr\}
\end{align*}
Hence using \propref{prop:commutator} and \propref{prop:triple} we have
\begin{align*}
\norm[\infty]{\frac{\sigma^\half}{\Zapabs^\half}\pap\bvarap} & \lesssim \norm[\infty]{\sigma^\half\Zapabs^\half\pap\frac{1}{\Zapabs}}\norm[2]{\pap\frac{1}{\Zap}}\norm[2]{\Ztapbar} + \norm[2]{\frac{\sigma^\half}{\Zapabs^\half}\pap^2\frac{1}{\Zap}}\norm[2]{\Ztapbar} \\
& \quad + \norm[2]{\pap\frac{1}{\Zap}}\norm[2]{\frac{\sigma^\half}{\Zapabs^\half}\pap\Ztapbar} + \norm[\infty]{\frac{\sigma^\half}{\Zapabs^\half}\pap\Dap\Zt} 
\end{align*}
\medskip

\item $\dis \frac{\sigma^\half}{\Zapabs^\threebytwo}\pap^2\bvarap \in \Ltwo$, $\dis \frac{\sigma^\half}{\Zapabs^\half}\pap\Dapabs\bvarap \in \Ltwo$ and $\dis \frac{\sigma^\half}{\Zapabs^\half}\pap\Dap\bvarap \in \Ltwo$
\medskip\\
Proof:  We will first show that $\dis (\Id - \Hil)\cbrac*[\Bigg]{\frac{\sigma^\half}{\Zap^\threebytwo}\pap^2\bvarap} \in \Ltwo$.  Using the formula of $\bvarap$ from \eqref{form:bvarapnew} we see that 
\begin{align*}
\lpar (\Id - \Hil)\cbrac*[\Bigg]{\frac{\sigma^\half}{\Zap^\threebytwo}\pap^2\bvarap} = (\Id - \Hil) &  \Biggl\{\frac{\sigma^\half}{\Zap^\threebytwo}\pap^2\Dap\Zt + \brac*[\bigg]{\frac{\sigma^\half}{\Zap^\threebytwo}\pap\Ztap}\brac{\pap\frac{1}{\Zap}}  \\
& +  2(\Dap\Zt)\brac*[\bigg]{\frac{\sigma^\half}{\Zap^\half}\pap^2\frac{1}{\Zap}}\Biggr\} + (\Id - \Hil)\cbrac*[\bigg]{\brac{\frac{\Zt}{\Zap}} \frac{\sigma^\half}{\Zaphalf}\pap^3\frac{1}{\Zap}}
\end{align*}
Now
\begin{align*}
(\Id - \Hil)\cbrac*[\bigg]{\brac{\frac{\Zt}{\Zap}} \frac{\sigma^\half}{\Zaphalf}\pap^3\frac{1}{\Zap}} & = -\frac{1}{2}\sqbrac{\Zt,\Hil}\cbrac{\brac{\pap\frac{1}{\Zap}}\brac{\frac{\sigma^\half}{\Zaphalf}\pap^2\frac{1}{\Zap}} } \\
& \quad + \sqbrac{\Pa\brac{\frac{\Zt}{\Zap}},\Hil}\pap\brac{\frac{\sigma^\half}{\Zaphalf}\pap^2\frac{1}{\Zap}}
\end{align*}
Hence using \propref{prop:commutator} we have
\begin{align*}
&  \norm*[\Bigg][2]{(\Id - \Hil)\cbrac*[\Bigg]{\frac{\sigma^\half}{\Zap^\threebytwo}\pap^2\bvarap}} \\
& \lesssim \norm[2]{\frac{\sigma^\half}{\Zapabs^\threebytwo}\pap^2\Dap\Zt} + \norm[2]{\frac{1}{\Zapabs^2}\pap\Ztap}\norm[\infty]{\sigma^\half\Zapabs^\half\pap\frac{1}{\Zap}} \\
& \quad  + \norm[2]{\frac{\sigma^\half}{\Zapabs^\half}\pap^2\frac{1}{\Zap}}\cbrac{\norm[\infty]{\Dap\Zt} + \norm[2]{\Ztap}\norm[2]{\pap\frac{1}{\Zap}}} 
\end{align*}

Now lets come back to prove $\dis \frac{\sigma^\half}{\Zapabs^\threebytwo}\pap^2\bvarap \in \Ltwo$. We see that
\begin{align*}
\frac{\sigma^\half}{\Zapabs^\threebytwo}\pap^2\bvarap = \Real\cbrac*[\Bigg]{\frac{\sigma^\half\w^\threebytwo}{\Zap^\threebytwo}(\Id - \Hil)\pap^2\bvarap }
\end{align*}
Hence it is enough to show that $\dis \frac{\sigma^\half}{\Zap^\threebytwo}(\Id - \Hil)\pap^2\bvarap \in \Ltwo$. Now we have
\begin{align*}
\frac{\sigma^\half}{\Zap^\threebytwo}(\Id - \Hil)\pap^2\bvarap  = -\sqbrac{\frac{\sigma^\half}{\Zap^\threebytwo},\Hil}\pap^2\bvarap + (\Id - \Hil)\cbrac*[\Bigg]{\frac{\sigma^\half}{\Zap^\threebytwo}\pap^2\bvarap}
\end{align*}
From this and \propref{prop:commutator} we finally have the estimate
\begin{align*}
\norm[2]{ \frac{\sigma^\half}{\Zapabs^\threebytwo}\pap^2\bvarap} & \lesssim \norm[\infty]{\bvarap}\cbrac{\norm[2]{\frac{\sigma^\half}{\Zapabs^\half}\pap^2\frac{1}{\Zap}} + \norm[\infty]{\sigma^\half\Zapabs^\half\pap\frac{1}{\Zap}}\norm[2]{\pap\frac{1}{\Zap}} } \\
& \quad + \norm*[\Bigg][2]{(\Id - \Hil)\cbrac*[\Bigg]{\frac{\sigma^\half}{\Zap^\threebytwo}\pap^2\bvarap}}
\end{align*}
We also see that
\begin{align*}
\norm[2]{\frac{\sigma^\half}{\Zapabs^\half}\pap\Dapabs\bvarap} \lesssim \norm[\infty]{\sigma^\half\Zapabs^\half\pap\frac{1}{\Zapabs}}\norm[2]{\Dapabs\bvarap} + \norm[2]{ \frac{\sigma^\half}{\Zapabs^\threebytwo}\pap^2\bvarap}
\end{align*}
The other term $\dis \frac{\sigma^\half}{\Zapabs^\half}\pap\Dap\bvarap \in \Ltwo$ is obtained similarly.
\medskip

\item $\dis \frac{\sigma^\half}{\Zapabs^\half}\pap\Aone \in \Linfty$
\medskip\\
Proof: We know that $\Aone  = 1 - \Imag[\Zt,\Hil]\Ztapbar$ and hence using \propref{prop:tripleidentity} we have
\begin{align*}
\quad \frac{\sigma^\half}{\Zapabs^\half}\pap\Aone = -\Imag\cbrac{\sqbrac{\frac{\sigma^\half}{\Zapabs^\half}\Ztap,\Hil}\Ztapbar + \sqbrac{\Zt,\Hil}\pap\brac{\frac{\sigma^\half}{\Zapabs^\half}\Ztapbar} - \sqbrac{\frac{\sigma^\half}{\Zapabs^\half}, \Zt ; \Ztapbar} }
\end{align*}
Hence using \propref{prop:commutator} and \propref{prop:triple} we have
\begin{align*}
\norm[\infty]{\frac{\sigma^\half}{\Zapabs^\half}\pap\Aone} \lesssim \norm[\infty]{\sigma^\half\Zapabs^\half\pap\frac{1}{\Zapabs}}\norm[2]{\Ztapbar}^2 + \norm[2]{\frac{\sigma^\half}{\Zapabs^\half}\pap\Ztapbar}\norm[2]{\Ztapbar}
\end{align*}

\medskip

\item $\dis \frac{\sigma^\half}{\Zapabs^\threebytwo}\pap^2\Aone \in \Ltwo$
\medskip\\
Proof: Observe that $\dis  \frac{\sigma^\half}{\Zapabs^\threebytwo}\pap^2\Aone = \Real \cbrac*[\bigg]{\frac{\sigma^\half\w^\threebytwo}{(\Zap)^\threebytwo}(\Id - \Hil)\pap^2\Aone}$ and hence it is enough to show that  $\dis \frac{\sigma^\half}{(\Zap)^\threebytwo}(\Id - \Hil)\pap^2\Aone \in \Ltwo$. Now
\begin{align*}
\lpar \frac{\sigma^\half}{(\Zap)^\threebytwo}(\Id - \Hil)\pap^2\Aone = (\Id - \Hil)\cbrac*[\bigg]{\frac{\sigma^\half}{(\Zap)^\threebytwo}\pap^2\Aone} - \sqbrac*[\Bigg]{\frac{\sigma^\half}{(\Zap)^\threebytwo} ,\Hil}\pap^2\Aone
\end{align*}
Using the formula of $\Aone$ from \eqref{form:Aonenew} we see that  
\begin{align*}
\lpar (\Id - \Hil)\cbrac*[\bigg]{\frac{\sigma^\half}{(\Zap)^\threebytwo}\pap^2\Aone} & = i(\Id - \Hil)\cbrac*[\bigg]{\frac{\sigma^\half}{(\Zap)^\threebytwo}\cbrac*[\Big]{ (\pap\Ztap)(\Ztapbar) + 2(\Ztap)(\pap\Ztapbar)}}  \\
& + i\sqbrac{\Zt,\Hil}\cbrac*[\bigg]{\frac{\sigma^\half}{(\Zap)^\threebytwo}\pap^2\Ztapbar}
\end{align*}
With 
\begin{flalign*}
&& \sqbrac{\Zt,\Hil}\cbrac*[\bigg]{\frac{\sigma^\half}{(\Zap)^\threebytwo}\pap^2\Ztapbar} = \sqbrac{\Zt,\Hil}\pap\cbrac*[\bigg]{\frac{\sigma^\half}{(\Zap)^\threebytwo}\pap\Ztapbar} - \threebytwo\sqbrac{\Zt,\Hil}\cbrac*[\bigg]{\brac{\pap\frac{1}{\Zap}}\frac{\sigma^\half}{\Zaphalf}\pap\Ztapbar} 
\end{flalign*}
 Hence using \propref{prop:commutator} we have
\begin{flalign*}
\lpar \norm*[\bigg][2]{\frac{\sigma^\half}{\Zapabs^\threebytwo}\pap^2\Aone} & \lesssim \cbrac{\norm*[\bigg][2]{\frac{\sigma^\half}{\Zapabs^\half}\pap^2\frac{1}{\Zap}} + \norm[\infty]{\sigma^\half\Zapabs^\half\pap\frac{1}{\Zap}}\norm[2]{\pap\frac{1}{\Zap}}} \norm[\infty]{\Aone}  \\
& \quad  + \norm[2]{\Ztap}\cbrac{\norm*[\bigg][\infty]{\frac{\sigma^\half}{\Zapabs^\threebytwo}\pap\Ztapbar} + \norm[2]{\pap\frac{1}{\Zap}}\norm*[\bigg][2]{\frac{\sigma^\half}{\Zapabs^\half}\pap\Ztapbar } }
\end{flalign*}
\medskip

\item $\dis (\Id - \Hil)\Dt^2\Th \in \Ltwo$, $\dis (\Id - \Hil)\Dt^2\Ztapbar \in \Ltwo$, $\dis (\Id - \Hil)\Dt^2 \Dap\Ztbar \in \Hhalf$
\medskip\\
Proof: For a function $f$ satisfying $\Pa f = 0$ we use \propref{prop:tripleidentity} to get
\begin{align*}
(\Id - \Hil)\Dt^2f & = \sqbrac{\Dt,\Hil}\Dt f + \Dt\sqbrac{\Dt,\Hil}f \\
& = \sqbrac{\bvar,\Hil}\pap\Dt f + \Dt\sqbrac{\bvar,\Hil}\pap f \\
& = 2\sqbrac{\bvar,\Hil}\pap\Dt f + \sqbrac{\Dt\bvar,\Hil}\pap f - \sqbrac{\bvar, \bvar ; \pap f}
\end{align*}
Hence using \propref{prop:commutator} and \propref{prop:triple} we have
\begingroup
\allowdisplaybreaks
\begin{align*}
 \norm[2]{(\Id - \Hil)\Dt^2\Th}  & \lesssim \norm[\Hhalf]{\bvarap}\norm[2]{\Dt\Th} + \norm[\Hhalf]{\pap\Dt\bvar}\norm[2]{\Th} + \norm[\infty]{\bvarap}^2\norm[2]{\Th} \\
 \norm[2]{(\Id - \Hil)\Dt^2\Ztapbar} & \lesssim \norm[\Hhalf]{\bvarap}\norm[2]{\Dt\Ztapbar} + \norm[\Hhalf]{\pap\Dt\bvar}\norm[2]{\Ztapbar} + \norm[\infty]{\bvarap}^2\norm[2]{\Ztapbar}  \\
  \norm[\Hhalf]{(\Id - \Hil)\Dt^2\Dap\Ztbar} & \lesssim 
 \begin{aligned}[t]
 & \norm[\Hhalf]{\bvarap}\norm[\Hhalf]{\Dt\Dap\Ztbar} + \norm[\Hhalf]{\pap\Dt\bvar}\norm[\Hhalf]{\Dap\Ztbar}  \\
 & + \norm[\infty]{\bvarap}^2\norm[\Hhalf]{\Dap\Ztbar} 
 \end{aligned}
\end{align*}
\endgroup
\medskip

\item $\dis \sigma(\Id - \Hil)\Dapabs^3\Th \in \Ltwo $, $\dis \sigma(\Id - \Hil)\Dapabs^3\Ztapbar \in \Ltwo $, $\dis \sigma(\Id - \Hil)\Dapabs^3\Dap\Ztbar \in \Hhalf $
\medskip\\
Proof:  We use  \eqref{form:Dapabs^3f} for a function $f$ satisfying $\Pa f = 0$ to get 
\begin{align*}
\sigma(\Id - \Hil)\Dapabs^3 f & = \sigma(\Id - \Hil)\cbrac*[\bigg]{ \brac{\frac{1}{\Zapabs}\pap^2\frac{1}{\Zapabs}}\Dapabs f +  \brac{\pap\frac{1}{\Zapabs}}^2\Dapabs f } \\
& \quad  + \sigma\sqbrac{\pap\frac{1}{\Zapabs^3},\Hil}\pap^2 f + \sigma\sqbrac{\frac{1}{\Zapabs^3},\Hil}\pap^3 f
\end{align*}
Hence using \propref{prop:commutator} we have
\begingroup
\allowdisplaybreaks
\begin{flalign*}
 \quad \  \norm*[\big][2]{\sigma(\Id - \Hil)\Dapabs^3\Ztapbar}  \lesssim \enspace & \norm*[\bigg][\Hhalf]{\sigma\pap^3\frac{1}{\Zapabs^3}}\norm[2]{\Ztapbar} + \norm*[\bigg][\Ccal]{\frac{\sigma^\half}{\Zapabs^\threebytwo}\pap^2\frac{1}{\Zapabs}}\norm*[\bigg][\Ccal]{\frac{\sigma^\half}{\Zapabs^\threebytwo}\pap\Ztapbar} \\*
& + \norm*[\bigg][\infty]{\sigma^\half\Zapabs^\half\pap\frac{1}{\Zapabs}}^2\norm*[\bigg][2]{\frac{1}{\Zapabs^2}\pap\Ztapbar} \\
\norm*[\big][2]{\sigma(\Id - \Hil)\Dapabs^3\Th}  \lesssim \enspace & \norm*[\bigg][\Hhalf]{\sigma\pap^3\frac{1}{\Zapabs^3}}\norm[2]{\Th} + \norm*[\bigg][\Ccal]{\frac{\sigma^\half}{\Zapabs^\threebytwo}\pap^2\frac{1}{\Zapabs}}\norm*[\bigg][\Ccal]{\frac{\sigma^\half}{\Zapabs^\threebytwo}\pap\Th} \\*
&  + \norm[\Wcal]{\sigma^\half\Zapabs^\half\pap\frac{1}{\Zapabs}}\norm[\Ccal]{\frac{1}{\Zapabs}\pap\frac{1}{\Zapabs}}\norm*[\bigg][\Ccal]{\frac{\sigma^\half}{\Zapabs^\threebytwo}\pap\Th} 
\end{flalign*}
\endgroup
and also 
\begin{align*}
& \norm*[\big][\Hhalf]{\sigma(\Id - \Hil)\Dapabs^3\Dap\Ztbar} \\
& \lesssim \enspace  \norm*[\bigg][\Hhalf]{\sigma\pap^3\frac{1}{\Zapabs^3}}\norm[\Hhalf]{\Dap\Ztbar} + \norm*[\bigg][\Ccal]{\frac{\sigma^\half}{\Zapabs^\threebytwo}\pap^2\frac{1}{\Zapabs}}\norm*[\bigg][\Wcal]{\frac{\sigma^\half}{\Zapabs^\half}\pap\Dap\Ztbar}  \\*
& \quad + \norm[\Wcal]{\sigma^\half\Zapabs^\half\pap\frac{1}{\Zapabs}}\norm[\Ccal]{\frac{1}{\Zapabs}\pap\frac{1}{\Zapabs}}\norm*[\bigg][\Wcal]{\frac{\sigma^\half}{\Zapabs^\half}\pap\Dap\Ztbar}
\end{align*}
\medskip

\item $\dis \sqbrac{\Dt^2, \frac{1}{\Zap}}\Ztapbar \in \Ccal$, $\dis \sqbrac{\Dt^2, \frac{1}{\Zapbar}}\Ztapbar \in \Ccal$
\medskip\\
Proof: We will only show $\dis \sqbrac{\Dt^2, \frac{1}{\Zap}}\Ztapbar \in \Ccal$ and $\dis \sqbrac{\Dt^2, \frac{1}{\Zapbar}}\Ztapbar \in \Ccal$ is proved similarly. We recall from \eqref{form:DtoneoverZap} that $\dis  \Dt \frac{1}{\Zap} =  \frac{1}{\Zap}\cbrac{(\bvarap - \Dap\Zt - \Dapbar \Ztbar) + \Dapbar \Ztbar}$ and hence
\begin{align*}
\lpar \sqbrac{\Dt^2, \frac{1}{\Zap}}\Ztapbar &= \Ztapbar\Dt^2\frac{1}{\Zap} + 2(\Dt\Ztapbar)\Dt\frac{1}{\Zap} \\
& = \Ztapbar\Dt\cbrac{ \frac{1}{\Zap}\cbrac{(\bvarap - \Dap\Zt - \Dapbar \Ztbar) + \Dapbar \Ztbar}} \\
& \quad  + 2(\Dap\Zttbar - \bvarap\Dap\Ztbar)(\bvarap - \Dap\Zt) \\
& = (\Dap\Ztbar)(\bvarap - \Dap\Zt)^2 + (\Dap\Ztbar)\Dt(\bvarap - \Dap\Zt - \Dapbar\Ztbar) \\
& \quad + (\Dap\Ztbar)\Dt\Dapbar\Ztbar + 2(\Dap\Zttbar - \bvarap\Dap\Ztbar)(\bvarap - \Dap\Zt)
\end{align*}
Now using \propref{prop:Leibniz} we have the estimates
\begin{align*}
\lpar \norm[2]{\Zapabs(\Dap\Ztbar)\Dt(\bvarap - \Dap\Zt - \Dapbar\Ztbar)} & \lesssim \norm[2]{\Ztbarap}\norm[\infty]{\Dt(\bvarap - \Dap\Zt - \Dapbar\Ztbar)} \\
\norm[\Hhalf]{(\Dap\Ztbar)\Dt(\bvarap - \Dap\Zt - \Dapbar\Ztbar)} & \lesssim \norm[\Hhalf]{\Dap\Ztbar}\norm[\infty]{\Dt(\bvarap - \Dap\Zt - \Dapbar\Ztbar)} \\
& \quad +  \norm[\infty]{\Dap\Ztbar}\norm[\Hhalf]{\Dt(\bvarap - \Dap\Zt - \Dapbar\Ztbar)}
\end{align*}
This implies that $\dis (\Dap\Ztbar)\Dt(\bvarap - \Dap\Zt - \Dapbar\Ztbar) \in \Ccal$. Hence using \lemref{lem:CW} we have
\begin{align*}
\lpar \norm[\Ccal]{\sqbrac{\Dt^2, \frac{1}{\Zap}}\Ztapbar } & \lesssim  \norm[\Ccal]{\Dap\Ztbar}\norm[\Wcal]{\bvarap -\Dap\Zt}^2 + \norm[\Ccal]{(\Dap\Ztbar)\Dt(\bvarap - \Dap\Zt - \Dapbar\Ztbar)} \\
&  \quad + \brac{\norm[\Ccal]{\Dap\Zttbar} + \norm[\Wcal]{\bvarap}\norm[\Ccal]{\Dap\Ztbar}}\brac{\norm[\Wcal]{\bvarap} + \norm[\Wcal]{\Dap\Zt}} \\
& \quad + \norm[\Wcal]{\Dap\Ztbar}\norm[\Ccal]{\Dt\Dapbar\Ztbar}
\end{align*}
\medskip

\item $\dis \sqbrac{i\frac{\Aone}{\Zapabs^2}\pap, \frac{1}{\Zap}}\Ztapbar \in \Ccal$, $\dis \sqbrac{i\frac{\Aone}{\Zapabs^2}\pap, \frac{1}{\Zapbar}}\Ztapbar \in \Ccal$
\medskip\\
Proof: Observe that $\dis \sqbrac{i\frac{\Aone}{\Zapabs^2}\pap, \frac{1}{\Zap}}\Ztapbar = i\Aone(\Dapabs\Ztbar)\Dapabs\frac{1}{\Zap}$ and so using \lemref{lem:CW} we have
\[
\norm[\Ccal]{\sqbrac{i\frac{\Aone}{\Zapabs^2}\pap, \frac{1}{\Zap}}\Ztapbar} \lesssim \norm[\Wcal]{\Aone}\norm[\Wcal]{\Dapabs\Ztbar}\norm[\Ccal]{\Dapabs\frac{1}{\Zap}}
\]
The other term is proved similarly.
\medskip

\item  $\dis (\Id - \Hil)\sqbrac{i\sigma\Dapabs^3, \frac{1}{\Zap}}\Ztapbar \in \Hhalf$, $\dis (\Id - \Hil)\sqbrac{i\sigma\Dapabs^3, \frac{1}{\Zapbar}}\Ztapbar \in \Hhalf$ and we also have $\dis \Zapabs\sqbrac{i\sigma\Dapabs^3, \frac{1}{\Zap}}\Ztapbar \in \Ltwo$, $\dis \Zapabs\sqbrac{i\sigma\Dapabs^3, \frac{1}{\Zapbar}}\Ztapbar \in \Ltwo$
\medskip\\
Proof: We will only show $\dis (\Id - \Hil)\sqbrac{i\sigma\Dapabs^3, \frac{1}{\Zap}}\Ztapbar \in \Ccal$ and $\dis \Zapabs\sqbrac{i\sigma\Dapabs^3, \frac{1}{\Zap}}\Ztapbar \in \Ltwo$  and the other terms are proved similarly. Note that we are not making the stronger claim that $\dis \sqbrac{i\sigma\Dapabs^3, \frac{1}{\Zap}}\Ztapbar \in \Ccal$. This is not true and the use of $(\Id - \Hil)$ in the $\Hhalf$ estimate is essential. We have
\begin{align*}
\lpar \sqbrac{i\sigma\Dapabs^3, \frac{1}{\Zap}}\Ztapbar &= i\sigma\brac{\Dapabs^3\frac{1}{\Zap}}\Ztapbar + 3i\sigma\brac{\Dapabs^2\frac{1}{\Zap}}\Dapabs\Ztapbar \\
& \quad + 3i\sigma\brac{\Dapabs\frac{1}{\Zap}}\Dapabs^2\Ztapbar
\end{align*}
We control each term seperately:
\begin{enumerate}
\item We use the expansion in  \eqref{form:Dapabs^3f} to get
\begin{flalign*}
\lpar & \sigma\brac{\Dapabs^3\frac{1}{\Zap}}\Ztapbar  \\
& =  \sigma\brac{\frac{1}{\Zapabs}\pap^2\frac{1}{\Zapabs}}\brac{\Dapabs\frac{1}{\Zap}}\Ztapbar + \sigma\brac{\pap\frac{1}{\Zapabs}}^2\brac{\Dapabs\frac{1}{\Zap}}\Ztapbar \\
& \quad + 3\sigma\brac{\pap\frac{1}{\Zapabs}}\brac*[\bigg]{\frac{1}{\Zapabs^2}\pap^2\frac{1}{\Zap}}\Ztapbar + \sigma\brac*[\bigg]{\frac{1}{\Zapabs^3}\pap^3\frac{1}{\Zap}}\Ztapbar
\end{flalign*}
Hence using \lemref{lem:CW} we have the estimate
\begin{flalign*}
&& & \norm[\Ccal]{\sigma\brac{\Dapabs^3\frac{1}{\Zap}}\Ztapbar}  \\
&& & \lesssim \norm[\Wcal]{\Dapabs\Ztbar} 
\begin{aligned}[t]
& \Bigg\{\norm*[\bigg][\Ccal]{\frac{\sigma^\half}{\Zapabs^\threebytwo}\pap^2\frac{1}{\Zapabs}}\norm[\Wcal]{\sigma^\half\Zapabs^\half\pap\frac{1}{\Zap}} + \norm[\Ccal]{\Dapabs\frac{1}{\Zap}}\norm[\Wcal]{\sigma^\half\Zapabs^\half\frac{1}{\Zapabs}}^2 \\
& \quad + \norm*[\bigg][\Ccal]{\frac{\sigma^\half}{\Zapabs^\threebytwo}\pap^2\frac{1}{\Zap}}\norm[\Wcal]{\sigma^\half\Zapabs^\half\pap\frac{1}{\Zapabs}} + \norm*[\bigg][\Ccal]{\frac{\sigma}{\Zapabs^2}\pap^3\frac{1}{\Zap}} \Bigg\}
\end{aligned}
\end{flalign*}

\item We observe that 
\begin{flalign*}
\lpar \qq& \sigma\brac{\Dapabs^2\frac{1}{\Zap}}\Dapabs\Ztapbar = \sigma\cbrac*[\bigg]{\brac{\pap\frac{1}{\Zapabs}}\Dapabs\frac{1}{\Zap} + \frac{1}{\Zapabs^2}\pap^2\frac{1}{\Zap}  }\Dapabs\Ztapbar 
\end{flalign*}
Hence using \lemref{lem:CW}  we have the estimate
\begin{flalign*}
&& & \norm[\Ccal]{\sigma\brac{\Dapabs^2\frac{1}{\Zap}}\Dapabs\Ztapbar} \lesssim 
\begin{aligned}[t]
& \norm[\Wcal]{\sigma^\half\Zapabs^\half\frac{1}{\Zapabs}}\norm[\Ccal]{\Dapabs\frac{1}{\Zap}}\norm*[\bigg][\Wcal]{\frac{\sigma^\half}{\Zapabs^\threebytwo}\pap\Ztapbar} \\
& + \norm*[\bigg][\Ccal]{\frac{\sigma^\half}{\Zapabs^\threebytwo}\pap^2\frac{1}{\Zap}}\norm*[\bigg][\Wcal]{\frac{\sigma^\half}{\Zapabs^\threebytwo}\pap\Ztapbar}
\end{aligned} 
\end{flalign*}

\item We observe that
\begin{flalign*}
\lpar\qq \sigma\brac{\Dapabs\frac{1}{\Zap}}\Dapabs^2\Ztapbar  = \sigma\brac{\Dapabs\frac{1}{\Zap}} \cbrac*[\bigg]{\brac{\pap\frac{1}{\Zapabs}}\Dapabs\Ztapbar +\frac{1}{\Zapabs^2}\pap^2\Ztapbar }
\end{flalign*}
The first term is easily controlled using \lemref{lem:CW} 
\begin{align*}
& \norm[\Ccal]{ \sigma\brac{\Dapabs\frac{1}{\Zap}} \brac{\pap\frac{1}{\Zapabs}}\Dapabs\Ztapbar } \\
& \lesssim \norm[\Ccal]{\Dapabs\frac{1}{\Zap}}\norm[\Wcal]{\sigma^\half\Zapabs^\half\pap\frac{1}{\Zapabs}}\norm*[\bigg][\Wcal]{\frac{\sigma^\half}{\Zapabs^\threebytwo}\pap\Ztapbar}
\end{align*}
Hence we are only left with $\dis \sigma \brac{\Dapabs\frac{1}{\Zap}}\frac{1}{\Zapabs^2}\pap^2\Ztapbar$. We see that 
\begin{flalign*}
\lpar \qq\norm*[\bigg][2]{\sigma \Zapabs \brac{\Dapabs\frac{1}{\Zap}}\frac{1}{\Zapabs^2}\pap^2\Ztapbar} \lesssim \norm[\infty]{\sigma^\half\Zapabs^\half\pap\frac{1}{\Zap}}\norm*[\bigg][2]{\frac{\sigma^\half}{\Zapabs^\fivebytwo}\pap^2\Ztapbar}
\end{flalign*}
This conclude the proof of $\dis \Zapabs\sqbrac{i\sigma\Dapabs^3, \frac{1}{\Zap}}\Ztapbar \in \Ltwo$. To finish the $\Hhalf$ estimate we rewrite the term $\dis \frac{1}{\Zapabs^2}\pap^2\Ztapbar$ as $\dis \frac{\w^2}{\Zap^2}\pap^2\Ztapbar$ and commute one derivative outside to obtain
\begin{align*}
\lpar \qq (\Id - \Hil)\cbrac*[\bigg]{ \sigma \brac{\Dapabs\frac{1}{\Zap}}\frac{1}{\Zapabs^2}\pap^2\Ztapbar} &= -2(\Id - \Hil)\cbrac*[\bigg]{\sigma \brac{\pap\frac{1}{\Zap}}^2\frac{\w}{\Zapabs^2}\pap\Ztapbar } \\
& \quad + \sigma\sqbrac{\frac{\w}{\Zapbar}\pap\frac{1}{\Zap},\Hil}\pap\brac*[\bigg]{\frac{1}{\Zap^2}\pap\Ztapbar}
\end{align*}
We can bound each of the terms using \lemref{lem:CW} and \propref{prop:commutator}
\begin{flalign*}
 \lpar \qq \quad & \norm*[\bigg][\Ccal]{\sigma \brac{\pap\frac{1}{\Zap}}^2\frac{\w}{\Zapabs^2}\pap\Ztapbar} & \\
&  \lesssim \norm[\Ccal]{\Dapabs\frac{1}{\Zap}}\norm[\Wcal]{\sigma^\half\Zapabs^\half\pap\frac{1}{\Zap}}\norm*[\bigg][\Wcal]{\frac{\sigma^\half}{\Zapabs^\threebytwo}\pap\Ztapbar}\norm[\Wcal]{\w} 
\end{flalign*}
and also
\begin{flalign*}
\lpar \qq   & \norm*[\bigg][\Hhalf]{\sigma\sqbrac{\frac{\w}{\Zapbar}\pap\frac{1}{\Zap},\Hil}\pap\brac*[\bigg]{\frac{1}{\Zap^2}\pap\Ztapbar}}  \\
& \lesssim \norm*[\bigg][2]{\frac{1}{\Zap^2}\pap\Ztapbar}\Biggl\{
\begin{aligned}[t]
& \norm*[\bigg][2]{\frac{\sigma^\twobythree}{\Zapbar}\pap^2\w}\norm[\infty]{\sigma^\onebythree\pap\frac{1}{\Zap}} + \norm[2]{\Dapabs\w}\norm*[\bigg][\infty]{\sigma^\half\Zapabs^\half\pap\frac{1}{\Zap}}^2 \\
& + \norm*[\big][\infty]{\sigma^\onebythree\Dapbar\w}\norm[2]{\sigma^\twobythree\pap^2\frac{1}{\Zap}} + \norm[\infty]{\sigma^\onebythree\pap\frac{1}{\Zap}}\norm[2]{\sigma^\twobythree \pap^2\frac{1}{\Zap}}  \\
& + \norm[2]{\frac{\sigma}{\Zapbar}\pap^3\frac{1}{\Zap}} \Bigg\}
\end{aligned} 
\end{flalign*}
\end{enumerate}
\medskip

\item $\dis \Rone \in \Ccal$
\medskip\\
Proof: We recall from \eqref{form:Rone} the formula of $\Rone$
\begin{align*} 
\begin{split}  
\Rone & =  -2(\Dapbar\Ztbar)(\Dt\Dapbar\Ztbar) -2\sigma\Real(\Dap\Zt)\Dapbar\Dap\Th -\sigma(\Dapbar\Dap\Zt)\Dap\Th \\ 
& \quad  +i\sigma\brac{2i\Real \brac{\Dapabs \Th} + \brac{\Real \Th}^2} \Dapabs \Dapbar\Ztbar -\sigma \Real \brac{\Dapabs^2 \Th}\Dapbar\Ztbar  \\
& \quad   + i\sigma \brac{\Real \Th} \brac{ \Real \brac{\Dapabs \Th}}\Dapbar\Ztbar 
\end{split}
\end{align*}
All the terms are easily controlled using \lemref{lem:CW}
\begin{align*}
&  \norm[\Ccal]{\Rone}\\
& \lesssim  \norm[\Wcal]{\Dapbar\Ztbar}\norm[\Ccal]{\Dt\Dapbar\Ztbar} + \norm[\Wcal]{\Dap\Zt}\norm[\Ccal]{\sigma\Dapbar\Dap\Th} + \norm*[\bigg][\Wcal]{\frac{\sigma^\half}{\Zapabs^\half}\pap\Dap\Zt}\norm*[\bigg][\Ccal]{\frac{\sigma^\half}{\Zapabs^\threebytwo}\pap\Th}  \\
&  \quad + \cbrac*[\Bigg]{\norm*[\bigg][\Ccal]{\frac{\sigma^\half}{\Zapabs^\threebytwo}\pap\Th} + \norm[\Ccal]{\frac{\Th}{\Zapabs}}\norm[\Wcal]{\sigma^\half\Zapabs^\half\Real\Th}}\norm*[\bigg][\Wcal]{\frac{\sigma^\half}{\Zapabs^\half}\pap\Dapbar\Ztbar} \\
 & \quad  + \norm[\Ccal]{\sigma\Dapabs^2\Th}\norm[\Wcal]{\Dapbar\Ztbar} + \norm[\Wcal]{\sigma^\half\Zapabs^\half\Real \Th}\norm*[\bigg][\Ccal]{\frac{\sigma^\half}{\Zapabs^\threebytwo}\pap\Th}\norm[\Wcal]{\Dapbar\Ztbar}
\end{align*}
\medskip

\item $\dis \Jone \in \Linfty\cap\Hhalf$
\medskip\\
Proof: We recall from \eqref{form:Jone} the formula for $\Jone$
\begin{align*} 
\Jone &=  \Dt\Aone +  \Aone\brac{\bvarap - \Dap\Zt - \Dapbar\Ztbar} + \sigma \pap \Real \brac{\Id - \Hil} \cbrac{(\Dapabs + i\Real \Th) \Dapbar \Ztbar} \\
& \quad - \sigma\pap \Imag\brac{\Id - \Hil} \Dt\Th 
\end{align*}
We have already shown that $\dis \Dt\Aone \in \Linfty\cap \Hhalf$ and we have using \propref{prop:Leibniz}
\begin{align*}
\norm[\Linfty\cap\Hhalf]{\Aone\brac{\bvarap - \Dap\Zt - \Dapbar\Ztbar}} \lesssim & \norm[\Linfty\cap\Hhalf]{\Aone}\norm[\infty]{\bvarap - \Dap\Zt - \Dapbar\Ztbar} \\
& + \norm[\infty]{\Aone}\norm[\Linfty\cap\Hhalf]{\bvarap - \Dap\Zt - \Dapbar\Ztbar}
\end{align*}
Let us now control the other terms.
\begin{enumerate}
\item Observe that 
\begin{align*}
\lpar \quad \sigma\pap(\Id - \Hil)\cbrac{\Dapabs\Dapbar\Ztbar} & =  \sigma\pap(\Id - \Hil)\cbrac{\brac{\Dapabs\frac{1}{\Zapbar}}\Ztapbar + \frac{1}{\Zapabs\Zapbar}\pap\Ztapbar} \\
& = \sigma\pap\sqbrac{\Dapabs\frac{1}{\Zapbar} ,\Hil}\Ztapbar + \sigma\pap\sqbrac{\frac{1}{\Zapabs\Zapbar},\Hil}\pap\Ztapbar
\end{align*}
Hence using \propref{prop:commutator} we have
\begin{align*}
 & \norm[\Linfty\cap\Hhalf]{\sigma\pap(\Id - \Hil)\cbrac{\Dapabs\Dapbar\Ztbar}} && \\
 & \lesssim \norm[2]{\Ztapbar}
\begin{aligned}[t]
& \Bigg\{\norm[2]{\frac{\sigma}{\Zapabs}\pap^3\frac{1}{\Zapbar}} + \norm[2]{\sigma^\twobythree\pap^2\frac{1}{\Zapabs}}\norm[\infty]{\sigma^\onebythree\pap\frac{1}{\Zapbar}} \\
& \quad +  \norm[2]{\sigma^\twobythree\pap^2\frac{1}{\Zapbar}}\norm[\infty]{\sigma^\onebythree\pap\frac{1}{\Zapabs}}\Bigg\}
\end{aligned}
\end{align*}
\medskip

\item We note that $i\Real\Th = \Dap\w$ and hence we have
\begin{align*}
\sigma\pap(\Id - \Hil)\cbrac{(i\Real\Th)\Dapbar\Ztbar} & = \sigma\pap(\Id - \Hil)\cbrac{(\Dap\w)\Dapbar\Ztbar} \\
& = \sigma\pap\sqbrac{\frac{1}{\Zapabs^2}\pap\w,\Hil}\Ztapbar
\end{align*}
From this and \propref{prop:commutator} we obtain
\begin{flalign*}
&&& \norm[\Linfty\cap\Hhalf]{\sigma\pap(\Id - \Hil)\cbrac{(i\Real\Th)\Dapbar\Ztbar}} \\
&&& \lesssim \norm[2]{\Ztapbar}
\begin{aligned}[t]
& \Biggl\{\norm*[\bigg][2]{\frac{\sigma}{\Zapabs^2}\pap^3\w} + \norm*[\bigg][2]{\frac{\sigma^\twobythree}{\Zapabs}\pap^2\w}\norm[\infty]{\sigma^\onebythree\pap\frac{1}{\Zapabs}} +  \norm[2]{\sigma^\twobythree\pap^2\frac{1}{\Zapabs}}\norm[\infty]{\sigma^\onebythree\Dapabs\w} \\
& \quad  + \norm[\infty]{\sigma^\half\Zapabs^\half\pap\frac{1}{\Zapabs}}^2\norm[2]{\Dapabs\w} \Biggl\}
\end{aligned}
\end{flalign*}
\medskip

\item We see that as $\Pa \Th = 0 $ we have
\begin{align*}
\sigma\pap(\Id - \Hil) \Dt\Th = \sigma\pap\sqbrac{\Dt,\Hil}\Th = \sigma\pap\sqbrac{\bvar,\Hil}\pap\Th
\end{align*}
Hence using \propref{prop:commutator} we have
\[
\norm[\Linfty\cap\Hhalf]{\sigma\pap(\Id - \Hil) \Dt\Th} \lesssim \norm*[\big][2]{\sigma^\onebythree\pap\bvarap}\norm*[\big][2]{\sigma^\twobythree\pap\Th}
\]

\end{enumerate}
\medskip

\item $\dis \Dapabs\Jone \in \Ltwo$ and hence $\dis \Jone \in \Wcal$
\medskip\\
Proof: As $\Jone$ is real valued we have
\[
\Dapabs\Jone = \Real\cbrac{\frac{\w}{\Zap}(\Id - \Hil)\pap\Jone} = \Real\cbrac{\w(\Id - \Hil)\Dap\Jone - \w\sqbrac{\frac{1}{\Zap},\Hil}\pap\Jone}
\]
We recall the equation of $\Ztapbar$ from \eqref{eq:Ztbarap}
\begin{align*}
\begin{split}
& \brac{\Dt^2 +i\frac{\Aone}{\Zapabs^2}\pap  -i\sigma\Dapabs^3} \Ztbarap \\
 & =  \Rone\Zapbar -i\brac{\pap\frac{1}{\Zap}} \Jone -i\Dap \Jone - \Zapbar\sqbrac{\Dt^2 +i\frac{\Aone}{\Zapabs^2}\pap -i\sigma\Dapabs^3, \frac{1}{\Zapbar}} \Ztapbar
\end{split}
\end{align*}
By applying $(\Id - \Hil)$ to the above equation we get
\begin{align*}
\qq & \norm[2]{(\Id - \Hil)\Dap\Jone} \\
& \lesssim \norm[2]{(\Id - \Hil)\Dt^2\Ztapbar} + \norm[\infty]{\Aone}\norm*[\bigg][2]{\frac{1}{\Zapabs^2}\pap\Ztapbar} + \norm[2]{\sigma(\Id - \Hil)\Dapabs^3\Ztapbar} \\
& \quad + \norm[2]{\Rone\Zapabs} + \norm[2]{\pap\frac{1}{\Zap}}\norm[\infty]{\Jone} + \norm*[\bigg][2]{\Zapabs\sqbrac{\Dt^2 +i\frac{\Aone}{\Zapabs^2}\pap -i\sigma\Dapabs^3, \frac{1}{\Zapbar}} \Ztapbar}
\end{align*}
Hence using \propref{prop:commutator} we easily have
\[
\norm[2]{\Dapabs\Jone} \lesssim \norm[2]{(\Id - \Hil)\Dap\Jone} + \norm[2]{\pap\frac{1}{\Zap}}\norm[\infty]{\Jone}
\]
\medskip

\item $\dis \Rtwo \in \Ltwo$
\medskip\\
Proof: We recall from \eqref{form:Rtwo} the formula for $\Rtwo$
\begin{align*}
\begin{split}
\Rtwo &=  -2i\brac{\Dapbar \Ztbar}\brac{\Dapabs \Dapbar \Ztbar} + \brac{\Real \Th }\cbrac{ \brac{\Dapbar \Ztbar}^2 +i\Dapbar\brac{\frac{\Aone}{\Zap}} +i\sigma\brac{\Real\Th} \Dapabs\Th} \\
& \quad +\sigma \Real\brac{\Dapabs\Th} \Dapabs\Th + \brac{\Dapabs\frac{\Aone}{\Zapabs}}\brac{\frac{\Zap}{\Zapabs}\pap \frac{1}{\Zap}} + \Dapabs\brac{\frac{1}{\Zapabs^2}\pap \Aone}  \\
& \quad +  (\Id + \Hil) \Imag\cbrac{\Real(\Dapbar\Ztbar)\Dapabs\Dapbar\Ztbar  -i\Real (\Dt \Th)\Dapbar\Ztbar}\\
& \quad + \frac{\Aone}{\Zapabs^2}\pap\Real (\Id - \Hil) \brac{\frac{\Zap}{\Zapabs} \pap \frac{1}{\Zap}} 
\end{split} 
\end{align*}
Most of the terms are easily controlled and using \lemref{lem:CW} we have
\begin{flalign*}
&&& \norm*[\bigg][2]{\Rtwo - \frac{\Aone}{\Zapabs^2}\pap\Real (\Id - \Hil) \brac{\frac{\Zap}{\Zapabs} \pap \frac{1}{\Zap}}} \\
&&& \lesssim \norm[\infty]{\Dapbar\Ztbar}\brac*[\Big]{\norm[2]{\Dapabs\Dapbar\Ztbar} + \norm[2]{\Dt\Th}} + \norm[2]{\Th}\norm[\infty]{\Dapbar\Ztbar}^2  + \norm[\Ccal]{\frac{\Th}{\Zapabs}}\norm[\Ccal]{\Dapabs\frac{1}{\Zap}}\norm[\infty]{\Aone} \\
&&& \quad + \norm[2]{\Th}\norm*[\bigg][\infty]{\frac{1}{\Zapabs^2}\pap\Aone} + \norm[\Ccal]{\frac{\Th}{\Zapabs}}\norm*[\bigg][\Ccal]{\frac{\sigma^\half}{\Zapabs^\threebytwo}\pap\Th}\norm[\infty]{\sigma^\half\Zapabs^\half\Real\Th} + \norm*[\bigg][\Ccal]{\frac{\sigma^\half}{\Zapabs^\threebytwo}\pap\Th}^2 \\
&&& \quad + \norm*[\bigg][\infty]{\frac{1}{\Zapabs^2}\pap\Aone}\norm[2]{\pap\frac{1}{\Zap}} + \norm[\Ccal]{\Dapabs\frac{1}{\Zapabs}}\norm[\Ccal]{\Dapabs\frac{1}{\Zap}}\norm[\infty]{\Aone} + \norm*[\bigg][\Wcal]{\frac{1}{\Zapabs^2}\pap\Aone} 
\end{flalign*}
We now control the last term. We have
\begin{align*}
\frac{\Aone}{\Zapabs^2}\pap\Real (\Id - \Hil) \brac{\frac{\Zap}{\Zapabs} \pap \frac{1}{\Zap}} & = \Real(\Id - \Hil)\cbrac*[\bigg]{\frac{\Aone}{\Zapabs^2}\pap \brac{\frac{\Zap}{\Zapabs}\pap\frac{1}{\Zap}}} \\
& \quad - \Real\cbrac{\sqbrac{\frac{\Aone}{\Zapabs^2},\Hil}\pap \brac{\frac{\Zap}{\Zapabs}\pap\frac{1}{\Zap}}}
\end{align*}
The first term can be written as
\begin{align*}
(\Id - \Hil)\cbrac*[\bigg]{\frac{\Aone}{\Zapabs^2}\pap \brac{\frac{\Zap}{\Zapabs}\pap\frac{1}{\Zap}}} &= \sqbrac{\frac{\Aone}{\Zapabs^2}\pap\w, \Hil}\pap\frac{1}{\Zap} + \sqbrac{\frac{\Aone\w}{\Zapabs^2},\Hil}\pap^2\frac{1}{\Zap} 
\end{align*}
Hence using \propref{prop:commutator} and \lemref{lem:CW} we have
\begin{flalign*}
&&& \norm*[\bigg][2]{\frac{\Aone}{\Zapabs^2}\pap\Real (\Id - \Hil) \brac{\frac{\Zap}{\Zapabs} \pap \frac{1}{\Zap}}} \\
&&& \lesssim \norm*[\bigg][2]{\pap\frac{1}{\Zap}}\cbrac*[\bigg]{\norm[\Wcal]{\Aone}\norm*[\bigg][\Ccal]{\frac{1}{\Zapabs^2}\pap\w}  + \norm*[\bigg][\Ccal]{\frac{1}{\Zapabs^2}\pap\Aone}\norm[\Wcal]{\w} + \norm[\Ccal]{\Dapabs\frac{1}{\Zapabs}}\norm[\Wcal]{\Aone}\norm[\Wcal]{\w}}
\end{flalign*}
\medskip

\item $\dis \Jtwo \in \Ltwo$
\medskip\\
Proof: Let us recall the equation of $\Th$ from \eqref{eq:Th}
\[
 \brac{\Dt^2 +i\frac{\Aone}{\Zapabs^2}\pap  -i\sigma\Dapabs^3} \Th = \Rtwo +i\Jtwo
\]
Observe that $\Jtwo$ is real valued. Hence by applying $\Imag (\Id - \Hil)$ to the above equation and using \propref{prop:commutator} and \lemref{lem:CW} we have
\begin{align*}
\norm[2]{\Jtwo} & \lesssim \norm[2]{(\Id - \Hil)\Dt^2\Th} + \brac{\norm*[\bigg][\Ccal]{\frac{1}{\Zapabs^2}\pap\Aone} + \norm[\Wcal]{\Aone}\norm[\Ccal]{\Dapabs\frac{1}{\Zapabs}}}\norm[2]{\Th} \\
& \quad + \norm*[\big][2]{\sigma(\Id - \Hil)\Dapabs^3\Th} + \norm[2]{\Rtwo}
\end{align*}
\medskip

\item $\dis \sigma\sqbrac{\frac{1}{\Zap},\Hil}\Dapabs^3\Ztapbar \in \Hhalf$, $\dis \sigma\sqbrac{\frac{1}{\Zapbar},\Hil}\Dapabs^3\Ztapbar \in \Hhalf$
\medskip\\
Proof: We will only prove $\dis \sigma\sqbrac{\frac{1}{\Zap},\Hil}\Dapabs^3\Ztapbar \in \Hhalf$ and the other one is proved exactly in the same way. We first observe that
\begin{align*}
\frac{\sigma}{\Zap}\pap^2\brac{\frac{1}{\Zapbar}\Dapabs\Ztapbar} &= \frac{\sigma}{\Zap}\pap\cbrac{\brac{\pap\frac{1}{\Zapbar}}\Dapabs\Ztapbar + \w\Dapabs^2\Ztapbar} \\
& = \sigma\brac{\pap\frac{1}{\Zapbar}}\wbar\Dapabs^2\Ztapbar + \sigma\brac{\frac{1}{\Zap}\pap^2\frac{1}{\Zapbar}}\Dapabs\Ztapbar \\
& \quad + \sigma(\Dap\w)\Dapabs^2\Ztapbar + \sigma\Dapabs^3\Ztapbar
\end{align*}
Hence
\begin{flalign*}
\enspace \   & \norm[2]{\frac{\sigma}{\Zap}\pap^2\brac{\frac{1}{\Zapbar}\Dapabs\Ztapbar} - \sigma\Dapabs^3\Ztapbar } \\
& \lesssim \brac{\norm*[\bigg][\infty]{\sigma^\half\Zapabs^\half\pap\frac{1}{\Zapbar}} + \norm*[\bigg][\infty]{\frac{\sigma^\half}{\Zapabs^\half}\pap\w}}\norm*[\bigg][2]{\frac{\sigma^\half}{\Zapabs^\half}\Dapabs^2\Ztapbar} \\
& + \norm[2]{\sigma^\twobythree\pap^2\frac{1}{\Zapbar}}\norm[\infty]{\frac{\sigma^\onebythree}{\Zapabs^2}\pap\Ztapbar}
\end{flalign*}
Now we have
\begin{align*}
 \sigma\sqbrac{\frac{1}{\Zap},\Hil}\Dapabs^3\Ztapbar &= \sqbrac{\frac{1}{\Zap},\Hil}\cbrac{\sigma\Dapabs^3\Ztapbar - \frac{\sigma}{\Zap}\pap^2\brac{\frac{1}{\Zapbar}\Dapabs\Ztapbar}}\\
& \quad + \sigma\sqbrac{\frac{1}{\Zap},\Hil}\frac{1}{\Zap}(\Ph + \Pa)\pap^2\brac{\frac{1}{\Zapbar}\Dapabs\Ztapbar}
\end{align*}
We can control each of the terms
\begin{enumerate}
\item The first term is easily controlled using \propref{prop:commutator}
\begin{flalign*}
\lpar \qq \quad & \norm[\Hhalf]{\sqbrac{\frac{1}{\Zap},\Hil}\cbrac{\sigma\Dapabs^3\Ztapbar - \frac{\sigma}{\Zap}\pap^2\brac{\frac{1}{\Zapbar}\Dapabs\Ztapbar}}} &\\
& \lesssim \norm[2]{\pap\frac{1}{\Zap}}\norm[2]{\frac{\sigma}{\Zap}\pap^2\brac{\frac{1}{\Zapbar}\Dapabs\Ztapbar} - \sigma\Dapabs^3\Ztapbar }
\end{flalign*}

\item We have
\[
\lpar \sigma\sqbrac{\frac{1}{\Zap},\Hil}\frac{1}{\Zap}\Ph \pap^2\brac{\frac{1}{\Zapbar}\Dapabs\Ztapbar} = \sigma\sqbrac{\frac{1}{\Zap^2},\Hil} \pap^2\Ph\brac{\frac{1}{\Zapbar}\Dapabs\Ztapbar}
\]
and hence we obtain using \propref{prop:commutator}
\begin{flalign*}
\lpar \qq \quad & \norm[\Hhalf]{\sigma\sqbrac{\frac{1}{\Zap},\Hil}\frac{1}{\Zap}\Ph \pap^2\brac{\frac{1}{\Zapbar}\Dapabs\Ztapbar}} & \\
& \lesssim \norm*[\bigg][2]{\frac{1}{\Zapabs^2}\pap\Ztapbar}\cbrac{\norm[2]{\frac{\sigma}{\Zap}\pap^3\frac{1}{\Zap}} + \norm[\infty]{\sigma^\onebythree\pap\frac{1}{\Zap}}\norm[2]{\sigma^\twobythree\pap^2\frac{1}{\Zap}}  }
\end{flalign*}

\item We see that
\begin{flalign*}
&&\frac{2\sigma}{\Zap}\Pa \pap^2\brac{\frac{1}{\Zapbar}\Dapabs\Ztapbar} & = -\sigma\sqbrac{\frac{1}{\Zap},\Hil}\pap^2\brac{\frac{1}{\Zapbar}\Dapabs\Ztapbar} + \sigma(\Id - \Hil)\Dapabs^3\Ztapbar \\
&&& \quad + (\Id -\Hil)\cbrac{\frac{\sigma}{\Zap}\pap^2\brac{\frac{1}{\Zapbar}\Dapabs\Ztapbar} - \sigma\Dapabs^3\Ztapbar}
\end{flalign*}
Hence using \propref{prop:commutator} we have
\begin{align*}
& \norm[\Hhalf]{\sigma\sqbrac{\frac{1}{\Zap},\Hil}\frac{1}{\Zap}\Pa \pap^2\brac{\frac{1}{\Zapbar}\Dapabs\Ztapbar}} \\
& \lesssim \norm[2]{\pap\frac{1}{\Zap}}
\begin{aligned}[t]
& \Biggl\{ \norm[2]{\sigma^\twobythree\pap^2\frac{1}{\Zap}}\norm*[\bigg][\infty]{\frac{\sigma^\onebythree}{\Zapabs^2}\pap\Ztapbar} + \norm[2]{\sigma(\Id - \Hil)\Dapabs^3\Ztapbar} \\
& \quad + \norm[2]{\frac{\sigma}{\Zap}\pap^2\brac{\frac{1}{\Zapbar}\Dapabs\Ztapbar} - \sigma\Dapabs^3\Ztapbar} \Biggr\}
\end{aligned}
\end{align*}
\end{enumerate}
\medskip  
  
\item $\dis (\Id - \Hil)\Dt^2 \Dapbar\Ztbar \in \Hhalf$
\medskip\\
Proof: We have already shown that $\dis (\Id - \Hil)\Dt^2 \Dap\Ztbar \in \Hhalf$. Hence
\begin{align*}
(\Id - \Hil)\cbrac{\Dt^2 \Dap\Ztbar} = (\Id - \Hil)\cbrac{\sqbrac{\Dt^2,\frac{1}{\Zap}}\Ztbarap} + \sqbrac{\frac{1}{\Zap},\Hil}\Dt^2\Ztapbar + \frac{1}{\Zap}(\Id - \Hil)\Dt^2\Ztapbar
\end{align*}
Let us recall the equation of  $\Ztapbar$ from \eqref{eq:Ztbarap} 
\begin{align*}
\begin{split}
& \brac{\Dt^2 +i\frac{\Aone}{\Zapabs^2}\pap  -i\sigma\Dapabs^3} \Ztbarap \\
 & =  \Rone\Zapbar -i\brac{\pap\frac{1}{\Zap}} \Jone -i\Dap \Jone - \Zapbar\sqbrac{\Dt^2 +i\frac{\Aone}{\Zapabs^2}\pap -i\sigma\Dapabs^3, \frac{1}{\Zapbar}} \Ztapbar
\end{split}
\end{align*}
By replacing $\Dt^2\Ztapbar$ with all the other terms from the above equation and using \propref{prop:commutator} we get
\begin{align*}
\enspace & \norm[\Hhalf]{\sqbrac{\frac{1}{\Zap},\Hil}\Dt^2\Ztapbar} & \\
& \lesssim \norm[\Hhalf]{\sigma\sqbrac{\frac{1}{\Zap},\Hil}\Dapabs^3\Ztapbar} + \norm*[\bigg][2]{\pap\frac{1}{\Zap}}\cbrac{  \norm[2]{\Rone\Zapabs} + \norm[2]{\pap\frac{1}{\Zap}}\norm[\infty]{\Jone} + \norm[2]{\Dapabs\Jone} } \\
&  \quad  + \norm[2]{\pap\frac{1}{\Zap}}\norm*[\bigg][2]{\Zapabs\sqbrac{\Dt^2 +i\frac{\Aone}{\Zapabs^2}\pap -i\sigma\Dapabs^3, \frac{1}{\Zapbar}} \Ztapbar} \\
& \quad +  \norm[\infty]{\Aone}\norm[2]{\pap\frac{1}{\Zap}}\norm*[\bigg][2]{\frac{1}{\Zapabs^2}\pap\Ztapbar}
\end{align*}
We can similarly prove that $\dis \sqbrac{\frac{1}{\Zapbar},\Hil}\Dt^2\Ztapbar \in \Hhalf$. Using this we have
\begin{align*}
\norm[\Hhalf]{ \frac{1}{\Zap}(\Id - \Hil)\Dt^2\Ztapbar} & \lesssim \norm[\Ccal]{\sqbrac{\Dt^2,\frac{1}{\Zap}}\Ztbarap} + \norm[\Hhalf]{\sqbrac{\frac{1}{\Zap},\Hil}\Dt^2\Ztapbar} \\
& \quad + \norm[\Hhalf]{(\Id - \Hil)\Dt^2 \Dap\Ztbar}
\end{align*}
As $\dis (\Id - \Hil)\Dt^2\Ztapbar \in \Ltwo$, this implies that $\dis  \frac{1}{\Zap}(\Id - \Hil)\Dt^2\Ztapbar \in \Ccal$. Now as $\w \in \Wcal$, by using \lemref{lem:CW} we get that $\dis  \frac{1}{\Zapbar}(\Id - \Hil)\Dt^2\Ztapbar \in \Ccal$. Now 
\begin{align*}
(\Id - \Hil)\cbrac{\Dt^2 \Dapbar\Ztbar} = (\Id - \Hil)\cbrac{\sqbrac{\Dt^2,\frac{1}{\Zapbar}}\Ztbarap} + \sqbrac{\frac{1}{\Zapbar},\Hil}\Dt^2\Ztapbar + \frac{1}{\Zapbar}(\Id - \Hil)\Dt^2\Ztapbar
\end{align*}
Hence we obtain
\begin{align*}
\norm[\Hhalf]{(\Id - \Hil)\Dt^2 \Dapbar\Ztbar} & \lesssim \norm[\Ccal]{\sqbrac{\Dt^2,\frac{1}{\Zapbar}}\Ztbarap} + \norm[\Hhalf]{\sqbrac{\frac{1}{\Zapbar},\Hil}\Dt^2\Ztapbar}  \\
& \quad +  \norm[\Hhalf]{ \frac{1}{\Zapbar}(\Id - \Hil)\Dt^2\Ztapbar}
\end{align*}
\medskip

\item $\dis \sigma(\Id - \Hil)\Dapabs^3\Dapbar\Ztbar \in \Hhalf $
\medskip\\
Proof: We have already shown that $\dis \sigma(\Id - \Hil)\Dapabs^3 \Dap\Ztbar \in \Hhalf$. Hence
\begin{align*}
\norm[\Hhalf]{ \frac{\sigma}{\Zap}(\Id - \Hil)\Dapabs^3\Ztapbar} &  \lesssim \norm[\Hhalf]{\sigma\sqbrac{\frac{1}{\Zap},\Hil}\Dapabs^3\Ztapbar} + \norm[\Hhalf]{\sigma(\Id - \Hil)\cbrac{\sqbrac{\Dapabs^3,\frac{1}{\Zap}}\Ztbarap}}   \\
& \quad + \norm[\Hhalf]{\sigma(\Id - \Hil)\Dapabs^3 \Dap\Ztbar} 
\end{align*}
As $\dis \sigma(\Id - \Hil)\Dapabs^3\Ztapbar \in \Ltwo$, this implies that $\dis  \frac{\sigma}{\Zap}(\Id - \Hil)\Dapabs^3\Ztapbar \in \Ccal$. Now as $\w \in \Wcal$, by using \lemref{lem:CW} we get that $\dis  \frac{\sigma}{\Zapbar}(\Id - \Hil)\Dapabs^3\Ztapbar \in \Ccal$. Hence
\begin{align*}
 \norm[\Hhalf]{\sigma(\Id - \Hil)\Dapabs^3 \Dapbar\Ztbar}  &  \lesssim  \norm[\Hhalf]{\sigma(\Id - \Hil)\cbrac{\sqbrac{\Dapabs^3,\frac{1}{\Zapbar}}\Ztbarap}} + \norm[\Hhalf]{\sigma\sqbrac{\frac{1}{\Zapbar},\Hil}\Dapabs^3\Ztapbar}  \\
& \quad +  \norm[\Hhalf]{ \frac{\sigma}{\Zapbar}(\Id - \Hil)\Dapabs^3\Ztapbar}
\end{align*}
\medskip

\item $\dis \frac{1}{\Zapabs^2}\pap\Jone \in \Hhalf$ and hence $\dis \frac{1}{\Zapabs^2}\pap\Jone \in \Ccal$
\medskip\\
Proof:  Let us recall the equation of $\Dapbar\Ztbar$ from \eqref{eq:DapbarZtbar}
\begin{align*} 
 \brac{\Dt^2 +i\frac{\Aone}{\Zapabs^2}\pap  -i\sigma\Dapabs^3} \Dapbar\Ztbar  =  \Rone -i\brac{\Dapbar\frac{1}{\Zap}} \Jone -i\frac{1}{\Zapabs^2}\pap \Jone 
\end{align*}
We see that 
\[
i\frac{\Aone}{\Zapabs^2}\pap\Dapbar\Ztbar = \brac{\frac{2i\w\Aone}{\Zapabs^2}\pap\w}\Dap\Ztbar + \frac{i\w^2\Aone}{\Zapabs^2}\pap\Dap\Ztbar
\]

Now observe that $\dis \frac{1}{\Zapabs^2}\pap\Jone $ is real valued. Hence by applying $\Imag (\Id - \Hil)$ to the equation of $\Dapbar\Ztbar$ and using \lemref{lem:CW} and \propref{prop:commutator} we get
\begin{align*}
\quad \norm[\Hhalf]{\frac{1}{\Zapabs^2}\pap\Jone} & \lesssim  \norm[\Ccal]{\Dapbar\frac{1}{\Zap}}\norm[\Wcal]{\Jone} +   \norm[\Hhalf]{(\Id - \Hil)\Dt^2\Dapbar\Ztbar} + \norm*[\big][\Hhalf]{\sigma(\Id - \Hil)\Dapabs^3\Dapbar\Ztbar}  \\
& \quad + \norm[2]{\Dap^2\Ztbar}\brac{\norm[\infty]{\Aone}\norm[2]{\pap\frac{1}{\Zapbar}} + \norm[\infty]{\Aone}\norm[2]{\Dapabs\w} + \norm[2]{\Dapabs\Aone}}\\
& \quad   + \norm[\Wcal]{\Dap\Ztbar}\norm[\Wcal]{\w}\norm[\Wcal]{\Aone}\norm[\Ccal]{\frac{1}{\Zapabs^2}\pap\w} + \norm[\Ccal]{\Rone}
\end{align*}
\end{enumerate} 
\medskip

\subsection{\texorpdfstring{Closing the energy estimate \nopunct}{}}\label{sec:closeEsigma}   \hspace*{\fill} \medskip

We are now ready to close the energy $\Esigma$. To simplify the calculations we will use the following notation: If $a(t), b(t) $ are functions of time we write $a \approx b$ if there exists a universal non-negative polynomial $P$ with $\abs{a(t)-b(t)} \leq P(\Esigma(t))$. Observe that $\approx$ is an equivalence relation. With this notation proving \thmref{thm:aprioriEsigma} is equivalent to showing $ \frac{d\Esigma(t)}{dt} \approx 0$.

\begin{lem} \label{lem:timederiv}
Let $T>0$ and let $f,\bvar \in C^2([0,T), H^2(\Rsp))$ with $\bvar$ being real valued. Let $\Dt = \pt + \bvar\pap$. Then 
\begin{enumerate}[leftmargin =*, align=left]
\item $\dis \frac{\diff }{\diff t} \int f \diff \ap = \int \Dt f \diff\ap + \int \bvarap f \diff\ap$
\item $\dis \abs*[\Big]{ \frac{d}{dt}\int \abs{f}^2 \diff \ap - 2\Real \int \bar{f} (\Dt f) \diff \ap} \lesssim \norm[2]{f}^2 \norm[\infty]{\bvarap}$

\item $\dis \abs{ \frac{d}{dt}\int (\papabs\bar{f})f \diff\ap- 2\Real \cbrac{\int (\papabs\bar{f})\Dt f \diff \ap}} \lesssim   \norm[\Hhalf]{f}^2 \norm[\infty]{\bvarap}$
\end{enumerate}
\end{lem}
\begin{proof}
The first two follow directly from the fact that $\Dt = \pt + \bvar\pap$. For the third estimate we observe that  
\begin{align*}
 & \frac{d}{dt}\int (\papabs\bar{f})f \diff\ap \\
 & = 2\Real\cbrac{\int (\papabs\bar{f})\pt f \diff\ap} \\
 & = 2\Real\cbrac{\int (\papabs\bar{f})\Dt f \diff\ap} - 2\Real\cbrac{\int (\papabs\bar{f})(\bvar\pap f) \diff\ap} \\
 & = 2\Real\cbrac{\int (\papabs\bar{f})\Dt f \diff\ap} -2\Real \cbrac{\int \brac{\papabs^\half\bar{f}} \brac{\sqbrac{\papabs^\half,\bvar}\pap f} \diff \ap}  + \nobrac{\int \abs{\papabs^\half f}^2 \bvarap \diff \ap }
\end{align*}
The estimate now follows from \propref{prop:commutator}.
\end{proof}
This lemma helps us move the time derivative inside the integral as a material derivative. We will now control the time derivative of the energy. 
\smallskip

\subsubsection{Controlling $\Esigmazero$ \nopunct} \hspace*{\fill} \medskip

Recall that
\[
 \Esigmazero =  \norm[\infty]{\sigma^{\half} \Zapabs^\half \pap\frac{1}{\Zap}}^2  +   \norm[2]{\sigma^\onebysix\Zapabs^\half\pap\frac{1}{\Zap}}^6  +  \norm[2]{\pap\frac{1}{\Zap}}^2  + \norm[2]{\frac{\sigma^\half}{\Zapabs^\half}\pap^2\frac{1}{\Zap}}^2
\]
We control the terms individually
\begin{enumerate}[leftmargin =*, align=left, label=\arabic*)]

\item As mentioned in  \remref{rem:aprioriEsigma} we will substitute the time derivative with the upper Dini derivative for the $\Linfty$ term. 
\begin{align*}
\lpar & \limsup_{s \to 0^+} \frac{\norm*[\Big][\infty]{ \sigma^{\half}  \Zapabs^\half \pap\frac{1}{\Zap}}^2(t+s) -  \norm*[\Big][\infty]{  \sigma^{\half} \Zapabs^\half \pap\frac{1}{\Zap}}^2(t)  }{s} \\
& = 2\cbrac{\norm*[\Big][\infty]{  \sigma^{\half} \Zapabs^\half \pap\frac{1}{\Zap}}(t)} \limsup_{s \to 0^+} \frac{ \norm*[\Big][\infty]{  \sigma^{\half}\Zapabs^\half \pap\frac{1}{\Zap}}(t+s) -  \norm*[\Big][\infty]{ \sigma^{\half} \Zapabs^\half \pap\frac{1}{\Zap}}(t)  }{s} 
\end{align*}
\medskip\\
Now as $\dis \norm*[\Big][\infty]{  \sigma^{\half} \Zapabs^\half \pap\frac{1}{\Zap}}(t)$ is part of the energy we only need to concentrate on the second term. As $\pt (f(\cdot,t)\compose \h) = (\Dt f(\cdot,t)) \compose \h$ we use  \propref{prop:DtLinfty} to get
\begin{align*}
& \limsup_{s \to 0^+} \frac{ \norm*[\Big][\infty]{  \sigma^{\half}\Zapabs^\half \pap\frac{1}{\Zap}}(t+s) -  \norm*[\Big][\infty]{ \sigma^{\half} \Zapabs^\half \pap\frac{1}{\Zap}}(t)  }{s}  \\
& = \limsup_{s \to 0^+} \frac{ \norm*[\Big][\infty]{  \brac{\sigma^{\half}\Zapabs^\half \pap\frac{1}{\Zap}} \compose \h}(t+s) -  \norm*[\Big][\infty]{\brac{ \sigma^{\half} \Zapabs^\half \pap\frac{1}{\Zap}}\compose \h}(t)  }{s} \\
& \leq \norm*[\bigg][\infty]{\Dt\brac*[\bigg]{\sigma^\half\Zapabs^\half \pap\frac{1}{\Zap}}}(t)
\end{align*}
Recall from \eqref{form:DtZapabs} that $\Dt \Zapabs =  \Zapabs \brac*[\big]{\Real(\Dap\Zt) - \bvarap}$ and hence we have
\begin{align*}
\Dt\brac*[\bigg]{\Zapabs^\half \pap\frac{1}{\Zap}} & = \frac{1}{2} \brac*[\big]{\Real(\Dap\Zt) - \bvarap}\brac*[\bigg]{\Zapabs^\half \pap\frac{1}{\Zap}} - \bvarap\Zapabs^\half \pap\frac{1}{\Zap} \\
& \quad + \Zapabs^\half \pap\Dt\frac{1}{\Zap}
\end{align*}
Now as $\dis  \Dt \frac{1}{\Zap}  = \frac{1}{\Zap}( \bvarap - \Dap\Zt )$ we obtain
\begin{align*}
\Dt\brac*[\bigg]{\Zapabs^\half \pap\frac{1}{\Zap}} & = \frac{1}{2} \brac*[\big]{\Real(\Dap\Zt) - \bvarap - 2\Dap\Zt}\brac*[\bigg]{\Zapabs^\half \pap\frac{1}{\Zap}}  \\
& \quad +  \Zapabs^\half \Dap ( \bvarap - \Dap\Zt)
\end{align*}
Hence 
\begin{align*}
\norm*[\bigg][\infty]{\Dt\brac*[\bigg]{\sigma^\half\Zapabs^\half \pap\frac{1}{\Zap}}} & \lesssim \brac{\norm[\infty]{\Dap\Zt} + \norm[\infty]{\bvarap}}\norm[\infty]{\sigma^\half\Zapabs^\half\pap\frac{1}{\Zap}}  \\
& \quad + \norm*[\bigg][\infty]{\frac{\sigma^\half}{\Zapabs^\half}\pap\bvarap} + \norm*[\bigg][\infty]{\frac{\sigma^\half}{\Zapabs^\half}\pap\Dap\Zt } \\
& \lesssim P(\Esigma)
\end{align*}

\item 
By using the calculation above we first obtain 
\begin{align*}
\norm*[\bigg][2]{\Dt\brac*[\bigg]{\sigma^\onebysix\Zapabs^\half \pap\frac{1}{\Zap}}} & \lesssim \brac{\norm[\infty]{\Dap\Zt} + \norm[\infty]{\bvarap}}\norm[2]{\sigma^\onebysix\Zapabs^\half\pap\frac{1}{\Zap}}  \\
& \quad + \norm*[\bigg][2]{\frac{\sigma^\onebysix}{\Zapabs^\half}\pap\bvarap} + \norm*[\bigg][2]{\frac{\sigma^\onebysix}{\Zapabs^\half}\pap\Dap\Zt }
\end{align*}
Hence by using \lemref{lem:timederiv} we get
\begin{align*}
& \frac{d}{dt} \norm[2]{\sigma^\onebysix \Zapabs^\half\pap\frac{1}{\Zap}}^6 \\
& \lesssim \norm[\infty]{\bvarap}\norm[2]{\sigma^\onebysix \Zapabs^\half\pap\frac{1}{\Zap}}^6 + \norm[2]{\sigma^\onebysix \Zapabs^\half\pap\frac{1}{\Zap}}^5\norm[2]{ \Dt\brac{\sigma^\onebysix\Zapabs^\half\pap\frac{1}{\Zap}}} \\
& \lesssim P(\Esigma)
\end{align*}

\item  By using \lemref{lem:timederiv} we obtain
\begin{align*}
\frac{d}{dt} \norm[2]{\pap\frac{1}{\Zap}}^2 \lesssim \norm[\infty]{\bvarap}\norm[2]{\pap\frac{1}{\Zap}}^2 + \norm[2]{\pap\frac{1}{\Zap}}\norm[2]{\Dt\pap\frac{1}{\Zap}} \lesssim P(\Esigma)
\end{align*}

\item We first note from \eqref{form:DtoneoverZap} that
\begin{align*}
\Dt\pap\frac{1}{\Zap} = \Dap\bvarap - \Dap^2\Zt - \brac{\pap\frac{1}{\Zap}}\Dap\Zt
\end{align*}
From this and \eqref{form:DtZapabs} we see that
\begin{align*}
\quad \norm[2]{\Dt\brac{\frac{\sigma^\half}{\Zapabs^\half}\pap^2\frac{1}{\Zap}}} & \lesssim \brac{\norm[\infty]{\bvarap} + \norm[\infty]{\Dap\Zt}}\norm[2]{\frac{\sigma^\half}{\Zapabs^\half}\pap^2\frac{1}{\Zap}} + \norm[2]{\frac{\sigma^\half}{\Zapabs^\half}\pap\Dap\bvarap} \\
& \quad + \norm[2]{\frac{\sigma^\half}{\Zapabs^\half}\pap\Dap^2\Zt} + \norm[2]{\pap\frac{1}{\Zap}}\norm[\infty]{\frac{\sigma^\half}{\Zapabs^\half}\pap\Dap\Zt}
\end{align*}
Hence by using \lemref{lem:timederiv} we get
\begin{align*}
& \frac{d}{dt} \norm[2]{\frac{\sigma^\half}{\Zapabs^\half}\pap^2\frac{1}{\Zap}}^2 \\
& \lesssim \norm[\infty]{\bvarap}\norm[2]{\frac{\sigma^\half}{\Zapabs^\half}\pap^2\frac{1}{\Zap}}^2 + \norm[2]{\frac{\sigma^\half}{\Zapabs^\half}\pap^2\frac{1}{\Zap}}\norm[2]{ \Dt\brac{\frac{\sigma^\half}{\Zapabs^\half}\pap^2\frac{1}{\Zap}}} \\
& \lesssim P(\Esigma)
\end{align*}

\end{enumerate}
\medskip

\subsubsection{Controlling $\Esigmaone$ \nopunct} \hspace*{\fill} \medskip

Recall that 
\begin{align*}
 \Esigmaone = \norm[\Hhalf]{(\Zttbar -i)\Zap}^2 + \norm[2]{\sqrt{\Aone}\Ztapbar}^2 + \norm[2]{\frac{\sigma^\half}{\Zapabs^\half}\pap\Ztapbar}^2
\end{align*}
We will first simplify the time derivative of each of the individual terms before combining them. 
\begin{enumerate}[leftmargin =*, align=left, label=\arabic*)]
\item 
As $\bvarap, \Hil\bvarap \in \Linfty$, by using \lemref{lem:timederiv} we get
\begin{align*}
\qquad \lpar  \frac{d}{dt} \int \abs{\papabs^\half \cbrac{ (\Zttbar -i)\Zap }}^2\difff\ap \approx 2\Real \int \cbrac{\papabs \brac{ (\Ztt +i)\Zapbar }}\Dt\brac{(\Zttbar -i)\Zap}  \difff\ap
\end{align*}
Now from \eqref{form:Zttbar} we have
\begin{align*}
 \Dt\brac{(\Zttbar -i)\Zap} &= \Ztttbar\Zap + (\Dap\Zt -\bvarap)(\Zttbar-i)\Zap \\
 & = \Ztttbar\Zap + (\Dap\Zt -\bvarap)(-i\Aone + \sigma\pap\Th)
\end{align*}
and using \propref{prop:Leibniz} we observe that
\begin{align*}
& \norm[\Hhalf]{(\Dap\Zt -\bvarap)(-i\Aone + \sigma\pap\Th)} \\
& \lesssim \brac{\norm[\Linfty\cap\Hhalf]{\Dap\Zt} + \norm[\Linfty\cap\Hhalf]{\bvarap} }\norm[\Linfty\cap\Hhalf]{\Aone} + \brac{\norm*[\big][2]{\sigma^\onebythree\pap\Dap\Zt} + \norm*[\big][2]{\sigma^\onebythree\pap\bvarap} }\norm*[\big][2]{\sigma^\twobythree\pap\Th} \\
& \quad + \brac{\norm[\infty]{\Dap\Zt} + \norm[\infty]{\bvarap}}\norm[\Hhalf]{\sigma\pap\Th}
\end{align*}
Hence we have
\[
\qquad \lpar  \frac{d}{dt} \int \abs{\papabs^\half \cbrac{ (\Zttbar -i)\Zap }}^2\difff\ap \approx 2\Real \int (\Ztttbar \Zap)\papabs \brac{ (\Ztt + i)\Zapbar } \difff\ap
\]

\item
As $\bvarap, \Aone, \Dt\Aone \in \Linfty$ and $\Ztapbar, \Zttbarap \in \Ltwo$ we get
\begin{flalign*}
&& \frac{d}{dt}\int \Aone \abs{ \Ztapbar }^2\difff\ap & = \int (\bvarap\Aone + \Dt\Aone)\abs{\Ztapbar}^2 \diff\ap + 2\Real \int \Aone \Ztapbar \brac{-\bvarap\Ztap + \Zttap} \diff \ap \\
&&& \approx 0
\end{flalign*}
Now observe from \propref{prop:Leibniz} that
\begin{align*}
\norm[\Hhalf]{\Aone\Dapbar\Ztbar} \lesssim \norm[\Linfty\cap\Hhalf]{\Aone}\norm[\Linfty\cap\Hhalf]{\Dapbar\Ztbar}
\end{align*}
Hence we have
\begin{align*}
\frac{d}{dt}\int \Aone \abs{ \Ztapbar }^2\difff\ap \approx 2\Real \int \brac{i\Aone \Dapbar\Ztbar}\papabs\brac{(\Ztt+i)\Zapbar} \diff\ap
\end{align*}

\item 
By \lemref{lem:timederiv} we get
\[
 \sigma \frac{d}{dt} \int \abs*[\Big]{\frac{1}{\Zapabs^\half}\pap \Ztapbar }^2\difff\ap \approx 2 \sigma\Real \int \brac{\frac{1}{\Zapabs^\half}\pap \Ztapbar }\Dt\brac{\frac{1}{\Zapabs^\half}\pap \Ztap } \difff\ap
\]
Using \eqref{form:DtZapabs} we obtain
\begin{align*}
\sigma^\half \Dt\brac{\frac{1}{\Zapabs^\half}\pap \Ztap } &= \brac{ -\frac{3}{2}\bvarap - \frac{1}{2}\Real(\Dap\Zt) }\brac{\frac{\sigma^\half}{\Zapabs^\half}\pap\Ztap} \\
& \quad - \brac{\frac{\sigma^\half}{\Zapabs^\half}\pap\bvarap}\Ztap + \frac{\sigma^\half}{\Zapabs^\half}\pap\Zttap
\end{align*}
As $\dis \bvarap, \Real(\Dap\Zt), \frac{\sigma^\half}{\Zapabs^\half}\pap\bvarap \in \Linfty$ and $\dis \Ztap, \frac{\sigma^\half}{\Zapabs^\half}\pap\Ztap \in\Ltwo$ we have
\[
 \sigma \frac{d}{dt} \int \abs*[\Big]{\frac{1}{\Zapabs^\half}\pap \Ztapbar }^2\difff\ap \approx 2 \sigma\Real \int \brac{\frac{1}{\Zapabs^\half}\pap \Ztapbar }\brac{\frac{1}{\Zapabs^\half}\pap \Zttap } \difff\ap
\] 
Now
\begin{flalign*}
&&\frac{\sigma^\half}{\Zapabs^\half}\pap \Zttap &= \frac{\sigma^\half}{\Zapabs^\half}\pap^2\brac{\frac{1}{\Zapbar}(\Ztt+i)\Zapbar} \\
&&& = \brac{\frac{\sigma^\half}{\Zapabs^\half}\pap^2\frac{1}{\Zapbar}}(\Ztt+i)\Zapbar + 2\brac{\frac{\sigma^\half}{\Zapabs^\half}\pap\frac{1}{\Zapbar}}\pap\cbrac{(\Ztt+i)\Zapbar} \\
&&& \quad + \frac{\sigma^\half}{\Zapabs^\half\Zapbar}\pap^2\cbrac{(\Ztt+i)\Zapbar}
\end{flalign*}
Using $(\Ztt+i)\Zapbar = i\Aone + \sigma\pap\Thbar$ from \eqref{form:Zttbar} we obtain 
\begin{align*}
\lpar & \norm[2]{ \brac{\frac{\sigma^\half}{\Zapabs^\half}\pap^2\frac{1}{\Zapbar}}(\Ztt+i)\Zapbar}  \lesssim \norm[2]{\frac{\sigma^\half}{\Zapabs^\half}\pap^2\frac{1}{\Zapbar}}\norm[\infty]{\Aone}  + \norm[\infty]{\frac{\sigma^\fivebysix}{\Zapabs^\half}\pap^2\frac{1}{\Zapbar}}\norm*[\big][2]{\sigma^\twobythree\pap\Th} \\
& \norm[2]{\brac{\frac{\sigma^\half}{\Zapabs^\half}\pap\frac{1}{\Zapbar}}\pap\cbrac{(\Ztt+i)\Zapbar}} \lesssim \norm[\infty]{\sigma^\half\Zapabs^\half\pap\frac{1}{\Zap}}\brac{\norm*[\big][2]{\Dapabs\Aone} + \norm[2]{\frac{\sigma}{\Zapabs}\pap^2\Th }}
\end{align*}
Hence 
\[
\lpar \quad  \sigma \frac{d}{dt} \int \abs*[\Big]{\frac{1}{\Zapabs^\half}\pap \Ztapbar }^2\difff\ap \approx 2 \sigma\Real \int \brac{\frac{1}{\Zapabs^\half}\pap \Ztapbar }\brac{ \frac{1}{\Zapabs^\half\Zapbar}\pap^2\cbrac{(\Ztt+i)\Zapbar} } \difff\ap
\]
Now using $\pap= i\papabs + (\Id + \Hil)\pap $ and $(\Ztt + i)\Zapbar = i\Aone + \sigma\pap\Thbar$ we obtain
\begin{align*}
\frac{\sigma^\half}{\Zapabs^\half\Zapbar}\pap^2\cbrac{(\Ztt+i)\Zapbar}  & = \frac{i\sigma^\half}{\Zapabs^\half\Zapbar}\pap\papabs\cbrac{(\Ztt+i)\Zapbar} \\
& \quad + \frac{i\sigma^\half}{\Zapabs^\half\Zapbar}(\Id + \Hil)\pap^2\Aone
\end{align*}
By commuting the Hilbert transform outside and using \propref{prop:commutator} we obtain
\begin{align*}
\lpar \quad \enspace \norm[2]{\frac{\sigma^\half}{\Zapabs^\threebytwo}(\Id + \Hil)\pap^2\Aone} & \lesssim \norm[\infty]{\Aone}\brac{\norm[2]{\frac{\sigma^\half}{\Zapabs^\half}\pap^2\frac{1}{\Zapabs}} + \norm[2]{\pap\frac{1}{\Zapabs}}\norm[\infty]{\sigma^\half\Zapabs^\half\pap\frac{1}{\Zapabs}} } \\
& \quad + \norm[2]{\frac{\sigma^\half}{\Zapabs^\threebytwo}\pap^2\Aone }
\end{align*}
Hence 
\[
\lpar \quad  \sigma \frac{d}{dt} \int \abs*[\Big]{\frac{1}{\Zapabs^\half}\pap \Ztapbar }^2\difff\ap \approx 2 \Real \int \cbrac{-i\sigma\pap\brac{\frac{1}{\Zapbar}\Dapabs\Ztapbar}}\papabs\brac{(\Ztt+i)\Zapbar  } \difff\ap
\]

\item Now by combining all the three terms we obtain
\begin{align*}
\frac{d}{dt} \Esigmaone \approx 2\Real \int \cbrac{ \Ztttbar\Zap + i\Aone\Dapbar\Ztbar -i\sigma\pap\brac{\frac{1}{\Zapbar}\Dapabs\Ztapbar}}\papabs\brac{(\Ztt+i)\Zapbar  } \difff\ap
\end{align*}
Recall from \eqref{eq:ZtZap} that
\begin{align*} 
\begin{split}
& \Ztttbar\Zap + i\Aone\Dapbar\Ztbar -i\sigma\pap\brac{\frac{1}{\Zapbar}\Dapabs\Ztapbar  } \\
& = i\sigma\pap\cbrac{\brac{\Dapabs \frac{1}{\Zapbar}}\Ztapbar } -\sigma(\Dap\Zt)\pap\Th -\sigma\pap\cbrac{(\Real\Th)\Dapbar\Ztbar} -i\Jone
\end{split}
\end{align*}
Hence it is sufficient to show that each of the terms on the right hand side is in $\Hhalf$. We have already shown that $\Jone \in \Hhalf$. We also have from \propref{prop:Leibniz} and \lemref{lem:CW}
\begin{align*}
\norm[\Hhalf]{\sigma(\Dap\Zt)\pap\Th} & \lesssim \norm*[\big][2]{\sigma^\onebythree\pap\Dap\Zt}\norm*[\big][2]{\sigma^\twobythree\pap\Th} + \norm[\infty]{\Dap\Zt}\norm[\Hhalf]{\sigma\pap\Th} \\
\norm[\Hhalf]{\sigma\pap\cbrac{(\Real\Th)\Dapbar\Ztbar}} & \lesssim \norm*[\big][2]{\sigma^\onebythree\pap\Dapbar\Ztbar}\norm*[\big][2]{\sigma^\twobythree\pap\Th} + \norm[\infty]{\Dapbar\Ztbar}\norm[\Hhalf]{\sigma\pap\Th} \\
& \quad + \norm*[\big][\Wcal]{\sigma^\half\Zapabs^\half\Real \Th}\norm*[\bigg][\Ccal]{\frac{\sigma^\half}{\Zapabs^\half}\pap\Dapbar\Ztbar}
\end{align*}
Now observe that
\begin{align*}
i\sigma\pap\cbrac{\brac{\Dapabs \frac{1}{\Zapbar}}\Ztapbar } & = i\sigma\brac{\pap\frac{1}{\Zapabs}}\brac{\pap\frac{1}{\Zapbar}}\Ztapbar + i\sigma\brac{\pap\frac{1}{\Zapbar}}\Dapabs\Ztapbar \\
& \quad + i\sigma\brac{\pap^2\frac{1}{\Zapbar}}\Dapabs\Ztbar
\end{align*}
We have the estimates from \lemref{lem:CW}
\begin{align*}
\norm[\Hhalf]{\sigma\brac{\pap\frac{1}{\Zapabs}}\brac{\pap\frac{1}{\Zapbar}}\Ztapbar} & \lesssim \norm[\Wcal]{\sigma^\half\Zapabs^\half\pap\frac{1}{\Zapabs}}\norm[\Wcal]{\sigma^\half\Zapabs^\half\pap\frac{1}{\Zapbar}}\norm[\Ccal]{\Dapabs\Ztbar} \\
\norm[\Hhalf]{\sigma\brac{\pap\frac{1}{\Zapbar}}\Dapabs\Ztapbar} & \lesssim \norm[\Wcal]{\sigma^\half\Zapabs^\half\pap\frac{1}{\Zapbar}}\norm*[\bigg][\Ccal]{\frac{\sigma^\half}{\Zapabs^\threebytwo}\pap\Ztapbar}
\end{align*}
For the last term we use  \propref{prop:Hhalfweight} with $\dis f = \sigma^\twobythree\pap^2\frac{1}{\Zapbar}, w = \frac{\sigma^\onebysix}{\Zapabs^\half}, h = \frac{\sigma^\onebysix}{\Zapabs^\half}\Ztapbar$. Hence
\begin{align*}
& \norm[\Hhalf]{\sigma\brac{\pap^2\frac{1}{\Zapbar}}\Dapabs\Ztbar} \\
& \lesssim \norm*[\bigg][\Hhalf]{\frac{\sigma^\fivebysix}{\Zapabs^\half}\pap^2\frac{1}{\Zapbar}}\norm*[\bigg][\infty]{\frac{\sigma^\onebysix}{\Zapabs^\half}\Ztapbar} + \norm[2]{\sigma^\twobythree\pap^2\frac{1}{\Zapbar}}\norm[2]{\sigma^\onebythree\pap\Dapabs\Ztbar} \\
& \quad + \norm[2]{\sigma^\twobythree\pap^2\frac{1}{\Zapbar}}\norm[2]{\sigma^\onebysix\Zapabs^\half\pap\frac{1}{\Zapabs}}\norm*[\bigg][\infty]{\frac{\sigma^\onebysix}{\Zapabs^\half}\Ztapbar}
\end{align*}
This completes the proof of $\dis \frac{d}{dt} \Esigmaone(t) \lesssim P(\Esigma(t))$
\end{enumerate}
\medskip

\subsubsection{Controlling $\Esigmatwo$ and $\Esigmathree$ \nopunct}\label{sec:controlEsigmatwothree} \hspace*{\fill} \medskip

Recall that both $\Esigmatwo$ and $\Esigmathree$ are of the form
\begin{align*}
E_{\sigma,i} = \norm[2]{\Dt f}^2 + \norm[\Hhalf]{\sqrt{\Aone}\frac{f}{\Zapabs}}^2 + \norm[\Hhalf]{\frac{\sigma^\half}{\Zapabs^\threebytwo}\pap f}^2
\end{align*}
Where $f = \Ztapbar$ for $i=2$ and $f=\Th$ for $i=3$. Also note that $\Ph f = f$ for these choices of $f$. We will simplify the time derivative of each of the terms individually before combining them.
\begin{enumerate}[leftmargin =*, align=left, label=\arabic*)]
\item 
As $\bvarap \in \Linfty$ we have from \lemref{lem:timederiv}
\[
\frac{d}{dt} \int \abs{\Dt f}^2 \diff\ap \approx 2\Real \int (\Dt^2 f)(\Dt \bar{f}) \diff\ap
\]

\item 
By using \lemref{lem:timederiv} we have
\[
\frac{d}{dt}  \int \abs*[\bigg]{ \papabs^\half \brac*[\bigg]{\frac{\sqrt{\Aone}}{\Zapabs} f}}^2\difff\ap \approx  2\Real \int  \cbrac{\papabs\brac*[\bigg]{\frac{\sqrt{\Aone}}{\Zapabs} f}} \Dt\brac*[\bigg]{\frac{\sqrt{\Aone}}{\Zapabs} \bar{f}} \difff\ap
\]
Observe from \eqref{form:DtZapabs} that
\begin{align*}
 \Dt\brac*[\bigg]{\frac{\sqrt{\Aone}}{\Zapabs} \bar{f}}  = \cbrac{\frac{\Dt\Aone}{2\Aone} + \bvarap - \Real(\Dap\Zt) }\frac{\sqrt{\Aone}}{\Zapabs}\bar{f} + \frac{\sqrt{\Aone}}{\Zapabs}\Dt\bar{f}
\end{align*}
We note that for  $f = \Ztapbar$ or $f = \Th$ we have $\dis \frac{f}{\Zapabs} \in \Ccal$. Hence by \lemref{lem:CW} we have
\begin{align*}
\lpar \quad \norm[\Hhalf]{ \cbrac{\frac{\Dt\Aone}{2\Aone} + \bvarap - \Real(\Dap\Zt) }\frac{\sqrt{\Aone}}{\Zapabs}\bar{f}} \lesssim \norm[\Ccal]{\frac{f}{\Zapabs}}\norm[\Wcal]{\sqrt{\Aone}}
\begin{aligned}[t]
& \biggl\{\norm[\Wcal]{\Dt\Aone}\norm[\Wcal]{\frac{1}{\Aone}} \\
& \quad  + \norm[\Wcal]{\bvarap} + \norm[\Wcal]{\Dap\Zt} \biggr\}
\end{aligned}
\end{align*}
Hence
\[
\frac{d}{dt}  \int \abs*[\bigg]{ \papabs^\half \brac*[\bigg]{\frac{\sqrt{\Aone}}{\Zapabs} f}}^2\difff\ap \approx  2\Real \int  \cbrac{\frac{\sqrt{\Aone}}{\Zapabs}\papabs\brac*[\bigg]{\frac{\sqrt{\Aone}}{\Zapabs} f}} (\Dt\bar{f}) \diff\ap
\]
We simplify further using $\papabs = i\Hil\pap$ and $\Hil f = f$
\begin{align*}
&\frac{\sqrt{\Aone}}{\Zapabs}\papabs\brac*[\bigg]{\frac{\sqrt{\Aone}}{\Zapabs} f} \\
& = i\sqbrac{\frac{\sqrt{\Aone}}{\Zapabs} ,\Hil}\pap\brac{\frac{\sqrt{\Aone}}{\Zapabs}f} + i\Hil\cbrac{\frac{\sqrt{\Aone}}{\Zapabs}\pap\brac{\frac{\sqrt{\Aone}}{\Zapabs}}f + \frac{\Aone}{\Zapabs^2}\pap f  } \\
& =  i\sqbrac{\frac{\sqrt{\Aone}}{\Zapabs} ,\Hil}\pap\brac{\frac{\sqrt{\Aone}}{\Zapabs}f} + i\Hil\cbrac*[\Bigg]{\frac{1}{2}\brac*[\Bigg]{\frac{1}{\Zapabs^2}\pap\Aone}f + \Aone\brac{\Dapabs\frac{1}{\Zapabs}}f  } \\
& \quad - i\sqbrac{\frac{\Aone}{\Zapabs^2},\Hil}\pap f + i\frac{\Aone}{\Zapabs^2}\pap f
\end{align*} 
Hence by \lemref{lem:CW}, \propref{prop:commutator} and using $\Aone \geq 1$ we have  
\begin{align*}
\qq \lpar & \norm*[\bigg][2]{\frac{\sqrt{\Aone}}{\Zapabs}\papabs\brac*[\bigg]{\frac{\sqrt{\Aone}}{\Zapabs} f} -  i\frac{\Aone}{\Zapabs^2}\pap f} \\
& \lesssim \brac{\norm[2]{\Dapabs\Aone} + \norm[\infty]{\sqrt{\Aone}}\norm[2]{\pap\frac{1}{\Zapabs}} }\norm[\Hhalf]{\frac{\sqrt{\Aone}}{\Zapabs}f} +  \norm[\Wcal]{\Aone}\norm[\Ccal]{\Dapabs\frac{1}{\Zapabs}}\norm[\Ccal]{\frac{f}{\Zapabs}}  \\
& \quad + \norm[2]{f}\norm*[\bigg][\Linfty\cap\Hhalf]{\frac{1}{\Zapabs^2}\pap\Aone} 
\end{align*}
As $\Dt f \in \Ltwo$ this shows that 
\[
\frac{d}{dt}  \int \abs*[\bigg]{ \papabs^\half \brac*[\bigg]{\frac{\sqrt{\Aone}}{\Zapabs} f}}^2\difff\ap \approx  2\Real \int  \brac*[\bigg]{i\frac{\Aone}{\Zapabs^2}\pap f} (\Dt\bar{f}) \diff\ap
\]

\item
By \lemref{lem:timederiv} we have
\[
\lpar \quad  \sigma \frac{d}{dt} \int  \abs*[\bigg]{\papabs^\half \brac*[\Bigg]{\frac{1}{\Zapabs^\threebytwo} \pap f} }^2\difff\ap \approx 2\sigma \Real \int  \cbrac*[\Bigg]{\papabs \brac*[\Bigg]{\frac{1}{\Zapabs^\threebytwo} \pap f}}\Dt \brac*[\Bigg]{\frac{1}{\Zapabs^\threebytwo} \pap \bar{f}} \diff\ap 
\]
We note that 
\begin{align*}
\lpar \sigma^\half\Dt\brac{\frac{1}{\Zapabs^\threebytwo} \pap \bar{f}} = \sigma^\half\brac{\frac{1}{2}\bvarap - \frac{3}{2}\Real(\Dap\Zt)}\brac{\frac{1}{\Zapabs^\threebytwo} \pap \bar{f}} + \frac{\sigma^\half}{\Zapabs^\threebytwo}\pap\Dt \bar{f}
\end{align*}
As $\dis \frac{\sigma^\half}{\Zapabs^\threebytwo}\pap f \in \Ccal$ for $f= \Ztapbar$ or $f = \Th$ we use \lemref{lem:CW} to obtain
\[
\lpar \quad  \norm[\Hhalf]{\sigma^\half\brac{\frac{1}{2}\bvarap - \frac{3}{2}\Real(\Dap\Zt)}\brac{\frac{1}{\Zapabs^\threebytwo} \pap \bar{f}} } \lesssim \brac{\norm[\Wcal]{\bvarap} + \norm[\Wcal]{\Dap\Zt}}\norm[\Ccal]{\frac{\sigma^\half}{\Zapabs^\threebytwo}\pap f}
\]
Hence
\[
\lpar \quad  \sigma \frac{d}{dt} \int  \abs*[\bigg]{\papabs^\half \brac*[\Bigg]{\frac{1}{\Zapabs^\threebytwo} \pap f} }^2\difff\ap \approx -2\sigma \Real \int  \pap\cbrac*[\Bigg]{\frac{1}{\Zapabs^\threebytwo}\papabs \brac*[\Bigg]{\frac{1}{\Zapabs^\threebytwo} \pap f}}\Dt\bar{f} \diff\ap
\]
Now using $\papabs = i\Hil\pap$ we see that
\begin{align*}
\sigma\pap\cbrac*[\Bigg]{\frac{1}{\Zapabs^\threebytwo}\papabs \brac*[\Bigg]{\frac{1}{\Zapabs^\threebytwo} \pap f}} & = i\sigma\pap\sqbrac{\frac{1}{\Zapabs^\threebytwo},\Hil}\pap \brac*[\Bigg]{\frac{1}{\Zapabs^\threebytwo} \pap f} \\
& \quad + i\sigma\Hil\pap\cbrac*[\Bigg]{\frac{1}{\Zapabs^\threebytwo}\pap \brac*[\Bigg]{\frac{1}{\Zapabs^\threebytwo} \pap f}}
\end{align*}
Using \propref{prop:commutator} we have the estimate
\begin{align*}
\lpar \quad & \norm[2]{\sigma\pap\sqbrac{\frac{1}{\Zapabs^\threebytwo},\Hil}\pap \brac*[\Bigg]{\frac{1}{\Zapabs^\threebytwo} \pap f} } \\
& \lesssim \norm[\Hhalf]{\frac{\sigma^\half}{\Zapabs^\threebytwo} \pap f} \cbrac{\norm[2]{\frac{\sigma^\half}{\Zapabs^\half}\pap^2\frac{1}{\Zapabs}} + \norm[\infty]{\sigma^\half\Zapabs^\half\pap\frac{1}{\Zapabs}}\norm[2]{\pap\frac{1}{\Zapabs}}}
\end{align*}
By using the expansion in \eqref{form2:Dapabs^3f} for $f=\Ztapbar$ we get
\begin{align*}
\lpar \qq & \norm[2]{\sigma\Dapabs^3\Ztapbar - \sigma\pap\cbrac*[\Bigg]{\frac{1}{\Zapabs^\threebytwo}\pap \brac*[\Bigg]{\frac{1}{\Zapabs^\threebytwo} \pap \Ztapbar}}} \\
& \lesssim \norm[\infty]{\sigma^\half\Zapabs^\half\pap\frac{1}{\Zapabs}}\norm[2]{\frac{\sigma^\half}{\Zapabs^\fivebytwo}\pap^2\Ztapbar} + \norm[\infty]{\sigma^\half\Zapabs^\half\pap\frac{1}{\Zapabs}}^2\norm[2]{\frac{1}{\Zapabs^2}\pap\Ztapbar} \\
& \quad + \norm[2]{\frac{\sigma^\half}{\Zapabs^\half}\pap^2\frac{1}{\Zapabs}}\norm[\infty]{\frac{\sigma^\half}{\Zapabs^\threebytwo}\pap\Ztapbar}
\end{align*}
Similarly using  the expansion in \eqref{form2:Dapabs^3f} for $f=\Th$ and using \lemref{lem:CW} we obtain
\begin{align*}
\lpar \qq & \norm[2]{\sigma\Dapabs^3\Th - \sigma\pap\cbrac*[\Bigg]{\frac{1}{\Zapabs^\threebytwo}\pap \brac*[\Bigg]{\frac{1}{\Zapabs^\threebytwo} \pap \Th}}} \\
& \lesssim \norm[\Ccal]{\Dapabs\frac{1}{\Zapabs}}\norm[\Ccal]{\frac{\sigma}{\Zapabs^2}\pap^2\Th} + \norm[\infty]{\sigma^\half\Zapabs^\half\pap\frac{1}{\Zapabs}}\norm[\Ccal]{\Dapabs\frac{1}{\Zapabs}}\norm[\Ccal]{\frac{\sigma^\half}{\Zapabs^\threebytwo}\pap\Th} \\
& \quad + \norm[\Ccal]{\frac{\sigma^\half}{\Zapabs^\threebytwo}\pap^2\frac{1}{\Zapabs}}\norm[\Ccal]{\frac{\sigma^\half}{\Zapabs^\threebytwo}\pap\Th}
\end{align*}
Using these we now have
\[
\lpar \quad  \sigma \frac{d}{dt} \int  \abs*[\bigg]{\papabs^\half \brac*[\Bigg]{\frac{1}{\Zapabs^\threebytwo} \pap f} }^2\difff\ap \approx -2 \Real \int \brac{i\sigma\Hil\Dapabs^3 f} \Dt\bar{f} \diff\ap
\]
But we have already shown that $\sigma(\Id - \Hil)\Dapabs^3 f \in \Ltwo$ for both $f = \Ztapbar$ and $f = \Th$. Hence we finally have
\[
\lpar \quad  \sigma \frac{d}{dt} \int  \abs*[\bigg]{\papabs^\half \brac*[\Bigg]{\frac{1}{\Zapabs^\threebytwo} \pap f} }^2\difff\ap \approx 2 \Real \int \brac{-i\sigma\Dapabs^3 f} \Dt\bar{f} \diff\ap
\]

\item Combining all three terms we obtain for $i=2,3$
\[
\frac{d}{dt} E_{\sigma,i} \approx 2 \Real \int \brac{ \Dt^2 f +i\frac{\Aone}{\Zapabs^2}\pap f  -i\sigma\Dapabs^3 f}(\Dt \bar{f}) \diff\ap
\]
For $f = \Ztapbar$ we obtain from  \eqref{eq:Ztbarap}
\begin{align*}
\begin{split}
& \brac{\Dt^2 +i\frac{\Aone}{\Zapabs^2}\pap  -i\sigma\Dapabs^3} \Ztbarap \\
 & =  \Rone\Zapbar -i\brac{\pap\frac{1}{\Zap}} \Jone -i\Dap \Jone - \Zapbar\sqbrac{\Dt^2 +i\frac{\Aone}{\Zapabs^2}\pap -i\sigma\Dapabs^3, \frac{1}{\Zapbar}} \Ztapbar
\end{split}
\end{align*}
We have already shown that $\Rone \in \Ccal$, $\Jone \in \Wcal$, $\dis \pap\frac{1}{\Zap} \in \Ltwo$ and the last term in $\Ltwo$. Now for $f = \Th$ we have from \eqref{eq:Th}
\[
 \brac{\Dt^2 +i\frac{\Aone}{\Zapabs^2}\pap  -i\sigma\Dapabs^3} \Th = \Rtwo +i\Jtwo
\]
In this case too we have shown that $\Rtwo, \Jtwo \in \Ltwo$. Hence this shows that 
\[
\frac{d}{dt}E_{\sigma,i}(t) \lesssim P(\Esigma(t)) \qq \text{ for } i=2,3
\]

\end{enumerate}
\medskip

\subsubsection{Controlling $\Esigmafour$ \nopunct} \hspace*{\fill} \medskip

Recall that 
\begin{align*}
 \Esigmafour = \norm[\Hhalf]{\Dt\Dapbar\Ztbar}^2 + \norm[2]{\sqrt{\Aone}\Dapabs\Dapbar\Ztbar}^2 + \norm[2]{\frac{\sigma^\half}{\Zapabs^\half}\pap\Dapabs\Dapbar\Ztbar}^2
\end{align*}
As before we first simplify the terms individually before combining them.

\begin{enumerate}[leftmargin =*, align=left, label=\arabic*)]
\item 
By \lemref{lem:timederiv} and the fact that $\papabs = i\Hil\pap$ we have
\begin{align*}
\frac{d}{dt}  \int \abs*[\Big]{\papabs^\half \brac{\Dt \Dapbar \Ztbar}}^2\difff\ap & \approx 2\Real \int (\Dt^2\Dapbar\Ztbar)\papabs(\Dt\Dap\Zt) \diff \ap \\
& \approx 2\Real \int (\Hil\Dt^2\Dapbar\Ztbar)\cbrac{-i\pap(\Dt\Dap\Zt)} \diff \ap
\end{align*}
But we have shown that $(\Id - \Hil)\Dt^2\Dapbar\Ztbar \in \Hhalf$. Hence we have
\[
\frac{d}{dt}  \int \abs*[\Big]{\papabs^\half \brac{\Dt \Dapbar \Ztbar}}^2\difff\ap  \approx 2\Real \int (\Dt^2\Dapbar\Ztbar)\cbrac{-i\pap(\Dt\Dap\Zt)} \diff \ap
\]

\item
By \lemref{lem:timederiv} and as $\bvarap \in \Linfty$ we have
\begin{align*}
\frac{d}{dt} \int \Aone \abs{\Dapabs \Dapbar \Ztbar}^2\difff\ap & \approx \int \brac{\frac{\Dt\Aone}{\Aone}}\Aone\abs{\Dapabs \Dapbar \Ztbar}^2 \diff\ap \\
& \quad + 2\Real \int \Aone \brac{\Dapabs\Dapbar\Ztbar}\Dt\brac{\Dapabs\Dap\Zt} \diff\ap
\end{align*}
As $\dis \frac{\Dt\Aone}{\Aone} \in \Linfty$, the first term is controlled. We now see that
\[
\Dt\Dapabs\Dap\Zt = -\Real(\Dap\Zt)\Dapabs\Dap\Zt + \Dapabs\Dt\Dap\Zt
\]
Now as $\Real(\Dap\Zt) \in \Linfty$ we obtain 
\[
\frac{d}{dt} \int \Aone \abs{\Dapabs \Dapbar \Ztbar}^2\difff\ap \approx 2\Real \int \brac*[\bigg]{i\frac{\Aone}{\Zapabs^2}\pap\Dapbar\Ztbar}\cbrac{-i\pap(\Dt\Dap\Zt)} \diff\ap
\]

\item
By \lemref{lem:timederiv} and as $\bvarap\in \Linfty$ we have
\begin{align*}
& \sigma \frac{d}{dt}\int \abs*[\bigg]{\frac{1}{\Zapabs^\half}\pap \Dapabs \Dapbar \Ztbar }^2\difff\ap \\
 & \approx 2\sigma\Real \int \cbrac{\frac{1}{\Zapabs^\half}\pap \Dapabs \Dapbar \Ztbar}\Dt\cbrac{\frac{1}{\Zapabs^\half}\pap \Dapabs \Dap \Zt} \diff \ap
\end{align*}
We see that
\begin{align*}
\lpar \quad \sigma^\half\Dt\cbrac{\frac{1}{\Zapabs^\half}\pap \Dapabs \Dap \Zt} & = \brac{-\frac{3}{2}\Real(\Dap\Zt) -  \frac{\bvarap}{2}}\cbrac{\frac{\sigma^\half}{\Zapabs^\half}\pap\Dapabs\Dap\Zt } \\
& -\Real\brac{\frac{\sigma^\half}{\Zapabs^\half}\pap\Dap\Zt}\brac{\Dapabs\Dap\Zt} + \frac{\sigma^\half}{\Zapabs^\half}\pap\Dapabs\Dt\Dap\Zt
\end{align*}
As $\Dap\Zt, \bvarap \in \Linfty$, the first term is controlled in $\Ltwo$. The second term is also  in $\Ltwo$ as we have $\dis \frac{\sigma^\half}{\Zapabs^\half}\pap\Dap\Zt \in \Linfty$ and $\Dapabs\Dap\Zt \in \Ltwo$. Hence we have
\[
 \sigma \frac{d}{dt}\int \abs*[\bigg]{\frac{1}{\Zapabs^\half}\pap \Dapabs \Dapbar \Ztbar }^2\difff\ap \approx 2\Real \int \brac{-i\sigma\Dapabs^3\Dapbar\Ztbar}\cbrac{-i\pap(\Dt\Dap\Zt)} \diff\ap
\]

\item Combining the three terms we get
\[
\frac{d}{dt} \Esigmafour \approx 2\Real \int  \brac*[\bigg]{ \Dt^2 \Dapbar\Ztbar +i\frac{\Aone}{\Zapabs^2}\pap \Dapbar\Ztbar  -i\sigma\Dapabs^3 \Dapbar\Ztbar}\cbrac{-i\pap(\Dt\Dap\Zt)} \diff\ap
\]
From equation  \eqref{eq:DapbarZtbar} we see that
\[
 \brac{\Dt^2 +i\frac{\Aone}{\Zapabs^2}\pap  -i\sigma\Dapabs^3} \Dapbar\Ztbar  =  \Rone -i\brac{\Dapbar\frac{1}{\Zap}} \Jone -i\frac{1}{\Zapabs^2}\pap \Jone 
\]
But we have already shown that $\dis \Rone,\frac{1}{\Zapabs^2}\pap\Jone \in \Hhalf$ and the second term is controlled in $\Hhalf$ by using \lemref{lem:CW}
\[
\norm[\Hhalf]{\brac{\Dapbar\frac{1}{\Zap}} \Jone} \lesssim \norm[\Ccal]{\Dapbar\frac{1}{\Zap}}\norm[\Wcal]{\Jone}
\]
Hence we have
\[
\lpar \frac{d}{dt} \Esigmafour(t) \lesssim P(\Esigma(t))
\]
This concludes the proof of \thmref{thm:aprioriEsigma}
\end{enumerate}
\medskip

\section{Equivalence of Energy and relation with Sobolev norm}\label{sec:equivalence}
\medskip

In this section we prove the equivalence of the energies $\Esigma(t)$ and $\Ecalsigma(t)$. We also explain the relation of the energy $\Ecalsigma(t)$ to the Sobolev norm of the solution.

\begin{prop}\label{prop:equivEsigma} There exists universal polynomials $P_1, P_2$ with non-negative coefficients so that if $(\Z,\Zt)(t)$ is a smooth solution to the water wave equation \eqref{eq:systemone} for $\sigma \geq 0$ in the time interval $[0,T]$ satisfying $(\Zap-1,\frac{1}{\Zap} - 1, \Zt) \in \Linfty([0,T], H^{s+\half}(\Rsp)\times H^{s+\half }(\Rsp)\times H^{s }(\Rsp))$ for all $s\geq 3$,  then for all $t \in [0,T]$ we have
\begin{align*}
\Esigma(t) \leq P_1(\Ecalsigma(t)) \quad \tx{ and }\quad  \Ecalsigma(t) \leq P_2(\Esigma(t))
\end{align*}
\end{prop}
\begin{proof}
From now on we will suppress the time variable. Let us first prove that $ \Ecalsigma \leq P_2(\Esigma)$. Note that from \secref{sec:quantEsigma} we already have most of the terms of $\Ecalsigma$ controlled. The terms which are not controlled namely $\dis \frac{\sigma^\half}{\Zap^\threebytwo}\pap^2\frac{1}{\Zap}$ and $\dis \frac{\sigma}{\Zap^2}\pap^3\frac{1}{\Zap}$ can be easily controlled in $\Hhalf$ as we have $\w \in \Wcal$ and we have already shown that $\dis \frac{\sigma^\half}{\Zapabs^\threebytwo}\pap^2\frac{1}{\Zap} \in \Ccal$, $\dis \frac{\sigma}{\Zapabs^2}\pap^3\frac{1}{\Zap} \in \Ccal$ and using \lemref{lem:CW}

Let us now show that $\Esigma \leq P_1(\Ecalsigma)$. 
\begin{enumerate}[leftmargin =*, align=left]
\item As we have $\Ztapbar \in \Ltwo$, we see from \secref{sec:quantEsigma} that $\Aone \in \Linfty\cap\Hhalf$. Hence we have that $(\Zttbar -i)\Zap \in \Hhalf$ by using equation \eqref{form:Zttbar}. We now show that $\Dap\Ztbar \in \Linfty$. Observe that
\begin{align*}
\pap (\Dap\Ztbar)^2 = 2(\Ztapbar)(\Dap^2\Ztbar) = 2(\Ztapbar)\brac{\pap\frac{1}{\Zap}}\Dap\Ztbar + 2(\Ztapbar)\brac{\frac{1}{\Zap^2}\pap\Ztapbar}
\end{align*}
Hence we have
\begin{align*}
\norm[\infty]{\Dap\Ztbar}^2 \leq 2\norm[2]{\Ztapbar}\norm[2]{\pap\frac{1}{\Zap}}\norm[\infty]{\Dap\Ztbar} + 2\norm[2]{\Ztapbar}\norm[2]{\frac{1}{\Zap^2}\pap\Ztapbar}
\end{align*}
Now using the inequality $ab \leq \frac{a^2}{2\epsilon} + \frac{\epsilon b^2}{2}$ on the first term, we obtain $\Dap\Ztbar \in \Linfty$.

\item Following the proof in \secref{sec:quantEsigma}  we now have the terms $\Dapabs\Dapbar\Ztbar \in \Ltwo$, $\dis \pap\frac{1}{\Zapabs} \in \Ltwo$, $\w \in \Wcal$,  $\Dapbar\Ztbar \in \Wcal\cap\Ccal$, $\dis \pap\Pa\brac{\frac{\Zt}{\Zap}} \in \Linfty$, $\Aone \in \Wcal$, $\bvarap \in \Wcal$, $\frac{1}{\Zapabs^2}\pap\Aone \in \Wcal\cap\Ccal$. We also see that $\Th \in \Ltwo$, $\Dt\Th \in \Ltwo$ by using the formula \eqref{form:Th} and \eqref{form:DtTh}. By following the proof of $\dis \Dap\frac{1}{\Zap} \in \Ccal$ in \secref{sec:quantEsigma},  we easily obtain $\dis \Dapabs\frac{1}{\Zap} \in \Ccal$ and $\dis \frac{\Th}{\Zapabs}\in \Ccal$. Hence we have $\dis \frac{\sqrt{\Aone}}{\Zapabs}\Th \in \Ccal$ and $\dis \frac{\sqrt{\Aone}}{\Zapabs}\Ztapbar \in \Ccal$ from \lemref{lem:CW}. 

\item Again by following the approach in \secref{sec:quantEsigma} we automatically have $\dis \sigma^\half\Zapabs^\half\pap\frac{1}{\Zap} \in \Wcal$, $\sigma^\twobythree\pap\Th \in \Ltwo$, $\dis \sigma^\twobythree\pap^2\frac{1}{\Zap} \in \Ltwo$, $\sigma^\onebythree \Th \in \Linfty\cap\Hhalf$, $\dis \sigma^\onebythree\pap\frac{1}{\Zap} \in \Linfty\cap\Hhalf$ etc. Hence we now have $\dis \frac{\sigma}{\Zapabs}\pap^2\Th \in \Ltwo$, $\sigma\pap\Dap\Th \in \Ltwo$ by following the proof of $\dis \frac{\sigma}{\Zapabs}\pap^3\frac{1}{\Zap} \in \Ltwo$ in \secref{sec:quantEsigma}. In particular we now have $\Dt\Ztapbar \in \Ltwo$ by using equation \eqref{form:Zttbar}. 

\item By following the proof of $\dis \frac{\sigma^\half}{\Zapabs^\threebytwo}\pap^2\frac{1}{\Zap} \in \Ccal$ in \secref{sec:quantEsigma} we see that $\dis \frac{\sigma^\half}{\Zapabs^\threebytwo}\pap\Th \in \Ccal$. Similarly by following the proof of $\dis \frac{\sigma^\half}{\Zapabs^\fivebytwo}\pap^2\Ztapbar \in \Ltwo$ in \secref{sec:quantEsigma} we also obtain $\dis \frac{\sigma^\half}{\Zapabs^\half}\pap\Dapabs\Dapbar\Ztbar \in \Ltwo$. We use  \propref{prop:LinftyHhalf} with $\dis f = \frac{\sigma^\half}{\Zapabs^\threebytwo}\pap\Ztapbar$ and $\dis w = \frac{1}{\Zapabs}$ to obtain $\dis \frac{\sigma^\half}{\Zapabs^\threebytwo}\pap\Ztapbar \in \Ccal$. 

\item As $\w \in \Wcal$ we have $\dis \frac{\sigma}{\Zapabs^2}\pap^3\frac{1}{\Zap} \in \Ccal$. Hence by following the proof of $\dis \frac{\sigma}{\Zapabs^2}\pap^3\frac{1}{\Zap} \in \Ccal$ in \secref{sec:quantEsigma} we obtain $\dis \frac{\sigma}{\Zapabs^2}\pap^2\Th \in \Ccal$, $\sigma\Dapbar\Dap\Th \in \Ccal$ etc. Hence by using equation \eqref{form:Zttbar} we now have $\Dapbar\Zttbar \in \Ccal$ and hence $\Dt\Dapbar\Ztbar \in \Ccal$. This finishes the proof of \propref{prop:equivEsigma}

\end{enumerate}
\end{proof}

We now explain the relation between the energy $\Ecalsigma(t)$ and the Sobolev norm of the data. 

\begin{lem}\label{lem:equivsobolev}  Let $(\Z,\Zt)(t)$ be a smooth solution to the water wave equation \eqref{eq:systemone} for $\sigma \geq 0$ in the time interval $[0,T^*]$ for $T^*>0$, satisfying $(\Zap-1,\frac{1}{\Zap} - 1, \Zt) \in \Linfty([0,T^*], H^{s+\half}(\Rsp)\times H^{s+\half }(\Rsp)\times H^{s }(\Rsp))$ for all $s\geq 3$. Then we have the following estimates
\begin{enumerate}
\item For $\sigma>0$ there exists universal polynomials $P_1, P_2$ with non-negative coefficients so that 
for each $t \in [0,T^*]$ we have
\begin{align*}
\norm[H^2]{\Ztapbar}(t) + \norm[H^{2.5}]{\pap\Zap}(t) \leq P_1\brac{\Ecalsigma(t) + \norm[\infty]{\Zap}(t) + \frac{1}{\sigma}} \quad \tx{and}\\
 \Ecalsigma(t) \leq P_2\brac{\norm[H^2]{\Ztapbar}(t) +  \norm[H^{2.5}]{\pap\Zap}(t) + \norm[\infty]{\frac{1}{\Zap}}(t) + \sigma}
\end{align*}
\item For $\sigma\geq 0$ there exists universal polynomials $P_3, P_4$ with non-negative coefficients so that for each $t \in [0,T^*]$ we have
\begin{align*}
\norm[H^1]{\Ztapbar}(t) + \norm[H^{\half}]{\pap\Zap}(t) \leq P_3\brac{\Ecalsigma(t)\vert_{\sigma=0} + \norm[\infty]{\Zap}(t)} \quad \tx{and}\\
 \Ecalsigma(t)\vert_{\sigma=0} \leq P_4\brac{\norm[H^1]{\Ztapbar}(t) +  \norm[H^{\half}]{\pap\Zap}(t) + \norm[\infty]{\frac{1}{\Zap}}(t) }
\end{align*}
\item  There exists a universal increasing function $F:[0,\infty) \to [0,\infty)$ so that if $0\leq T\leq T^*$ and we define
\begin{align*}
A & = \sup_{t\in[0,T]} \cbrac{\norm[H^{3.5}]{\Zap -1}(t) + \norm[H^{3.5}]{\frac{1}{\Zap} - 1}(t) + \norm[H^3]{\Zt}(t)} <\infty \\
B & = \sup_{t\in[0,T]} \cbrac{\norm[H^{1.5}]{\Zap -1}(t) + \norm[H^{1.5}]{\frac{1}{\Zap} - 1}(t) + \norm[H^2]{\Zt}(t)} <\infty  \\
D & = \norm[2]{\frac{1}{\Zap} - 1}(0) + \norm[2]{\Zt}(0) + \sup_{t\in[0,T]}\Ecalsigma(t) < \infty
\end{align*}
Then for $\sigma>0$
\begin{align*}
A \leq F\brac{D + \norm[\infty]{\Zap}(0) + T +  \sigma + \frac{1}{\sigma}}
\end{align*}
and for $\sigma\geq 0$
\begin{align*}
B \leq F\brac{D + \norm[\infty]{\Zap}(0) + T +  \sigma + 1}
\end{align*}
\end{enumerate}

\end{lem}
\begin{rmk}\label{rem:equivsobolev}
For $\sigma>0$ if the interface is non-self intersecting and if $\Ecalsigma$ is well defined with $\Ecalsigma  <\infty$, then we in fact have $\Zap \in \Linfty$ but the norm $\norm[\infty]{\Zap}$ depends on $\sigma^{-\onebythree}$ and the rate at which $\Zap \to 1$ as $\abs{\ap} \to \infty$. To see this observe that in the proof of \thmref{thm:aprioriEsigma} we showed that $\norm*[\infty]{\sigma^\onebythree \kap} \leq C(\Ecalsigma)$ and hence the curvature $\kap \in \Linfty$. Therefore by the Kellogg-Warschawski theorem (see chapter 3 of \cite{Po92}), the derivative of the Riemann mapping extends continuously to the boundary and hence $\Zap \in \Linfty_{loc}$. As $\Zap \to 1$ when $\abs{\ap} \to \infty$, we have that $\Zap \in \Linfty$. Hence by part 1 of the above lemma we have $\Ztapbar \in H^2(\Rsp)$ and $\pap\Zap \in H^{2.5}(\Rsp)$. Hence for $\sigma>0$ the condition $\Ecalsigma<\infty$ is essentially equivalent to the condition that the solution is in a suitable Sobolev space.   
\end{rmk}

\begin{proof}
We prove each part seperately. 

\begin{enumerate}[leftmargin =*, align=left]

\item  Let $\sigma>0$ and assume that $\dis \Ecalsigma + \norm[\infty]{\Zap}  <\infty$. Hence we have that $\Ztapbar \in \Ltwo$ and we have
\begin{align*}
\norm[2]{\pap^2\Ztapbar} \lesssim \frac{1}{\sigma^\half}\norm[\infty]{\Zap}^\fivebytwo\norm*[\Bigg][2]{\frac{\sigma^\half}{\Zap^\fivebytwo}\pap^2\Ztapbar} \quad \tx{ and }\quad \norm[2]{\Dapabs\Zap} \lesssim \norm[\infty]{\Zap}\norm[2]{\pap\frac{1}{\Zap}}
\end{align*}  
Hence $\Ztapbar \in H^2$ and as $\Zap \in \Linfty$ we obtain $\Zap \in \Wcal$. Hence from \lemref{lem:CW} we have
\begin{align*}
\norm[\Hhalf]{\pap^3\frac{1}{\Zap}} \lesssim \frac{1}{\sigma}\norm[\Wcal]{\Zap}^2\norm[\Ccal]{\frac{\sigma}{\Zap^2}\pap^3\frac{1}{\Zap}}
\end{align*}
Hence $\dis \pap\frac{1}{\Zap} \in H^{2.5}$. As $\Zap \in \Linfty$, we clearly have $\pap\Zap \in \Ltwo$ as. Now for $s \geq 1 $ we see from \propref{prop:Leibniz} that
\begin{align*}
\norm[2]{\papabs^s \pap\Zap} & = \norm[2]{\papabs^s \brac{\Zap^2\pap\frac{1}{\Zap}}} \\
& \lesssim \norm[2]{\papabs^s \Zap}\norm[\infty]{\Zap}\norm[\infty]{\pap\frac{1}{\Zap}} + \norm[\infty]{\Zap}^2\norm[2]{\papabs^s\pap\frac{1}{\Zap}}
\end{align*}
Using this for $s=1,2, 2.5$ sequentially we obtain $\pap\Zap \in H^{2.5}$.  

\item Now assume that $\sigma>0 $ and $\dis \norm[H^2]{\Ztapbar} +  \norm[H^{2.5}]{\pap\Zap} + \norm[\infty]{\frac{1}{\Zap}} < \infty$. We first observe that $\Ecalsigmatwo$ is easily controlled and that $\dis \sigma^\onebysix\Zap^\half\pap\frac{1}{\Zap} \in \Ltwo$, $\dis \sigma^\half\Zap^\half\pap\frac{1}{\Zap} \in \Linfty$. Now we have
\begin{align*}
\norm[2]{\pap\frac{1}{\Zap}} \lesssim \norm[\infty]{\frac{1}{\Zap}}^2\norm[2]{\pap\Zap}
\end{align*}
and hence for $s\geq1$ we have from \propref{prop:Leibniz}
\begin{align*}
\norm[2]{\papabs^s\pap\frac{1}{\Zap}} & = \norm[2]{\papabs^s\brac{\frac{1}{\Zap^2}\pap\Zap}} \\
& \lesssim \norm[2]{\papabs^s\frac{1}{\Zap}}\norm[\infty]{\frac{1}{\Zap}}\norm[\infty]{\pap\Zap} + \norm[\infty]{\frac{1}{\Zap}}^2\norm[2]{\papabs^s\pap\Zap}
\end{align*}
Using this for $s=1,2,2.5$ sequentially we obtain $\dis \pap\frac{1}{\Zap} \in H^{2.5}$. Hence we easily see that $\dis \frac{\sigma^\half}{\Zap^\half}\pap^2\frac{1}{\Zap} \in \Ltwo$ and $\dis \frac{\sigma}{\Zap}\pap^3\frac{1}{\Zap} \in \Ltwo$. We also have from \propref{prop:Leibniz}
\begin{align*}
\norm[\Hhalf]{\frac{1}{\Zap}\pap\frac{1}{\Zap}} & \lesssim \norm[\infty]{\frac{1}{\Zap}}\norm[\Hhalf]{\pap\frac{1}{\Zap}} + \norm[2]{\pap\frac{1}{\Zap}}^2 \\
\norm[\Hhalf]{\frac{\sigma^\half}{\Zap^\threebytwo}\pap^2\frac{1}{\Zap}} & \lesssim \sigma^\half\norm[\infty]{\frac{1}{\Zap}}^\threebytwo\norm[\Hhalf]{\pap^2\frac{1}{\Zap}} + \sigma^\half\norm[\infty]{\frac{1}{\Zap}}^\half\norm[2]{\pap\frac{1}{\Zap}}\norm[2]{\pap^2\frac{1}{\Zap}}
\end{align*}
and similarly
\begin{align*}
\norm[\Hhalf]{\frac{\sigma}{\Zap^2}\pap^3\frac{1}{\Zap}} \lesssim \sigma\norm[\infty]{\frac{1}{\Zap}}^2\norm[\Hhalf]{\pap^3\frac{1}{\Zap}} + \sigma\norm[\infty]{\frac{1}{\Zap}}\norm[2]{\pap\frac{1}{\Zap}}\norm[2]{\pap^3\frac{1}{\Zap}}
\end{align*}
We are only left with $\sigma\pap\Th$. We first observe that as $\Zap = e^{\f + i\g}$ we have 
\begin{align*}
\pap\Zap = \Zap\pap(f+ ig) \quad \tx{ and } \quad \pap^2\Zap = \Zap\cbrac{\pap(\f + i\g)}^2 + \Zap \pap^2(f+i\g)
\end{align*}
and hence we have
\begin{align*}
\norm[\Ltwo\cap\Linfty]{\pap\frac{\Zap}{\Zapabs}} & = \norm[\Ltwo\cap\Linfty]{\pap g} \lesssim \norm[\infty]{\frac{1}{\Zap}}\norm[\Ltwo\cap\Linfty]{\pap\Zap} \\
\norm[2]{\pap^2\frac{\Zap}{\Zapabs}} & = \norm[2]{\pap(e^{i\g}\pap\g)} \\
& \lesssim \norm[2]{\pap\g}\norm[\infty]{\pap\g} + \norm[\infty]{\frac{1}{\Zap}}\norm[2]{\pap^2\Zap} + \norm[\infty]{\frac{1}{\Zap}}^2\norm[\infty]{\pap\Zap}\norm[2]{\pap\Zap}
\end{align*}
From this we see using \propref{prop:Leibniz}
\begin{align*}
\norm[2]{\papabs^{\threebytwo}\brac{\frac{\Zap}{\Zapabs}\pap\frac{1}{\Zap}}} \lesssim \norm[\Hhalf]{\pap^2\frac{1}{\Zap}} + \norm[2]{\pap^2\frac{\Zap}{\Zapabs}}\norm[2]{\pap\frac{1}{\Zap}}
\end{align*}
Hence $\sigma\pap\Th \in \Hhalf$ by using the formula \eqref{form:Th}.

\item Now let $\sigma \geq 0$ and assume $\Ecalsigma\vert_{\sigma =0} + \norm[\infty]{\Zap} < \infty$. We clearly see that $\Ztapbar \in H^1$ and $\pap\Zap \in \Ltwo$. By the argument shown earlier, we also have $\Zap \in \Wcal$. Hence from \lemref{lem:CW} we have
\begin{align*}
\norm[\Hhalf]{\pap\frac{1}{\Zap}} \lesssim \norm[\Wcal]{\Zap}\norm[\Ccal]{\frac{1}{\Zap}\pap\frac{1}{\Zap}}
\end{align*}
Now we see from \propref{prop:Leibniz}
\begin{align*}
\norm[\Hhalf]{\pap\Zap}  = \norm[\Hhalf]{\Zap^2\pap\frac{1}{\Zap}}  \lesssim \norm[\infty]{\Zap}^2\norm[\Hhalf]{\pap\frac{1}{\Zap}} + \norm[\infty]{\Zap}\norm[2]{\pap\Zap}\norm[2]{\pap\frac{1}{\Zap}}
\end{align*}

\item Let $\sigma \geq 0$ and assume that $\Ztapbar \in H^1, \pap\Zap \in H^\half $ and $\dis \frac{1}{\Zap} \in \Linfty$. We easily see that $\dis \Ztapbar \in \Ltwo, \frac{1}{\Zap^2}\pap\Ztapbar \in \Ltwo$ and $\dis \pap\frac{1}{\Zap} \in \Ltwo$. We also have from \propref{prop:Leibniz}
\begin{align*}
\norm[\Hhalf]{\pap\frac{1}{\Zap}} \lesssim \norm[\infty]{\frac{1}{\Zap}}^2\norm[\Hhalf]{\pap\Zap} + \norm[\infty]{\frac{1}{\Zap}}\norm[2]{\pap\frac{1}{\Zap}}\norm[2]{\pap\Zap}
\end{align*}
and hence again using \propref{prop:Leibniz} we have
\begin{align*}
\norm[\Hhalf]{\frac{1}{\Zap}\pap\frac{1}{\Zap}} \lesssim \norm[2]{\pap\frac{1}{\Zap}}^2 + \norm[\infty]{\frac{1}{\Zap}}\norm[\Hhalf]{\pap\frac{1}{\Zap}}
\end{align*}
\item Now assume that $\sigma\geq 0$ and let 
\begin{align*}
D = \norm[2]{\frac{1}{\Zap} - 1}(0) + \norm[2]{\Zt}(0) + \sup_{t\in [0,T]}\Ecalsigma(t) 
\end{align*}
Define
\begin{align*}
M = D + \norm[\infty]{\Zap}(0) + T + \sigma + 1 
\end{align*}
In the following $C(M)$ will denote a constant depending only on $M$.  As $\Ecalsigma(0) \leq D$ we see that $ \norm[2]{\pap\frac{1}{\Zap}}(0) \leq C(M)$ and hence $\norm[\infty]{\frac{1}{\Zap}}(0) \leq C(M)$. 

Now the evolution equation \eqref{eq:systemone} gives us
\begin{align*}
(\pt + \bvar\pap)\Zap = \Ztap -\bvarap \Zap = \brac{\Dap\Zt - \bvarap}\Zap
\end{align*}
Hence for all $0\leq t \leq T$ we have the estimate
\begin{align*}
\norm[\infty]{\Zap}(t) & \leq \norm[\infty]{\Zap}(0) \exp\cbrac{{\int_0^t \brac{\norm[\infty]{\Dap\Zt}(s) + \norm[\infty]{\bvarap}(s)} \diff s}}
\end{align*}
As $\norm[\infty]{\Dap\Zt}$ and $\norm[\infty]{\bvarap}$ are controlled by $\Ecalsigma$, we see that $\sup_{t\in [0,T]}\norm[\infty]{\Zap}(t) \leq C(M)$. By a similar argument we also obtain $\sup_{t\in [0,T]}\norm[\infty]{\frac{1}{\Zap}}(t)\leq C(M)$

We now control $\norm[2]{\frac{1}{\Zap} - 1}(t)$ and $\norm[2]{\Zt}(t)$. To do this define
\begin{align*}
f(t) = \norm[2]{\frac{1}{\Zap} - 1}^2(t) + \norm[2]{\Zt}^2(t) + 1
\end{align*}
Observe that $f(0) \leq C(M)$. We first see some of the quantities controlled by $f(t)$. 
\begin{enumerate}
\item Using the formula \eqref{form:bvar} we see that
\begin{align*}
\norm[2]{\bvar} \lesssim \norm[2]{\Zt}\norm[\infty]{\frac{1}{\Zap}} \lesssim C(M)f^\half 
\end{align*}
hence using the estimate $\norm[\Hhalf]{\bvarap} \leq C(\Esigma)$ from \secref{sec:quantEsigma} and \propref{prop:equivEsigma} we have
\begin{align*}
\norm[\infty]{\bvar} + \norm[2]{\bvarap} \lesssim \norm[2]{\bvar} + \norm[\Hhalf]{\bvarap} \lesssim C(M)f^\half + C(M) \lesssim C(M)f^\half
\end{align*}
\item Using the formula  $\Aone = 1 - \Imag\sqbrac{\Zt,\Hil}\Ztapbar $ from \eqref{eq:mainvariables} we see that
\begin{align*}
\norm[2]{\Aone - 1} \lesssim \norm[\infty]{\Zt}\norm[2]{\Ztap} \lesssim \brac{\norm[2]{\Zt} + \norm[2]{\Ztap}}\norm[2]{\Ztap} \lesssim C(M)f^\half + C(M) \lesssim C(M)f^\half
\end{align*}
\item Using \eqref{form:DtoneoverZap} we see that
\begin{align*}
\norm[2]{(\pt + \bvar\pap)\frac{1}{\Zap}} & \lesssim \norm[2]{\frac{1}{\Zap}(\bvarap - \Dap\Zt)} \\
& \lesssim \norm[\infty]{\frac{1}{\Zap}}\norm[2]{\bvarap} + \norm[\infty]{\frac{1}{\Zap}}^2\norm[2]{\Ztap} \\
& \lesssim C(M)f^\half + C(M) \\
& \lesssim C(M)f^\half
\end{align*}
\item From \eqref{form:Zttbar} we see that 
\begin{align*}
\norm[2]{(\pt + \bvar\pap)\Zt} & \lesssim \norm[2]{i-i\frac{\Aone}{\Zap} + \sigma\Dap\Th} \\
& \lesssim \norm[2]{\frac{1}{\Zap} - 1} + \norm[\infty]{\frac{1}{\Zap}}\norm[2]{\Aone - 1} + \sigma^\onebythree\norm[\infty]{\frac{1}{\Zap}}\norm[2]{\sigma^\twobythree\pap\Th}\\
& \lesssim f^\half + C(M)f^\half + C(M)\sigma^\onebythree \\
& \lesssim C(M)f^\half
\end{align*}
where we have used the fact that $\Ecalsigma$ controls $\sigma^\twobythree\pap\Th \in \Ltwo$ from \secref{sec:quantEsigma}. 
\end{enumerate}
Hence we see that
\begin{align*}
\pt f & \lesssim f^\half\cbrac{\norm[2]{\pt\frac{1}{\Zap}}  + \norm[2]{\pt\Zt}  } \\
& \lesssim f^\half\cbrac{\norm[2]{(\pt + \bvar\pap)\frac{1}{\Zap}} + \norm[\infty]{\bvar}\norm[2]{\pap\frac{1}{\Zap}} + \norm[2]{(\pt + \bvar\pap)\Zt} + \norm[\infty]{\bvar}\norm[2]{\Ztap} } \\
& \lesssim C(M)f
\end{align*}
Hence $f(t)$ remains bounded on $[0,T]$ and we have
\begin{align*}
\sup_{t\in[0,T]}\cbrac{\norm[2]{\frac{1}{\Zap} - 1}(t) + \norm[2]{\Zt}(t)} \leq C(M)
\end{align*}

Now using  $\sup_{t\in [0,T]}\norm[2]{\frac{1}{\Zap} - 1} \leq C(M)$ and the fact that $\sup_{t\in [0,T]} \norm[\infty]{\Zap}(t) \leq C(M)$ we see that $\sup_{t\in [0,T]}\norm[2]{\Zap - 1} \leq C(M)$.  Now using part 1 and 2 of this lemma we easily obtain 
\begin{align*}
\sup_{t\in[0,T]} \cbrac{\norm[H^{3.5}]{\Zap -1}(t) + \norm[H^{3.5}]{\frac{1}{\Zap} - 1}(t) + \norm[H^3]{\Zt}(t)} \leq C\brac{M + \frac{1}{\sigma}}
\end{align*}
for $\sigma>0$ and 
\begin{align*}
 \sup_{t\in[0,T]} \cbrac{\norm[H^{1.5}]{\Zap -1}(t) + \norm[H^{1.5}]{\frac{1}{\Zap} - 1}(t) + \norm[H^2]{\Zt}(t)} \leq C(M)
\end{align*}
for $\sigma\geq 0$ thereby proving the lemma
\end{enumerate}
\end{proof}

\section{Existence in Sobolev spaces}\label{sec:existence}

\smallskip

In this section we prove the existence of solutions in Sobolev spaces for $\sigma>0$ with the results being \thmref{thm:existencesobolev} and also \corref{cor:existencesobolev}. This existence result is then used to complete the proof of \thmref{thm:existence} in \secref{sec:proof}. The existence proof is standard and follows the general approach of \cite{Am03, CaCoFeGaGo12}. Even though \cite{CaCoFeGaGo12} already has an existence result in conformal mapping coordinates for the water wave equation with surface tension, we require a much stronger result than the one provided there. First we need to have lower regularity on the initial data in Sobolev spaces, we need lesser restrictions on the lower order terms  and we also need a blow up criterion not depending on the chord arc constant of the interface. This existence result is of independent interest as we do not use the vorticity formulation using the Birkhoff-Rott integral as was done in \cite{Am03, CaCoFeGaGo12}. We do not assume that the interface is a graph nor that it is non-self intersecting (as was explained in \secref{sec:systemone}). In this section we fix $\sigma>0$ and constants appearing in the computations may depend on $\sigma$.

\subsection{\texorpdfstring{A priori estimates for exact solutions \nopunct}{}}\hspace*{\fill} \medskip

In order to prove existence for system \eqref{eq:systemone}, it is more convenient to work with an equivalent system in the variables $(\g, \vvar)$ 
\begin{align}\label{eq:systemtwo}
\begin{split}
\cvar & = e^{-i\Hil \g} \\
\w & = e^{i\g}\\
\bstar & = 2i\Hil(\cvar\vvar) + i \sqbrac{\cvar^2, \Hil} \brac{\frac{\vvar}{\cvar}} \\
\avar & = i\cvar\Hil\brac{\frac{\vvar}{\cvar} } \\
\dvar & = -i e^{i\g} \cvar (\Id - \Hil) \brac{\frac{\vvar}{\cvar}} \\
\Aonestar & =  1 - \Imag \sqbrac{\dvar,\Hil} \pap\bar\dvar \\
\etwo & = \Real(\w) - \Aonestar \cvar + \sigma \Imag\cbrac{\sqbrac{\cvar,\Hil} \pap(\Id + \Hil) (\cvar \pap \g) }\\
\pt\g & = -(\cvar\pap)\vvar + \avar(\cvar\pap)\g -\bstar\pap\g \\
\pt\vvar & = -i\sigma \Hil (\cvar\pap)^2 \g - \avar (\cvar\pap)\vvar  -\bstar\pap\vvar + \avar^2 (\cvar\pap)\g   + \etwo
\end{split}
\end{align}
To get system \eqref{eq:systemtwo} from system \eqref{eq:systemone} we use the following transformation
\begin{align}\label{eq:systemonetwo}
\begin{split}
\g & = \Imag(\log( \Zap)) \\
\vvar & = \Imag \brac{\frac{\Zap}{\Zapabs}\Ztbar}
\end{split}
\end{align}
and the following to get system \eqref{eq:systemone} from system \eqref{eq:systemtwo}
\begin{align}\label{eq:systemtwoone}
\begin{split}
\Zap & = e^{i (\Id + \Hil)\g} \\
\Zt & = \dvar
\end{split}
\end{align}
Let us now prove that these two systems are equivalent. 
\begin{lem}\label{lem:systemequiv}
Let $s\geq 3$ and $T\geq 0$. Then $(\Z, \Zt)(t)$ solves \eqref{eq:systemone} with $(\Zap -1, \frac{1}{\Zap} -1,\Zt) \in \Linfty([0,T], H^{s+\half}(\Rsp)\times H^{s+\half}(\Rsp)\times H^s(\Rsp))$ if and only if $(\g,\vvar)(t)$ solves \eqref{eq:systemtwo} along with $(\g,\vvar) \in \Linfty([0,T], H^{s+\half}(\Rsp)\times H^s(\Rsp))$, where the transformations between them are given by \eqref{eq:systemonetwo} and  \eqref{eq:systemtwoone}
\end{lem}
\begin{proof}
Step 1: We first assume that $(\Z, \Zt)(t)$ solves \eqref{eq:systemone} and show that $(\g,\vvar)(t)$ solves  \eqref{eq:systemtwo}.
\begin{enumerate}
\item Now if $(\Zap -1, \frac{1}{\Zap} -1,\Zt) \in \Linfty([0,T], H^{s+\half}\times H^{s+\half}\times H^s)$ for $s\geq 3$ then for any $t\in[0,T]$ we see that $\norm[\infty]{\Zap}(t) + \norm[\infty]{\frac{1}{\Zap}}(t) + \norm[2]{\abs{\Zap} - 1}(t) \leq M < \infty $
for some $M>0$. Now as $\Zap -1 \in H^{s+\half}(\Rsp)$ we observe that $\Psizp$ extends continuously to $\Pminusbar$ and hence $\log(\Psizp)$ also extends continuously to the boundary. Hence it makes sense to talk about the function $\log(\Zap)$.  
Observe that if $C_1>0$ then
\begin{align*}
c_1 \abs{z} \leq \abs{e^z -1} \leq c_2 \abs{z} \qq \tx{ for all } z\in \Rsp, \abs{z} \leq C_1
\end{align*}
for some $c_1,c_2 >0$ depending only on $C_1$. Now as $\abs{\Zap} = e^{\Real(\log(\Zap))}$ we see that $\Real(\log(\Zap)) \in \Ltwo$. Hence we see that $\Imag(\log(\Zap)) \in \Ltwo$ and hence $\g \in \Ltwo$ and $\Zap = e^{i(\Id + \Hil)\g}$. Now using \eqref{eq:systemonetwo} and the formula $\pap g = \Imag\brac{\frac{1}{\Zap}\pap\Zap}$ we see that $(\g, \vvar) \in \Linfty([0,T], H^{s+\half}(\Rsp)\times H^s(\Rsp))$.

\item  Observe from \eqref{eq:systemonetwo} that $\cvar  = e^{-i\Hil \g} = \frac{1}{\Zapabs}$ and $\w =  e^{i\g} = \frac{\Zap}{\Zapabs}$. We also have
\begin{align*}
\avar =  i\cvar\Hil\brac{\frac{\vvar}{\cvar}} =  i\cvar\Hil\brac{\Imag\brac{\Ztbar\Zap}} = c\Real\brac{\Hil\brac{\Ztbar\Zap}} =  \Real \brac{\w\Ztbar}
\end{align*}
\item Observe 
\begin{align*}
\bstar  = 2i\Hil(\cvar\vvar) + i \sqbrac{c^2, \Hil} \brac{\frac{\vvar}{\cvar}} =  i\Hil(\cvar\vvar) + i\cvar^2\Hil \brac{\frac{\vvar}{\cvar}} =  i\Hil(\cvar\vvar) + \cvar\avar
\end{align*}
Hence we see that
\begin{align*}
\bstar =  i\Hil(\cvar\vvar) + \cvar\avar =  i\Hil\brac{\Imag\brac{\frac{\Ztbar}{\Zapbar}}} + \Real\brac{\frac{\Ztbar}{\Zapbar}} =  \Real (\Id + \Hil) \brac{\frac{\Ztbar}{\Zapbar}}  = \bvar
\end{align*}
\item Now observe that
\begin{align*}
\dvarbar  =  i e^{-i\g} \cvar (\Id + \Hil) \brac{\frac{\vvar}{\cvar}} = \frac{i}{\Zap} (\Id + \Hil) \Imag\brac{\Ztbar\Zap}  = \Ztbar
\end{align*}
\item Hence we also have $\Aonestar = \Aone$
\item We can now see that
\begin{align*}
(\pt + \bstar\pap)\g &= \Imag\brac{(\pt + \bstar\pap)\log(\Zap)} \\
&= \Imag\brac{\frac{1}{\Zap}(\pt + \bvar\pap)\Zap} \\
&= \Imag\brac{\frac{\Ztap}{\Zap}} \\
& = -\Imag\brac{\cvar\w\pap\Ztbar} \\
&= \Imag\brac{\Ztbar\cvar\pap(\w)} - \Imag\brac{c\pap(\w\Ztbar)} \\
&= \Imag\brac{i\brac{c\pap\g} \w\Ztbar} -\cvar\pap\vvar \\
& =  -\cvar\pap\vvar + \avar(\cvar\pap)\g
\end{align*}

\item We see that
\begin{align*}
& (\pt + \bstar\pap) \vvar \\
& = \Imag\cbrac{(\pt + \bstar\pap)\brac{\w\Ztbar}} \\
& = \Imag\cbrac{i\w\Ztbar (\pt + \bstar\pap)\g} + \Imag\cbrac{\w(\pt + \bstar\pap)\Ztbar} \\
& = \avar\cbrac{ -(\cvar\pap)\vvar + \avar(\cvar\pap)\g} + \Imag\cbrac{\Hil\brac{\w(\pt + \bstar\pap)\Ztbar}} +  \Imag\cbrac{(\Id - \Hil)\brac{\w(\pt + \bstar\pap)\Ztbar}} \\
\end{align*}
Observe that
\begin{align*}
 \Imag\cbrac{\Hil\brac{\w(\pt + \bstar\pap)\Ztbar}} &=\Imag\cbrac{\Hil\brac{\w(\pt + \bvar\pap)\Ztbar -i\w}} + \Imag\cbrac{\Hil(i\w) } \\
 &= -i\Hil\Real\cbrac{\w(\pt + \bvar\pap)\Ztbar -i\w} + \Real(\Hil \w) \\
 &= -i\sigma\Hil(\cvar\pap)^2\g  + \Real(\Hil \w)
\end{align*}
and we also have
\begin{align*}
& \Imag\cbrac{(\Id - \Hil)\brac{\w(\pt + \bstar\pap)\Ztbar}} \\
&= \Imag (\Id - \Hil)\cbrac{i\w - \frac{\Aone}{\Zapabs} + \frac{\sigma}{\Zapabs}\pap(\Id + \Hil)\brac{\frac{1}{\Zapabs}\pap\g}} \\
&= \Real(\Id-\Hil)(\w) - \frac{\Aone}{\Zapabs} +\sigma\Imag\cbrac{\sqbrac{\frac{1}{\Zapabs},\Hil}\pap(\Id + \Hil)\brac{\frac{1}{\Zapabs}\pap\g}}
\end{align*}
We can rewrite $\etwo$ as
\begin{align*}
\etwo & = \Real(\w) - \Aonestar \cvar + \sigma \Imag\cbrac{\sqbrac{\cvar,\Hil} \pap(\Id + \Hil) (\cvar \pap \g) }\\
& = \Real\brac{\w} - \frac{\Aone}{\Zapabs} + \sigma\Imag\cbrac{\sqbrac{\frac{1}{\Zapabs},\Hil} \pap(\Id + \Hil) \brac{\frac{1}{\Zapabs}\pap \g} }
\end{align*}
Hence combining these we get 
\begin{align*}
(\pt + \bstar\pap) \vvar & =  -i\sigma \Hil (\cvar\pap)^2 \g - \avar (\cvar\pap)\vvar + \avar^2 (\cvar\pap)\g + \etwo
\end{align*}

\end{enumerate}
Step 2:  We now assume that $(\g,\vvar)(t)$ solves \eqref{eq:systemtwo} and show that $(\Z, \Zt)(t)$ solves  \eqref{eq:systemone}. 
\begin{enumerate}
\item Observe that if $C_1>0$ then 
\begin{align*}
 \abs{e^z -1} \leq c_2 \abs{z} \qq \tx{ for all } z\in \Csp, \abs{z} \leq C_1
\end{align*}
where $c_2$ depends only on $C_1$. Hence via a similar calculation from step 1 and using \eqref{eq:systemtwoone} we see that if $(\g,\vvar) \in \Linfty([0,T], H^{s+\half}(\Rsp)\times H^s(\Rsp))$ then $(\Zap -1, \frac{1}{\Zap} -1,\Zt) \in \Linfty([0,T], H^{s+\half}(\Rsp)\times H^{s+\half}(\Rsp)\times H^s(\Rsp))$. We also observe that in this case we have $\log(\Psizp) = K_{-y} \conv (i(\Id + \Hil)\g)$ and hence $\log(\Psizp)$ is well defined. Hence we easily obtain 
\begin{align*}
\lim_{c \to \infty} \sup_{\abs{\zp}\geq c}\cbrac{\abs{\Psizp(\zp) - 1} + \abs{\U(\zp)}}  = 0
\qquad \tx{ and } \quad \Psizp(\zp) \neq 0 \quad \tx{ for all } \zp \in \Pminus 
\end{align*}
\item We again have $\cvar  = e^{-i\Hil \g} = \frac{1}{\Zapabs}$  and $\w =  e^{i\g} = \frac{\Zap}{\Zapabs}$.  Also
\begin{align*}
\Ztbar\Zap  = \dvarbar\Zap = \cbrac{i e^{-i\g}\cvar(\Id + \Hil)\brac{\frac{\vvar}{\cvar}}}\Zap = i (\Id + \Hil)(\vvar \Zapabs)
\end{align*}
Hence by taking imaginary parts we get
\begin{align*}
\Imag(\Ztbar\Zap) = \vvar\Zapabs
\end{align*}
and hence we have $\vvar = \Imag(\w\Ztbar)$. Also observe that $\dvarbar = \frac{i}{\Zap}(\Id + \Hil)\brac{\vvar\Zapabs}$ and hence $\Ztbar$ is the boundary value of a holomorphic function. 
\item Hence now from step 1 we automatically have $\avar = \Real(\w\Ztbar)$, $\bstar = \bvar$ and  $\Aonestar = \Aone$.
\item Observe that
\begin{align*}
\bstar = \bvar = \Real(\Id - \Hil)\brac{\frac{\Zt}{\Zap}} = \Pa\brac{\frac{\Zt}{\Zap}} + \Ph\brac{\frac{\Ztbar}{\Zapbar}}
\end{align*}
and hence we have
\begin{align*}
i\sqbrac{\bvar,\Hil}\g_\ap & = -i\sqbrac{\bvar,\Hil}\Imag\brac{\Zap\pap\frac{1}{\Zap}} \\
& = -\Imag\cbrac{i\sqbrac{\bvar,\Hil}\brac{\Zap\pap\frac{1}{\Zap}}} \\
& = -\Imag\cbrac{i\sqbrac{\frac{\Zt}{\Zap},\Hil}\brac{\Zap\pap\frac{1}{\Zap}}} \\
& = -\Imag\cbrac{i(\Id - \Hil)\brac{\Zt\pap\frac{1}{\Zap}}} \\
& = -\Real(\Id - \Hil)\brac{\Zt\pap\frac{1}{\Zap}}
\end{align*}
\item Now from \eqref{eq:systemtwo} we have
\begin{align*}
(\pt + \bstar\pap)\g =  -\cvar\pap\vvar + \avar(\cvar\pap)\g
\end{align*}
Now by again using the computation in step 1 we see that 
\begin{align*}
(\pt + \bstar\pap)\g =  -\cvar\pap\vvar + \avar(\cvar\pap)\g =  \Imag\brac{\frac{\Ztap}{\Zap}}
\end{align*}
and we have
\begin{align*}
(\pt + \bvar\pap)\Zap & = (\pt + \bvar\pap)e^{i(\Id + \Hil)\g} \\
& = \Zap \cbrac{i\sqbrac{\bvar,\Hil}\g_\ap + i(\Id + \Hil)(\pt + \bvar\pap)\g}\\
& = \Zap\cbrac{-\Real(\Id - \Hil)\brac{\Zt\pap\frac{1}{\Zap}} + i(\Id + \Hil)\Imag\brac{\frac{\Ztap}{\Zap}}} \\
& = \Zap\cbrac{-\Real(\Id - \Hil)\brac{\Zt\pap\frac{1}{\Zap}} + \frac{\Ztap}{\Zap} - \Real(\Id - \Hil)\brac{\frac{\Ztap}{\Zap}}} \\
& = \Ztap - \Zap\bvarap
\end{align*}
\item Observe that if $f$ is function satisfying $\Pa f = 0$, then we have
\begin{align*}
\Pa\cbrac{(\pt + \bvar\pap)f} = \half\sqbrac{\bvar,\Hil}\pap f = \half \sqbrac{\frac{\Zt}{\Zap},\Hil}\pap f = \Pa \brac{\frac{\Zt}{\Zap}\pap f}
\end{align*}
Hence we see that
\begin{align*}
\Pa\cbrac{\Zap (\pt + \bvar\pap)\Ztbar} & = \Pa\cbrac{\Zap \Pa\cbrac{(\pt + \bvar\pap)\Ztbar}} \\
& = \Pa\cbrac{\Zap \Pa\cbrac{\frac{\Zt}{\Zap}\pap\Ztbar}} \\
& = \Pa\cbrac{\Zt\Ztbarap}
\end{align*}
\item Now as $\Imag(\w\Ztbar) = \vvar$, apply $\pt + \bvar \pap$ to this equation to get
\begin{align*}
\Imag\brac{\w (\pt + \bvar\pap)\Ztbar } + \Imag\brac{i\cbrac{(\pt + \bvar\pap)g}\w\Ztbar  } = -i\sigma \Hil (\cvar\pap)^2 \g - \avar (\cvar\pap)\vvar  + \avar^2 (\cvar\pap)\g   + \etwo
\end{align*}
But we know that $(\pt + \bvar\pap)g =  -(\cvar\pap)\vvar + \avar(\cvar\pap)\g   $ and that $\Real(\w\Ztbar) = \avar$. Hence
\begin{align*}
\Imag\brac{\w (\pt + \bvar\pap)\Ztbar } = -i\sigma\Hil(\cvar\pap)^2\g +  \Real(\w) - \Aone \cvar + \sigma \Imag\cbrac{\sqbrac{\cvar,\Hil} \pap(\Id + \Hil) (\cvar \pap \g) }
\end{align*}
Now observe that
\begin{align*}
&\Hil(\cvar\pap)^2\g + i\Imag\cbrac{\sqbrac{\cvar,\Hil} \pap(\Id + \Hil) (\cvar \pap \g) } \\
& = \Hil(\cvar\pap)^2\g + i\Imag(\Id - \Hil)\cbrac{(\cvar\pap)(\Id + \Hil)(\cvar\pap\g)} \\
& = \Hil(\cvar\pap)^2\g -\Hil (\cvar\pap)^2\g + (\cvar\pap)\Hil(\cvar\pap)g \\
& = \nobrac{\frac{1}{\Zapabs}\pap}\Hil\brac{\frac{1}{\Zapabs}\pap \g}
\end{align*}
Hence by multiplying both sides by $\Zapabs = \frac{1}{\cvar}$ we get
\begin{align*}
\Imag(\Zap (\pt + \bvar\pap)\Ztbar ) = \Real(\Zap) - \Aone -i\sigma\pap\Hil\brac{\frac{1}{\Zapabs}\pap \g}
\end{align*}
Now apply $i(\Id + \Hil)$ to both sides
\begin{align*}
& \Zap (\pt + \bvar\pap)\Ztbar  - \Real(\Id - \Hil)\cbrac{\Zap (\pt + \bvar\pap)\Ztbar } \\
&= i(\Id + \Hil)\cbrac{ \Real(\Zap) - \Aone} + \sigma\pap(\Id + \Hil)\brac{\frac{1}{\Zapabs}\pap \g}
\end{align*}
Now observe that
\begin{align*}
i(\Id + \Hil)\Real(\Zap) = i(\Id + \Hil)\Real(\Zap - 1) + i(\Id + \Hil) 1 = i(\Zap -1) + i = i\Zap
\end{align*}
and we also have
\begin{align*}
-i(\Id + \Hil)\Aone & = -i\Aone -i\Hil\Aone \\
& = -i\Aone -i\Hil\cbrac{1-\Imag(\Id -\Hil)(\Zt\Ztapbar)} \\
& = -i\Aone - \Real(\Id - \Hil)(\Zt\Ztapbar)
\end{align*}
Hence we have
\begin{align*}
& \Zap (\pt + \bvar\pap)\Ztbar  - \Real(\Id - \Hil)\cbrac{\Zap (\pt + \bvar\pap)\Ztbar } \\
& = i\Zap -i\Aone - \Real(\Id - \Hil)(\Zt\Ztapbar) + \sigma\pap(\Id + \Hil)\brac{\frac{1}{\Zapabs}\pap \g}
\end{align*}
But we have already shown that $(\Id - \Hil)\cbrac{\Zap (\pt + \bvar\pap)\Ztbar } = (\Id - \Hil)(\Zt\Ztapbar) $. Hence
\begin{align*}
 \Zap (\pt + \bvar\pap)\Ztbar =  i\Zap -i\Aone + \sigma\pap(\Id + \Hil)\brac{\frac{1}{\Zapabs}\pap \g}
\end{align*}
Now dividing by $\Zap$ we finally get
\begin{align*}
 (\pt + \bvar\pap)\Ztbar = i -i\frac{\Aone}{\Zap} + \frac{\sigma}{\Zap}\pap(\Id + \Hil)\brac{\frac{1}{\Zapabs}\pap \g}
\end{align*}
\end{enumerate}
\end{proof}

As the proof establishes that $\bstar = \bvar$ and $\Aonestar = \Aone$, we will from now on use the variables $\bvar,\Aone$ instead of $\bstar, \Aonestar$ in the system \eqref{eq:systemtwo}. We now prove a priori estimates for \eqref{eq:systemtwo}. Let $N\geq 0$ and define the energy 
\begin{align}\label{eq:energy}
\begin{split}
\Ecalthreefive & = \frac{1}{2}\norm[H^{2.5}]{\g}^2 + \half\norm[H^2]{\vvar}^2 \\
\Ecalfourfivei &= \half\norm[2]{\frac{1}{\cvar^\half}\cbrac{(\cvar\pap)^{3+i}\vvar - \avar(\cvar\pap)^{3+i}\g }}^2 + \frac{\sigma}{2}\norm[\Hhalf]{(\cvar\pap)^{3+i}\g}^2 \\
\Ecal &= \Ecalthreefive + \sum_{i=0}^N \Ecalfourfivei
\end{split}
\end{align}
We also define $K(t) = \norm[H^{2.5}]{\g(\cdot,t)} + \norm[H^2]{\vvar(\cdot,t)}$. Then we have
\begin{prop}\label{prop:apriori}
Fix $N\geq 0$ and let $(\g,\vvar)(t)$ be a smooth solution to \eqref{eq:systemtwo} in the time interval $[0,T]$ with  $(\g,\vvar) \in C([0,T], H^{s+\half}\times H^s)$ for all $s\geq 0$. Then there exists a polynomial $C = C(t)$ with non-negative coefficients depending only on $\sigma$ so that for any $t\in[0,T)$ we have
\begin{align*}
\frac{\diff \Ecal (t)}{\diff t} \leq C(K(t))\Ecal(t)
\end{align*}
\end{prop}
\begin{proof}
The proof is divided into 3 steps.  We will freely use \lemref{lem:interpolationmult} to simplify the computations. 

\textbf{Step 1:} We first find quantities which can be controlled by the energy. We will use the notation $C(K) = C(K(t))$. Now 
\begin{enumerate}
\item From the definition of $K$ and $\cvar$, we have
\begin{align*}
\norm[\infty]{\g} + \norm[\infty]{\cvar} + \norm[\infty]{\frac{1}{\cvar}} \leq C(K)
\end{align*}
Hence from the definitions we easily see that 
\begin{align*}
\norm[H^{1.5}]{\cvar_\ap} +  \norm[H^1]{\pt \g}  +  \norm[H^1]{\pt \cvar} + \norm[H^2]{\bvar} + \norm[H^2]{\avar} + \norm[H^\half]{\pt\bvar} + \norm[H^\half]{\pt \avar} \leq C(K)
\end{align*}

\item Observe that $\cvar_\ap = (-i\Hil\g_\ap) \cvar$ and hence from $\Ecalthreefive$ we see that
\begin{align*}
\norm[H^{2.5}]{\g} + \norm[H^{1.5}]{\cvar_\ap} + \norm[H^{1.5}]{\pap\brac{\frac{1}{\cvar}}} \leq C(K)\Ecal^\half
\end{align*}
and hence
\begin{align*}
\norm[H^\half]{(\cvar\pap)^i \g} \leq C(K)\Ecal^\half \qq \tx{ for } i=1,2
\end{align*}
For $i\geq 3$
\begin{align*}
\norm[2]{\frac{1}{\cvar^\half}(\cvar\pap)^i \g}^2 = \int \frac{1}{\cvar} \brac{(\cvar\pap)^i \g} \brac{(\cvar\pap)^i \g} \diff \ap \lesssim \norm[\Hhalf]{(\cvar\pap)^{i-1}\g} \norm[\Hhalf]{(\cvar\pap)^{i}\g}
\end{align*}
Hence we see that $\norm[2]{(\cvar\pap)^i \g} \leq C(K)\Ecal^\half$ for $1\leq i \leq N+3$. Now using \corref{cor:commutatoreasy} we get $\norm[2]{\cvar^i\pap^i \g} \leq C(K)\Ecal^\half$ for $1\leq i \leq N+3$. Hence dividing by $\cvar^i$ and using $\cvar_\ap = (-i\Hil\g_\ap) \cvar$ we get $\norm[H^{N+3}]{\g} + \norm[H^{N+2}]{\cvar_\ap} \leq C(K)\Ecal^\half$. Now by using the fact that $\norm[\Hhalf]{(\cvar\pap)^{N+3}\g} \leq C(K) \Ecal^\half$ and by using  \corref{cor:commutatoreasy} we get $\norm[\Hhalf]{\cvar^{N+3}\pap^{N+3}\g} \leq C(K) \Ecal^\half$. Hence we have 
\begin{align*}
\norm[H^{N+3.5}]{\g} \leq C(K)\Ecal^\half
\end{align*}
Now observe that for $z \in \Csp$ we have $\abs{e^z - 1} \leq C_2\abs{z}$ for all $\abs{z} \leq C_1$, where $C_2$ depends only on $C_1$. Hence we have
\begin{align*}
\abs{\cvar - 1} + \abs{\w - 1} = \abs{e^{-i\Hil\g} - 1} + \abs{e^{i\g} - 1} \leq C(K)\brac{\abs{\Hil\g} + \abs{\g}}
\end{align*}
Using this and the fact that $\w_\ap = \pap(e^{i\g}) = i\w\g_\ap$  and $\cvar_\ap = (-i\Hil\g_\ap) \cvar$ we have
\begin{align*}
\norm[H^{N+3.5}]{\cvar -1} + \norm[H^{N+3.5}]{\frac{1}{\cvar} -1} +  \norm[H^{N+3.5}]{\w - 1} \leq C(K)\Ecal^\half
\end{align*}
\item From the definition of $\avar$ we have $\norm[\infty]{\avar} \leq C(K)$. Now by using the fact that $\norm[2]{(\cvar\pap)^{3+i}\g} \leq C(K)\Ecal^\half$ for all $0\leq i\leq N$, using the energy $\Ecalfourfivei$ we now have $\norm[2]{(\cvar\pap)^{3+i}\vvar} \leq C(K)\Ecal^\half$ for all $0\leq i\leq N$. Hence by using \corref{cor:commutatoreasy} repeatedly we get 
\begin{align*}
\norm[H^{N+3}]{\vvar} \leq C(K)\Ecal^\half
\end{align*}
Note that with these estimates one can also easily get the estimate $\norm[\Hhalf]{(\cvar\pap)^{2+i}\vvar} \leq C(K)\Ecal^\half$ for all $0\leq i\leq N$. Now using the definition of $\avar, \dvar, \bvar, \Aone$ and $\etwo$ we easily get using \propref{prop:commutator}
\begin{align*}
\norm[H^{N+3}]{\avar}  + \norm[H^{N+3}]{\dvar} + \norm[H^{N+3}]{\bvar} + \norm[H^{N+3}]{\Aone - 1} + \norm[H^{N+3}]{\etwo}  \leq C(K)\Ecal^\half
\end{align*}
Now by the equations we get 
\begin{align*}
\norm[H^{N+2}]{\pt \g} + \norm[H^{N+1.5}]{\pt \vvar}  \leq C(K)\Ecal^\half
\end{align*}
Also observe that $\pt \cvar = \cbrac{-i\Hil(\pt\g)}\cvar$. Hence we now get $\norm[H^{N+2}]{\pt\cvar} \leq C(K)\Ecal^\half $.
\end{enumerate}
\textbf{Step 2:} We now establish some identities which will be useful to prove the energy estimate. We define $\Dt = \pt + \bvar\pap$. We define the following notation:  If $a,b:\Rsp\times[0,T] \to \Csp$ are functions we write $a \approx_{\Ltwo} b$ if there exists a polynomial $C(t)$  with non-negative coefficients depending only on $\sigma$ such that $\norm[2]{a-b} \leq C(K)\Ecal(t)^\half $. Observe that $\approx_{\Ltwo}$ is an equivalence relation.  
\begin{enumerate}
\item Let us compute $\sqbrac{\Dt,\cvar\pap}$. We see that
\begin{align*}
\sqbrac{\Dt,\cvar\pap} = \sqbrac{\pt + \bvar\pap,\cvar\pap} = \brac{\frac{\cvar_t + \bvar\cvar_\ap - \cvar\bvar_\ap}{\cvar}}\cvar\pap
\end{align*}
Now using the above formula for $\cvar_t$ and the definition of $\bstar$ we have
\begin{align}\label{form:errorone}
\begin{split}
& \brac{\frac{\cvar_t + \bvar\cvar_\ap - \cvar\bvar_\ap}{\cvar}} \\
& =  \nobrac{-i\Hil\cbrac{ -(\cvar\pap)\vvar + \avar(\cvar\pap)\g -\bvar\pap\g}} + \bvar\frac{\cvar_\ap}{\cvar} - \cbrac{2i\Hil(\cvar\vvar_\ap + \cvar_\ap \vvar) + i\pap\sqbrac{\cvar^2,\Hil}\brac{\frac{\vvar}{\cvar}}} \\
& =  -i\Hil(\cvar\vvar_\ap)  + \cbrac{-i\Hil\cbrac{   \avar(\cvar\pap)\g -\bvar\pap\g} + \bvar\frac{\cvar_\ap}{\cvar} - \cbrac{2i\Hil( \cvar_\ap \vvar) + i\pap\sqbrac{\cvar^2,\Hil}\brac{\frac{\vvar}{\cvar}}}} \\
& = -i\Hil(\cvar\vvar_\ap) + \errorone
\end{split}
\end{align}
where $\errorone$ is defined as the term in the bracket and we observe that $\norm[H^{N+2.5}]{\errorone} \leq C(K)\Ecal^\half$. Hence we have
\begin{align*}
\sqbrac{\Dt,\cvar\pap} = \cbrac{-i\Hil(\cvar\vvar_\ap) + \errorone}\cvar\pap 
\end{align*}
\item Observe from the definition of $\avar$
\begin{align*}
\avar = i\cvar\Hil\brac{\frac{\vvar}{\cvar}} = i\Hil\vvar + i\sqbrac{\cvar,\Hil}\brac{\frac{\vvar}{\cvar}}
\end{align*}
Hence we see that 
\begin{align*}
\cvar\avar_\ap = i\Hil(\cvar\vvar_\ap) + i\sqbrac{\cvar,\Hil}\vvar_\ap + i\cvar\pap\sqbrac{\cvar,\Hil}\brac{\frac{\vvar}{\cvar}}
\end{align*}
and hence we have $\norm[H^{N+2.5}]{\cvar\avar_\ap - i\Hil(\cvar\vvar_\ap)} \leq C(K)\Ecal^\half$. Now 
\begin{align}\label{form:errortwo}
\begin{split}
& \Dt(\cvar\pap)\g\\
& = \sqbrac{\Dt,\cvar\pap}\g + \cvar\pap\Dt\g \\
& = \cbrac{-i\Hil(\cvar\vvar_\ap) + \errorone}\cvar\g_\ap + \cvar\pap\cbrac{-\cvar\vvar_\ap + \avar\cvar\g_\ap} \\
& = -(\cvar\pap)^2\vvar + \avar(\cvar\pap)^2\g + \nobrac{\cbrac{\cvar\avar_\ap - i\Hil(\cvar\vvar_\ap)}\cvar\g_\ap + (\errorone)(\cvar\g_\ap)}  \\
& = -(\cvar\pap)^2\vvar + \avar(\cvar\pap)^2\g + \cbrac*[\bigg]{\cbrac{ i\sqbrac{\cvar,\Hil}\vvar_\ap + i\cvar\pap\sqbrac{\cvar,\Hil}\brac{\frac{\vvar}{\cvar}}}\cvar\g_\ap + (\errorone)(\cvar\g_\ap)}  \\
& =  -(\cvar\pap)^2\vvar + \avar(\cvar\pap)^2\g + \errortwo
\end{split}
\end{align}
where $\errortwo$ is defined as the term in the bracket and we observe that $\norm[H^{N+2.5}]{\errortwo} \leq C(K)\Ecal^\half$

\item Using the estimates from above and \propref{prop:Leibniz} we have
\begin{align*}
& \papabs^\half\Dt(\cvar\pap)^{3+N} g \\
& \approxLtwo \papabs^\half(\cvar\pap)^{2+N}\Dt(\cvar\pap)g \\
&  \approxLtwo \papabs^\half(\cvar\pap)^{2+N}\cbrac{-(\cvar\pap)^2\vvar + \avar(\cvar\pap)^2 g} \\
& \approxLtwo - \papabs^\half(\cvar\pap)\cbrac{(\cvar\pap)^{3+N}\vvar - \avar(\cvar\pap)^{3+N} g}
\end{align*}

\item Using the above estimates and \propref{prop:commutator} we have
\begin{align*}
& \Dt\cbrac{(\cvar\pap)^{3+N}\vvar - \avar(\cvar\pap)^{3+N}\g} \\
& \approxLtwo (\cvar\pap)^{3+N}\Dt\vvar - \avar(\cvar\pap)^{3+N}\Dt\g \\
& \approxLtwo (\cvar\pap)^{3+N}\cbrac{-i\sigma\Hil(\cvar\pap)^2\g - \avar(\cvar\pap)\vvar + \avar^2(\cvar\pap)\g } - \avar(\cvar\pap)^{3+N}\cbrac{-(\cvar\pap)\vvar + \avar (\cvar\pap)\g} \\
& = -i\sigma(\cvar\pap)^{3+N}\Hil(\cvar\pap)^2\g - \sqbrac{(\cvar\pap)^{3+N},\avar}(\cvar\pap)\vvar - \avar(\cvar\pap)^{4+N}\vvar  + \sqbrac{(\cvar\pap)^{3+N},\avar^2}(\cvar\pap)\g \\
& \quad + \avar^2(\cvar\pap)^{4+N}\g  + \avar(\cvar\pap)^{4+N}\vvar - \avar\sqbrac{(\cvar\pap)^{3+N},\avar}(\cvar\pap)\g -\avar^2(\cvar\pap)^{4+N}\g \\
& \approxLtwo  -i\sigma(\cvar\pap)^{3+N}\Hil(\cvar\pap)^2\g \\
& \approxLtwo -i\sigma(\cvar\pap)^2\Hil(\cvar\pap)^{3+N}\g
\end{align*}
\end{enumerate}
\medskip
\textbf{Step 3:} We now prove the energy estimate. Observe that controlling the time derivative of $\Ecalthreefive$ and $\Ecalfourfivei$ for $0\leq i < N$ is immediate. Hence we now control the time derivative of the highest term in the energy namely $\EcalfourfiveN$. To simplify the calculations we will use the following notation: If $a(t), b(t) $ are functions of time we write $a \approx b$ if there exists a non-negative polynomial $C(t)$ with coefficients depending only on $\sigma$ so that $\abs{a(t)-b(t)} \leq C(K(t))\Ecal(t)$. Observe that $\approx$ is an equivalence relation. With this notation proving \propref{prop:apriori} is equivalent to showing $ \frac{d\Ecal(t)}{dt} \approx 0$. Hence now by using \lemref{lem:timederiv} we have
\begin{align*}
\frac{\diff \EcalfourfiveN}{\diff t} & \approx \int \cbrac{\frac{1}{\cvar^\half}\cbrac{(\cvar\pap)^{3+N}\vvar - \avar(\cvar\pap)^{3+N}\g }}\Dt\cbrac{\frac{1}{\cvar^\half}\cbrac{(\cvar\pap)^{3+N}\vvar - \avar(\cvar\pap)^{3+N}\g }} \diff \ap \\
& \quad + \sigma\int \brac{\papabs (\cvar\pap)^{3+N}\g} \Dt(\cvar\pap)^{3+N}\g \diff \ap
\end{align*}
Observe that $\Dt \cvar = (-i\Dt\Hil\g)\cvar$. Hence
\begin{align*}
\frac{\diff \EcalfourfiveN}{\diff t} & \approx \int \cbrac{\frac{1}{\cvar}\cbrac{(\cvar\pap)^{3+N}\vvar - \avar(\cvar\pap)^{3+N}\g }}\nobrac{\Dt\cbrac{(\cvar\pap)^{3+N}\vvar - \avar(\cvar\pap)^{3+N}\g }} \diff \ap \\
& \quad + \sigma\int \brac{\papabs (\cvar\pap)^{3+N}\g} \Dt(\cvar\pap)^{3+N}\g \diff \ap
\end{align*}
Now using the computation from step 2 we get
\begin{align*}
\frac{\diff \EcalfourfiveN}{\diff t} & \approx \int \cbrac{\frac{1}{\cvar}\cbrac{(\cvar\pap)^{3+N}\vvar - \avar(\cvar\pap)^{3+N}\g }}\nobrac{\cbrac{ -i\sigma(\cvar\pap)^2\Hil(\cvar\pap)^{3+N}\g }} \diff \ap \\
& \quad - \sigma\int \brac{\papabs (\cvar\pap)^{3+N}\g}  (\cvar\pap)\cbrac{(\cvar\pap)^{3+N}\vvar - \avar(\cvar\pap)^{3+N}g} \diff \ap \\
& \approx 0
\end{align*}
where at the last step we used that  $\papabs = i\Hil\pap$. This proves the a priori estimate. 
\end{proof}

We now prove a priori estimate for the difference of two solutions for system \eqref{eq:systemtwo}. This will prove uniqueness of the solution. Let $(\gone, \vvarone)(t)$ and $(\gvartwo,\vvartwo)(t)$ be two solutions of \eqref{eq:systemtwo}. We will use a subscript to denote which solution we are talking about. For example 
\begin{align*}
\cvarone = e^{-i\Hil\gone} \qquad \cvartwo = e^{-i\Hil\gtwo}
\end{align*}
and similarly for other variables as well. For $\Aone$ and $\etwo$ we will use the notation $(\Aone)_1, (\Aone)_2 $ and $(\etwo)_1, (\etwo)_2$ respectively. Now define the energy
\begin{align}\label{eq:energydelta}
\begin{split}
\EcalDeltatwo &= \half\norm[H^1]{\gone - \gtwo}^2 + \norm[H^\half]{\vvarone - \vvartwo}^2 \\
\EcalDeltatwofive &= \half\norm[2]{\frac{1}{\cvarone^\half}\cbrac{(\cvarone\pap)(\vvarone - \vvartwo)}}^2 + \frac{\sigma}{2}\norm[\Hhalf]{(\cvarone\pap)(\gone - \gtwo)}^2 \\
\EcalDeltathree &= \half\norm[2]{\papabs^\half \cbrac{(\cvarone\pap)\vvarone - (\cvartwo\pap)\vvartwo} - \avarone\papabs^\half\cbrac{(\cvarone\pap)\gone - (\cvartwo\pap)\gtwo }}^2 \\
& \quad + \frac{\sigma}{2}\norm[2]{\frac{1}{\cvarone^\half}\cbrac{(\cvarone\pap)^2\gone - (\cvartwo\pap)^2\gtwo }}^2  \\
\EcalDelta &= \EcalDeltatwo + \EcalDeltatwofive + \EcalDeltathree
\end{split}
\end{align}
\begin{prop}\label{prop:aprioridelta}
Let $(\gone,\vvarone)(t)$ and $(\gtwo,\vvartwo)(t)$ be two solutions of \eqref{eq:systemtwo} in the time interval $[0,T]$ with $(\g_i,\vvar_i) \in C^l([0,T], H^{3.5 - \threebytwo l}\times H^{3 - \threebytwo l})$ for $l=0,1$, for both $i=1,2$. Let $M>0$ be a constant so that for any $t\in[0,T]$
\begin{align*}
 \norm[H^{3.5}]{\gone(\cdot,t)} + \norm[H^3]{\vvarone(\cdot,t)} +  \norm[H^{3.5}]{\gtwo(\cdot,t)} + \norm[H^3]{\vvartwo(\cdot,t)} \leq M
\end{align*}
Then there exists a constant $C(M) >0$ depending only on $M$ and $\sigma$ such that for all $t\in [0,T)$ we have
\begin{align*}
\frac{\diff \EcalDelta(t)}{\diff t} \leq C(M)\EcalDelta(t)
\end{align*}
\end{prop}
\begin{proof}
We define $\Dtone = \pt + \bvarone\pap$ and $\Dttwo = \pt + \bvartwo\pap$. We will freely use \lemref{lem:interpolationmult} to simplify the computations.

\textbf{Step 1:} We first find quantities which can be controlled by $M$ and $\EcalDelta$ in the time interval $[0,T]$. 
\begin{enumerate}
\item We know that $ \norm[H^{3.5}]{\gone(\cdot,t)} + \norm[H^3]{\vvarone(\cdot,t)} +  \norm[H^{3.5}]{\gtwo(\cdot,t)} + \norm[H^3]{\vvartwo(\cdot,t)} \leq M$. As both $(\gone, \vvarone)$ and $(\gtwo, \vvartwo)$ solve  \eqref{eq:systemtwo} we can use the same estimates from step 1 of the proof of \propref{prop:apriori} corresponding to $N=0$. Hence for all $t\in[0,T]$ we have for $i=1,2$
\begin{align*}
\norm[H^{3.5}]{\cvar_i -1} + \norm[H^{3.5}]{\frac{1}{\cvar_i} -1}+  \norm[H^{3.5}]{\w_i - 1} &\leq C(M)\\
\norm[H^{3}]{\avar_i}  + \norm[H^{3}]{\dvar_i} + \norm[H^{3}]{\bvar_i} + \norm[H^{3}]{(\Aone)_i - 1} + \norm[H^{3}]{(\etwo)_i} & \leq C(M)\\
\norm[H^{2}]{\pt \g_i} + \norm[H^{1.5}]{\pt \vvar_i} + \norm[H^{2}]{\pt\cvar_i} & \leq C(M)
\end{align*}

\item From $\EcalDeltatwo$ we have
\begin{align*}
\norm[H^1]{\gone - \gtwo} + \norm[H^\half]{\vvarone - \vvartwo} \leq C(M)\EcalDelta^\half 
\end{align*}
Now from $\EcalDeltatwofive$ and bounds on $\cvarone$ we have
\begin{align*}
\norm[H^1]{\vvarone - \vvartwo} + \norm[\Hhalf]{\cvarone\pap(\gone - \gtwo)} \leq C(M)\EcalDelta^\half
\end{align*}
Now using \propref{prop:Leibniz} we have
\begin{align*}
\norm[\Hhalf]{\pap(\gone - \gtwo)} \lesssim  \norm[\Hhalf]{\cvarone\pap(\gone - \gtwo)}\norm[\infty]{\frac{1}{\cvarone}} +  \norm[2]{\cvarone\pap(\gone - \gtwo)}\norm[2]{\pap\frac{1}{\cvarone}}
\end{align*}
Hence $\norm[H^{1.5}]{\gone - \gtwo} \leq C(M)\EcalDelta^\half$. Hence from a similar argument to step 1 of \propref{prop:apriori}  we also obtain $\norm[H^{1.5}]{\cvarone - \cvartwo} \leq C(M)\EcalDelta^\half$. 

\item Observe that
\begin{align*}
\nobrac{(\cvarone\pap)^2\gone - (\cvartwo\pap)^2\gtwo } & = (\cvarone - \cvartwo)\pap \cbrac{(\cvarone\pap)\gone} + \cvartwo\pap\cbrac{(\cvarone - \cvartwo)\pap\gone} \\
& \quad  + \cbrac{(\cvartwo\pap)\cvartwo}\pap(\gone - \gtwo) + \cvartwo^2\pap^2(\gone - \gtwo) 
\end{align*}
Now using the estimate $\norm[H^{1.5}]{\cvarone - \cvartwo} \leq C(M)\EcalDelta^\half$ and the fact that $\norm[H^{3.5}]{\g_i} + \norm[H^{3.5}]{\cvar_i - 1} + \norm[H^{3.5}]{\frac{1}{\cvar_i} - 1} \leq C(M)$ for $i=1,2$ we immediately obtain from $\EcalDeltathree$
\begin{align*}
\norm[H^2]{\gone - \gtwo} \leq C(M)\EcalDelta^\half
\end{align*}
From the above estimate for $\cvarone - \cvartwo$ and the fact that $\norm[H^{3.5}]{\gtwo} \leq C(M)$ we see that
\begin{align*}
\norm[H^\half]{(\cvarone\pap) \gone - (\cvartwo\pap)\gtwo} \leq C(M)\EcalDelta^\half
\end{align*}
Hence from using $\norm[H^3]{\avarone} \leq C(M)$ and $\EcalDeltathree$ we obtain 
\begin{align*}
\norm[H^\half]{(\cvarone\pap) \vvarone - (\cvartwo\pap)\vvartwo} \leq C(M)\EcalDelta^\half
\end{align*}
Hence again by using the estimate for $\cvarone - \cvartwo$ we get
\begin{align*}
\norm[H^{1.5}]{\vvarone - \vvartwo} \leq C(M)\EcalDelta^\half
\end{align*}

\item Again by a similar computation from step 1 of \propref{prop:apriori} we get
\begin{align*}
\norm[H^2]{\cvarone - \cvartwo} + \norm[H^2]{\wone - \wtwo} \leq C(M)\EcalDelta^\half
\end{align*}
Now using the definitions of $\avar, \dvar, \bvar, \Aone$ and $\etwo$ we easily get using \propref{prop:commutator} 
\begin{align*}
&\norm[H^{1.5}]{\avarone - \avartwo} + \norm[H^{1.5}]{\dvarone - \dvartwo} + \norm[H^{1.5}]{\bvarone - \bvartwo} + \norm[H^{1.5}]{(\Aone)_1 - (\Aone)_2} + \norm[H^{1.5}]{(\etwo)_1 - (\etwo)_2} \\
& \leq C(M)\EcalDelta^\half
\end{align*}
Similarly from the equations of $\pt \g $ and $\pt \vvar$ we get 
\begin{align*}
\norm[H^\half]{\pt(\gone - \gtwo)} + \norm[2]{\pt (\vvarone - \vvartwo)} \leq C(M)\EcalDelta^\half
\end{align*}
Also observe that $\pt \cvar = \cbrac{-i\Hil(\pt\g)}\cvar$. Hence $\norm[H^\half]{\pt (\cvarone - \cvartwo)} \leq C(M)\EcalDelta^\half$

\end{enumerate}

\medskip
\textbf{Step 2:} We now prove some estimates required to prove the energy estimate. We define the following notation:  If $a,b:\Rsp\times[0,T] \to \Csp$ are functions we write $a \approx_{\Ltwo} b$ if there exists a constant $C(M)$ depending only on $M$ and $\sigma$ such that $\norm[2]{a-b} \leq C(M)\EcalDelta(t)^\half$. Observe that $\approx_{\Ltwo}$ is an equivalence relation.
\begin{enumerate}
\item  For the sake of convenience we define 
\begin{align}\label{eq:zeta}
\zeta(\ap, t) = \papabs^\half\cbrac{(\cvarone\pap)\vvarone - (\cvartwo\pap)\vvartwo} - \avarone\papabs^\half\cbrac{(\cvarone\pap)\gone - (\cvartwo\pap)\gtwo}
\end{align}
Using \propref{prop:commutator} and repeated use of the esimates from step 1 we see that
\begin{align*}
& \Dtone \zeta  \\
& \approx_{\Ltwo} \papabs^\half\Dtone  \cbrac{(\cvarone\pap)\vvarone  - (\cvartwo\pap)\vvartwo}  -\avarone\papabs^\half\Dtone \cbrac{(\cvarone\pap)\gone - (\cvartwo\pap)\gtwo} \\
& \approx_{\Ltwo} \papabs^\half\cbrac{\Dtone(\cvarone\pap)\vvarone - \Dttwo(\cvartwo\pap)\vvartwo} - \avarone\papabs^\half\cbrac{\Dtone(\cvarone\pap)\gone - \Dttwo(\cvartwo\pap)\gtwo} \\
& \approx_{\Ltwo} \papabs^\half\cbrac{(\cvarone\pap)\Dtone\vvarone - (\cvartwo\pap)\Dttwo\vvartwo} - \avarone\papabs^\half\cbrac{\Dtone(\cvarone\pap)\gone - \Dttwo(\cvartwo\pap)\gtwo} \\
& \approx_{\Ltwo} \cbrac{\papabs^\half(\cvarone\pap)\Dtone\vvarone - \avarone \papabs^\half\Dtone(\cvarone\pap)\gone} -  \cbrac{\papabs^\half(\cvartwo\pap)\Dttwo\vvartwo - \avartwo \papabs^\half\Dttwo(\cvartwo\pap)\gtwo}
\end{align*}
Now we use the equations \eqref{eq:systemtwo} and the formulae \eqref{form:errortwo} and \eqref{form:errorone} to obtain
\begin{align*}
& \papabs^\half(\cvarone\pap)\Dtone\vvarone  - \avarone\papabs^\half\Dtone(\cvarone\pap)\gone \\
& = \papabs^\half(\cvarone\pap)\cbrac{-i\sigma\Hil(\cvarone\pap)^2\gone - \avarone(\cvarone\pap)\vvarone + \avarone^2(\cvarone\pap)\gone + (\etwo)_1}  \\
& \quad -  \avarone\papabs^\half\cbrac{ - (\cvarone\pap)^2\vvarone + \avarone(\cvarone\pap)^2\gone + (\errortwo)_1} \\
& = -i\sigma\papabs^\half(\cvarone\pap)\Hil(\cvarone\pap)^2\gone - \papabs^\half\cbrac{(\cvarone\pap\avarone) \cvarone\pap\vvarone} - \sqbrac{\papabs^\half, \avarone}(\cvarone\pap)^2\vvarone \\
& \quad + \papabs^\half\cbrac{(2\avarone\cvarone\pap\avarone) \cvarone\pap\gone} + \sqbrac{\papabs^\half,\avarone}\cbrac{\avarone(\cvarone\pap)^2\gone} + \papabs^\half\cbrac{(\cvarone\pap)(\etwo)_1} \\
& \quad - \avarone\papabs^\half(\errortwo)_1
\end{align*}
A similar equation is true for the 2nd solution as well. Hence using \propref{prop:commutator} and the estimates from step 1 we see that 
\begin{align*}
 \Dtone \zeta & \approx_{\Ltwo} -i\sigma\papabs^\half\cbrac{ (\cvarone\pap)\Hil (\cvarone\pap)^2\gone - (\cvartwo\pap)\Hil (\cvartwo\pap)^2\gtwo } \\
 & \approx_{\Ltwo}  -i\sigma\papabs^\half (\cvarone\pap) \Hil\cbrac{ (\cvarone\pap)^2\gone - (\cvartwo\pap)^2\gtwo }
\end{align*}

\item Using the estimates from step 1 we observe that
\begin{align*}
 \Dtone \cbrac{(\cvarone\pap)^2\gone - (\cvartwo\pap)^2\gtwo}  & \approx_{\Ltwo} \Dtone(\cvarone\pap)^2\gone - \Dttwo(\cvartwo\pap)^2\gtwo \\
 & \approx_{\Ltwo} (\cvarone\pap)\Dtone(\cvarone\pap)\gone -  (\cvartwo\pap)\Dttwo(\cvartwo\pap)\gtwo \\
 & \approx_{\Ltwo} \cvarone\pap\cbrac{\Dtone(\cvarone\pap)\gone - \Dttwo(\cvartwo\pap)\gtwo}
\end{align*}
Now we use the formulae \eqref{form:errortwo} and \eqref{form:errorone} to obtain
\begin{align*}
 & \Dtone \cbrac{(\cvarone\pap)^2\gone - (\cvartwo\pap)^2\gtwo} \\
 & \approx_{\Ltwo} \cvarone\pap\cbrac{ -(\cvarone\pap)^2\vvarone + \avarone(\cvarone\pap)^2\gone + (\cvartwo\pap)^2\vvartwo - \avartwo(\cvartwo\pap)^2\gtwo  } \\
 & \approx_{\Ltwo} -\cvarone\pap\cbrac*[\Big]{ (\cvarone\pap)\cbrac{(\cvarone\pap)\vvarone - (\cvartwo\pap)\vvartwo} - \avarone(\cvarone\pap)\cbrac{(\cvarone\pap)\gone - (\cvartwo\pap)\gtwo}}
\end{align*}
\item Using \propref{prop:commutator} we observe that
\begin{align*}
& -i\sigma\Hil\pap\cbrac{\cvarone\papabs^\half (\zeta)} \\
& = -i\sigma\Hil\pap\cbrac{\cvarone\papabs^\half \cbrac{ \papabs^\half\cbrac{(\cvarone\pap)\vvarone - (\cvartwo\pap)\vvartwo} - \avarone\papabs^\half\cbrac{(\cvarone\pap)\gone - (\cvartwo\pap)\gtwo}}} \\
& \approx_{\Ltwo} -i\sigma\Hil\pap\cbrac*[\Big]{\cvarone\papabs\cbrac{(\cvarone\pap)\vvarone - (\cvartwo\pap)\vvartwo} - \avarone\cvarone\papabs\cbrac{(\cvarone\pap)\gone - (\cvartwo\pap)\gtwo}}
\end{align*}
Now as $\papabs = i\Hil\pap$ by repeated use of \propref{prop:commutator} and the estimates from step 1 we obtain 
\begin{align*}
& -i\sigma\Hil\pap\cbrac{\cvarone\papabs^\half (\zeta)} \\
& \approx_{\Ltwo} \sigma\pap\cbrac*[\Big]{(\cvarone\pap)\cbrac{(\cvarone\pap)\vvarone - (\cvartwo\pap)\vvartwo} - \avarone(\cvarone\pap)\cbrac{(\cvarone\pap)\gone - (\cvartwo\pap)\gtwo}}
\end{align*}
\end{enumerate}

\medskip
\textbf{Step 3:} We now prove the energy estimate. It is easy to see that the time derivative of $\EcalDeltatwo$ is easily controlled. Hence we now control the time derivative of the energies $\EcalDeltatwofive$ and $\EcalDeltathree$. To simplify the calculations as usual we will use the following notation: If $a(t), b(t) $ are functions of time we write $a \approx b$ if there exists a constant $C(M)$ depending only on $M$ and $\sigma$ such that $\abs{a(t)-b(t)} \leq C(M)\EcalDelta(t)$. Observe that $\approx$ is an equivalence relation. With this notation proving \propref{prop:aprioridelta} is equivalent to showing $ \frac{d\EcalDelta(t)}{dt} \approx 0$. 

\begin{enumerate}
\item Observe that
\begin{align*}
\frac{\diff \EcalDeltatwofive}{\diff t} & = \int \frac{1}{\cvarone}\cbrac{(\cvarone\pap)(\vvarone - \vvartwo)} \pt \cbrac{(\cvarone\pap)(\vvarone - \vvartwo)} \diff \ap \\
& \quad + \sigma\int \cbrac{\papabs(\cvarone\pap)(\gone - \gtwo)} \pt\cbrac{(\cvarone\pap)(\gone - \gtwo)} \diff \ap \\
& \approx -\int \pap\cbrac{(\cvarone\pap)(\vvarone - \vvartwo)} \pt(\vvarone - \vvartwo)\diff \ap \\
& \quad + \sigma \int \cbrac{\papabs(\cvarone\pap)(\gone - \gtwo)} \cbrac{(\cvarone\pap)\pt (\gone - \gtwo)} \diff \ap
\end{align*}
Now we use the equations for $\pt \g$ and $\pt \vvar$ to obtain 
\begin{align*}
\frac{\diff \EcalDeltatwofive}{\diff t} & \approx  i\sigma \int \pap\cbrac{(\cvarone\pap)(\vvarone - \vvartwo)} \Hil \cbrac{(\cvarone\pap)^2\gone - (\cvartwo\pap)^2\gtwo} \diff \ap \\
& \quad - \sigma \int \cbrac{\papabs(\cvarone\pap)(\gone - \gtwo)} \cbrac{(\cvarone\pap)\cbrac{ (\cvarone\pap)\vvarone - (\cvartwo\pap)\vvartwo }} \diff \ap
\end{align*}
Now we use the estimates from step 1 to obtain
\begin{align*}
\frac{\diff \EcalDeltatwofive}{\diff t} & \approx  i\sigma \int \pap\cbrac{(\cvarone\pap)(\vvarone - \vvartwo)} \Hil \cbrac{(\cvarone\pap)^2(\gone - \gtwo)} \diff \ap \\
& \quad - \sigma \int \cbrac{\papabs(\cvarone\pap)(\gone - \gtwo)} \cbrac{(\cvarone\pap)^2(\vvarone - \vvartwo)} \diff \ap \\
& = -i\sigma\int \pap\cbrac{(\cvarone\pap)(\vvarone - \vvartwo)}\cbrac{\sqbrac{\cvarone,\Hil}\pap\cbrac{(\cvarone\pap(\gone - \gtwo))}} \diff \ap \\
& \approx 0
\end{align*}
where we used \propref{prop:commutator} in the last step.

\item Now we control $\EcalDeltathree$. Using \lemref{lem:timederiv} and the estimates from step 2 we obtain
\begin{align*}
\frac{\diff \EcalDeltathree}{\diff t} & \approx  \int (\zeta)(\Dtone \zeta) \diff \ap + \sigma\int \frac{1}{\cvarone}\cbrac{(\cvarone\pap)^2\gone - (\cvartwo\pap)^2\gtwo} \Dtone \cbrac{(\cvarone\pap)^2\gone - (\cvartwo\pap)^2\gtwo} \diff \ap \\
& \approx -i\sigma\int (\zeta)\cbrac{\papabs^\half (\cvarone\pap) \Hil\cbrac{ (\cvarone\pap)^2\gone - (\cvartwo\pap)^2\gtwo }} \diff \ap \\
& \quad - \sigma \int \cbrac{(\cvarone\pap)^2\gone - (\cvartwo\pap)^2\gtwo}\pap\Big\{ (\cvarone\pap)\cbrac{(\cvarone\pap)\vvarone - (\cvartwo\pap)\vvartwo} \\
& \qquad \qquad \quad - \avarone(\cvarone\pap)\cbrac{(\cvarone\pap)\gone - (\cvartwo\pap)\gtwo}\Big\} \diff\ap \\
& \approx 0
\end{align*}
thereby proving the energy estimate. 
\end{enumerate}
\end{proof}
 
\smallskip

\subsection{\texorpdfstring{A priori estimates for approximate solutions \nopunct}{}}\hspace*{\fill} \medskip

We now work with a mollified system. Fix $\delta,\ep\geq0$ as parameters and let $\phi$ be a smooth bump function satisfying $\phi(\ap) \geq 0$ for all $\ap\in\Rsp$ and $\int \phi(\ap) \diff\ap = 1$.  For $s>0$ let $\phi_s(\ap) = \frac{1}{s}\phi\brac*[\big]{\frac{\ap}{s}}$. Consider the smoothing operator $\Jdel$ defined via $\Jdel(f) = f * \phi_\delta$ for $\del>0$ and $\Jdel(f) = f$ for $\del =0$. Consider the following system in the variables $(\g, \vvar)$ 
\begin{align}\label{eq:systemtwodelep}
\begin{split}
\cvar & = e^{-i\Hil \g} \\
\w & = e^{i\g}\\
\bvar & = 2i\Hil(\cvar\vvar) + i \sqbrac{\cvar^2, \Hil} \brac{\frac{\vvar}{\cvar}} \\
\avar & = i\cvar\Hil\brac{\frac{\vvar}{\cvar} } \\
\dvar & = -i e^{i\g} \cvar (\Id - \Hil) \brac{\frac{\vvar}{\cvar}} \\
\Aone & =  1 - \Imag \sqbrac{\dvar,\Hil} \pap\bar\dvar \\
\etwo & = \Real(\w) - \Aone \cvar + \sigma \Imag\cbrac{\sqbrac{\cvar,\Hil} \pap(\Id + \Hil) (\cvar \pap \g) }\\
\pt\g & = \Jdel^2\cbrac{ -(\cvar\pap)\vvar + \avar(\cvar\pap)\g -\bvar\pap\g} \\
\pt\vvar & = \Jdel^2\cbrac{ -i\sigma \Hil (\cvar\pap)^2 \g - \avar (\cvar\pap)\vvar  -\bvar\pap\vvar + \avar^2 (\cvar\pap)\g   + \etwo} -\ep\papabs\vvar
\end{split}
\end{align}
The evolution equations of $\g$ and $\vvar$ have changed so that now we have a smoothing term $\Jdel^2$ in both the equations. We also have a dissipative term for the time evolution equation for $\vvar$. This is very similar to the system used in \cite{CaCoFeGaGo12}.  

To prove existence for this system we will need some estimates.
\begin{lemma}\label{lem:commutatordel}
Let $f,g \in \Scalsp(\Rsp)$ and let $0 < \del,\del_1,\del_2 \leq 1$. Then we have
\begin{enumerate}
\item $\norm[2]{J_{\del_1}(f) - J_{\del_2}f} \lesssim \max\{\del_1^s,\del_2^s\}\norm[H^s]{f}$ \quad for $0< s\leq 1$
\item $\norm[2]{\sqbrac{\Jdel,f}\pap g} \lesssim \del\norm[H^2]{f}\norm[2]{g_\ap}$ 
\item $\norm[2]{\sqbrac{\Jdel,f}\pap g} \lesssim \norm[\infty]{f_\ap}\norm[2]{g}$ 
\item $\norm[2]{\papabs^\half\sqbrac{\Jdel,f}\pap g} \lesssim \norm[\infty]{f_\ap}\norm[\Hhalf]{g}$
\end{enumerate}
where the constants in the estimates are independent of $\del$.
\end{lemma}
\begin{proof}
We prove each of them separately. 
\begin{enumerate}
\item It is enough to assume that $\del_2 =0$. Using the Fourier transform we have
\begin{align*}
\norm[2]{\Jdel(f) - f} \lesssim & \norm[\Ltwo(\diff\xi)]{\brac*{\hat{\phi}(\del\xi) -1}\hat{f}(\xi)} \\
& \lesssim \norm[\Ltwo(\diff \xi)]{\hat{f}(\xi)1_{\abs{\del\xi}\geq 1}} + \norm[\Ltwo(\diff\xi)]{\brac*{\hat{\phi}(\del\xi) -\hat{\phi}(0)}\hat{f}(\xi)1_{\abs{\del\xi}\leq 1}} \\
& \lesssim \del^s\norm[H^s]{f}
\end{align*}
\item Observe that
\begin{align*}
\norm[2]{\sqbrac{\Jdel,f}\pap g} & \lesssim \norm[\Ltwo(\diff\xi)]{\int \hat{f}(\xi-\eta)\eta\hat{g}(\eta)\sqbrac{\hat{\phi}(\del\xi) - \hat{\phi}(\del\eta)} \diff \eta} \\
& \lesssim \del \norm[\Ltwo(\diff\xi)]{\int (\xi - \eta)\hat{f}(\xi-\eta)\eta\hat{g}(\eta)\sqbrac{\frac{\hat{\phi}(\del\xi) - \hat{\phi}(\del\eta)}{(\del\xi - \del\eta)}} \diff \eta} \\
& \lesssim \del\norm[\Lone(\diff\xi)]{\xi\hat{f}(\xi)}\norm[\Ltwo(\diff\xi)]{\xi\hat{g}(\xi)} \\
& \lesssim \del\norm[H^2]{f}\norm[2]{g_\ap}
\end{align*}
\item We see that
\begin{align*}
 \brac{\sqbrac{\Jdel,f}\pap g}(\ap) = \int \phi_\del(\bp) \brac{f(\ap - \bp) - f(\ap)} g_\ap (\ap - \bp) \diff \bp
\end{align*}
As $ g_\ap (\ap-\bp) = -\pbp g(\ap - \bp)$. Hence if we define $\chi(\bp) = \phi '(\bp)\bp$ then
\begin{align*}
 \brac{\sqbrac{\Jdel,f}\pap g}(\ap) & = \int \pbp\cbrac*[\big]{ \phi_\del(\bp) \brac{f(\ap - \bp) - f(\ap)}} g(\ap - \bp) \diff \bp\\
 & = \int \cbrac{\chi_\del(\bp) \brac{ \frac{f(\ap - \bp) - f(\ap)}{\bp} } - \phi_\del(\bp)f_\bp(\bp) } g(\ap-\bp) \diff \bp
\end{align*}
from which the estimate follows. 
\item Observe that
\begin{align*}
\papabs^\half\sqbrac{\Jdel,f}\pap g & = \sqbrac{\papabs^\half\Jdel,f}\pap g - \sqbrac{\papabs^\half,f}\pap(\Jdel g)\\
& = \Jdel\sqbrac{\papabs^\half,f}\pap g + \sqbrac{\Jdel,f}\pap \brac{\papabs^\half g} - \sqbrac{\papabs^\half,f}\pap(\Jdel g)
\end{align*}
The estimate now follows from \propref{prop:commutator}.
\end{enumerate}
\end{proof}

We will also require estimates similar to \lemref{lem:timederiv} but adapted to this system.  
\begin{lem} \label{lem:timederivdel}
Let $T>0$ and let $f,\bvar \in C^1([0,T), H^2(\Rsp))$ with $\bvar$ being real valued. Define the operator $\Dtdel = \pt + \Jdel^2 (\bvar\pap)$. Then we have the estimate
\begin{enumerate}[leftmargin =*, align=left]
\item $\dis \abs*[\Big]{ \frac{d}{dt}\int \abs{f}^2 \diff \ap - 2\Real \int \bar{f} (\Dtdel f) \diff \ap} \lesssim \norm[2]{f}^2 \norm[\infty]{\bvarap}$

\item $\dis \abs{ \frac{d}{dt}\int (\papabs\bar{f})f \diff\ap- 2\Real \cbrac{\int (\papabs\bar{f})\Dtdel f \diff \ap}} \lesssim   \norm[\Hhalf]{f}^2 \norm[\infty]{\bvarap}$
\end{enumerate}
\end{lem}
\begin{proof}
The proof is very similar to the proof of \lemref{lem:timederiv} with the only difference being that we now also use  \lemref{lem:commutatordel}. 
\end{proof}

The energy for this system $\Ecal(t)$ is again defined by \eqref{eq:energy} where $N\geq 0$
\begin{align}\label{eq:energydel}
\begin{split}
\Ecalthreefive & = \frac{1}{2}\norm[H^{2.5}]{\g}^2 + \half\norm[H^2]{\vvar}^2 \\
\Ecalfourfivei &= \half\norm[2]{\frac{1}{\cvar^\half}\cbrac{(\cvar\pap)^{3+i}\vvar - \avar(\cvar\pap)^{3+i}\g }}^2 + \frac{\sigma}{2}\norm[\Hhalf]{(\cvar\pap)^{3+i}\g}^2 \\
\Ecal &= \Ecalthreefive + \sum_{i=0}^N \Ecalfourfivei
\end{split}
\end{align}
We again define $K(t) = \norm[H^{2.5}]{\g(\cdot,t)} + \norm[H^2]{\vvar(\cdot,t)}$. We now have
\begin{prop}\label{prop:aprioridelep}
Fix $N\geq 0$ and let $(\g,\vvar)(t)$ be a smooth solution to \eqref{eq:systemtwodelep} with parameters $(\del,\ep)$ in the time interval $[0,T]$ with  $(\g,\vvar) \in C([0,T], H^{s+\half}\times H^s)$ for all $s\geq 0$. 
\begin{enumerate}
\item If $0\leq \del \leq 1$ and $0<\ep\leq 1$, then there exists a polynomial $C_\ep = C_\ep(t)$ with non-negative coefficients depending only on $\sigma,\ep$ and independent of $\del$ so that for any $t\in[0,T)$ we have
\begin{align*}
\frac{\diff \Ecal(t)}{\diff t} \leq C_\ep(K(t))\Ecal(t)
\end{align*}
\item  If $\del = 0$ and $0\leq \ep \leq 1$, then there exists a polynomial $C = C(t)$ with non-negative coefficients depending only on $\sigma$ and independent of $\ep$ so that for any $t\in[0,T)$ we have
\begin{align*}
\frac{\diff \Ecal(t)}{\diff t} \leq C(K(t))\Ecal(t)
\end{align*}
\end{enumerate}
\end{prop}
\begin{proof}
The proof of this proposition is similar to the proof of \propref{prop:apriori} and we will mostly focus on the changes that we need to make. As before we will freely use \lemref{lem:interpolationmult} to simplify the computations. 

\medskip
\textbf{Step 1:} The quantities controlled by the energy are the same as in \propref{prop:apriori}. We collect the estimates below. This applies for all $0\leq \del,\ep\leq 1$
\begin{enumerate}
\item We have $\norm[\infty]{\g} + \norm[\infty]{\cvar} + \norm[\infty]{\frac{1}{\cvar}} \leq C(K)$ and hence
\begin{align*}
\norm[H^{1.5}]{\cvar_\ap} +  \norm[H^1]{\pt \g}  +  \norm[H^1]{\pt \cvar} + \norm[H^2]{\bvar} + \norm[H^2]{\avar} + \norm[H^\half]{\pt\bvar} + \norm[H^\half]{\pt \avar} \leq C(K)
\end{align*}
\item $\norm[H^{N+3.5}]{\g} + \norm[H^{N+3.5}]{\cvar -1} + \norm[H^{N+3.5}]{\frac{1}{\cvar} -1} +  \norm[H^{N+3.5}]{\w - 1} \leq C(K)\Ecal^\half$
\item We have 
\begin{align*}
\norm[H^{N+3}]{\vvar} + \norm[H^{N+3}]{\avar}  + \norm[H^{N+3}]{\dvar} + \norm[H^{N+3}]{\bvar} + \norm[H^{N+3}]{\Aone - 1} + \norm[H^{N+3}]{\etwo} \leq C(K)\Ecal^\half
\end{align*}
and hence
\begin{align*}
\norm[H^{N+2}]{\pt \g} + \norm[H^{N+1.5}]{\pt \vvar} + \norm[H^{N+2}]{\pt\cvar} \leq C(K)\Ecal^\half
\end{align*}
\end{enumerate}

\medskip
\textbf{Step 2:} We now establish some identities involving $\Dtdel = \pt + \Jdel^2(\bvar\pap)$ similar to the ones obtained in the proof of \propref{prop:apriori} for the case of $0\leq \del\leq 1$ and $0<\ep\leq 1$. As a lot of calculations are similar we will skip some details. We define the following notation:  If $a,b:\Rsp\times[0,T] \to \Csp$ are functions we write $a \approx_{\Ltwo} b$ if there exists a polynomial $C(t)$  with non-negative coefficients depending only on $\sigma$ such that $\norm[2]{a-b} \leq C(K)\cbrac{\Ecal(t)^\half + \norm[H^{3.5+N}]{\vvar}}$. Observe that $\approx_{\Ltwo}$ is an equivalence relation.  
\begin{enumerate}
\item If $f_1,f_2$ are two functions then
\begin{align*}
\Dtdel(f_1 f_2) & = f_1(\pt f_2) + f_2(\pt f_1) + \Jdel^2\cbrac{f_1(\bvar\pap)f_2 + f_2(\bvar\pap)f_1} \\
& = f_1 (\Dtdel f_2) + f_2(\Dtdel f_1) + \sqbrac{\Jdel^2,f_1}(\bvar\pap)f_2 +  \sqbrac{\Jdel^2,f_2}(\bvar\pap)f_1
\end{align*}
and we observe that 
\begin{align*}
\sqbrac{\Jdel^2,f_1}(\bvar\pap)f_2 & = \sqbrac{\Jdel^2,f_1\bvar}\pap f_2 - f_1\sqbrac{\Jdel^2,\bvar}\pap f_2 \\
& = \Jdel\sqbrac{\Jdel,f_1\bvar}\pap f_2 + \sqbrac{\Jdel,f_1 \bvar}\pap(\Jdel f_2) - f_1 \Jdel\sqbrac{\Jdel,\bvar}\pap f_2 - f_1\sqbrac{\Jdel,\bvar}\pap(\Jdel f_2)
\end{align*}
We similarly have an expansion for $\sqbrac{\Jdel^2,f_2}(\bvar\pap)f_1$ as well.
\item Observe that
\begin{align*}
\sqbrac{\Dtdel, \cvar\pap} g & = \cvar \sqbrac{\Dtdel,\pap} g + \sqbrac{\Dtdel,\cvar} g_\ap \\
& = -\cvar\Jdel^2\cbrac{\bvar_\ap g_\ap} + (\Dtdel \cvar)g_\ap + \sqbrac{\Jdel^2,g_\ap} (\bvar\pap) \cvar + \sqbrac{\Jdel^2,\cvar}(\bvar\pap)g_\ap
\end{align*}
\item Using the above identities and \lemref{lem:commutatordel} we have
\begin{align*}
& \papabs^\half \Dtdel (\cvar\pap)^{3+N} g \\
& \approxLtwo \papabs^\half(\cvar\pap)^{3+N}\Dtdel g \\
& = \papabs^\half (\cvar\pap)^{3+N} \Jdel^2 \cbrac{ -(\cvar\pap)\vvar + \avar (\cvar\pap)g} \\
& \approxLtwo \papabs^\half\Jdel^2 (\cvar\pap)^{3+N}\cbrac{ -(\cvar\pap)\vvar + \avar(\cvar\pap)g} \\
& \approxLtwo -\papabs^\half\Jdel^2(\cvar\pap)\cbrac{ (\cvar\pap)^{3+N}\vvar - \avar(\cvar\pap)^{3+N} g}
\end{align*}
\item Again by using the above identities and \lemref{lem:commutatordel} we have
\begin{align*}
& \Dtdel \cbrac{(\cvar\pap)^{3+N}\vvar -\avar(\cvar\pap)^{3+N}g} \\
& \approxLtwo (\cvar\pap)^{3+N}\Dtdel \vvar -\avar(\cvar\pap)^{3+N}\Dtdel g\\
& \approxLtwo (\cvar\pap)^{3+N} \Jdel^2\cbrac{ -i\sigma\Hil(\cvar\pap)^2 g - \avar(\cvar\pap)\vvar + \avar^2(\cvar\pap)g} -\ep(\cvar\pap)^{3+N}\papabs\vvar \\
& \quad -\avar(\cvar\pap)^{3+N}\Jdel^2\cbrac{-(\cvar\pap)\vvar + \avar(\cvar\pap)g} \\
& \approxLtwo -i\sigma(\cvar\pap)^2\Hil(\cvar\pap)^{1+N}\Jdel^2(\cvar\pap)^2 g - \ep(\cvar\pap)^{3+N}\papabs\vvar
\end{align*}

\end{enumerate}

\medskip
\textbf{Step 3:} We now prove the energy estimate for the case of $0\leq \del\leq 1$ and $0<\ep\leq 1$. As before controlling the time derivative of $\Ecalthreefive$ and $\Ecalfourfivei$ for $0\leq i < N$ is immediate. Hence we now control the time derivative of the highest term in the energy namely $\EcalfourfiveN$. To simplify the calculations we will use the following notation: If $a(t), b(t) $ are functions of time we write $a \approx b$ if there exists a non-negative polynomial $C(t)$ with coefficients depending only on $\sigma$ so that $\abs{a(t)-b(t)} \leq C(K(t))\Ecal^\half(t)\cbrac{\Ecal^\half(t) + \norm[H^{3.5+N}]{\vvar}}$. Observe that $\approx$ is an equivalence relation. Hence now by using \lemref{lem:timederivdel} and doing a similar computation as done in \propref{prop:apriori} we obtain
\begin{align*}
& \frac{\diff \EcalfourfiveN}{\diff t} \\
& \approx \int \cbrac{\frac{1}{\cvar}\cbrac{(\cvar\pap)^{3+N}\vvar - \avar(\cvar\pap)^{3+N}\g }}\nobrac{\Dtdel\cbrac{(\cvar\pap)^{3+N}\vvar - \avar(\cvar\pap)^{3+N}\g }} \diff \ap \\
& \quad + \sigma\int \brac{\papabs (\cvar\pap)^{3+N}\g} \Dtdel(\cvar\pap)^{3+N}\g \diff \ap \\
& \approx \int \cbrac{\frac{1}{\cvar}\cbrac{(\cvar\pap)^{3+N}\vvar - \avar(\cvar\pap)^{3+N}\g }}\nobrac{\cbrac{-i\sigma(\cvar\pap)^2\Hil(\cvar\pap)^{1+N}\Jdel^2(\cvar\pap)^2 g - \ep(\cvar\pap)^{3+N}\papabs\vvar }} \diff \ap \\
& \quad - \sigma\int \brac{\papabs (\cvar\pap)^{3+N}\g} \Jdel^2(\cvar\pap)\cbrac{ (\cvar\pap)^{3+N}\vvar - \avar(\cvar\pap)^{3+N} g} \diff \ap \\
& \approx \int \cbrac{\frac{1}{\cvar}\cbrac{(\cvar\pap)^{3+N}\vvar - \avar(\cvar\pap)^{3+N}\g }}\nobrac{\cbrac{-i\sigma(\cvar\pap)^2\Hil \Jdel^2(\cvar\pap)^{3+N} g - \ep(\cvar\pap)^{3+N}\papabs\vvar }} \diff \ap \\
& \quad - \sigma\int \brac{\papabs (\cvar\pap)^{3+N}\g} \Jdel^2(\cvar\pap)\cbrac{ (\cvar\pap)^{3+N}\vvar - \avar(\cvar\pap)^{3+N} g} \diff \ap \\
& \approx -\ep\int \frac{1}{\cvar} \cbrac{(\cvar\pap)^{3+N}\vvar}(\cvar\pap)^{3+N}\papabs\vvar \diff\ap \\
& \approx -\ep\int \cvar^{5+2N}\abs{\papabs^\half \pap^{3+N}\vvar}^2 \diff\ap
\end{align*}
Hence we have the inequality
\begin{align*}
 \frac{\diff \EcalfourfiveN}{\diff t} \leq C(K)\Ecal(t) + C(K)\Ecal^\half(t)\norm[H^{3.5+N}]{\vvar} -\ep\norm[\infty]{\frac{1}{\cvar}}^{-(5+2N)}\norm[H^{3.5+N}]{\vvar}^2
\end{align*}
As $\ep>0$ the estimate follows. 

\medskip
\textbf{Step 4:} We now prove the energy estimate for the case of $\del = 0$ and $0\leq \ep\leq 1$. We will use the following notation: If $a(t), b(t) $ are functions of time we write $a \approx b$ if there exists a non-negative polynomial $C(t)$ with coefficients depending only on $\sigma$ so that $\abs{a(t)-b(t)} \leq C(K(t))\Ecal(t)$. Observe that $\approx$ is an equivalence relation. Hence now by using \lemref{lem:timederivdel} and doing a similar computation as done in \propref{prop:apriori} we obtain
\begin{align*}
& \frac{\diff \EcalfourfiveN}{\diff t} \\
& \approx \int \frac{1}{\cvar}\cbrac{(\cvar\pap)^{3+N}\vvar - \avar(\cvar\pap)^{3+N}\g}\cbrac{-\ep(\cvar\pap)^{3+N}\papabs\vvar} \diff\ap \\
& \approx -\ep\norm[\Hhalf]{\frac{1}{\cvar^\half}(\cvar\pap)^{3+N}\vvar}^2 + \ep\int\papabs^\half\cbrac{\avar\cvar^\half(\cvar\pap)^{3+N}\g}\papabs^\half\cbrac{\frac{1}{\cvar^\half}(\cvar\pap)^{3+N}\vvar} \diff \ap
\end{align*}
Hence we have the inequality
\begin{align*}
 \frac{\diff \EcalfourfiveN}{\diff t} \leq C(K)\Ecal(t) + \ep C(K)\Ecal^\half(t)\norm[\Hhalf]{\frac{1}{\cvar^\half}(\cvar\pap)^{3+N}\vvar} -\ep\norm[\Hhalf]{\frac{1}{\cvar^\half}(\cvar\pap)^{3+N}\vvar}^2
\end{align*}
as $0\leq \ep\leq 1$ the estimate follows. 
\end{proof}

We now prove the a priori estimate for the difference of two solutions for system \eqref{eq:systemtwodelep}. Let $(\gone, \vvarone)(t)$ and $(\gvartwo,\vvartwo)(t)$ be two solutions of \eqref{eq:systemtwodelep} with parameters $(\del_1,\ep_1)$ and $(\del_2,\ep_2)$ respectively. The energy for the difference $\EcalDelta(t)$ is once again defined by \eqref{eq:energydelta}. We now have
\begin{prop}\label{prop:apriorideltadelep}
Let $(\gone,\vvarone)(t)$ and $(\gtwo,\vvartwo)(t)$ be two solutions of \eqref{eq:systemtwo} with parameters $(\del_1,\ep_1)$ and $(\del_2,\ep_2)$ respectively in the time interval $[0,T]$. Assume that $(\g_i,\vvar_i) \in C^l([0,T], H^{3.5 - \threebytwo l}\times H^{3 - \threebytwo l})$ for $l=0,1$, for both $i=1,2$. Let $M>0$ be a constant so that for any $t\in[0,T]$
\begin{align*}
 \norm[H^{3.5}]{\gone(\cdot,t)} + \norm[H^3]{\vvarone(\cdot,t)} +  \norm[H^{3.5}]{\gtwo(\cdot,t)} + \norm[H^3]{\vvartwo(\cdot,t)} \leq M
\end{align*}
\begin{enumerate}
\item If $0\leq \del_1,\del_2\leq 1$ and $0 < \ep = \ep_1 = \ep_2 \leq 1$, then  for all $t\in [0,T)$ we have
\begin{align*}
\frac{\diff \EcalDelta(t)}{\diff t} \leq C_\ep(M)\brac{\EcalDelta(t) + \max\cbrac{\del_1^\half,\del_2^\half}}
\end{align*}
where $C_\ep(M)$ is a constant depending on $M,\ep$ and $\sigma$.
\item If $\del_1 = \del_2 = 0$ and $0\leq \ep_1,\ep_2\leq 1$, then for all $t\in [0,T)$ we have
\begin{align*}
\frac{\diff \EcalDelta(t)}{\diff t} \leq C(M)\brac{\EcalDelta(t) + \max\cbrac{\ep_1,\ep_2}}
\end{align*}
where $C(M)$ is a constant depending only on $M$ and $\sigma$.
\end{enumerate}
\end{prop}
\begin{proof}
The proof of this proposition is similar to the proof of \propref{prop:aprioridelta} and we will mostly focus on the changes that we need to make. We will first do the computation for for the first case namely $0\leq \del_1,\del_2 \leq 1$ and $0<\ep \leq 1$.We will freely use \lemref{lem:interpolationmult}  and \lemref{lem:commutatordel}  to simplify the computations.

\textbf{Step 1:} As before we first control the quantities controlled by the energy. This is essentially the same as done in step 1 of the proof of \propref{prop:aprioridelta} so we will just summarize the estimates. This applies for all $0\leq \del_1,\del_2, \ep_1, \ep_2 \leq 1$. We again note that $C(M)$ will denote a constant depending only on $M$ and $\sigma$. 
\begin{align*}
\norm[H^{3.5}]{\cvar_i -1} + \norm[H^{3.5}]{\frac{1}{\cvar_i} -1}+  \norm[H^{3.5}]{\w_i - 1} &\leq C(M)\\
\norm[H^{3}]{\avar_i}  + \norm[H^{3}]{\dvar_i} + \norm[H^{3}]{\bvar_i} + \norm[H^{3}]{(\Aone)_i - 1} + \norm[H^{3}]{(\etwo)_i} & \leq C(M)\\
\norm[H^{2}]{\pt \g_i} + \norm[H^{1.5}]{\pt \vvar_i} + \norm[H^{2}]{\pt\cvar_i} & \leq C(M)
\end{align*}
and now the difference of quantities
\begin{align*}
\norm[H^2]{\gone - \gtwo} + \norm[H^{1.5}]{\vvarone - \vvartwo} + \norm[H^2]{\cvarone - \cvartwo} + \norm[H^2]{\wone - \wtwo} & \leq C(M)\EcalDelta^\half \\
\norm[H^{1.5}]{\avarone - \avartwo} + \norm[H^{1.5}]{\dvarone - \dvartwo} + \norm[H^{1.5}]{\bvarone - \bvartwo} & \leq C(M)\EcalDelta^\half \\
\norm[H^{1.5}]{(\Aone)_1 - (\Aone)_2}  + \norm[H^{1.5}]{(\etwo)_1 - (\etwo)_2} & \leq C(M)\EcalDelta^\half \\
\end{align*}
For the difference of time derivatives, the estimate becomes a little different
\begin{align*}
\norm[H^\half]{\pt(\gone - \gtwo)} + \norm[2]{\pt (\vvarone - \vvartwo)} + \norm[H^\half]{\pt (\cvarone - \cvartwo)} & \leq C(M)\cbrac{\EcalDelta^\half + \max\cbrac{\del_1,\del_2} + \max\cbrac{\ep_1,\ep_2}}
\end{align*}

\medskip
\textbf{Step 2:} We now establish some estimates for the case of $0\leq \del_1,\del_2 \leq 1$ and $0<\ep = \ep_1 = \ep_2\leq 1$. We define the following notation:  If $a,b:\Rsp\times[0,T] \to \Csp$ are functions we write $a \approx_{\Ltwo} b$ if there exists a constant $C(M)$   depending only on $M$ and $\sigma$ such that $\norm[2]{a-b} \leq C(M)\cbrac{\EcalDelta(t)^\half + \norm[H^{2}]{\vvarone - \vvartwo} + \max\cbrac{\del_1^\half,\del_2^\half}}$. Observe that $\approx_{\Ltwo}$ is an equivalence relation. We define $\Dtdelone = \pt + \Jdelone^2(\bvarone\pap)$ and $\Dtdeltwo = \pt + \Jdeltwo^2(\bvartwo\pap)$
\begin{enumerate}
\item Recalling the definition of $\zeta(\ap,t)$ from \eqref{eq:zeta} and by an analogous computation done in \propref{prop:aprioridelta} and using \lemref{lem:commutatordel} we see that
\begin{align*}
& \Dtdelone \zeta \\
& \approxLtwo -i\sigma\papabs^\half(\cvarone\pap)\Hil\cbrac{\Jdelone^2(\cvarone\pap)^2\gone - \Jdeltwo^2(\cvartwo\pap)^2\gtwo} \\
& \quad  - \ep\papabs^\half\cbrac{(\cvarone\pap)\papabs\vvarone - (\cvartwo\pap)\papabs\vvartwo} \\
& \approxLtwo  -i\sigma\papabs^\half(\cvarone\pap)\Hil\cbrac{\Jdelone^2(\cvarone\pap)^2\gone - \Jdeltwo^2(\cvartwo\pap)^2\gtwo} -\ep\cvarone\pap\papabs^\threebytwo(\vvarone - \vvartwo)
\end{align*}
\item We also have
\begin{align*}
& \Dtdelone\cbrac{(\cvarone\pap)^2\gone - (\cvartwo\pap)^2\gtwo} \\
& \approxLtwo -\cvarone\pap\cbrac{(\cvarone\pap)\cbrac{\Jdelone^2(\cvarone\pap)\vvarone - \Jdeltwo^2(\cvartwo\pap)\vvartwo} - \avarone(\cvarone\pap)\cbrac{\Jdelone^2(\cvarone\pap)\gone - \Jdeltwo^2(\cvartwo\pap)\gtwo}}
\end{align*}
\item Similarly we have
\begin{align*}
& -i\sigma\Hil\pap\cbrac{\cvarone\papabs^\half \zeta} \\
& \approxLtwo \sigma\pap\cbrac{(\cvarone\pap)\cbrac{(\cvarone\pap)\vvarone - (\cvartwo\pap)\vvartwo} - \avarone(\cvarone\pap)\cbrac{(\cvarone\pap)\gone - (\cvartwo\pap)\gtwo}}
\end{align*}
\end{enumerate}

\medskip
\textbf{Step 3:} We now prove the energy estimate for the case of $0\leq \del_1,\del_2 \leq 1$ and $0<\ep\leq 1$. We only control the highest order energy as the lower order ones are easily controlled. To simplify the calculations we will use the following notation: If $a(t), b(t) $ are functions of time we write $a \approx b$ if there exists a constant $C(M)$ depending only on $M$ and $\sigma$ so that $\abs{a(t)-b(t)} \leq C(M)\cbrac{\EcalDelta(t) + \EcalDelta^\half(t)\norm[H^{2}]{\vvarone - \vvartwo} + \max\cbrac{\del_1^\half,\del_2^\half}}$.  Hence now by using \lemref{lem:timederivdel},  \lemref{lem:commutatordel} and doing a similar computation as done in \propref{prop:apriori} we obtain
\begin{align*}
& \frac{\diff \EcalDeltathree}{\diff t} \\
& \approx \int (\zeta)\cbrac{-i\sigma\papabs^\half (\cvarone\pap) \Hil\cbrac{\Jdelone^2(\cvarone\pap)^2\gone - \Jdeltwo^2(\cvartwo\pap)^2\gtwo } -\ep\cvarone\pap\papabs^\threebytwo(\vvarone - \vvartwo)} \diff \ap  \\
& \quad - \sigma \int \cbrac{(\cvarone\pap)^2\gone - (\cvartwo\pap)^2\gtwo}\pap\Big\{ (\cvarone\pap)\cbrac{\Jdelone^2(\cvarone\pap)\vvarone - \Jdeltwo^2(\cvartwo\pap)\vvartwo} \\
& \qquad \qquad \quad - \avarone(\cvarone\pap)\cbrac{\Jdelone^2(\cvarone\pap)\gone - \Jdeltwo^2(\cvartwo\pap)\gtwo}\Big\} \diff\ap \\
& \approx -i\sigma\int (\zeta)\cbrac{\papabs^\half (\cvarone\pap) \Hil\Jdelone^2\cbrac{(\cvarone\pap)^2\gone - (\cvartwo\pap)^2\gtwo }} \diff \ap   -\ep \int (\zeta)\cbrac{\cvarone\pap\papabs^\threebytwo(\vvarone - \vvartwo)}\diff\ap \\
& \quad - \sigma \int \cbrac{(\cvarone\pap)^2\gone - (\cvartwo\pap)^2\gtwo}\pap\Jdelone^2\Big\{ (\cvarone\pap)\cbrac{(\cvarone\pap)\vvarone - (\cvartwo\pap)\vvartwo} \\
& \qquad \qquad \quad - \avarone(\cvarone\pap)\cbrac{(\cvarone\pap)\gone - (\cvartwo\pap)\gtwo}\Big\} \diff\ap \\
& \approx -\ep\int \papabs^\half\cbrac{(\cvarone\pap)\vvarone - (\cvartwo\pap)\vvartwo} \cbrac{\cvarone\pap\papabs^\threebytwo(\vvarone - \vvartwo)}\diff\ap \\
& \approx -\ep\int \cvarone^2\abs*[\Big]{\pap^2(\vvarone - \vvartwo)}^2 \diff\ap
\end{align*}
Hence we have the estimate
\begin{align*}
\frac{\diff\EcalDeltathree}{\diff t} \leq C(M)\EcalDelta(t) + C(M)\EcalDelta^\half(t)\norm[H^2]{\vvarone - \vvartwo} + C(M)\max\cbrac{\del_1^\half, \del_2^\half} - \ep\norm[\infty]{\frac{1}{\cvarone}}^{-2}\norm[H^2]{\vvarone - \vvartwo}^2 
\end{align*}
As $\ep >0$  the estimate follows. 

\medskip
\textbf{Step 4:} The proof of the energy estimate for $\del_1 = \del_2 = 0$ and $0\leq \ep_1 , \ep_2 \leq 1$ is exactly the same as proved in \propref{prop:aprioridelta} as all the terms involving $\ep_1,\ep_2$ can be controlled by  $C(M)\max\cbrac{\ep_1,\ep_2}$.  
\end{proof}

We are now in a position to prove the local existence in Sobolev spaces. 
\begin{thm}\label{thm:existencesobolev}
For $\sigma>0$ we have the following
\begin{enumerate}
\item Let $s\geq 3$ and let $(\g,\vvar)(0) \in H^{s+\half}(\Rsp)\times H^s(\Rsp)$. Then there exists a time $T>0$ so that the initial value problem to \eqref{eq:systemtwo} has a unique solution $(\g,\vvar) \in C^l([0,T], H^{s+\half - \threebytwo l}\times H^{s-\threebytwo l})$ for $l = 0,1$. Moreover if $T_{max}$ is the maximum time of existence then either $T_{max}=\infty$ or $T_{max} < \infty$ with
\begin{align*}
\sup_{t\in [0,T_{max})}\cbrac{ \norm[H^{3.5}]{\g(\cdot,t)} + \norm[H^3]{\vvar(\cdot,t)}} = \infty
\end{align*}
\item If $(\g_1,\vvar_1)(t)$ and $(\g_2,\vvar_2)(t)$ are two solutions of \eqref{eq:systemtwo} in $[0,T]$ with 
\begin{align*}
\sup_{t\in [0,T)} \cbrac{\norm[H^{3.5}]{\g_1(\cdot,t)} + \norm[H^3]{\vvar_1(\cdot,t)} + \norm[H^{3.5}]{\g_2(\cdot,t)} + \norm[H^3]{\vvar_2(\cdot,t)} } = M <\infty
\end{align*}
Then there is constant $C(M)$ depending only on $M$ and $\sigma$ such that
\begin{align*}
& \sup_{t\in [0,T)}\cbrac{ \norm[H^{2}]{\g_1(\cdot,t) - \g_2(\cdot,t)} + \norm[H^{1.5}]{\vvar_1(\cdot,t) -\vvar_2(\cdot,t) }} \\
& \leq  e^{C(M)T}\cbrac{\norm[H^{2}]{\g_1(\cdot,0) - \g_2(\cdot,0)} + \norm[H^{1.5}]{\vvar_1(\cdot,0) -\vvar_2(\cdot,0) }}
\end{align*}
\end{enumerate}
\end{thm}
\begin{proof}
The proof of this result is essentially the same as the proof of Theorem 5.6 in Ambrose \cite{Am03}. There are a few minor changes that need to be made which we now describe.
\begin{enumerate}
\item First observe that even though the result in \cite{Am03} is for periodic solutions, the existence proof does not use compactness to prove existence. So the fact that we are working on $\Rsp$ makes no difference.
\item First fix $\ep,\del >0$ and consider the mollified initial data $(\g^{\ep,\del},\vvar^{\ep,\del})(0) = (J_\ep( \g(0)), J_\ep (\vvar(0)))$. For this smooth initial data we first prove a local existence result to the system \eqref{eq:systemtwodelep} by a standard  Picard iteration scheme by writing the corresponding integral equation and treating the $J_\del^2$ terms in the equations as forcing terms. We hence obtain smooth solutions $(\g^{\del,\ep}, \vvar^{\del,\ep})(t)$ to \eqref{eq:systemtwodelep} for $t\in[0,T]$ where $T$ depends on $\ep, \del$. Then by using the first estimate from \propref{prop:aprioridelep} we see that in fact $T$ is independent of $\del$. 
Hence using \propref{prop:apriorideltadelep} and by following the approach of Ambrose \cite{Am03} we obtain unique smooth solutions $(\g^\ep,\vvar^\ep)(t)$ to the system \eqref{eq:systemtwodelep} for $\del=0$ in the time interval $[0,T]$ where T depends on $\ep$.  
\item Now by the second a priori estimate in \propref{prop:aprioridelep} for $\del =0$, we see that $T$ is in fact independent of $\ep$ and now by using  \propref{prop:apriorideltadelep} and by following the approach of Ambrose \cite{Am03} we obtain a unique solution $(\g,\vvar) \in \Linfty([0,T], H^{s+\half}\times H^{s})$ to the system \eqref{eq:systemtwo} with $(\pt\g, \pt\vvar) \in \Linfty([0,T], H^{s -1}\times H^{s-\threebytwo})$. Then by following the approach of  Ambrose \cite{Am03} using the time reversible nature of the system  \eqref{eq:systemtwo} we obtain  $(\g,\vvar) \in C^l([0,T], H^{s+\half - \threebytwo l}\times H^{s-\threebytwo l})$ for $l = 0,1$.
\item The blow up criterion follows from  \propref{prop:aprioridelep} and part two of the theorem follows from  \propref{prop:aprioridelta}
\end{enumerate}
\end{proof}
\begin{cor}\label{cor:existencesobolev}
For $\sigma>0$ we have the following
\begin{enumerate}
\item Let $s\geq 3$ and let $(\Zap-1,\frac{1}{\Zap} - 1, \Zt)(0) \in H^{s+\half}(\Rsp)\times H^{s+\half}(\Rsp)\times H^s(\Rsp)$. Then there exists a time $T>0$ so that the initial value problem to \eqref{eq:systemone} has a unique solution $(\Z,\Zt)(t)$ satisfying $(\Zap-1,\frac{1}{\Zap} - 1, \Zt) \in C^l([0,T], H^{s+\half - \threebytwo l}(\Rsp)\times H^{s+\half - \threebytwo l}(\Rsp)\times H^{s - \threebytwo l}(\Rsp))$ for $l = 0,1$. Moreover if $T_{max}$ is the maximum time of existence then either $T_{max}=\infty$ or $T_{max} < \infty$ with
\begin{align*}
\sup_{t\in [0,T_{max})}\cbrac{ \norm[H^{3.5}]{\Zap - 1}(t) + \norm[H^{3.5}]{\frac{1}{\Zap} - 1 }(t) + \norm[H^3]{\Zt}(t)} = \infty
\end{align*}
\item Let $(\Z^1,\Zt^1)(t)$ and $(\Z^2,\Zt^2)(t)$ be two solutions of \eqref{eq:systemone} in $[0,T]$ with 
\begin{align*}
\sup_{t\in [0,T_{max})}\cbrac{ \norm[H^{3.5}]{\nobrac{\Zap^i - 1}}(t) + \norm[H^{3.5}]{\nobrac{\frac{1}{\Zap^i} - 1}}(t) + \norm[H^3]{\nobrac{\Zt^i}}(t)} \leq M <\infty
\end{align*}
for both $i=1,2$ for some $M>0$. Then there is constant $C(M)$ depending only on $M$ and $\sigma$ such that
\begin{align*}
& \sup_{t\in [0,T)}\cbrac{ \norm[H^{2}]{\Zap^1 - \Zap^2}(t) + \norm[H^{2}]{\frac{1}{\Zap^1} - \frac{1}{\Zap^2}}(t) + \norm[H^{1.5}]{\Zt^1 - \Zt^2}(t)} \\
& \leq  e^{C(M)T}\cbrac{ \norm[H^{2}]{\Zap^1 - \Zap^2}(0) + \norm[H^{2}]{\frac{1}{\Zap^1} - \frac{1}{\Zap^2}}(0) + \norm[H^{1.5}]{\Zt^1 - \Zt^2}(0)}
\end{align*}
(Note that the inequality above is for a weaker norm than the one the problem is stated for, however it is still sufficient to prove uniqueness.)
\end{enumerate}
\end{cor}
\begin{proof}
The first part is a direct consequence of \thmref{thm:existencesobolev} and \lemref{lem:systemequiv}. The second part follows from \thmref{thm:existencesobolev} and an easy modification of the argument of \lemref{lem:systemequiv}
\end{proof}

\medskip
\section{Proof of \thmref{thm:existence} and \corref{cor:example}}\label{sec:proof}
 
\begin{proof}[Proof of \thmref{thm:existence}]
Let $\ep>0$ and consider the mollified initial data given by $(\Z^\ep,\Zt^\ep)(0) = (P_\ep\conv\Z, P_\ep\conv\Zt)(0)$ where $P_\ep$ is the Poisson kernel \eqref{eq:Poissonkernel}. Observe that there exists an $\ep_0>0$ small enough so that for all $0<\ep\leq \ep_0$ we have
\begin{align*}
\Ecalsigma(\Z^\ep,\Zt^\ep)(0)  \leq 2\Ecalsigma(\Z,\Zt)(0) 
\end{align*}
Define 
\begin{align*}
M = \norm[2]{\frac{1}{\Zap} - 1}(0) + \norm[2]{\Zt}(0) + \norm[\infty]{\Zap}(0) < \infty
\end{align*}
We observe that for all $0<\ep\leq \ep_0$ we have
\begin{align*}
\norm[2]{\frac{1}{\Zap^\ep} - 1}(0) + \norm[2]{\Zt^\ep}(0) +  \norm[\infty]{\Zap^\ep}(0) \leq M
\end{align*} 
Now fix some $0<\ep\leq \ep_0$. Hence using \corref{cor:existencesobolev} we see that there exists a time $T_\ep>0$, so that the initial value problem to \eqref{eq:systemone} with initial data $(\Z^\ep,\Zt^\ep)(0)$ has a unique smooth solution $(\Z^\ep,\Zt^\ep)(t)$ in $[0,T_\ep]$ so that for all $s\geq 3$
\begin{align*}
\sup_{t\in[0,T_\ep]} \cbrac{\norm[H^{s+\half}]{\Zap^\ep - 1}(t) + \norm[H^{s+\half}]{\frac{1}{\Zap^\ep} -1}(t) + \norm[H^s]{\Zt^\ep}(t) } < \infty
\end{align*} 
Therefore by \thmref{thm:aprioriEsigma} we see that for all $t\in[0,T_\ep)$ we have
\begin{align*}
\frac{\diff \Esigma(\Z^\ep,\Zt^\ep)(t)}{\diff t} \leq P(\Esigma(\Z^\ep,\Zt^\ep)(t))
\end{align*}
Hence by using \propref{prop:equivEsigma}, \lemref{lem:equivsobolev} and the blow up criterion from \corref{cor:existencesobolev} we that there exists $T,C_1 > 0$ both depending only on $\Ecalsigma(0)$ so that $(\Z^\ep,\Zt^\ep)(t)$ in fact exists in $[0,T]$ with $\sup_{t\in[0,T]}\Ecalsigma(\Z^\ep,\Zt^\ep)(t) \leq C_1$. Also using \lemref{lem:equivsobolev} we see that 
\begin{align*}
\sup_{t\in[0,T]} \cbrac{\norm[H^{3.5}]{\Zap^\ep - 1}(t) + \norm[H^{3.5}]{\frac{1}{\Zap^\ep} -1}(t) + \norm[H^3]{\Zt^\ep}(t) } \leq C_2
\end{align*} 
where $C_2$ depends only on $C_1, M,T$ and $\sigma$.  Hence by passing to the limit $\ep \to 0$ using \corref{cor:existencesobolev} we have a unique solution $(\Z,\Zt)(t)$ in $[0,T]$ to \eqref{eq:systemone} with $\sup_{t\in[0,T]}\Ecalsigma(\Z,\Zt)(t) \leq C_1$ and 
\begin{align*}
\sup_{t\in[0,T]} \cbrac{\norm[H^{3.5}]{\Zap - 1}(t) + \norm[H^{3.5}]{\frac{1}{\Zap} -1}(t) + \norm[H^3]{\Zt}(t) } \leq C_2
\end{align*} 
thereby proving the result. 
\end{proof}

\begin{proof}[Proof of \corref{cor:example}]
Without loss of generality we assume that $c=1$ and define $\dis \tau = \frac{\sigma}{\ep^{3/2}}$ which implies that $\tau\leq 1$. Observe that the result for $\sigma=0$ and $0<\ep\leq 1$ follows directly from Theorem 3.9 of Wu \cite{Wu19}. Hence from now on we assume $\sigma>0$ and can therefore use \thmref{thm:existence}.  Using  \thmref{thm:existence}  we only need to show that if $\tau\leq 1$ then
\begin{align*}
\Ecalsigma(\Z^{\ep,\sigma}, \Zt^{\ep,\sigma})(0) \leq C(M) 
\end{align*}
where $C(M)$ is a constant depending only on $M$. We now prove this estimate. 

To simplify the proof we will suppress the dependence of $\Mconst$ in the inequalities i.e. when we write $a \lesssim b$, we mean that there exists a constant $C(M)$ depending only on $\Mconst$ such that $a\leq C(M)b$. As we only need to prove the estimates for $t=0$, we will suppress the time dependence of the solutions e.g. we will write $(\Z\conv P_\epsilon, \Zt\conv P_\epsilon)\vert_{t=0}$ by $(\Z,\Zt)_\ep$ for simplicity. 

\begin{enumerate}[leftmargin =*, align=left]

\item We first observe that for any $\ep>0$ we have
\begin{align*}
\norm[2]{\brac{\pap\frac{1}{\Zap}}_{\n\ep}} \lesssim \sup_{\yp<0}\norm[\Ltwo(\Rsp, \diff \xp)]{\partial_z \brac{\frac{1}{\Psi_z}}} \lesssim 1
\end{align*}
Similarly we have
\begin{align*}
& \norm[\Hhalf]{\brac{\frac{1}{\Zap}\pap\frac{1}{\Zap}}_{\n\ep}}^2 \\
& \lesssim \norm[2]{\brac{\pap\frac{1}{\Zap}}_{\n\ep}} \norm[2]{\brac{\frac{1}{\Zap^2}\pap^2\frac{1}{\Zap}}_{\n\ep}} +  \norm[2]{\brac{\pap\frac{1}{\Zap}}_{\n\ep}}^2 \norm[\Linfty]{\brac{\frac{1}{\Zap}\pap\frac{1}{\Zap}}_{\n\ep}} \\
& \lesssim \sup_{\yp<0}\norm[\Ltwo(\Rsp, \diff \xp)]{\partial_z \brac{\frac{1}{\Psi_z}}}\sup_{\yp<0}\norm[\Ltwo(\Rsp, \diff \xp)]{\frac{1}{\Psizp^2}\partial_z^2 \brac{\frac{1}{\Psi_z}}} \\
& \quad + \sup_{\yp<0}\norm[\Ltwo(\Rsp, \diff \xp)]{\partial_z \brac{\frac{1}{\Psi_z}}}^2\sup_{\yp<0}\norm[\Linfty(\Rsp, \diff \xp)]{\frac{1}{\Psizp}\partial_z \brac{\frac{1}{\Psi_z}}} \\
& \lesssim 1
\end{align*}

\item Observe that
\begin{align*}
\sup_{\yp<0}\norm[\Linfty(\Rsp,\diff \xp)]{\frac{1}{\Psi_z}} \lesssim \sup_{\yp<0}\norm[\Ltwo(\Rsp,\diff \xp)]{\frac{1}{\Psi_z} - 1}^\half \sup_{\yp<0}\norm[\Ltwo(\Rsp, \diff \xp)]{\partial_z \brac{\frac{1}{\Psi_z}}}^\half \lesssim 1
\end{align*}
Hence using $ \sup_{\yp<0}\norm[H^{3.5}(\Rsp,\diff \xp)]{\U} \leq M$ and that $0<\sigma\leq 1$ we obtain
\begin{align*}
\norm[2]{\brac{\Ztapbar}_{\ep}}^2 + \norm*[\Bigg][2]{\brac{\frac{1}{\Zap^2}\pap\Ztapbar}_{\n\ep}}^2 + \norm*[\Bigg][2]{\brac{\frac{\sigma^\half}{\Zap^\half}\pap\Ztapbar}_{\n\ep}}^2 + \norm*[\Bigg][2]{\brac{\frac{\sigma^\half}{\Zap^\fivebytwo}\pap^2\Ztapbar}_{\n\ep}}^2 \lesssim 1
\end{align*}

\item By \lemref{lem:conv} we have
\begin{align*}
\norm[2]{\brac{\sigma^\onebyfour\Zap^\threebyfour\pap\frac{1}{\Zap}}_{\n\ep}} + \norm[\infty]{\brac{\sigma^\sevenbytwelve\Zap^\threebyfour\pap\frac{1}{\Zap}}_{\n\ep}}  \lesssim \sup_{\yp<0}\norm[\Leightbyseven(\Rsp, \diff \xp)]{\Psi_z^\threebyfour\partial_z \brac{\frac{1}{\Psi_z}}}\brac{\tau^\onebyfour + \tau^\sevenbytwelve} \lesssim 1
\end{align*}
Now using \eqref{form:RealImagTh} we obtain
\begin{align*}
 \norm[2]{\brac*[\Bigg]{\sigma^\onebyfour \frac{1}{\Zap^\onebyfour}\pap\w }_{\n\ep}} + \norm[\infty]{\brac*[\Bigg]{\sigma^\sevenbytwelve \frac{1}{\Zap^\onebyfour}\pap\w }_{\n\ep}} \lesssim 1
\end{align*}

\item Using \lemref{lem:conv} we obtain
\begin{align*}
& \norm[2]{\brac{\sigma^\onebysix\Zap^\half\pap\frac{1}{\Zap}}_{\n\ep}} + \norm[\infty]{\brac{\sigma^\half\Zap^\half\pap\frac{1}{\Zap}}_{\n\ep}} +  \norm[2]{\sigma^\fivebysix\pap\brac{\Zap^\half\pap\frac{1}{\Zap}}_{\n\ep}} \\
& \lesssim \sup_{\yp<0}\norm[\Lfourbythree(\Rsp, \diff \xp)]{\Psi_z^\half\partial_z \brac{\frac{1}{\Psi_z}}}\brac{\tau^\onebysix + \tau^\half + \tau^\fivebysix} \\
& \lesssim 1
\end{align*}
Now using \eqref{form:RealImagTh} we obtain
\begin{align*}
 \norm[2]{\brac*[\Bigg]{\sigma^\onebysix \frac{1}{\Zap^\half}\pap\w }_{\n\ep}} + \norm[\infty]{\brac*[\Bigg]{\sigma^\half \frac{1}{\Zap^\half}\pap\w }_{\n\ep}} \lesssim 1
\end{align*}
We observe that
\begin{align*}
& \norm[2]{\brac{\sigma^\fivebysix\Zap^\half\pap^2\frac{1}{\Zap}}_{\n\ep}} \\
& \lesssim \norm[2]{\sigma^\fivebysix\pap\brac{\Zap^\half\pap\frac{1}{\Zap}}_{\n\ep}} + \norm[2]{\brac{\sigma^\onebyfour\Zap^\threebyfour\pap\frac{1}{\Zap}}_{\n\ep}}\norm[\infty]{\brac{\sigma^\sevenbytwelve\Zap^\threebyfour\pap\frac{1}{\Zap}}_{\n\ep}} \\
& \lesssim 1
\end{align*}
and in particular again by using \eqref{form:RealImagTh} we obtain
\begin{align*}
& \norm[2]{\brac{\sigma^\fivebysix \Zap^\half\pap\Dap\w}_{\n\ep}} \\
& \lesssim  \norm[2]{\brac{\sigma^\fivebysix \Zap^\half\pap\brac{\w\pap\frac{1}{\Zap}}}_{\n\ep}} \\
& \lesssim  \norm[2]{\brac{\sigma^\fivebysix\Zap^\half\pap^2\frac{1}{\Zap}}_{\n\ep}} +  \norm[2]{\brac*[\Bigg]{\sigma^\onebyfour \frac{1}{\Zap^\onebyfour}\pap\w }_{\n\ep}}\norm[\infty]{\brac{\sigma^\sevenbytwelve\Zap^\threebyfour\pap\frac{1}{\Zap}}_{\n\ep}} \\
& \lesssim 1
\end{align*}
and hence we have
\begin{align*}
& \norm[2]{\sigma^\fivebysix\pap\brac*[\Bigg]{\frac{\w_\ap}{\Zap^\half}}_{\n\ep}} \\
& \lesssim \norm[2]{\brac{\sigma^\fivebysix \Zap^\half\pap\Dap\w}_{\n\ep}} +   \norm[2]{\brac*[\Bigg]{\sigma^\onebyfour \frac{1}{\Zap^\onebyfour}\pap\w }_{\n\ep}}\norm[\infty]{\brac{\sigma^\sevenbytwelve\Zap^\threebyfour\pap\frac{1}{\Zap}}_{\n\ep}} \\
& \lesssim 1
\end{align*}

%


\item Using \lemref{lem:conv} we have
\begin{align*}
& \norm[\infty]{\brac{\sigma^\onebythree\pap\frac{1}{\Zap}}_{\n\ep}} +  \norm[2]{\brac{\sigma^\twobythree\pap^2\frac{1}{\Zap}}_{\n\ep}} + \norm[\Hhalf]{\brac{\sigma\pap^2\frac{1}{\Zap}}_{\n\ep}} \\
& \lesssim \sup_{\yp<0}\norm[\Ltwo(\Rsp, \diff \xp)]{\partial_z \brac{\frac{1}{\Psi_z}}} \brac{\tau^\onebythree + \tau^\twobythree + \tau} \\
& \lesssim 1
\end{align*}
Hence using \eqref{form:RealImagTh} we also get $\dis \norm[\infty]{\brac{\sigma^\onebythree\pap\frac{1}{\Zapabs}}_\ep} \lesssim 1$

\item Using \eqref{form:RealImagTh} and \propref{prop:Hhalfweight} with $\dis f = \brac{\Zapabs^\half\pap^2\frac{1}{\Zap}}_{\n\ep}$, $\dis w = \brac{\frac{1}{\Zapabs^\half}}_{\n\ep}$ and $h = \w_\ep$ we get
\begin{align*}
& \norm[\Hhalf]{\brac{\sigma\w\pap^2\frac{1}{\Zap}}_{\n\ep}} \\
& \lesssim \norm[\Hhalf]{\brac{\sigma\pap^2\frac{1}{\Zap}}_{\n\ep}} +  \norm[2]{\brac{\sigma^\fivebysix\Zap^\half\pap^2\frac{1}{\Zap}}_{\n\ep}} \norm[\Hhalf]{\brac{\sigma^\onebysix\Zap^\half\pap\frac{1}{\Zap}}_{\n\ep}} \\
& \lesssim 1
\end{align*}
Hence using \eqref{form:Th} and \propref{prop:Leibniz} we can finally control
\begin{align*}
& \norm[\Hhalf]{\brac{\sigma\pap\Th}_{\ep}} \\
& \lesssim \norm[\Hhalf]{\sigma\pap\brac{\w\pap\frac{1}{\Zap}}_{\n\ep}} \\
& \lesssim \norm[\Hhalf]{\sigma\brac*[\Bigg]{\frac{\w_\ap}{\Zap^\half}}_{\n\ep}\brac{\Zap^\half\pap\frac{1}{\Zap}}_{\n\ep}} + \norm[\Hhalf]{\brac{\sigma\w\pap^2\frac{1}{\Zap}}_{\n\ep}} \\
& \lesssim \norm[2]{\sigma^\fivebysix\pap\brac*[\Bigg]{\frac{\w_\ap}{\Zap^\half}}_{\n\ep}}\norm[2]{\brac{\sigma^\onebysix\Zap^\half\pap\frac{1}{\Zap}}_{\n\ep}} + \norm[2]{\sigma^\onebysix\brac*[\Bigg]{\frac{\w_\ap}{\Zap^\half}}_{\n\ep}}\norm[2]{\sigma^\fivebysix\pap\brac{\Zap^\half\pap\frac{1}{\Zap}}_{\n\ep}} \\
& \quad + \norm[\Hhalf]{\brac{\sigma\w\pap^2\frac{1}{\Zap}}_{\n\ep}} \\
& \lesssim 1
\end{align*}


\item Using \lemref{lem:conv} we get
\begin{align*}
 \norm[2]{\sigma\pap\brac{\frac{1}{\Zap}\pap^2\frac{1}{\Zap}}_{\n\ep}} \lesssim \sup_{\yp<0}\norm[\Lone(\Rsp, \diff \xp)]{\frac{1}{\Psi_z}\partial_z^2 \brac{\frac{1}{\Psi_z}}}  \tau \lesssim 1
\end{align*}
Hence using \eqref{form:RealImagTh} we have
\begin{align*}
&\norm[2]{\brac{\frac{\sigma}{\Zap}\pap^3\frac{1}{\Zap}}_{\n\ep}} + \norm*[\Bigg][2]{\cbrac*[\Bigg]{\sigma\Zap^\half\pap\brac*[\Bigg]{\frac{1}{\Zap^\threebytwo}\pap^2\frac{1}{\Zap}}}_{\n\ep}} +  \norm[2]{\sigma\pap\brac{\frac{1}{\Zapabs}\pap^2\frac{1}{\Zap}}_{\n\ep}}\\
& \lesssim \norm[2]{\sigma\pap\brac{\frac{1}{\Zap}\pap^2\frac{1}{\Zap}}_{\n\ep}} + \norm[\infty]{\brac{\sigma^\onebythree\pap\frac{1}{\Zap}}_{\n\ep}} \norm[2]{\brac{\sigma^\twobythree\pap^2\frac{1}{\Zap}}_{\n\ep}} \\
& \lesssim 1
\end{align*}

\item We observe that
\begin{align*}
\norm*[\Bigg][2]{\brac*[\Bigg]{\frac{\sigma^\half}{\Zap^\half}\pap^2\frac{1}{\Zap}}_{\n\ep}}^2 \lesssim \norm[2]{\brac{\pap\frac{1}{\Zap}}_{\n\ep}}\norm[2]{\sigma\pap\brac{\frac{1}{\Zapabs}\pap^2\frac{1}{\Zap}}_{\n\ep}} \lesssim 1
\end{align*}
Hence we also have
\begin{align*}
\norm*[\Bigg][2]{\sigma^\half\pap\brac*[\Bigg]{\frac{1}{\Zap^\half}\pap\frac{1}{\Zap}}_{\n\ep}} \lesssim \norm*[\Bigg][2]{\brac*[\Bigg]{\frac{\sigma^\half}{\Zap^\half}\pap^2\frac{1}{\Zap}}_{\n\ep}} + \norm[\infty]{\brac{\sigma^\half\Zap^\half\pap\frac{1}{\Zap}}_{\n\ep}}\norm[2]{\brac{\pap\frac{1}{\Zap}}_{\n\ep}} \lesssim 1
\end{align*}


\item We observe that
\begin{align*}
\norm*[\Bigg][\Hhalf]{\brac*[\Bigg]{\frac{\sigma^\half}{\Zap^\threebytwo}\pap^2\frac{1}{\Zap}}_{\n\ep}}^2 & \lesssim \norm[2]{\brac{\frac{1}{\Zap^2}\pap^2\frac{1}{\Zap}}_{\n\ep}}\norm*[\Bigg][2]{\cbrac*[\Bigg]{\sigma\Zap^\half\pap\brac*[\Bigg]{\frac{1}{\Zap^\threebytwo}\pap^2\frac{1}{\Zap} }}_{\n\ep}} \\
& \lesssim \sup_{\yp<0}\norm[\Ltwo(\Rsp, \diff \xp)]{\frac{1}{\Psi_z^2}\partial_z^2 \brac{\frac{1}{\Psi_z}}} \norm*[\Bigg][2]{\cbrac*[\Bigg]{\sigma\Zap^\half\pap\brac*[\Bigg]{\frac{1}{\Zap^\threebytwo}\pap^2\frac{1}{\Zap} }}_{\n\ep}} \\
& \lesssim 1
\end{align*}

\item Using \lemref{lem:conv} we get
\begin{align*}
\norm*[\Bigg][2]{\sigma^\twobythree\pap\brac*[\Bigg]{\frac{1}{\Zap^2}\pap^2\frac{1}{\Zap}}_{\n\ep}} + \norm*[\Bigg][\Hhalf]{\sigma\pap\brac*[\Bigg]{\frac{1}{\Zap^2}\pap^2\frac{1}{\Zap}}_{\n\ep}}\lesssim  \sup_{\yp<0}\norm[\Ltwo(\Rsp, \diff \xp)]{\frac{1}{\Psi_z^2}\partial_z^2 \brac{\frac{1}{\Psi_z}}} \brac{\tau^\twobythree + \tau} \lesssim 1
\end{align*}
Now we use \propref{prop:Hhalfweight} with $\dis f = \brac*[\Bigg]{\frac{1}{\Zap^\half}\pap^2\frac{1}{\Zap}}_{\n\ep}$, $\dis w = \brac{\frac{1}{\Zap}}_{\n\ep}$ and $\dis h = \Zap^\half\pap\frac{1}{\Zap}$ to get
\begin{align*}
& \norm[\Hhalf]{\sigma\brac{\Zap^\half\pap\frac{1}{\Zap}}_{\n\ep} \brac*[\Bigg]{\frac{1}{\Zap^\threebytwo}\pap^2\frac{1}{\Zap}}_{\n\ep}} \\
& \lesssim \norm[\Hhalf]{\brac*[\Bigg]{\frac{\sigma^\half}{\Zap^\threebytwo}\pap^2\frac{1}{\Zap}}_{\n\ep}}\norm[\infty]{\brac{\sigma^\half\Zap^\half\pap\frac{1}{\Zap}}_{\n\ep}} +  \norm[2]{\brac*[\Bigg]{\frac{\sigma^\half}{\Zap^\half}\pap^2\frac{1}{\Zap}}_{\n\ep}}\norm*[\Bigg][2]{\sigma^\half\pap\brac*[\Bigg]{\frac{1}{\Zap^\half}\pap\frac{1}{\Zap}}_{\n\ep}} \\
& \quad +  \norm[2]{\brac*[\Bigg]{\frac{\sigma^\half}{\Zap^\half}\pap^2\frac{1}{\Zap}}_{\n\ep}}\norm[2]{\brac{\pap\frac{1}{\Zap}}_{\n\ep}}\norm[\infty]{\brac{\sigma^\half\Zap^\half\pap\frac{1}{\Zap}}_{\n\ep}} \\
& \lesssim 1
\end{align*}
Hence we finally have
\begin{align*}
 \norm[\Hhalf]{\brac{\frac{\sigma}{\Zap^2}\pap^3\frac{1}{\Zap}}_\ep} \lesssim \norm*[\Bigg][\Hhalf]{\sigma\pap\brac*[\Bigg]{\frac{1}{\Zap^2}\pap^2\frac{1}{\Zap}}_\ep} +  \norm[\Hhalf]{\sigma\brac{\Zap^\half\pap\frac{1}{\Zap}}_{\n\ep} \brac*[\Bigg]{\frac{1}{\Zap^\threebytwo}\pap^2\frac{1}{\Zap}}_{\n\ep}}  \lesssim 1
\end{align*}

%

\end{enumerate}

This completes the proof of the first part of the corollary. To see the rate of growth of curvature, observe that the $\Linfty$ norm of the curvature of the initial interface of $Z^{\ep,\sigma}(\cdot,0)$ is 
\begin{align*}
 \norm[\infty]{\kappa^{\ep,\sigma}} = \norm[\infty]{\frac{1}{(\Zapabs)_\ep}\pap(g_\ep)} \qq \tx{ where} \qq \brac{\frac{\Zap}{\abs{\Zap}}}_{\!\!\ep} = e^{ig_\ep}
\end{align*}
Now if the interface $\Z(\cdot,0)$ has an angled crest at $\ap =0$, then we see that $\pap(g_\ep)(0,0) \sim \ep^{-1} $  as $\ep \to 0$, due to $g(\cdot,0)$ having a jump for $\ep=0$ and $\g_\ep = K_\ep\conv \g$ (where $K_\ep$ is the Poisson kernel \eqref{eq:Poissonkernel}). But we know from the local description of the conformal map that $(\Zap)_\ep(0,0) \sim \ep^{\nu -1}$ as $\ep \to 0$ (see  \cite{Wi65} for a proof). Hence we see that $\norm[\infty]{\kap^{\ep,\sigma}} \sim \ep^{-\nu}$ as $\ep\to0$, thereby proving the lemma.

\end{proof}

\medskip
\section{Appendix}\label{sec:appendix}

Here we will prove all the identities and estimates used in the paper. We will state most of the statements only for functions in the Schwartz class and it can be extended to more general functions by an approximation argument. Let us first recall some of the notation used. Let $\Dt = \pt + \bvar\pap$ where $\bvar$ is given by \eqref{form:bvar} and recall that $\sqbrac{f,g ; h}$ is defined as 
\begin{align*}
\sqbrac{f_1, f_2;  f_3}(\ap) = \frac{1}{i\pi} \int \brac{\frac{f_1(\ap) - f_1(\bp)}{\ap - \bp}}\brac{\frac{f_2(\ap) - f_2(\bp)}{\ap-\bp}} f_3(\bp) \diff \bp
\end{align*}

\begin{prop}\label{prop:tripleidentity}
Let $f,g,h \in \mathcal{S}(\Rsp)$. Then we have the following identities
\begin{enumerate}
\item $h\pap[f,\Hil]\pap g = \sqbrac{h\pap f,\Hil}\pap g + \sqbrac{f, \Hil}\pap\brac{h\pap g} - \sqbrac{h, f ; \pap g} $ 

\item $\Dt [f,\Hil]\pap g = \sqbrac{\Dt f, \Hil}\pap g + \sqbrac{f, \Hil}\pap(\Dt g) - \sqbrac{\bvar, f; \pap g} $  
\end{enumerate}
\end{prop}
\begin{proof}
The second identity is a direct consequence of the first. Now we see that
\begingroup
\allowdisplaybreaks
\begin{align*}
& h(\ap)\pap[f,\Hil]\pap g  \\
& = h(\ap)\pap\brac{\frac{1}{i\pi}\int \frac{f(\ap)-f(\bp)}{\ap-\bp}\pbp g(\bp) \diff\bp } \\
& = h(\ap)f'(\ap) \brac{\frac{1}{i\pi}\int \frac{1}{\ap-\bp}\pbp g(\bp) \diff \bp} - \frac{1}{i\pi}\int \brac{\frac{h(\ap)-h(\bp)}{\ap-\bp}}\brac{ \frac{f(\ap)-f(\bp)}{\ap-\bp}} \pbp g(\bp) \diff \bp \\*
& \quad - \frac{1}{i\pi} \int \frac{f(\ap)-f(\bp)}{(\ap-\bp)^2} h(\bp)\pbp g(\bp) \diff \bp \\
& = \frac{1}{i\pi}\int \frac{h(\ap)f'(\ap) - h(\bp)f'(\bp)}{\ap-\bp} \pbp g(\bp) \diff\bp + \frac{1}{i\pi}\int  \frac{f(\ap)-f(\bp)}{\ap-\bp}\pbp(h(\bp)\pbp g(\bp) )\diff\bp \\*
& \quad  - \frac{1}{i\pi}\int \brac{\frac{h(\ap)-h(\bp)}{\ap-\bp}}\brac{ \frac{f(\ap)-f(\bp)}{\ap-\bp}} \pbp g(\bp) \diff \bp 
\end{align*}
\endgroup
\end{proof}

\begin{prop}\label{prop:Coifman} Let $H \in C^1(\Rsp),A_i \in C^1(\Rsp) $ for $i=1,\cdots m$ and $F\in C^\infty(\Rsp)$. Define 
\begin{align*}
C_1(H,A,f)(x) & = p.v. \int F\brac{\frac{H(x)-H(y)}{x-y}}\frac{\Pi_{i=1}^{m}(A_i(x) - A_i(y))}{(x-y)^{m+1}}f(y)\diff y \\
C_2(H,A,f)(x) & = p.v. \int F\brac{\frac{H(x)-H(y)}{x-y}}\frac{\Pi_{i=1}^{m}(A_i(x) - A_i(y))}{(x-y)^{m}} \partial_y f(y)\diff y
\end{align*}
then there exists constants $c_1,c_2,c_3,c_4$ depending only on $F$ and $\norm[\infty]{H'}$ so that
\begin{enumerate}
\item $\norm[2]{C_1(H,A,f)} \leq c_1\norm[\infty]{A_1'}\cdots\norm[\infty]{A_m'}\norm[2]{f}$

\item $\norm[2]{C_1(H,A,f)} \leq c_2\norm[2]{A_1'}\norm[\infty]{A_2'}\cdots\norm[\infty]{A_m'}\norm[\infty]{f}$ 

\item $\norm[2]{C_2(H,A,f)} \leq c_3\norm[\infty]{A_1'}\cdots\norm[\infty]{A_m'}\norm[2]{f}$ 

\item $\norm[2]{C_2(H,A,f)} \leq c_4\norm[2]{A_1'}\norm[\infty]{A_2'}\cdots\norm[\infty]{A_m'}\norm[\infty]{f}$ 
\end{enumerate}
\end{prop}
\begin{proof}
The first estimate is a theorem by Coifman, McIntosh and Meyer \cite{CoMcMe82}. See also chapter 9 of \cite{MeCo97}. Estimate 2 is a consequence of the Tb theorem and a proof can be found in \cite{Wu09}. The third and fourth estimates can be obtained from the first two by integration by parts. 
\end{proof}

\begin{prop}\label{prop:Lemarie}
Let $T:\Dcalsp(\Rsp) \to \Dcalsp'(\Rsp)$ be a linear operator with kernel $K(x,y)$ such that on the open set $\{(x,y):x\neq y\} \subset \Rsp\times\Rsp$, $K(x,y)$ is a function satisfying
\begin{align*}
\abs{K(x,y)} \leq \frac{C_0}{\abs{x-y}} \quad \tx{ and } \quad \abs{\grad_x K(x,y)} \leq \frac{C_0}{\abs{x-y}^2}
\end{align*}
where $C_0$ is a constant. If $T$ is continuous on $\Ltwo(\Rsp)$ with $\norm[\Ltwo\to\Ltwo]{T} \leq C_0$  and if $T(1) =0$, then $T$ is bounded on $\dot{H}^s$ for $0<s<1$ with $\norm[\dot{H}^s\to\dot{H}^s]{T} \lesssim C_0$ 
\end{prop}
\begin{proof}
This proposition is a direct consequence of the result of Lemarie \cite{Le85} where only weak boundedness of $T$ on $\Ltwo$ (in the sense of David and Journe) is assumed. As boundedness on $\Ltwo$ implies weak boundedness, the proposition follows. See also chapter 10 of \cite{MeCo97} for another proof of the result of Lemarie. 
\end{proof}

\begin{lemma}\label{lem:interpolation}
Let $r,s \in \Rsp$, $k,m\in \Zsp$. If $f \in \Scalsp(\Rsp)$, then we have the following 
\begin{enumerate}
\item $\norm[2]{\papabs^r f} \lesssim \norm[2]{f}^{\thvar}\norm[2]{\papabs^s f}^{1-\thvar}$ \quad  for $0\leq r<s $ with $1-\thvar = \frac{r}{s}$
\item  $\norm[2]{\papabs^r f} \lesssim \norm[2]{f''}^{\thvar}\norm[2]{\papabs^s f}^{1-\thvar}$ \quad  for $2\leq r<s $ with $1-\thvar = \frac{r-2}{s-2}$
\item $ \norm[\infty]{\pap^k f} \lesssim \norm[2]{f}^\thvar\norm[2]{\pap^m f}^{1-\thvar}$ \quad for $0\leq k <m$ with $1-\thvar = \frac{k+\half}{m}$
\item $ \norm[\infty]{\pap^k f} \lesssim \norm[2]{f''}^\thvar\norm[2]{\pap^m f}^{1-\thvar}$ \quad for $2\leq k <m$ with $1-\thvar = \frac{k-\threebytwo}{m-2}$
\end{enumerate}
\end{lemma}
\begin{proof}
 The first estimate is a standard interpolation estimate which can be easily proved by using the Fourier transform. We skip its proof. The second one follows from the first by applying it on the function $f''$ with $r,s$ replaced by $r-2,s-2$ respectively.

The third estimate is a consequence of the Gagliardo-Nirenberg interpolation estimate (see Theorem 12.87 in \cite{Le17}). The last one follows from the third estimate by applying it on the function $f''$ with $k,m$ replaced by $k-2,m-2$ respectively.
\end{proof}

\begin{lemma}\label{lem:interpolationmult}
Let $k,n \in \Nsp$ and $f_1,f_2, \cdots, f_k \in \Scalsp(\Rsp)$. Let $r_1,r_2\cdots, r_k \in \Zsp$ with $r_1 + \cdots + r_k = n$ and $ r_i\geq 0$ for all $1\leq i\leq k$ and. Let $r = \max\cbrac{r_1,r_2,\cdots, r_k} \geq 1$. Then 
\begin{enumerate}
\item $\norm*[\big][2]{f_1^{(r_1)}\cdots f_k^{(r_k)} } \leq C(K)\cbrac{\norm[H^s]{f_1'} + \cdots + \norm[H^s]{f_k'} }$ \quad for $s = \max\cbrac{r-1,n-2}$
\item $\norm*[\big][\Hhalf]{f_1^{(r_1)}\cdots f_k^{(r_k)} } \leq C(K)\cbrac{\norm[H^s]{f_1'} + \cdots + \norm[H^s]{f_k'} }$  \quad for $s = \max\cbrac{r-\half,n-2}$
\end{enumerate}
with $K = (\norm[\infty]{f_1} + \norm[H^1]{f_1'}) + \cdots +  (\norm[\infty]{f_k} + \norm[H^1]{f_k'})$ and $C(K)$ is a constant depending only on $K$.
\end{lemma}
\begin{proof}
Let us begin by proving the first estimate. Without loss of generality $0\leq r_1\leq r_2 \leq \cdots \leq r_k$. Clearly the estimate holds if $k=1$ or $r=1$. Hence we can now assume that $k\geq 2$ and $r\geq 2$. If  $r_1\leq \cdots \leq r_j \leq 1$ for some $j<k$ with $r_{j+1} \geq 2$, then we have
\begin{align*}
\norm*[\big][2]{f_1^{(r_1)}\cdots f_k^{(r_k)} } \leq C(K)\norm*[\big][2]{f_{j+1}^{(r_{j+1})} \cdots f_{k}^{(r_k)}}
\end{align*}
Hence without loss of generality we can assume that $r_1\geq 2$. As $k\geq 2$ this implies that $n\geq 4$ and we also have $r\geq 2$, $r\leq n-2$ and $s = n-2$. Hence using \lemref{lem:interpolation} we have
\begin{align*}
& \norm*[\big][2]{f_1^{(r_1)}\cdots f_k^{(r_k)} } \\
& \leq \norm*[\big][\infty]{f_1^{(r_1)}}\cdots \norm*[\big][\infty]{f_{k-1}^{(r_{k-1})}}\norm*[\big][2]{f_k^{(r_k)}} \\
& \lesssim \brac{\norm[2]{f_1''}^{\thvar_1} \cdots \norm[2]{f_{k-1}''}^{\thvar_{k-1}}}\brac{\norm[2]{f_1^{(s+1)}}^{1-\thvar_1} \cdots \norm[2]{f_{k-1}^{(s+1)}}^{1-\thvar_{k-1}}}\norm[2]{f_k''}^{\thvar_k}\norm[2]{f_k^{(s+1)}}^{1-\thvar_k}
\end{align*}
where $1-\thvar_j = \frac{r_j - \threebytwo}{s-1}$ for $j<k$ and $1-\thvar_k = \frac{r_k - 2}{s-1}$. Now observe that
\begin{align*}
(1-\thvar_1) + \cdots + (1-\thvar_k) = \frac{r_1 - \threebytwo}{s-1} + \cdots + \frac{r_{k-1} - \threebytwo}{s-1} + \frac{r_k - 2}{s-1} \leq 1
\end{align*}
Hence by using $AM-GM$ inequality the estimate follows. The proof of the second estimate is very similar and we skip it. 
\end{proof}
\begin{cor}\label{cor:commutatoreasy}
Let $f,g \in \Scalsp(\Rsp)$ and let $n\in \Nsp$ with $n\geq 2$. Then
\begin{enumerate}
\item $\norm[2]{(f\pap)^n g - f^n \pap^n g} \leq C(K)\cbrac{\norm[H^s]{f'} + \norm[H^s]{g'}}$ for $s = n-2$
\item $\norm[\Hhalf]{(f\pap)^n g - f^n \pap^n g} \leq C(K)\cbrac{\norm[H^s]{f'} + \norm[H^s]{g'}}$ for $s = n-\threebytwo$
\end{enumerate}
where $K = \norm[\infty]{f} + \norm[H^1]{f'} +  \norm[\infty]{g} + \norm[H^1]{g'}$ and $C(K)$ is a constant depending only on $K$.
\end{cor}
\begin{proof}
This follows directly from \lemref{lem:interpolationmult}
\end{proof}

\begin{prop}\label{prop:Hardy}
Let $f \in \Scalsp(\Rsp)$. Then we have
\begin{enumerate}

\item $\norm[\infty]{f} \lesssim \norm[H^s]{f}$ if $s>\frac{1}{2}$ and for $s=\half$ we have $\norm[BMO]{f} \lesssim \norm[\Hhalf]{f}$

\item $
\begin{aligned}[t]
\int \abs{\frac{f(\ap) - f(\bp)}{\ap - \bp}}^2 \diff \bp \lesssim \norm[2]{f'}^2 
\end{aligned}
$

\item $
\begin{aligned}[t]
\norm[\Ltwo(\Rsp, \diff \ap)]{\sup_{\bp}\abs{\frac{f(\ap) - f(\bp)}{\ap - \bp}}} \lesssim \norm[2]{f'}
\end{aligned}
$

\item $
\begin{aligned}[t]
\norm[\Hhalf]{f}^2 = \frac{1}{2\pi}\int\!\! \!\int \abs{\frac{f(\ap) - f(\bp)}{\ap - \bp}}^2 \diff \bp \diff\ap
\end{aligned}
$

\item $
\begin{aligned}[t]
\norm[\Ltwo(\Rsp^2, \diff\ap\diff\bp)]{\pbp\brac{\frac{f(\ap) - f(\bp)}{\ap - \bp}} } \lesssim \norm[\Hhalf]{f'}
\end{aligned}
$
\end{enumerate}
\end{prop}
\begin{proof}
\begin{enumerate}[leftmargin =*]
\item This is a standard Sobolev embedding result.
\item This is a consequence of Hardy's inequality. 
\item We see that
\begin{align*}
\sup_{\bp}\abs{\frac{f(\ap) - f(\bp)}{\ap - \bp}} \leq \sup_{\bp} \frac{\int_\ap^\bp \abs{f'(s)} \diff s}{\abs{\ap -\bp}} \leq M(f')(\ap)
\end{align*}
where $M$ is the uncentered Hardy Littlewood maximal operator. As the maximal operator is bounded on $\Ltwo$, the estimate follows.
\item Observe that as $\papabs = i\Hil\pap$ and $\Hil(1) =0$ we have
\begin{align*}
\norm[\Hhalf]{f}^2 & = -\frac{1}{\pi}\int \bar{f}(\ap) \pap\brac{\int \frac{f(\ap)-f(\bp)}{\ap-\bp}\diff \bp } \diff \ap \\
& = \frac{1}{\pi}\int \bar{f}(\ap) \brac{\int \frac{f(\ap)-f(\bp)}{(\ap-\bp)^2}\diff \bp } \diff \ap \\
& = \frac{1}{\pi}\int \!\!\! \int \abs{\frac{f(\ap)-f(\bp)}{\ap-\bp} }^2\diff \bp \diff\ap + \frac{1}{\pi}\int \!\!\! \int\frac{f(\ap)-f(\bp)}{(\ap-\bp)^2} \bar{f}(\bp) \diff \bp \diff\ap
\end{align*}
Now observe that 
\begin{align*}
 \frac{1}{\pi}\int \bar{f}(\ap) \brac{\int \frac{f(\ap)-f(\bp)}{(\ap-\bp)^2}\diff \bp } \diff \ap = -  \frac{1}{\pi}\int \!\!\! \int\frac{f(\ap)-f(\bp)}{(\ap-\bp)^2} \bar{f}(\bp) \diff \bp \diff\ap
\end{align*}
The identity now follows.
\item We see that
\begin{align*}
\pbp \brac{\frac{f(\ap)-f(\bp)}{(\ap-\bp)}} & = \frac{f(\ap)-f(\bp)}{(\ap-\bp)^2} - \frac{f'(\bp)}{\ap-\bp} \\
& = \int_0^1 \frac{f'(\bp + s(\ap-\bp))-f'(\bp)}{(\ap-\bp)}\diff s \\
& = \int_0^1 s\sqbrac{ \frac{f'(\bp + sl)-f'(\bp)}{sl}} \diff s \quad \tx{ using } \ap = \bp + l
\end{align*}
Hence we have
\begin{align*}
\norm[\Ltwo(\Rsp^2, \diff\ap\diff\bp)]{\pbp\brac{\frac{f(\ap) - f(\bp)}{\ap - \bp}} } & \lesssim \int_0^1 s\norm[\Ltwo(\Rsp^2, \diff\bp\diff l)]{ \frac{f'(\bp + sl)-f'(\bp)}{sl}} \diff s \\
& \lesssim \int_0^1 \sqrt{s} \norm[\Hhalf]{f'} \diff s \\
& \lesssim \norm[\Hhalf]{f'}
\end{align*}
\end{enumerate}
\end{proof}

\begin{prop}\label{prop:commutator}
Let $f,g \in \mathcal{S}(\Rsp)$ with $s,a\in \Rsp$ and $m,n \in \Zsp$. Then we have the following estimates
\begin{enumerate}
\item $\norm*[\big][2]{\papabs^s\sqbrac{f,\Hil}(\papabs^{a} g )} \lesssim  \norm*[\big][BMO]{\papabs^{s+a}f}\norm[2]{g}$ \quad  for $s,a \geq 0$

\item $\norm*[\big][2]{\papabs^s\sqbrac{f,\Hil}(\papabs^{a} g )} \lesssim  \norm*[\big][2]{\papabs^{s+a}f}\norm[BMO]{g}$ \quad  for $s\geq 0$ and $a>0$

\item $\norm*[\big][2]{\sqbrac*[\big]{f,\papabs^\half}g } \lesssim \norm*[\big][BMO]{\papabs^\half f}\norm[2]{g}$

\item $\norm*[\big][2]{\sqbrac*[\big]{f,\papabs^\half}(\papabs^\half g) } \lesssim \norm*[\big][BMO]{\papabs f}\norm[2]{g}$

\item $\norm[\Linfty\cap\Hhalf]{\pap^m\sqbrac{f,\Hil}\pap^n g} \lesssim \norm*[\big][2]{\pap^{(m+n+1)}f}\norm[2]{g}$ \quad  for $m,n \geq 0$

\item $\norm[2]{\partial_{\ap}^m\sqbrac{f,\Hil}\partial_{\ap}^n g} \lesssim \norm*[\infty]{\partial_\ap^{(m+n)} f}\norm[2]{g}$ \quad  for $m,n \geq 0$

\item $\norm[2]{\partial_{\ap}^m\sqbrac{f,\Hil}\partial_{\ap}^n g} \lesssim \norm*[2]{\partial_\ap^{(m+n)} f}\norm[\infty]{g}$ \quad for $m\geq 0$ and $n\geq 1$

\item $\norm[2]{\sqbrac{f,\Hil}g} \lesssim \norm[2]{f'}\norm[1]{g}$
\end{enumerate}
\end{prop}
\begin{proof}
The first four estimates are all variants of the Kato Ponce commutator estimate and are proved using the paraproduct decomposition. See Lemma 2.1 in \cite{HuIfTa16} for the first two estimates and Theorem 1.2 in \cite{Li19} for the third and fourth estimates. The fourth estimate is not explicitly stated as part of Theorem 1.2 in \cite{Li19} however the proof is identical to the proof of estimate 3 with the only change being at the last step where you move half a derivative from $g$ to $f$. 

The $\Hhalf$ estimate of the fifth estimate follows from the first estimate. For the $\Linfty$ estimate note that 
\begin{align*}
\pap^m\sqbrac{f,\Hil}\pap^n g &= \pap^m \int \frac{f(\ap) -f(\bp)}{\ap-\bp} \pbp^n g(\bp) \diff \bp \\
& = \pap^m \int \!\!\!\int_0^1 f'((1-s)\bp + s\ap)\pbp^n g(\bp) \diff s \diff \bp \\
& = (-1)^n\int_0^1 s^m (1-s)^n \brac{\int  f^{(m+n+1)}((1-s)\bp + s\ap) g(\bp)\diff \bp} \diff s
\end{align*}
The estimate now follows from the Cauchy Schwarz inequality. The sixth and seventh estimates follow from the first two estimates. For the last estimate observe that
\begin{align*}
\abs{\sqbrac{f,\Hil}g}(\ap) \lesssim  \int \abs{\frac{f(\ap) -f(\bp)}{\ap-\bp}} \abs{g(\bp)}^\half\abs{g(\bp)}^\half  \diff \bp \lesssim \brac{\int  \abs{\frac{f(\ap) -f(\bp)}{\ap-\bp}}^2 \abs{g(\bp)} \diff \bp}^\half \norm[1]{g}^\half
\end{align*}
The estimate now follows from Hardy's inequality as stated in \propref{prop:Hardy}. 
\end{proof}

\begin{prop}\label{prop:Leibniz}
Let $f,g,h \in \mathcal{S}(\Rsp)$ with $s,a\in \Rsp$ and $m,n \in \Zsp$. Then we have the following estimates
\begin{enumerate}
\item $\norm[2]{\papabs^s (fg)} \lesssim  \norm[2]{\papabs^s f}\norm[\infty]{g} + \norm[\infty]{f}\norm[2]{\papabs^s g}$ \quad for $s > 0$
\item $\norm[\Hhalf]{fg} \lesssim \norm[\Hhalf]{f}\norm[\infty]{g} + \norm[\infty]{f}\norm[\Hhalf]{g}$
\item $\norm[\Hhalf]{fg} \lesssim \norm[2]{f'}\norm[2]{g} + \norm[\infty]{f}\norm[\Hhalf]{g}$
\end{enumerate}
\end{prop}
\begin{proof}
See \cite{KaPo88} for the first estimate. The second one is a special case of the first. For the third one observe that
\begin{align*}
\papabs^\half (fg) = \sqbrac*{\papabs^\half,f}g + f\papabs^\half g
\end{align*}
and hence from \propref{prop:commutator}
\begin{align*}
\norm[\Hhalf]{ fg} \lesssim \norm*[BMO]{\papabs^\half f}\norm[2]{g} + \norm[\infty]{f}\norm[\Hhalf]{g}
\lesssim \norm[2]{f'}\norm[2]{g} + \norm[\infty]{f}\norm[\Hhalf]{g}
\end{align*}
\end{proof}

\begin{prop}\label{prop:triple}
Let $f,g,h \in \mathcal{S}(\Rsp)$ . Then we have the following estimates
\begin{enumerate}

\item $\norm[2]{\sqbrac{f,g;h}} \lesssim \norm[2]{f'}\norm[2]{g'}\norm[2]{h}$

\item $\norm[2]{\pap\sqbrac{f,\sqbrac{g,\Hil}}h} \lesssim \norm[2]{f'}\norm[2]{g'}\norm[2]{h}$

\item $\norm[2]{\sqbrac{f,g; h'}} \lesssim \norm[\infty]{f'}\norm[\infty]{g'}\norm[2]{h}$

\item $\norm[\Hhalf]{\sqbrac{f,g; h'}} \lesssim \norm[\infty]{f'}\norm[\infty]{g'}\norm[\Hhalf]{h}$

\item $\norm[\Linfty\cap\Hhalf]{\sqbrac{f,g;h}} \lesssim \norm[\infty]{f'}\norm[2]{g'}\norm[2]{h}$

\end{enumerate}
\end{prop}
\begin{proof}
\begin{enumerate}[leftmargin = *]
\item We see that
\begin{align*}
& \abs{\sqbrac{f,g;h}}(\ap) \\
& \lesssim  \int \abs{\frac{f(\ap) -f(\bp)}{\ap-\bp}} \abs{\frac{g(\ap) -g(\bp)}{\ap-\bp}}\abs{h(\bp)}\diff\bp \lesssim \norm[2]{f'} \brac{\int  \abs{\frac{g(\ap) -g(\bp)}{\ap-\bp}}^2\abs{h(\bp)}^2 \diff \bp  }^\half
\end{align*}
The estimate now follows from Hardy's inequality.

\item We see that
\begin{align*}
& \pap\sqbrac{f,\sqbrac{g,\Hil}}h \\
& = \pap\brac{f\sqbrac{g,\Hil}h - \sqbrac{g,\Hil}fh } \\
& = \frac{1}{i\pi}\pap\int \frac{(g(\ap) -g(\bp))(f(\ap)-f(\bp))}{\ap-\bp} h(\bp) \diff \bp \\
& = -\frac{1}{i\pi}\int \frac{g(\ap) -g(\bp)}{\ap-\bp}\frac{f(\ap) -f(\bp)}{\ap-\bp}h(\bp)\diff\bp + g'(\ap)\brac{\frac{1}{i\pi}\int \frac{f(\ap) -f(\bp)}{\ap-\bp} h(\bp) \diff\bp} \\
& \quad + f'(\ap) \brac{\frac{1}{i\pi}\int \frac{g(\ap) -g(\bp)}{\ap-\bp} h(\bp) \diff\bp}
\end{align*}
The estimate now follows by previous estimates.

\item This is a special case of  \propref{prop:Coifman}

\item  From the third estimate we observe that the operator $T$ defined by the action $h \mapsto \sqbrac{f,g;h'}$ is bounded on $\Ltwo$. Also we clearly see that $T(1) =0$. It is also easy to see that the kernel of this operator satisfies the conditions for \propref{prop:Lemarie}. Hence the operator $T$ is bounded on $\Hhalf$. 

\item The $\Linfty$ estimate is obtained easily by an application of Cauchy Schwarz and Hardy's inequality. Now we use $\norm[\Hhalf]{f} \lesssim \norm[\Ltwo(\Rsp\times\Rsp, \diff\ap\diff\bp)]{\frac{f(\ap)-f(\bp)}{\ap-\bp}}$ and see that
\begin{align*}
\frac{\sqbrac{f,g;h}(\ap) - \sqbrac{f,g;h}(\bp)}{\ap - \bp} & = \frac{1}{i\pi} \int \frac{h(s)}{\ap-\bp}\sqbrac{\frac{f(\ap)-f(s)}{\ap-s} - \frac{f(\bp)-f(s)}{\bp-s} }\frac{g(\ap)-g(s)}{\ap-s} \diff s \\
& \quad +  \frac{1}{i\pi} \int \frac{h(s)}{\ap-\bp}\sqbrac{\frac{g(\ap)-g(s)}{\ap-s} - \frac{g(\bp)-g(s)}{\bp-s} }\frac{f(\bp)-f(s)}{\bp-s} \diff s
\end{align*}
Now we use the following notation to simplify the calculation
\begin{align*}
F(a,b) = \frac{f(a) -f(b)}{a-b} \quad \tx{ and } G(a,b) = \frac{g(a)-g(b)}{a-b}
\end{align*}
Hence we have
\begin{align*}
\frac{\sqbrac{f,g;h}(\ap) - \sqbrac{f,g;h}(\bp)}{\ap - \bp}  & = -\frac{1}{i\pi}\int \frac{h(s)}{\bp-s}F(\ap,s)G(\ap,s) \diff s + \frac{1}{i\pi}\int \frac{h(s)}{\bp-s}F(\ap,\bp)G(\ap,s) \diff s \\
& \quad + \frac{1}{i\pi}\int \frac{h(s)}{\ap-s}F(\bp,s)G(\ap,\bp) \diff s - \frac{1}{i\pi}\int \frac{h(s)}{\ap-s}F(\bp,s)G(\bp,s) \diff s \\
& = - \Hil(F(\ap,\cdot)G(\ap,\cdot)h(\cdot))(\bp) + F(\ap,\bp)\Hil(G(\ap,\cdot)h(\cdot))(\bp) \\
& \quad + G(\ap,\bp)\Hil(F(\bp,\cdot)h(\cdot))(\ap) - \Hil(F(\bp,\cdot)G(\bp,\cdot)h(\cdot))(\ap)
\end{align*}
and we see that
\begin{align*}
\norm[\Ltwo(\Rsp\times\Rsp, \diff\ap\diff\bp)]{\Hil(F(\ap,\cdot)G(\ap,\cdot)h(\cdot))(\bp) } & \lesssim \norm[\Ltwo(\Rsp,\diff \ap)]{\norm[\Ltwo(\Rsp,\diff \bp)]{F(\ap,\bp)G(\ap,\bp)h(\bp) }} \\
& \lesssim \norm[\infty]{f'}\norm[2]{h}\norm[\Ltwo(\Rsp,\diff \ap)]{\norm[\Linfty(\Rsp,\diff \bp)]{G(\ap,\bp)} } \\
& \lesssim  \norm[\infty]{f'}\norm[2]{g'}\norm[2]{h}
\end{align*}
The other terms are handled similarly.
\end{enumerate}
\end{proof}

\begin{prop} \label{prop:LinftyHhalf}
Let $f \in \mathcal{S}(\Rsp)$ and let $w$ be a smooth non-zero weight with $w,\frac{1}{w} \in \Linfty(\Rsp) $ and $w' \in \Ltwo(\Rsp)$. Then 
\begin{enumerate}
\item $\norm[\infty]{f}^2 \lesssim \norm[2]{\frac{f}{w}}\norm[2]{wf'}$
\item $\norm[\Linfty\cap\Hhalf]{f}^2 \lesssim \norm[2]{\frac{f}{w}}\norm[2]{(wf)'} +  \norm[2]{\frac{f}{w}}^2\norm[2]{w'}^2$
\end{enumerate}
\end{prop}
\begin{proof}
1) We see that
\begin{align*}
\pap(f^2) = 2\brac{\frac{f}{w}}(wf')
\end{align*}
Now we integrate and use Cauchy Schwarz to get the estimate. 

\noindent 2) The $\Linfty$ estimate is obtained from the first estimate by observing that
\begin{align*}
\norm[\infty]{f}^2 \lesssim \norm[2]{\frac{f}{w}}\norm[2]{wf'} \lesssim \norm[2]{\frac{f}{w}}\norm[2]{(wf)'} + \norm[2]{\frac{f}{w}}\norm[2]{w'}\norm[\infty]{f}
\end{align*}
Now use the inequality $ab \leq \frac{a^2}{2\ep} + \frac{\ep b^2}{2}$ on the last term to obtain the estimate. For the $\Hhalf $ estimate, using $\papabs = i\Hil\pap$ we see that
\begin{align*}
\norm[\Hhalf]{f}^2 \lesssim  \abs{\int \brac{\frac{\bar{f}}{w}}\brac{w\Hil f'} \diff \ap } \lesssim \norm[2]{\frac{f}{w}}\norm[2]{w\Hil f'}
\end{align*}
Now as $w\Hil f' = \sqbrac{w,\Hil}f' + \Hil(wf')$ we have
\begin{align*}
\norm[\Hhalf]{f}^2 \lesssim \norm[2]{\frac{f}{w}}\brac{\norm[2]{w'}\norm[\infty]{f} + \norm[2]{wf'} } \lesssim \norm[2]{\frac{f}{w}}\norm[2]{w'}\norm[\infty]{f} + \norm[2]{\frac{f}{w}}\norm[2]{(wf)'}
\end{align*}
Hence using the inequality $ab \leq \frac{a^2}{2} + \frac{b^2}{2}$, we see that
\begin{align*}
\norm[\Hhalf]{f}^2  \lesssim  \norm[2]{\frac{f}{w}}\norm[2]{(wf)'} +  \norm[2]{\frac{f}{w}}^2\norm[2]{w'}^2 + \norm[\infty]{f}^2  \lesssim \norm[2]{\frac{f}{w}}\norm[2]{(wf)'} +  \norm[2]{\frac{f}{w}}^2\norm[2]{w'}^2
\end{align*}
\end{proof}

\begin{prop}\label{prop:Hhalfweight}
Let $f,g \in \mathcal{S}(\Rsp)$ and let $w,h \in \Linfty(\Rsp)$ be smooth functions with $w',h' \in \Ltwo(\Rsp)$. Then 
\begin{align*}
\norm[\Hhalf]{fwh} \lesssim \norm[\Hhalf]{fw}\norm[\infty]{h} + \norm[2]{f}\norm[2]{(wh)'} + \norm[2]{f}\norm[2]{w'}\norm[\infty]{h}
\end{align*}
If in addition we assume that $w$ is real valued then
\begin{align*}
\norm[2]{fgw} \lesssim \norm[\Hhalf]{fw}\norm[2]{g} + \norm[\Hhalf]{gw}\norm[2]{f} + \norm[2]{f}\norm[2]{g}\norm[2]{w'} 
\end{align*}
\end{prop}
\begin{proof}
1) We see that
\begin{align*}
\papabs^\half (fwh) =  \sqbrac*{\papabs^\half, h}fw + h\papabs^\half (fw) = \sqbrac*{\papabs^\half, hw}f + h\sqbrac*{\papabs^\half, w}f + h\papabs^\half (fw)
\end{align*}
The estimate now follows from the estimate $\norm*[2]{\sqbrac*{\papabs^\half, g}f} \lesssim \norm*[BMO]{\papabs^\half g}\norm[2]{f} \lesssim \norm[2]{g'}\norm[2]{f}$

\noindent 2) We observe that
\begin{align*}
fgw & = (\Ph f)(\Ph g)w + (\Ph f)(\Pa g)w + (\Pa f)(\Ph g)w + (\Pa f)(\Pa g)w \\
& = (\Ph f)\overline{(\Pa \bar{g})}w + (\Ph f)(\Pa g)w + (\Pa f)(\Ph g)w + (\Pa f)\overline{(\Ph \bar{g})}w
\end{align*}
We will control only the first term and the other terms are controlled similarly. Now see that
\begin{align*}
\norm[2]{(\Ph f)\overline{(\Pa \bar{g})}w} = \norm[2]{(\Ph f)(\Pa \bar{g})w} 
\end{align*}
Hence we have
\begin{align*}
2(\Ph f)(\Pa \bar{g})w & = (\Id - \Hil)\cbrac{(\Ph f)(\Pa \bar{g})w } + (\Id + \Hil)\cbrac{(\Ph f)(\Pa \bar{g})w } \\
& = \sqbrac{w\Pa \bar{g} ,\Hil}\Ph f - \sqbrac{w\Ph f ,\Hil}\Pa \bar{g}
\end{align*}
Now observe that as $w $ is real valued we have
\begin{align*}
\norm[2]{ \sqbrac{w\Pa \bar{g} ,\Hil}\Ph f} \lesssim \norm[\Hhalf]{ \sqbrac{w\Pa \bar{g}}}\norm[2]{\Ph f} & \lesssim \brac{\norm[\Hhalf]{\sqbrac{w,\Hil}\bar{g} } + \norm[\Hhalf]{w\bar{g}} }\norm[2]{f} \\
& \lesssim \norm[2]{w'}\norm[2]{g}\norm[2]{f} + \norm[\Hhalf]{wg}\norm[2]{f}
\end{align*}
Similarly we have
\begin{align*}
\norm[2]{\sqbrac{w\Ph f ,\Hil}\Pa \bar{g}} \lesssim \norm[\Hhalf]{w\Ph f}\norm[2]{\Pa \bar{g}} & \lesssim \brac{\norm[\Hhalf]{\sqbrac{w,\Hil}f} + \norm[\Hhalf]{wf} }\norm[2]{g} \\
& \lesssim \norm[2]{w'}\norm[2]{f}\norm[2]{g} + \norm[\Hhalf]{wf}\norm[2]{g}
\end{align*}
\end{proof}

\begin{prop} \label{prop:DtLinfty}
Let $f  \in  C^3([0,T), H^3(\Rsp))$. Then for any $t\in [0,T)$ we have
\begin{align*}
\limsup_{s \to 0^+} \frac{\norm[\infty]{f(\cdot,t+s)} - \norm[\infty]{f(\cdot,t)}}{s} \leq \norm[\infty]{\pt f(\cdot,t)}
\end{align*}
\end{prop}
\begin{proof}
Fix $s > 0$ satisfying $t+s \in [0,T)$ and for every $\epsilon >0$ we find $a_{\epsilon} \in \Rsp$ such that $\norm[\infty]{f(\cdot,t+s)} \leq \abs{f}(a_{\epsilon},t+s) + \epsilon$. Observe that $\abs{f}(a_{\epsilon}, t) \leq \norm[\infty]{f(\cdot,t)}$ and hence we have
\begin{align*}
\norm[\infty]{f(\cdot,t+s)} - \norm[\infty]{f(\cdot,t)} & \leq  \abs{f}(a_{\epsilon},t+s) - \abs{f}(a_{\epsilon},t) + \epsilon \\
& \leq  \abs{f(a_{\epsilon},t+s) - f(a_{\epsilon},t)} + \epsilon \\
& \leq \sup_{\substack{\ap \in \Rsp \\ u \in (0,s)}}  \abs{\partial_t f(\ap,t+u)}s + \epsilon
\end{align*}
Now let $\epsilon \to 0$ to get 
\[
 \frac{\norm[\infty]{f(\cdot,t+s)} - \norm[\infty]{f(\cdot,t)}}{s} \leq \sup_{u \in (0,s)}\norm[\infty]{\pt f(\cdot,t+u)}
\]
As $\pt^2 f \in \Linfty(\Rsp\times[0,T))$, we take the limit as $s\to 0$ to finish the proof.
\end{proof}
\begin{lemma}\label{lem:conv}
Let $K_\ep$ be the Poisson kernel from \eqref{eq:Poissonkernel}. If $\f \in L^q(\Rsp)$, then for $s\geq 0$ an integer we have
\begin{align*}
\norm[p]{(\pap^sf)\conv P_\ep} \lesssim \norm[q]{f}\ep^{-s-\brac{\frac{1}{q} - \frac{1}{p}}} \quad \tx{ for } 1\leq q\leq p \leq \infty
\end{align*}
Similarly for $s \in\Rsp, s\geq 0$ we have
\begin{align*}
\norm[p]{(\papabs^s f)\conv P_\ep} \lesssim \norm[q]{f}\ep^{-s-\brac{\frac{1}{q} - \frac{1}{p}}} \quad \tx{ for } 1\leq q\leq p \leq \infty
\end{align*}
\end{lemma}
\begin{proof}
The proof follows from basic properties of convolution. 
\end{proof}


\bibliographystyle{amsplain}
\bibliography{Main.bib}

\end{document}

%% file: definitions.tex
\newtheorem{thm}{Theorem}[section]

\newtheorem{lem}[thm]{Lemma}
\newtheorem{lemma}[thm]{Lemma}
\newtheorem{prop}[thm]{Proposition}
\newtheorem{cor}[thm]{Corollary}

\theoremstyle{definition}

\theoremstyle{remark}
\newtheorem{rem}[thm]{Remark}
\newtheorem{rmk}[thm]{Remark}

\newcommand{\thmref}[1]{Theorem~\ref{#1}}
\newcommand{\corref}[1]{Corollary~\ref{#1}}
\newcommand{\secref}[1]{\S\ref{#1}}
\newcommand{\lemref}[1]{Lemma~\ref{#1}} 
\newcommand{\propref}[1]{Proposition~\ref{#1}}
\newcommand{\remref}[1]{Remark~\ref{#1}}

\newcommand*{\dis}{\displaystyle}

\newcommand*{\qq}{\qquad}
\newcommand*{\lpar}{ }

\newcommand*{\tx}[1]{\text{#1}}

\newcommand*{\lamb}{\lambda}
\newcommand{\del}{\delta}

\newcommand*{\ep}{\epsilon}
\newcommand*{\kap}{\kappa}
\newcommand*{\degree}{^{\circ}}
\newcommand*{\suchthat}{\, \middle| \,}

\newcommand*{\n}{\!\!}

\newcommand*{\Pminus}{P_{-}}
\newcommand*{\Pminusbar}{\myoverline{-3}{0}{P}_{-}}

\newcommand*{\myoverline}[3]{\mkern -#1mu\overline{\mkern#1mu#3\mkern#2mu}\mkern -#2mu}	
\newcommand*{\vboldbar}{\myoverline{0}{0}{\vbold}}
\newcommand*{\Zbar}{\myoverline{-3}{0}{\Z}}
\newcommand*{\zbar}{\myoverline{-2}{0}{\z}}
\newcommand*{\Thbar}{\myoverline{-1}{-1.2}{\Th}}
\newcommand*{\Dapbar}{\myoverline{-3}{-1}{D}_\ap}
\newcommand*{\wbar}{\myoverline{0}{-0.5}{\omega}}

\newcommand*{\dvarbar}{\myoverline{0}{0}{\dvar}}

\newcommand*{\Mconst}{M}

\newcommand*{\half}{\frac{1}{2}}
\newcommand*{\threebytwo}{\frac{3}{2}}
\newcommand*{\fivebytwo}{\frac{5}{2}}

\newcommand*{\onebythree}{\frac{1}{3}}
\newcommand*{\twobythree}{\frac{2}{3}}
\newcommand*{\fourbythree}{\frac{4}{3}}

\newcommand*{\onebyfour}{\frac{1}{4}}
\newcommand*{\threebyfour}{\frac{3}{4}}

\newcommand*{\onebysix}{\frac{1}{6}}
\newcommand*{\fivebysix}{\frac{5}{6}}

\newcommand{\eightbyseven}{\frac{8}{7}}

\newcommand{\sevenbytwelve}{\frac{7}{12}}

\newcommand*{\Rsp}{\mathbb{R}}
\newcommand*{\Csp}{\mathbb{C}}
\newcommand*{\Nsp}{\mathbb{N}}
\newcommand*{\Zsp}{\mathbb{Z}}

\newcommand*{\Scalsp}{\mathcal{S}}
\newcommand*{\Dcalsp}{\mathcal{D}}

\newcommand*{\Lone}{L^1}
\newcommand*{\Ltwo}{L^2}
\newcommand*{\Lfourbythree}{L^\fourbythree}
\newcommand{\Leightbyseven}{L^\eightbyseven}

\newcommand*{\Linfty}{L^{\infty}}
\newcommand*{\Hhalf}{\dot{H}^\half}
\newcommand*{\Ccal}{\mathcal{C}}
\newcommand*{\Wcal}{\mathcal{W}}

\newcommand{\approxLtwo}{\approx_{\Ltwo}}

\newcommand*{\al}{\alpha}
\newcommand*{\ap}{{\alpha'}}

\newcommand*{\bp}{{\beta'}}

\newcommand*{\xp}{{x'}}
\newcommand*{\yp}{{y'}}
\newcommand*{\zp}{{z'}}

\newcommand*{\diff}{\mathop{}\! d}
\newcommand*{\difff}{\mathop{}\!\! d}
\newcommand*{\compose}[1]{\circ{#1}}

\newcommand*{\conv}{*}
\newcommand*{\Hil}{\mathbb{H}}

\newcommand*{\Pa}{\mathbb{P}_A}
\newcommand*{\Ph}{\mathbb{P}_H}
\newcommand*{\Phol}{\mathbb{P}_{hol}}
\newcommand*{\Id}{\mathbb{I}}
\newcommand*{\Imag}{\tx{Im}}
\newcommand*{\Real}{\tx{Re}}
\newcommand{\Jdel}{J_\delta}
\newcommand{\Jdelone}{J_{\del_1}}
\newcommand{\Jdeltwo}{J_{\del_2}}

\newcommand*{\grad}{\nabla}
\newcommand*{\Dt}{D_t}
\newcommand*{\Dtone}{\Dt^1}
\newcommand*{\Dttwo}{\Dt^2}
\newcommand{\Dtdel}{\Dt^\del}
\newcommand{\Dtdelone}{\Dt^{\del_1}}
\newcommand{\Dtdeltwo}{\Dt^{\del_2}}

\newcommand*{\pt}{\partial_t}
\newcommand*{\ps}{\partial_s}
\newcommand*{\px}{\partial_x}
\newcommand*{\py}{\partial_y}
\newcommand*{\pz}{\partial_z}
\newcommand*{\pxp}{\partial_\xp}
\newcommand*{\pyp}{\partial_\yp}
\newcommand*{\pzp}{\partial_\zp}
\newcommand*{\pap}{\partial_\ap}
\newcommand*{\papabs}{\abs{\pap}}
\newcommand*{\pbp}{\partial_\bp}
\newcommand*{\pal}{\partial_\al}

\newcommand*{\Dap}{D_{\ap}}

\newcommand*{\Dapabs}{\abs{D_{\ap}}}
\newcommand*{\Dapfrac}{\frac{1}{\Zap}\pap}
\newcommand*{\Dapbarfrac}{\frac{1}{\Zapbar}\pap}

\newcommand*{\Dapabsfrac}{\frac{1}{\Zapabs}\pap}

\newcommand*{\Ecal}{\mathcal{E}}

\newcommand*{\Ecalthreefive}{\mathcal{E}_{3.5}}
\newcommand*{\Ecalfourfivei}{\mathcal{E}_{4.5 + i}}
\newcommand*{\EcalfourfiveN}{\mathcal{E}_{4.5 + N}}

\newcommand*{\Esigmazero}{E_{\sigma,0}}
\newcommand*{\Esigmaone}{E_{\sigma,1}}
\newcommand*{\Esigmatwo}{E_{\sigma,2}}
\newcommand*{\Esigmathree}{E_{\sigma,3}}
\newcommand*{\Esigmafour}{E_{\sigma,4}}
\newcommand*{\Esigma}{E_{\sigma}}
\newcommand*{\Ecalsigmaone}{\mathcal{E}_{\sigma,1}}
\newcommand*{\Ecalsigmatwo}{\mathcal{E}_{\sigma,2}}
\newcommand*{\Ecalsigma}{\mathcal{E}_{\sigma}}

\newcommand*{\EcalDelta}{\mathcal{E}_{\Delta}}

\newcommand*{\EcalDeltatwo}{\mathcal{E}_{\Delta,2}}
\newcommand*{\EcalDeltatwofive}{\mathcal{E}_{\Delta,2.5}}
\newcommand*{\EcalDeltathree}{\mathcal{E}_{\Delta,3}}

\newcommand*{\vbold}{\mathbf{v}}

\newcommand*{\A}{\mathcal{A}}
\newcommand*{\Aone}{A_1}
\newcommand*{\Aonesigma}{A_{1,\sigma}}
\newcommand*{\Aonestar}{\Aone^*}

\newcommand*{\bvar}{b}
\newcommand*{\bvarone}{b_1}
\newcommand*{\bvartwo}{b_2}

\newcommand*{\bap}{\bvar_\ap}
\newcommand*{\bvarap}{\bap}
\newcommand*{\bstar}{b^*}

\newcommand*{\avar}{a}
\newcommand{\avarone}{\avar_1}
\newcommand{\avartwo}{\avar_2}

\newcommand*{\atay}{\mathfrak{\underline{a}}}

\newcommand*{\h}{h}

\newcommand*{\hal}{\h_\al}

\newcommand*{\hinv}{\h^{-1}}

\newcommand*{\g}{g}
\newcommand*{\gvar}{g}

\newcommand*{\gone}{\g_1}
\newcommand*{\gtwo}{\g_2}

\newcommand*{\gvartwo}{\gvar_2}

\newcommand*{\vvar}{v}
\newcommand*{\vvarone}{v_1}
\newcommand*{\vvartwo}{v_2}

\newcommand*{\thvar}{\theta}
\newcommand*{\Th}{\Theta}

\newcommand*{\cvar}{c}
\newcommand{\cvarone}{\cvar_1}
\newcommand{\cvartwo}{\cvar_2}

\newcommand*{\dvar}{d}
\newcommand*{\dvarone}{d_1}
\newcommand*{\dvartwo}{d_2}

\newcommand*{\etwo}{e_2}

\newcommand{\errorone}{error_1}
\newcommand{\errortwo}{error_2}

\newcommand*{\f}{f}

\newcommand*{\w}{\omega}
\newcommand*{\wone}{\omega_1}
\newcommand*{\wtwo}{\omega_2}

\newcommand*{\Psizp}{\Psi_{\zp}}

\newcommand*{\Jone}{J_1}
\newcommand*{\Jtwo}{J_2}

\newcommand*{\Rone}{R_1}
\newcommand*{\Rtwo}{R_2}

\newcommand{\U}{U}

\newcommand*{\z}{z}

\newcommand*{\zal}{\z_\al}

\newcommand*{\zalbar}{\zbar_\al}

\newcommand*{\zalabs}{\abs{\zal}}

\newcommand*{\zt}{\z_t}

\newcommand*{\ztbar}{\zbar_t}
\newcommand*{\ztal}{\z_{t\al}}

\newcommand*{\ztalbar}{\zbar_{t\al}}

\newcommand*{\ztt}{\z_{tt}}

\newcommand*{\zttbar}{\zbar_{tt}}
\newcommand*{\zttal}{\z_{tt\al}}

\newcommand*{\Z}{Z}

\newcommand*{\Zap}{\Z_{,\ap}}

\newcommand*{\Zaphalf}{\Zap^{1/2}}
\newcommand*{\Zapbar}{\Zbar_{,\ap}}

\newcommand*{\Zbarap}{\Zapbar}
\newcommand*{\Zapabs}{\abs{\Zap}}

\newcommand*{\Zt}{\Z_t}

\newcommand*{\Ztbar}{\Zbar_t}

\newcommand*{\Ztap}{\Z_{t,\ap}}

\newcommand*{\Ztapbar}{\Zbar_{t,\ap}}

\newcommand*{\Ztbarap}{\Ztapbar}

\newcommand*{\Ztt}{\Z_{tt}}

\newcommand*{\Zttbar}{\Zbar_{tt}}

\newcommand*{\Zttap}{\Z_{tt,\ap}}
\newcommand*{\Zttapbar}{\Zbar_{tt,\ap}}
\newcommand*{\Zttbarap}{\Zttapbar}

\newcommand*{\Ztttbar}{\Zbar_{ttt}}

\newcommand*{\nobrac}[1]{ #1 }

\DeclarePairedDelimiter{\oldbrac}{\lparen}{\rparen}			
\NewDocumentCommand{\brac}{ s o m }{						
	\IfBooleanT{#1}{
  		\IfValueT{#2}{\oldbrac[#2]{#3}}
		\IfValueF{#2}{\oldbrac{#3}} 
	}
	\IfBooleanF{#1}{
  		\IfValueT{#2}{\PackageError{mypackage}{Incorrect use of brac. Insert star}{}}
		\IfValueF{#2}{\oldbrac*{#3}} 
	}		
}

\DeclarePairedDelimiter\oldcbrac{\lbrace}{\rbrace}				
\NewDocumentCommand{\cbrac}{ s o m }{					
	\IfBooleanT{#1}{
  		\IfValueT{#2}{\oldcbrac[#2]{#3}}
		\IfValueF{#2}{\oldcbrac{#3}} 
	}
	\IfBooleanF{#1}{
  		\IfValueT{#2}{\PackageError{mypackage}{Incorrect use of cbrac. Insert star}{}}
		\IfValueF{#2}{\oldcbrac*{#3}} 
	}		
}

\DeclarePairedDelimiter\oldsqbrac{\lbrack}{\rbrack}				
\NewDocumentCommand{\sqbrac}{ s o m }{					
	\IfBooleanT{#1}{
  		\IfValueT{#2}{\oldsqbrac[#2]{#3}}
		\IfValueF{#2}{\oldsqbrac{#3}} 
	}
	\IfBooleanF{#1}{
  		\IfValueT{#2}{\PackageError{mypackage}{Incorrect use of sqbrac. Insert star}{}}
		\IfValueF{#2}{\oldsqbrac*{#3}} 
	}		
}

\DeclarePairedDelimiter{\oldabs}{\lvert}{\rvert}
\NewDocumentCommand{\abs}{ s o m }{						
	\IfBooleanT{#1}{
  		\IfValueT{#2}{\oldabs[#2]{#3}}
		\IfValueF{#2}{\oldabs{#3}} 
	}
	\IfBooleanF{#1}{
  		\IfValueT{#2}{\PackageError{mypackage}{Incorrect use of abs. Insert star}{}}
		\IfValueF{#2}{\oldabs*{#3}} 
	}		
}

\DeclarePairedDelimiterX{\oldnorm}[1]{\lVert}{\rVert}{#1}
\NewDocumentCommand{\norm}{ s o o m }{					
	\IfValueT{#2} {
		\IfBooleanT{#1}{
  			\IfValueT{#3}{\oldnorm[#2]{#4}_{#3}}
			\IfValueF{#3}{\oldnorm{#4}_{#2}} 
		}
		\IfBooleanF{#1}{
  			\IfValueT{#3}{\PackageError{mypackage}{Incorrect use of norm. Insert star}{}}
			\IfValueF{#3}{\oldnorm*{#4}_{#2}} 
		}
	}
	\IfValueF{#2} {
		\IfBooleanT{#1}{\oldnorm{#4}}	
		\IfBooleanF{#1}{\oldnorm*{#4}}		
	}	
}